\documentclass[12pt,psamsfonts]{article}

\usepackage{amsmath}
\usepackage{amsthm}
\usepackage{amssymb}
\usepackage[all]{xy}
\usepackage{comment}
\usepackage{graphics}
\usepackage{hyperref}
\usepackage{graphicx,tikz}

\newcommand{\di}{\displaystyle}
\newcommand{\CC}{\mathbb C}
\newcommand{\FF}{\mathbb F}

\newcommand{\PP}{\mathbb P}
\newcommand{\QQ}{\mathbb Q}
\newcommand{\RR}{\mathbb R}
\newcommand{\ZZ}{\mathbb Z}

\newcommand{\DDD}{\mathbb D}
\newcommand{\SSS}{\mathbb S}
\newcommand{\TT}{\mathbb T}

\theoremstyle{plain}

\newtheorem{theorem}{Theorem}[section]
\newtheorem{lemma}[theorem]{Lemma}
\newtheorem{proposition}[theorem]{Proposition}
\newtheorem{corollary}[theorem]{Corollary}

\theoremstyle{remark}
\newtheorem{remark}[theorem]{Remark}
\newtheorem{example}[theorem]{Example}

\theoremstyle{definition}
\newtheorem{definition}[theorem]{Definition}

\begin{document}

\title{Signatures in algebra, topology and dynamics}
\date{\today}
\author{\'Etienne Ghys, Andrew Ranicki}

\maketitle

\begin{center}
\includegraphics[width=.7\textwidth]{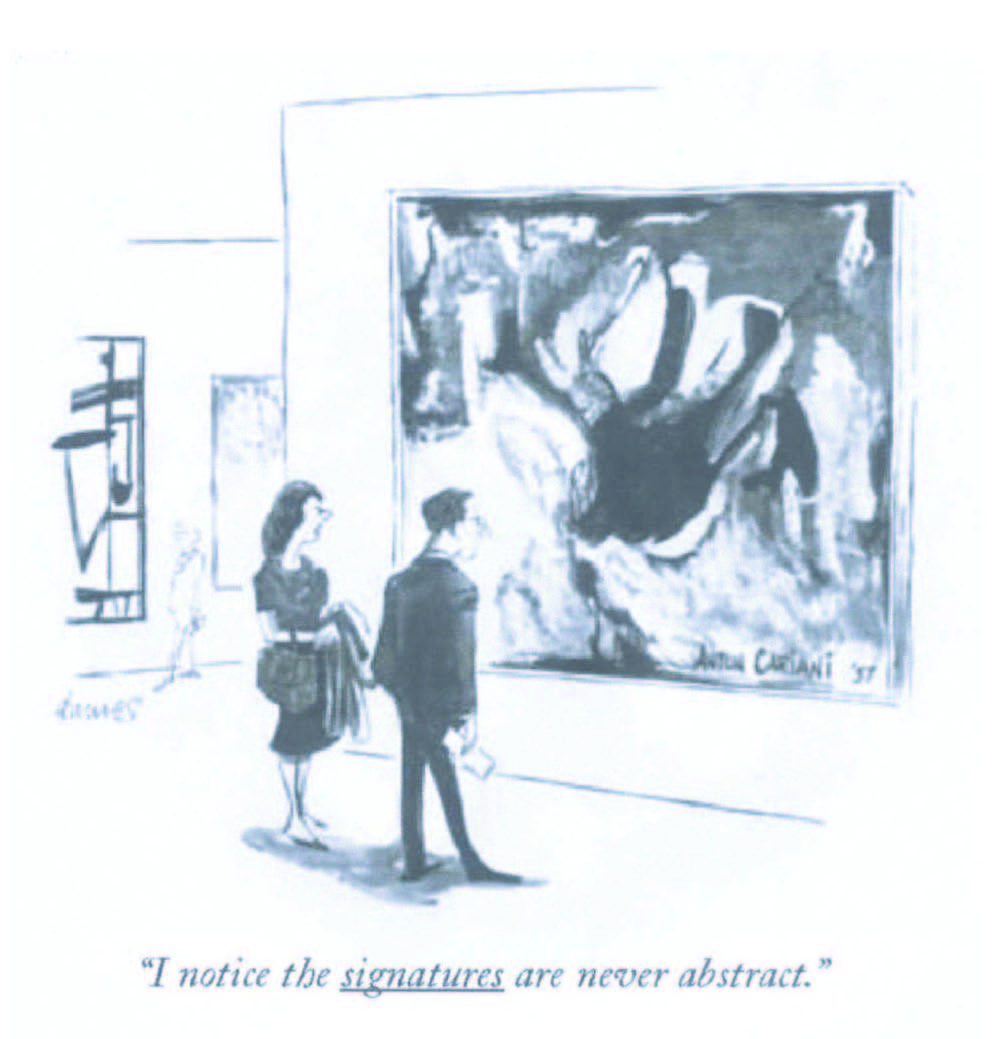}
\end{center}

\bigskip

\begin{minipage}[b]{0.45\linewidth}
{\footnotesize
E.G

UMPA UMR 5669 CNRS

ENS Lyon

Site Monod

46 All\'ee d'Italie

69364 Lyon

France

etienne.ghys@ens-lyon.fr
}
\end{minipage}
\hfill
\begin{minipage}[b]{0.45\linewidth}

{\footnotesize

A.R.

School of Mathematics,

University of Edinburgh

James Clerk Maxwell Building

Peter Guthrie Tait Road

Edinburgh EH9 3FD

Scotland, UK

a.ranicki@ed.ac.uk}
\end{minipage}

\newpage

\tableofcontents

\newpage

\section*{Introduction}

There is no doubt that quadratic forms are among the most basic and important objects in mathematics.
The distinction between an ellipse, a parabola and a hyperbola belongs to the standard curriculum of undergraduate students in mathematics,
and Sylvester's 1852 {\it Law of Inertia} for quadratic forms represents the formalization of this fact.
In a suitable basis, any quadratic form in $\RR^n$ can be written as
$$x_1^2+x_2^2+ \dots +x_p^2 -x_{p+1}^2 - \dots - x_{p+q}^2$$
for some integers $0\leqslant p,q \leqslant n$ with $p+q\leqslant n$.
These integers are uniquely determined by the quadratic form, and are the same for two forms \emph{if and only if} the forms are linearly congruent.
The sum $p+q$ is the {\it rank} and the difference $p-q$ is the {\it signature}
of the quadratic form.
In other words, Sylvester's Law of Inertia asserts that for any $n \geqslant   1$ the number of linear congruence classes of quadratic forms in $\RR^n$ is the number of ordered pairs  $(p,q)$ with $0 \leqslant p,q \leqslant n$, $p+q \leqslant n$ (which is equal to $(n+1)(n+2)/2$).

It is not a surprise that from these humble beginnings  the theory of quadratic forms has percolated to the full body of mathematics, in many different guises.
It would require a large encyclopedia to keep track of all these appearances, and it is certainly not our goal to provide such a compendium here.

\medskip

In this survey paper, we would like to present some {\it morceaux choisis}  with the main purpose of serving as a general introduction to the other papers in these volumes, which deal with the signatures of braids.
As the title suggests, we shall concentrate on algebra, topology, and dynamics.
Even though all the material that we describe is standard, it is usually scattered in the literature.
For instance, books and papers dealing with Hamiltonian dynamics usually do not discuss Wall non additivity of signatures.

\medskip

{\it Section 1} contains very standard facts about quadratic forms in $\RR^n$.
We tried to give some historical flavour to our presentation.
A common thread will be Sturm's theorem, counting real roots of polynomials.
From its original formulation as a number of sign variations, we shall describe its reformulation by Sylvester as the signature of a symmetric matrix, and then as an equality in a suitable Witt group.

{\it Section 2} is topological. The homology of a manifold is equipped with intersection forms.
Therefore signatures came into play, around 1940, and serve as invariants of manifolds.

{\it Section 3} deals with the Maslov class. One could easily write thick books on this topic (that probably nobody would read)
and we had to limit ourselves to the basics. It is amazing how ubiquitous this class can be and it is sometimes difficult to recognize it.
There are connections with mathematical physics, topology, number theory, dynamics, geometry, bounded cohomology etc.

\medskip

In the final sections, we choose three typical applications, among many other possibilities:

{\it Section 4} presents some aspects of dynamical systems, specifically hamiltonian, for which signatures are relevant.

{\it Section 5} goes deeper into some topological aspects. In particular, we describe the link between the Wall non additivity and the Maslov class.

{\it Section 6} describes some fascinating connections with number theory.

\medskip

Finally, {\it Section 7}~~is an appendix on the algebraic $L$-theory used in surgery obstruction theory, which generalizes some of the algebraic techniques presented here to forms over (noncommutative) rings with involution, in particular localization.

\medskip

We thank the {\it American Institute of Mathematics} in Palo Alto, at which we organized a meeting on the ``Many facets of the Maslov index'' in April 2014.

\medskip

Here are our signatures, as a symbolic renewal of the {\it Auld Alliance} between Scotland and France in 1295.

\medskip

\begin{center}
\includegraphics[width=.8\textwidth]{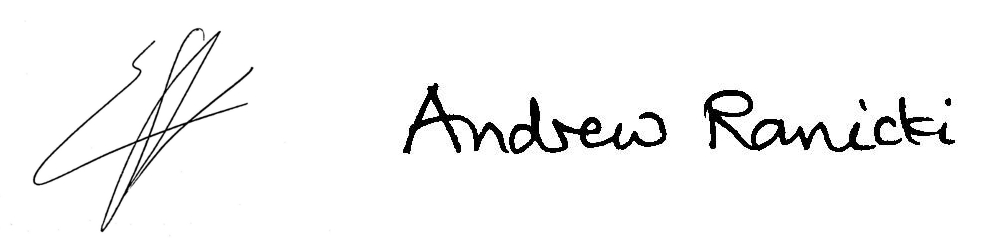}
\end{center}

\newpage

\section{Algebra}\label{algebra}

\subsection{The basic definitions}

Let us begin with some very basic definitions from linear algebra.
In this section we shall be mainly considering matrices over $\RR$.

An $m \times n$ {\it matrix} $S=(s_{ij})$ with entries $s_{ij} \in \RR$ ($1 \leqslant i \leqslant m$, $1 \leqslant j \leqslant n$) corresponds to a bilinear form (or pairing) of real vector spaces:
$$
S: ((x_1,x_2,\dots,x_m),(y_1,y_2,\dots,y_n))  \in \RR^m \times \RR^n  \mapsto \sum
\limits^m_{i=1}\sum\limits^n_{j=1} s_{ij}x_iy_j \in \RR.
$$

The {\it transpose} of an $m \times n$ matrix $S=(s_{ij})$ is the $n \times m$ matrix $S^*=(s^*_{ji})$ with $s^*_{ji}=s_{ij}.$

An $n \times n$ matrix $S$ is {\it symmetric} if $S^*=S$.

{\it Quadratic forms} correspond to symmetric matrices $S$. For $v = (x_1, \dots, x_n)  \in \RR^n$, we set
$$Q(v) =S(v,v)/2,\, {\rm with } \,\, S(v_1,v_2)=Q(v_1+v_2)-Q(v_1)-Q(v_2).$$

A quadratic form $Q(v)$ (or the corresponding symmetric matrix $S$) is {\it  positive definite} (resp. {\it negative definite}) if $Q(v) >0$ (resp. $Q(v)<0$) for all non zero $v\in \RR^n$.

Two symmetric $n \times n$ matrices $S,T$ are {\it linearly  congruent} if  $T=A^*SA$ for an invertible $n \times n$ matrix $A$.
A linear congruence can be regarded as an isomorphism $A:\RR^n \to \RR^n$ of real vector spaces such that
$$T(v_1,v_2)=S(Av_1,Av_2) \in \RR.$$

The {\it spectrum} of an $n\times n$ matrix $S$ is the set of eigenvalues, i.e. the roots $\lambda \in \CC$ of the characteristic polynomial $\det(X I_n-S) \in \RR[X]$.

In general, linearly congruent symmetric matrices $S,T$ do not have the same characteristic polynomial and spectrum.
However, it is significant (and by no means obvious) that symmetric matrices have real eigenvalues, and that the eigenvalues of linearly congruent symmetric matrices have the same signs.

Two $n \times n$ matrices $S,T$ are {\it conjugate} if $T=A^{-1}SA$ for an invertible $n \times n$ matrix $A$, in which case $S,T$ have the same characteristic polynomial $\det(X I_n-T)=\det(X I_n-S)$ and hence the same eigenvalues.
An $n \times n$ matrix $A$ is {\it orthogonal} if it is invertible and  $A^{-1}=A^*$.
This is equivalent to the preservation of the standard quadratic form $x_1^2+\dots+x_n^2$ in $\RR^n$,
i.e. to being an automorphism $A:(\RR^n,\langle~,~\rangle) \to (\RR^n,\langle~,~\rangle)$
of the standard symmetric form
$$\langle (x_1,\dots,x_n),(y_1,\dots,y_n)\rangle=\sum\limits^n_{i=1}x_iy_i \in \RR$$
corresponding to the identity $n \times n$  matrix $I_n$.
(However, note that the quadratic form associated to the standard symmetric form is
$Q(x_1,\dots,x_n)=(x_1^2+\dots+x_n^2)/2$.)
Two $n \times n$ matrices $S,T$ are {\it orthogonally congruent} if $T=A^*SA$ for an orthogonal $n \times n$ matrix $A$, in which case they are both linearly congruent and conjugate.

\subsection{Linear congruence}\label{linear}

The first basic theorem is that any quadratic form $Q(x_1, \dots, x_n)$ with real coefficients can be written as a sum, or difference, of squares of linear forms.
We refer to Bourbaki~\cite[Chapter 9]{bourbaki} for a description of the complex history of this theorem, with deep roots in the traditional geometry of conic sections and quadrics.
Usually, the theorem is attributed to Lagrange who stated it and proved it in full generality in 1759~\cite{lagrange1759}.

\bigskip
\begin{center}
\includegraphics[width=.4 \textwidth]{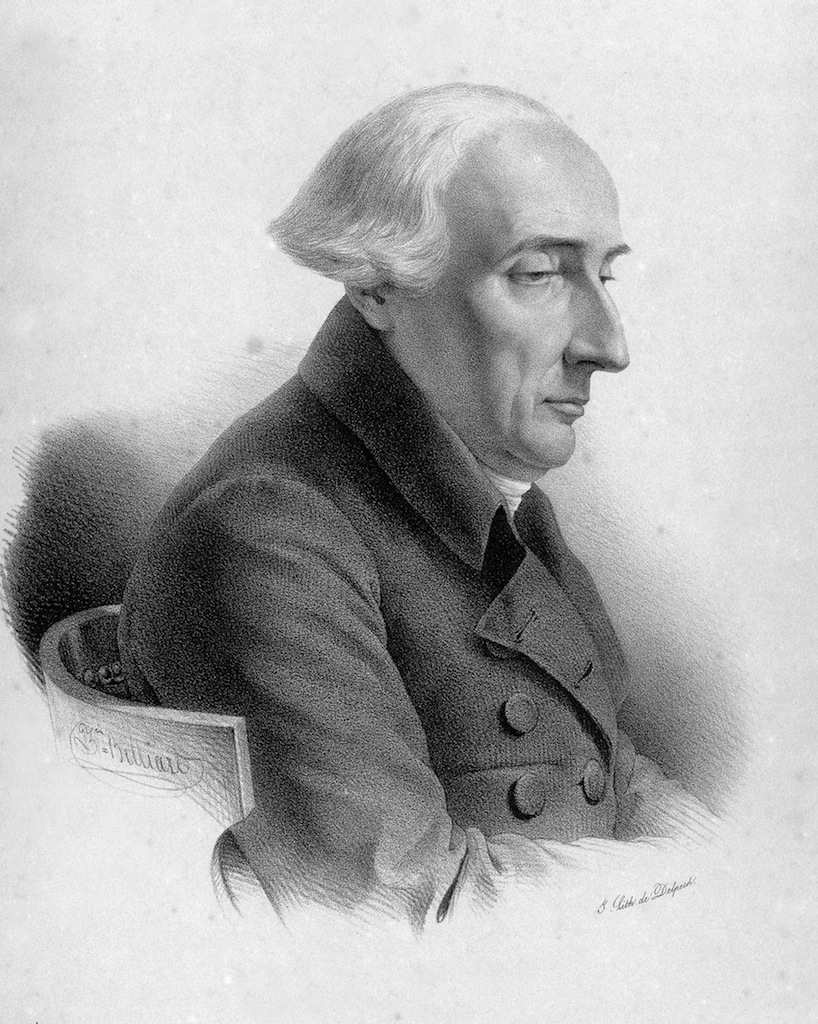}

J. L. Lagrange (1736--1813)
\end{center}
\bigskip

\begin{theorem}[18th century] \label{century}
A symmetric matrix is linearly congruent to a diagonal matrix.
\end{theorem}

\begin{proof}
Of course, this fact was not expressed in terms of matrices  but in terms of quadratic forms:
The proof is very easy and is best seen in an example:
$$
Q(x_1,x_2) = x_1^2+2x_1x_2+8x_2^2= (x_1+x_2)^2 + 7 x_2^2
$$
involving completing the square. Formally, the proof goes by induction, and it may be assumed that $S$ is invertible. As $Q \neq 0$ there exists some $v\in \RR^n$ such that $Q(v) =S(v,v)/2 \neq 0$.
Then $S$ is linearly congruent to $\begin{pmatrix} S(v,v) & 0 \\ 0 & S'\end{pmatrix}$ with $S'$ the $(n-1)\times (n-1)$-matrix of $S$ restricted to the $(n-1)$-dimensional subspace $v^{\perp}=\{w \in \RR^n\, \vert\, S(v,w)=0\} \subset \RR^n$.
\end{proof}

Note that the same proof generalizes to a symmetric matrix over {\it any field of characteristic $\neq 2$}.
This hypothesis is hidden in the formula $Q(v)=S(v,v)/2$.

Since any positive real number is a square, one can restate the previous result as:

\begin{theorem}
A symmetric matrix is linearly congruent to  some
$$I_{p,q}=\begin{pmatrix} I_p & 0 & 0 \\ 0 & 0  & 0 \\ 0 & 0 & -I_q \end{pmatrix}.$$
\end{theorem}

The {\it crucial observation} is that the integers $p$ and $q$ are invariant under linear congruence:

\begin{figure}[ht]
\centering
\begin{minipage}[b]{0.45\linewidth}
\includegraphics[width=.9\textwidth]{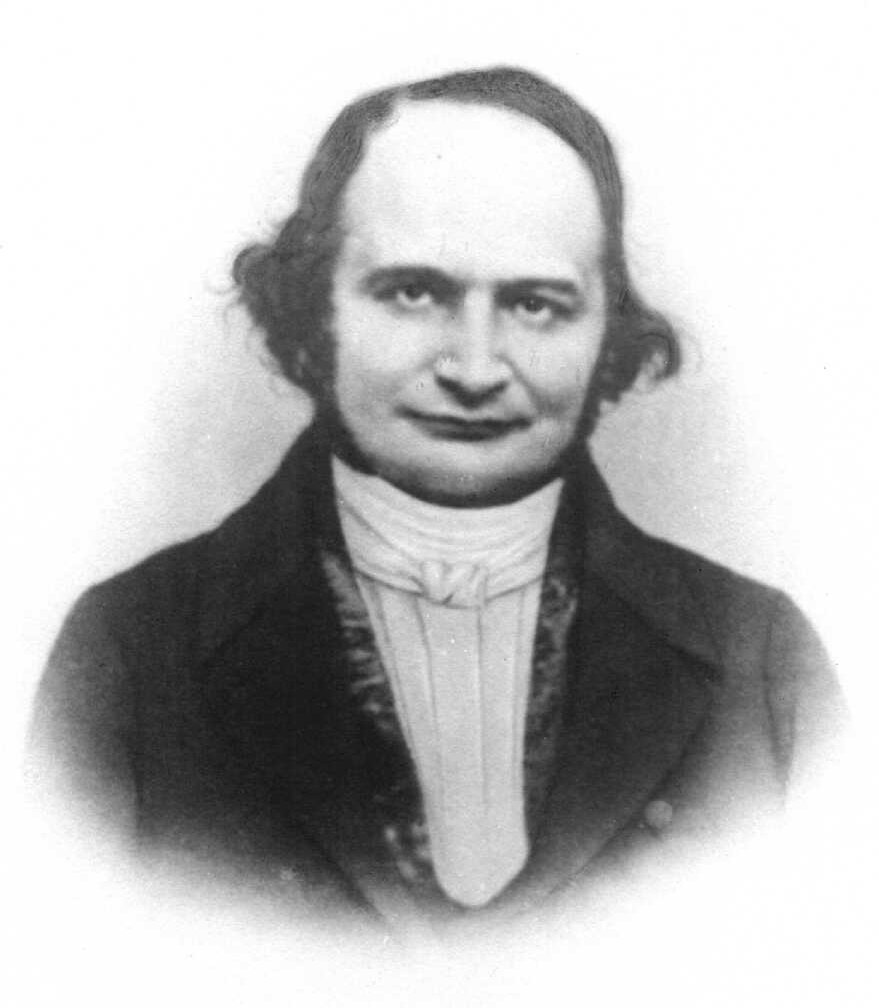}
\center{Carl Gustav Jacob Jacobi (1804-1851)}
\end{minipage}
\hfill
\begin{minipage}[b]{0.45\linewidth}
\includegraphics[width=.85\textwidth]{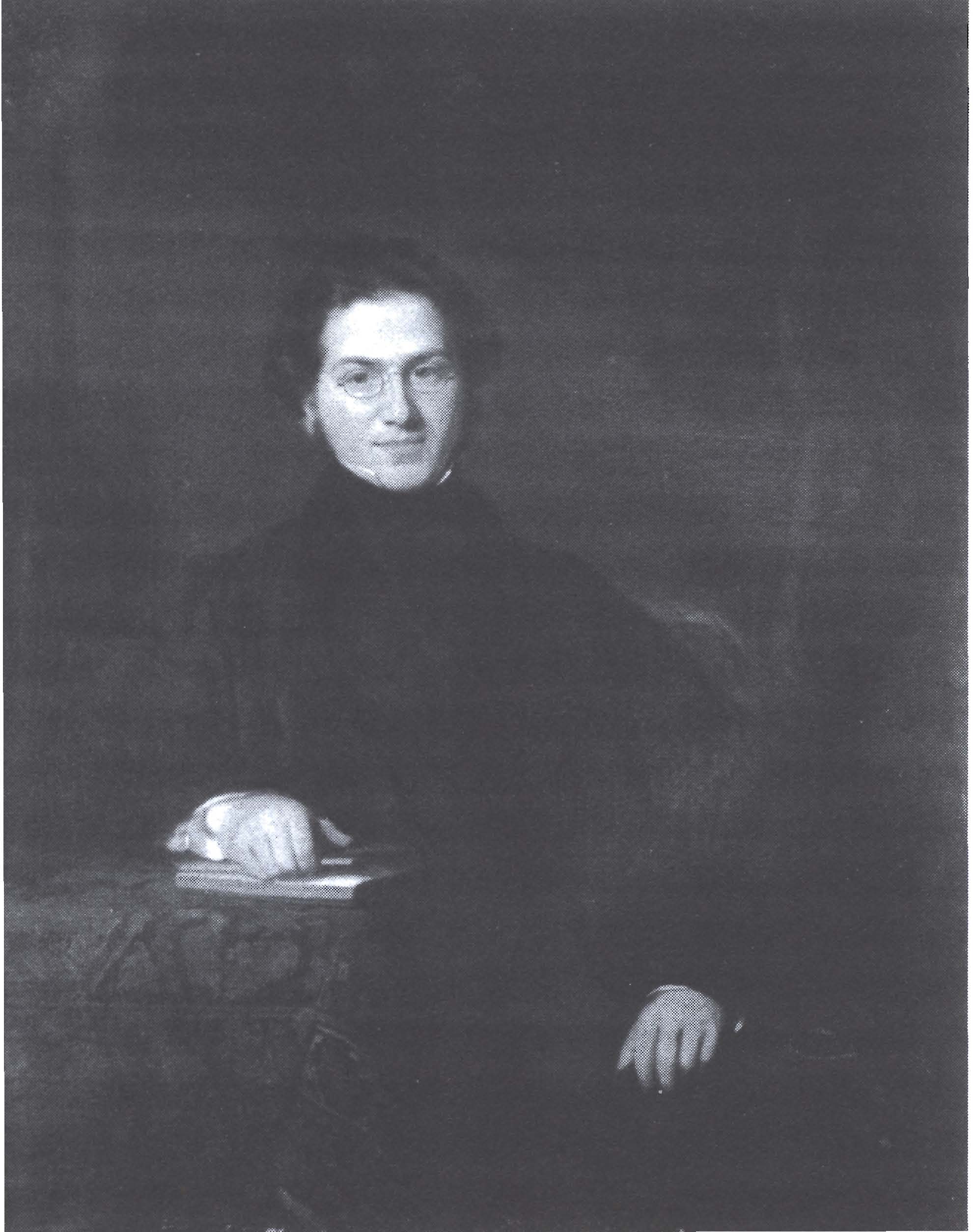}
\center{James Joseph Sylvester (1814-1897) }
\end{minipage}
\end{figure}

\begin{theorem}[Sylvester's Law of Inertia, 1852-3,~\cite{sylvester1852, sylvester1853b}] \label{sylvester}
Any symmetric $n \times n$ matrix $S$ over $\RR$ is linearly congruent to a \emph{unique} matrix of the form $I_{p,q}$.
\end{theorem}

\begin{proof}
One has to show that if $I_{p,q}$ and $I_{p',q'}$ are congruent, one has necessarily $p=p'$ and $q=q'$.
By contradiction, assume for instance that $p'>p$.
This would imply that there exists a subspace $V_+ \subseteq \RR^n$ of dimension $p'>p$ with $I_{p,q} (v,v)>0$ for all $v \in V_+\backslash \{0\}$.
Such a $V_{+}$ intersects non trivially any subspace of dimension $n-p$, and in particular $\{ 0 \} \times \RR^{n-p}$.
This implies that there is a vector $v \in V_+\backslash \{0\}$ with $I_{p,q}(v,v)\leqslant 0$: a contradiction.
\end{proof}

The original formulation of the Law of Inertia by Sylvester in 1852~\cite{sylvester1852} was as a ``remark $\dots$ easily proved'' (but for which he did not provide a proof).

$$\begin{array}{c}
\includegraphics[width=.9\textwidth]{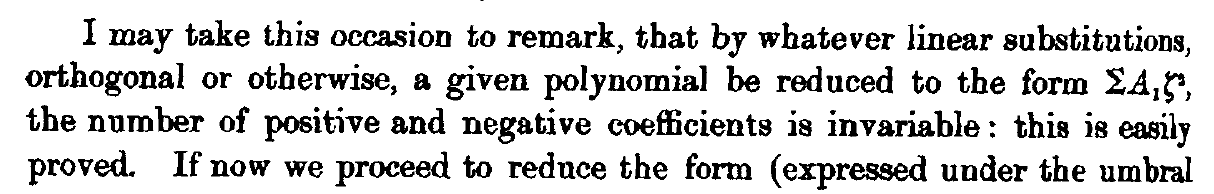}\\
***\\
\includegraphics[width=.9\textwidth]{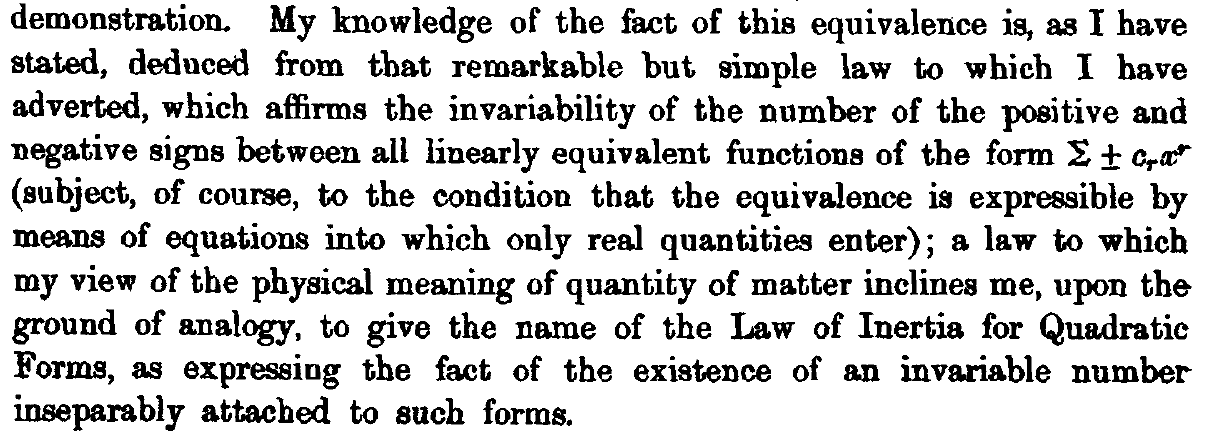}\\
\end{array}$$

Here is Sylvester's original 1853 proof~\cite{sylvester1853b} of the Law of Inertia:

\begin{center}
\includegraphics[width=1\textwidth]{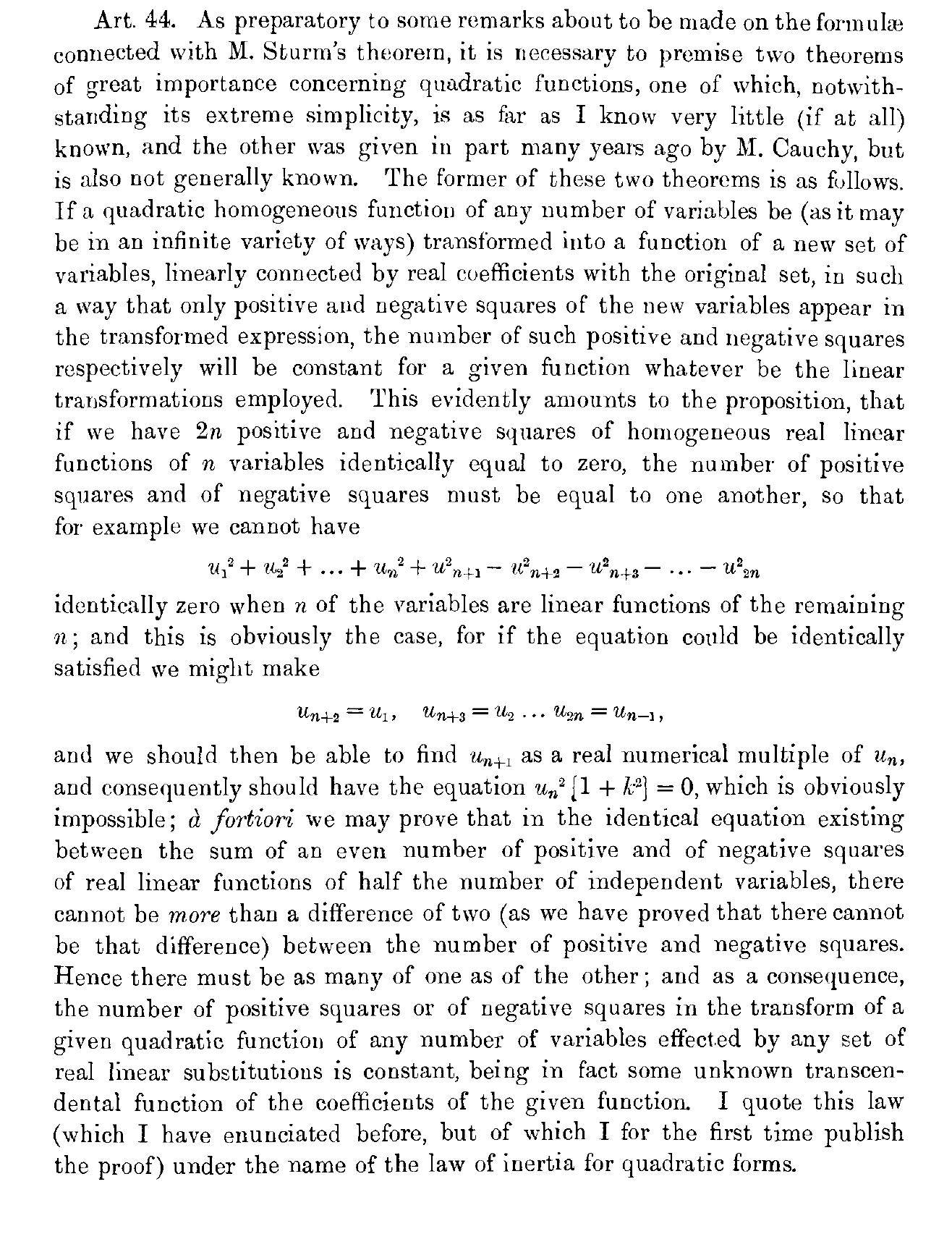}
\end{center}

Sylvester could not have been aware of an earlier unpublished proof by Jacobi~\cite{jacobi3}.
The proofs by Jacobi and Sylvester are essentially identical.
The paper of Jacobi was published in 1857, after his death in 1851,
along with a letter by Hermite~\cite{hermite} and a comment by
Borchardt~\cite{borchardt2}.
Moreover, Gauss included the Law of Inertia in his 1846 lectures, according to his student Riemann (see~\cite{bourbaki}).
Perhaps, one should speak of the Gauss-Jacobi-Sylvester law.
One could even suggest the addition of the name of Hermite ``qui l'a trouv\'ee en m\^eme temps que Sylvester'', according to Picard~\cite[page xix]{hermiteoeuvres}.
A student in the twenty-first century might be surprised to discover that so many great names had to join their efforts to prove such a simple theorem.
But he/she should not forget however that the simple proof uses concepts such as vector spaces, linear maps, ranks of matrices etc. which were far from clear in the middle of the nineteenth century.

\subsection{The signature of a symmetric matrix}\label{regular1}

\begin{definition} For an $n \times n$ symmetric matrix $S$ the integer $p$ in Sylvester's Law of Inertia~\ref{sylvester} is called the \emph{positive index} of $S$ and denoted by $\tau_+(S)$.
Similarly $q$ is called the  \emph{negative index} of $S$ and denoted by $\tau_-(S)$.
The \emph{signature} (also known as the \emph{index of inertia}) of  a symmetric $n \times n$ matrix $S$ with entries in $\RR$ is the difference $$\tau(S)=\tau_+(S)-\tau_-(S) \in \{-n,-n+1,\dots,-1,0,1,\dots,n\}.$$
\end{definition}

The signature was only implicit in Sylvester~\cite{sylvester1852}.
The terminology was introduced by Frobenius~\cite{frobenius}.

The {\it rank} of $S$  is the sum $${\rm dim}_{\RR}(S(\RR^n))=\tau_+(S)+\tau_-(S) \in \{0,1,2,\dots,n\}.$$
Note that $S$ is invertible if and only if $\tau_+(S)+\tau_-(S)=n$.

\begin{definition}
\label{regular} \leavevmode

1. For $k=1,2,\dots,n$ the {\it principal $k \times k$ minor} of an $n \times n$ matrix $S=(s_{ij})_{1 \leqslant i,j \leqslant n}$ is
$$\mu_k(S)={\rm det}(S_k) \in \RR$$
with $S_k=(s_{ij})_{1 \leqslant i,j \leqslant k}$ the principal $k \times k$ submatrix:
$$S=\begin{pmatrix} S_k & \dots \\ \vdots & \ddots \end{pmatrix}.$$
By convention, we set $\mu_0(S)=1$.

2. An $n \times n$ matrix $S$ is {\it regular} if $\mu_k(S) \neq 0$ for $1 \leqslant k \leqslant n,$
that is if each $S_k$ is invertible. In particular, $S_n=S$ is invertible. This is a generic assumption, satisfied on an open dense set of matrices.
\end{definition}

Jacobi~\cite{jacobi2} and Sylvester~\cite{sylvester1852,sylvester1853} initiated the expression of the signature of a symmetric matrix as the number of changes of sign in the principal minors.

\begin{definition} \label{var}
The  \emph{variation} of $\mu=(\mu_0,\mu_1,\dots,\mu_n) \in \RR^{n+1}$ with each $\mu_k \neq 0$ is
$${\rm var}(\mu)=\text{number of changes of sign in $\mu_0,\mu_1,\dots,\mu_n$} \in \{0,1,2,\dots,n\}.$$
\end{definition}

We shall make use of the following two very elementary properties of the variation.

\begin{proposition} \label{variation} \leavevmode

1. The variation of $\mu$ is related to the signs of the successive quotients $\mu_1/\mu_0$, $\mu_2/\mu_1$, $\dots$, $\mu_n/\mu_{n-1}$ by the identity
$$\sum\limits^n_{k=1} {\rm sign}(\mu_k/\mu_{k-1})=
n-2\,{\rm var}(\mu) .$$

2. If the components  of $\mu$, $\mu'$ have the same signs except for the $k$th entry, with ${\rm sign}(\mu'_k)=-{\rm sign}(\mu_k)$, then
$${\rm var}(\mu')-{\rm var}(\mu)=
\begin{cases}
{\rm sign}(\mu_k/\mu_{k-1})+{\rm sign}(\mu_{k+1}/\mu_k)&
\text{if $1 \leqslant k \leqslant n-1$}\\
{\rm sign}(\mu_1/\mu_0)&\text{if $k=0$}\\
{\rm sign}(\mu_n/\mu_{n-1})&\text{if $k=n$}.
\end{cases}$$

\end{proposition}

\begin{proof}
\indent 1. Consider first the special case $n=1$: if  $\mu_0,\mu_1$ have the same (resp. different) signs than ${\rm var}(\mu_0,\mu_1)=0$ (resp. 1)
and the equation is $1=1-0$ (resp. $-1=1-2$).  For the general case assume inductively true for $n$ and note that ${\rm sign}(\mu_{n+1}/\mu_n)=1-2\,{\rm var}(\mu_n,\mu_{n+1})$ by the special case.

2. Apply 1. twice.
\end{proof}

We shall now prove that a regular symmetric matrix $S$ is linearly congruent to the diagonal matrix with entries $\mu_k(S)/\mu_{k-1}(S)$ ($1 \leqslant k \leqslant n$).
The proof is by  the ``algebraic plumbing" of matrices -- the algebraic analogue of the geometric plumbing of manifolds, discussed in section~\ref{topology} below.

\begin{definition} \label{plumbing}
The \emph{plumbing} of a regular symmetric $n \times n$ matrix $S$ with respect to $v \in \RR^n$, $w \neq vS^{-1}v^* \in \RR$ is the regular symmetric $(n+1) \times (n+1)$ matrix
$$S'=\begin{pmatrix} S & v^* \\ v & w \end{pmatrix}.$$
\end{definition}

\begin{figure}[ht]
\centering
\begin{minipage}[b]{0.425\linewidth}
\includegraphics[width=.9\textwidth]{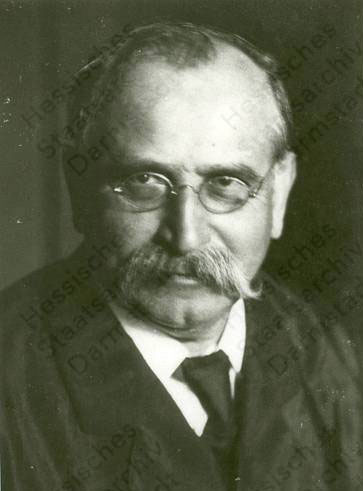}
\center{Sigmund Gundelfinger (1846-1910)}
\end{minipage}
\hfill
\begin{minipage}[b]{0.4\linewidth}
\includegraphics[width=.9\textwidth]{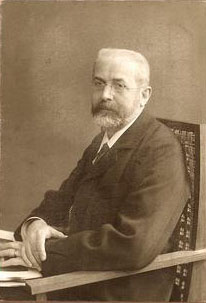}
\center{Ferdinand Georg Frobenius (1849-1917)}
\end{minipage}
\end{figure}

The {\it Sylvester-Jacobi-Gundelfinger-Frobenius theorem} gives a very simple way of calculating the signature, at least in the case of a regular matrix.
\begin{theorem}[Sylvester~\cite{sylvester1852}, Jacobi~\cite{jacobi2}, Gundelfinger~\cite{gundelfinger}, Frobenius~\cite{frobenius}] \label{sjgf}
The signature of a regular symmetric $n \times n$ matrix $S$ is given by:
$$
\begin{array}{ll}
\tau(S)&=\sum\limits^n_{k=1} {\rm sign}(\mu_k(S)/\mu_{k-1}(S))\\
&=n-2\,{\rm var}(\mu_0(S),\mu_1(S),\dots,\mu_n(S))\in \{-n,-n+1,\dots,n\}.
 \end{array}
$$
  \end{theorem}
 \begin{proof}
 Since every regular symmetric $n \times n$ matrix $S$ is obtained from the $0 \times 0$ matrix by $n$ successive plumbings, it suffices to calculate the jump in signature under one plumbing. The matrix identity
$$S'=\begin{pmatrix} S & v^* \\ v & w \end{pmatrix}=
\begin{pmatrix} 1 & 0 \\ vS^{-1} & 1 \end{pmatrix}
\begin{pmatrix} S & 0\\ 0 & w-vS^{-1}v^* \end{pmatrix}
\begin{pmatrix} 1 & S^{-1}v^* \\ 0 & 1 \end{pmatrix}$$
shows that the symmetric $(n+1)\times (n+1)$ matrix $S'$ is linearly congruent to $\begin{pmatrix}
S & 0 \\ 0 & w - vS^{-1}v^* \end{pmatrix}$ with
$$\mu_k(S')=\mu_k(S)~(1 \leqslant k \leqslant n),~\mu_{n+1}(S')=
\mu_n(S)(w-vS^{-1}v^*).$$
\end{proof}

Another way of expressing the same result is the following.
The positive index of $S$ is the number of indices $k=1, \dots, n$ such that  $\mu_{k-1}(S)$ and $\mu_k(S)$ have the same signs
$$\tau_+(S)=p=n-{\rm var}(\mu_0(S),\mu_1(S),\dots,\mu_n(S)),$$
the negative index of $S$ is the number of indices for which the signs are different
$$\tau_-(S)=q={\rm var}(\mu_0(S),\mu_1(S),\dots,\mu_n(S)).$$
and the signature is the difference
$$\tau(S)=\tau_+(S)-\tau_-(S)=p-q=n-2\,{\rm var}(\mu_0(S),\mu_1(S),\dots,\mu_n(S)).$$

We could even add the name of Gauss to the previous theorem since he proved it in 1801 when $n=3$ in his {\it Disquisitiones Arithmeticae}~\cite[page 305]{gauss}!

\begin{corollary} If $S$ is an \emph{invertible} symmetric $n \times n$ matrix which is not regular then for sufficiently small $\epsilon \neq 0$ the symmetric $n \times n$ matrix $S_{\epsilon}=S+\epsilon I_n$ is regular, and
$$\begin{array}{l}
\tau(S)=\tau(S_\epsilon)=\sum\limits^n_{k=1}
{\rm sign}(\mu_k(S_\epsilon)/\mu_{k-1}(S_\epsilon)) \in \{-n,-n-1,\dots,n-1,n\}.
\end{array}$$
\end{corollary}

\subsection{Orthogonal congruence}\label{orthogonal}

\bigskip
\begin{center}
\includegraphics[width=.35 \textwidth]{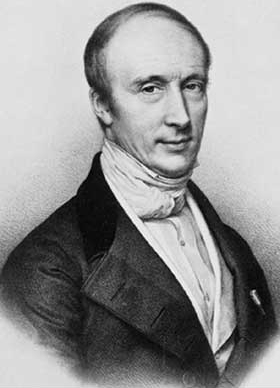}

Augustin Cauchy (1789--1857)
\end{center}
\bigskip

The next step in the understanding of symmetric matrices concerns {\it orthogonal congruence} rather than linear congruence.

\begin{theorem}[Spectral theorem, Cauchy 1829,~\cite{cauchy1829}]
\label{spectral} \leavevmode

1. Any symmetric matrix $S$ with real coefficients is \emph{orthogonally} congruent to a diagonal matrix.
The eigenvalues of $S$ are real numbers.

2. Two symmetric $n \times n$ matrices $S,T$ are orthogonally congruent if and only if their eigenvalues are the same.
\end{theorem}

There are many proofs of this truly fundamental fact.
We shall hint at nine proofs!

For a history of this result, we recommend Hawkins~\cite{hawkins}.
It is interesting (and important) to compare this spectral theorem with theorem~\ref{century}.
The classification of real quadratic forms under congruence is elementary and its proof extends to any field with characteristic different from $2$.
In the case of $\RR$, there is a {\it finite} number of normal forms for each dimension $n$.
On the other hand, the spectral theorem is specific to the field of real numbers and involves an infinite number of normal forms since the eigenvalues are invariant under conjugacy.
Historically, the subtle difference between these two important theorems has not always been clear to the founding fathers of linear algebra.

An immediate corollary of the Law of Inertia and the spectral theorem is that the signs of the eigenvalues of a symmetric matrix are invariant under congruence.

\begin{corollary}[of Sylvester's Law of Inertia~\ref{sylvester}]\label{sylvestercor2}  \label{corsign} \leavevmode

1. Two symmetric matrices $S,T$ are \emph{linearly congruent} if and only if their eigenvalues have the same signs.

2. The positive (resp. negative) index $\tau_+(S)$ (resp. $\tau_-(S)$) of a symmetric $n \times n$ matrix $S$ is the number of eigenvalues $\lambda_k>0$ (resp. $\lambda_k<0$). The signature
of $S$ is the sum of the signs of the eigenvalues
$$\tau(S)=\sum\limits^n_{k=1}{\rm sign}(\lambda_k) \in \{-n,-n+1,\dots,n-1,n\}.$$

\end{corollary}

We shall now describe many proofs of the spectral theorem, following more or less the chronological order.
Each of these proofs casts some new light on the theorem.

We denote by $\langle~,~\rangle$ the usual scalar product on $\RR^n$.
A symmetric $n \times n$ matrix $S=(s_{ij})$ determines an endomorphism
$$S:v=(x_1,\dots,x_n) \in \RR^n \mapsto S(v)=(\sum\limits^n_{i=1}S_{ij}x_j) \in \RR^n$$
such that $\langle S(v),w \rangle=S(v,w)=\langle v,S(w)\rangle \in \RR$ for $v,w \in \RR^n$. 
The orthogonal of any $S$-invariant subspace $V \subseteq \RR^n$ is an $S$-invariant subspace $V^{\perp}=\{w \in \RR^n\,\vert\, S(v,w)=0\,{\rm for~all}\, v \in V\}$.
The (real) eigenvalues of $S$ are the eigenvalues $\lambda \in \RR$ of $S$, with
$S(v)=\lambda v \in \RR^n$ for some eigenvector $v\neq 0 \in \RR^n$, and
the eigenspace ${\rm ker}(\lambda I - S:\RR^n \to \RR^n)$ is $S$-invariant.
Observe also that two eigenvectors $v_1,v_2$ with different real eigenvalues $\lambda_1,\lambda_2$ have to be orthogonal for $\langle~ ,~ \rangle$.
Indeed
$$\langle S(v_1),v_2 \rangle= \langle v_1,S(v_2) \rangle = \lambda_1 \langle v_1,v_2 \rangle = \lambda_2 \langle v_1,v_2 \rangle.$$
so that $\langle v_1 , v_2 \rangle =0$.
In other words, eigenspaces are orthogonal so that  the point of the Theorem is to prove that the spectrum is real.

\bigskip

The spectral theorem is trivial in dimension 1 and very easy in dimension 2.
Indeed, if $n=2$, one can for example subtract from $S$ a suitable  multiple of the identity matrix so that one can assume that
$$S=\begin{pmatrix} a & b \\ b& -a \end{pmatrix}$$
which has $\pm \sqrt{a^2+b^2}$ as eigenvalues.
This fact has been known for a very long time and was initially expressed by saying that a conic section with a centre of symmetry has two perpendicular axes of symmetry.

Going from dimension 2 to dimension 3 is easy {\it today}, since every $3\times 3$ matrix has at least a real eigenvalue.
In the case of a symmetric matrix, one can look at the $2$ dimensional invariant plane orthogonal to an eigenvector and use the theorem in dimension 2. 
The first rigorous exposition of this proof in dimension 3, was published in 1810 by Hachette and Poisson~\cite{hachettepoisson}\footnote{A very early example of a joint paper.}, who were describing axes of symmetries for quadric surfaces in $3$-space.

\bigskip

We now turn to the proof in the $n$-dimensional case.

\bigskip

{\it The heuristic first ``proof''} is physical and half convincing for a contemporary mathematician but it was considered solid by Lagrange in 1762~\cite{lagrange}.
Let us explain it in modern language.
Suppose first that the quadratic form $U(v) =\frac{1}{2}\langle S(v) ,v \rangle$ is positive definite and consider this function as a mechanical potential, producing a force $- {\rm grad}\, U (v) = -S(v) $.
Assuming a unit mass, the equation of motion is
$$
\frac{d^2 v}{dt^2} = -S(v) .
$$
Conservation of energy yields
$$
\frac{1}{2} \Big\vert \! \Big\vert \frac{dv}{dt}\Big\vert \! \Big\vert ^2 + U(v)  = {\rm Constant}
$$
and this implies stability: any trajectory stays in a bounded neighborhood of the origin since $U(v) $ remains bounded.
Recall that the differential equation $d^2 v / dt^2 + \omega^2 v =0$ has solutions $v = a \exp(\pm i \omega t)$.
These solutions can only be bounded (as the time runs over $\RR$) if $\omega$ is a real number.
Indeed  $\exp( i (\alpha+i \beta )t)=  \exp ( i \alpha t) \exp(-bt)$ is bounded if an only if $b=0$.
``Therefore'', $S$ has only real (and positive) eigenvalues.
If $S$ is not positive definite, just consider $S + \lambda Id$ with a large positive real number $\lambda$.

\begin{quote}
{\it Au reste, quoiqu'il soit difficile, peut-\^etre impossible, de d\'eterminer en g\'en\'eral les racines de l'\'equation $P=0$, on peut cependant s'assurer, par la nature m\^eme du probl\`eme, que ces racines sont n\'ecessairement toutes r\'eelles [...]; car sans cela les valeurs de $y', y'',y''',\dots$ pourraient cro\^{\i}tre \`a l'infini, ce qui serait absurde~~\cite{lagrange}.}
\end{quote}

\bigskip

{\it Cauchy's proof} (1829)~\cite{cauchy1829}.
Let  $\lambda$ be a complex eigenvalue of $S$ and let $v$ be a non zero eigenvector in $\CC^n$.
Of course, $\bar{v}$ is an eigenvector with eigenvalue $\bar{\lambda}$.
We still denote by the same symbol  $\langle~ ,~ \rangle$ the extension to $\CC^n$ of the scalar product as a {\it bilinear} form (not as a hermitian form).
We have
$$
\langle S(v) ,\bar{v} \rangle = \lambda \langle v, \bar{v} \rangle = \langle v, S(\bar{v} )\rangle = \bar{\lambda} \langle v, \bar{v} \rangle \in \CC.
$$
Since $v\neq 0$, we have $\langle v, \bar{v} \rangle \neq 0$. This implies that $\lambda = \bar{\lambda}$ so that $\lambda$ is indeed real.

Amazingly, this proof, which looks crystal clear today, was not so convincing at the beginning of the nineteenth century, probably because complex numbers still sounded mysterious (they were called {\it imaginary} or even  {\it  impossible}).

\bigskip

{\it A more algebraic proof.} If $n\geqslant   3$, one uses the fact that any  matrix with real entries has an invariant $2$ dimensional subspace $E$.
Of course, since $S$ is symmetric, the orthogonal of $E$ is also invariant and the proof goes by induction.
One could argue that the standard proof of the existence of an invariant $2$ dimensional subspace uses complex numbers so that this proof is very close to Cauchy's proof.

\bigskip
{\it Just for fun}, let us describe a magic proof given by Sylvester (1852)~\cite{sylvester1852}.
Consider the characteristic polynomial $P(X) = \det (X I_n-S)$.
Note that $P(X) P(-X) = \det (X^2 I_n-S^2)$.
Since $S$ is symmetric, all diagonal entries of $S^2$ are sums of squares and are therefore non negative.
Let us write
$$P(X) P(-X) = (-X^2)^n  + a_1 (-X^2)^{n-1}  + a_2 (-X^2)^{n-2} + \ldots + a_n.$$
The first coefficient $a_1$ is the trace of $S^2$, the sum of the squares of the entries of $S$, and is therefore non negative.
More generally $a_k$ is the sum of squares of the principal $k\times k$ minors of $S$ and is therefore non negative as well.
Sylvester concludes that this implies that $P$ cannot have an imaginary non real root.
Assuming some root of the form $\alpha+i\beta$, with $\alpha,\beta$ real,  one could replace $S$ by $S- \alpha I_n$ and reduce the problem to proving that a purely imaginary root $i \beta $ is not possible.
Now, that would imply that the positive number $\beta^2$ would be a root of $X^n + a_1X^{n-1} + a_2 X^{n-2} + \ldots + a_n = 0 $ which is clearly impossible.
$$
\includegraphics[width=.9\linewidth]{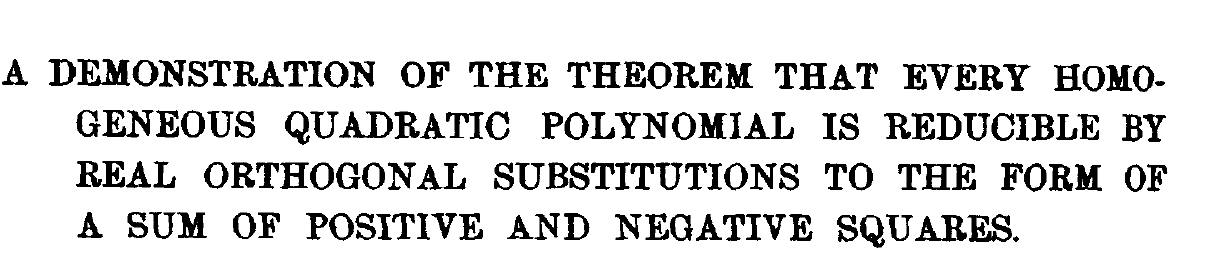}
$$
Sylvester describes his proof as ``very simple".
He was aware of the fact that Cauchy had published a proof ``somewhere'' (in his words) but one can understand why he preferred his own proof.
Indeed this approach is truly algebraic and does not allow any intrusion of the complex numbers.
It only uses the fact that a sum of squares is positive and that a positive number is not zero.

\bigskip

{\it The classical topological proof.}
One considers a vector $v$ that maximizes the value of $Q(v)=S(v,v)/2$ on the unit sphere of $\RR^n$.
Note that $Q(v+w)= Q(v) +  S(v,w) + Q(w)$ so that the differential of $Q$ at $v$ is $dQ_v(w)= S(v,w)$.
A point $v$ is critical for the restriction of $Q$ to the unit sphere if $S(v,w)=0$ for all vectors $w$ in the tangent space of the unit sphere at $v$, i.e. for all $w$ such that $\langle v,w\rangle=0$.
In other words, the critical points are the eigenvectors of the matrix $S$.
Therefore, the maximum of $Q$ on the unit sphere provides such an eigenvector.
One concludes by induction.
\bigskip

Let $Q$ be a quadratic form on some $n$-dimensional real vector space $E$ equipped with a positive definite scalar product $\langle , \rangle$ and let $S$ be the symmetric endomorphism associated to $Q$, i.e. such that $Q(v)  = \langle S(v) ,v\rangle$. (Do not confuse the endomorphism $S:E\to E$ with the symmetric bilinear pairing $S:E \times E \to \RR$ determined by $Q$).
Denote by $\lambda_1(S) \leqslant \dots \leqslant \lambda_n(S)$ the $n$ real eigenvalues of $S$ (whose existence follows from the spectral theorem).
Clearly, $\lambda_n(S)$ is the maximum value of the quotient ${Q(v) }/{ \langle v, v \rangle}$ as $v$ describes $E \backslash \{0\}$.
More generally, the \emph{Rayleigh minimax theorem} asserts that
$$
\lambda_k(S) = \min_{F \subseteq E}\max_{v \in F} \{ {Q(v) }/{\langle v, v \rangle}~ \vert~ v \neq 0 ; \text{dim}(F) = k\}.
$$

\bigskip
Actually, one can get a more pedantic (sixth) proof of the spectral theorem using this minimax idea and some modern algebraic topology.
This proof is not by induction.
Note that if $S$ is invertible, the function $f(v)=\langle S(v),v \rangle/\langle v,v\rangle= Q(v)/\langle v,v\rangle$ is constant on lines and therefore defines a function on the real projective space $\mathbb{RP}^{n-1}$.
As before, critical points of $F$ are eigenvectors of $S$.
If $S(v)= \lambda v$, it is easy to compute the second derivative of $F$ at $v$: one finds $d^2(w)= Q(w)-\lambda \langle w,w\rangle$ (which is indeed a quadratic form on the tangent space of $\mathbb{RP}^{n-1}$ at $\RR v$, which is  $\RR^n / \RR v$).
If the eigenvalue $\lambda$ is simple, we conclude that $d^2F$ is non degenerate so that $v$ is a Morse critical point.
A knowledgeable reader knows that a Morse function on a compact manifold must have at least as many critical points as the sum of the Betti numbers (for instance modulo 2).
In this case, the sum is equal to $n$ and each of these critical points yields an eigenvector.
We conclude that a symmetric matrix with a simple spectrum has real eigenvalues.
Clearly, the set of symmetric matrices with a simple spectrum is open and dense and the set of real matrices with real spectrum is closed, so that we get a proof of the spectrum theorem in the general case.

\bigskip

Let $S$ be a symmetric $n \times n$ matrix, and let $T$ be the $(n-1)\times (n-1)$ minor defined by the determinant of the submatrix obtained by
deleting the $k$th row and $k$th column, for $k=1,2,\dots,n$.
It is a corollary of Rayleigh minimax that the $(n-1)$ eigenvalues 
$\lambda_1(T),\lambda_2(T),\dots,\lambda_{n-1}(T)$ are {\it interlaced} with the $n$ eigenvalues $\lambda_1(S),\lambda_2(S),\dots,\lambda_n(S)$, that is up to reordering
$$\lambda_1(S) < \lambda_1(T) < \lambda_2(S) < \dots < \lambda_{n-1}(T) < \lambda_n(S)~.$$
One can give a simple algebraic proof of this fact, which will provide our seventh proof of the spectral theorem.
This proof is due to Cauchy~\cite{cauchy2} for $n=3$ and was generalized by Hermite~\cite{hermitecours}.
It goes by induction.
Assume that all $(n-1)\times (n-1)$ symmetric matrices are diagonalizable in an orthonormal basis.
Let $S$ be an $n\times n$ symmetric matrix.
We can therefore assume it has the following form in some orthonormal basis.
$$S=\begin{pmatrix}
\lambda_1 & 0 & 0 & \dots & 0 & a_1 \\
0 & \lambda_2 & 0 & \dots & 0 & a_2 \\
0 & 0& \lambda_3 & \dots & 0 & a_3 \\
\vdots & \vdots & \vdots & \ddots & \vdots & \vdots \\
0 & 0 & 0 & \dots  & \lambda_{n-1} & a_{n-1}\\
a_1 & a_2 & a_3 & \dots & a_{n-1} & \lambda_n
\end{pmatrix}$$
with $\lambda_1 \leqslant \lambda_2 \leqslant \dots \leqslant \lambda_{n-1}$ the eigenvalues of the principal $(n-1)\times (n-1)$ minor $T=S_{n-1}$.
Let us evaluate the characteristic polynomial $P(X)= \det(X I_n-S)$ of $S$.
$$
P(X) = - \sum_{i=1}^{n-1} a_i^2 \prod_{j\neq i ;  1 \leqslant j \leqslant n-1} (X-\lambda_j)+ \prod_{i=1}^n(X-\lambda_i).
$$
Let us first make the generic assumption that all the $\lambda_i$ are distinct and that the $a_i$'s are not zero.
It follows that the signs of $P(-\infty)$, $P(\lambda_1), \dots,  P(\lambda_{n-1}),\allowbreak  P(+\infty)$ alternate.
Hence there exist $n$ real roots in the corresponding intervals so that $S$ is indeed diagonalizable.
The general case follows since the set of real polynomials having all their roots in $\RR$ is closed.

\bigskip

As a bonus, we provide a {\it eighth proof}, from Appell (1925)~\cite[Chapter 1]{appell}, even though it is probably neither the best nor the easiest!
Let $S=(s_{ij})$ be a real symmetric $n \times n$ matrix.
As before we denote by $\mu_k(S)={\rm det}(S_k)$ the principal minors.
Denote by $C=(c_{ij})$ the \emph{comatrix} of $S$ so that $c_{ij}$ is $(-1)^{i+j}$ times the determinant of the $(n-1)\times (n-1)$ matrix obtained from $S$ by deleting the $i^{th}$ row and the $j^{th}$ column.
We claim that
$$
\begin{vmatrix}
c_{n-1,n-1} & c_{n-1,n}\\
c_{n,n-1} & c_{n,n}
\end{vmatrix}
=\mu_n (S)  \mu_{n-2} (S).
$$
This follows from the computation of the product of the two determinants
$$
\begin{vmatrix}
1 & 0 & \dots & 0 & 0 & 0 \\
0 & 1 &  \dots & 0 & 0 & 0 \\
. & .  & \dots  & . & .  & .  \\
\vdots & \vdots & \ddots & \vdots & \vdots & \vdots \\
0 & 0 & \dots  & 1 & 0 & 0\\
c_{n-1,1} & c_{n-1,2}  & \dots & c_{n-1,n-2} & c_{n-1,n-1}  & c_{n-1,n} \\
c_{n,1} & c_{n,2}  & \dots & c_{n,n-2} & c_{n,n-1}  & c_{n,n}
\end{vmatrix}
=
\begin{vmatrix}
c_{n-1,n-1} & c_{n-1,n}\\
c_{n,n-1} & c_{n,n}
\end{vmatrix}
$$
and
$$
\begin{vmatrix}
s_{1,1}  & s_{1,2}  & \dots & s_{1,n-2}  & s_{1,n-1}  & s_{1,n}  \\
s_{2,1}  & s_{2,2}  &  \dots & s_{2,n-2}  & s_{2,n-1}  & s_{2,n}  \\
. & .  & \dots  & . & .  & .  \\
\vdots & \vdots & \ddots & \vdots & \vdots & \vdots \\
s_{n-2,1}  & s_{n-2,2}  & \dots  & s_{n-2,n-2}  & s_{n-2,n-1}  & s_{n-2,n} \\
s_{n-1,1} & s_{n-1,2}  & \dots & s_{n-1,n-2} & s_{n-1,n-1}  & s_{n-1,n} \\
s_{n,1} & s_{n,2}  & \dots & s_{n,n-2} & s_{n,n-1}  & s_{n,n}
\end{vmatrix} = \mu_n(S)
$$
and the observation that $C.S= (\det S) I_n = \mu_n(S) I_n$.
Let us make the generic assumption that the polynomials $\mu_k(X I_n-S)_{1\leqslant k \leqslant n}$ have distinct roots and they are all distinct.

Let us prove by induction on $k$ that $\mu_k(X I_n-S)$ has $k$ real roots,
and that there is a unique root of $\mu_{k-1}(X I_n-S)$ between two roots of $\mu_{k}(X I_n-S)$.
Assume this is true up to $k=n-1$ and denote by $\alpha_1 < \dots < \alpha_{n-1}$ the real roots of $\mu_{n-1}(X I_n-S)$.

Apply the previous identity to the symmetric matrix $\alpha_i I_n-S$ (so that $\mu_{n-1}(\alpha_i I_n-S)= 0$).
We get that $\mu_{n}(\alpha_i I_n-S ) \mu_{n-2}(\alpha_i I_n-S)$ is negative.
Hence there is at least one root of $\mu_n(X I_n-S)$ in each interval $[\alpha_i, \alpha_{i+1}]$.
A degree consideration also shows that there is also at least a root $< \alpha_1$ and another root $> \alpha_{n-1}$ and this concludes the inductive step.

One should still get rid of the generic assumption on the roots of $\mu_k(X I_n-S)_{1\leqslant k \leqslant n}$.
As in the sixth proof, this follows from the fact that the set of real polynomials having all their roots in $\RR$ is closed.

\bigskip

The ninth proof will be explained later (see~\ref{eighth}).

\bigskip

The spectral theorem expresses the fact that there is a basis in $\RR^n$ which is orthonormal for both $\langle, \rangle$ and $S$.
A more invariant way of expressing the same theorem is the following.
Let $S,T$ be two quadratic forms on a real vector space.
If $T$ is positive definite, then there exists a basis in which the matrix of $T$ is the identity (i.e. orthonormal for $T$) and the matrix of $S$ is diagonal.

There is a nice generalization of the previous statement due to Milnor (cf. Greub~\cite[Chapter IX.3]{greub}).
Let $S,T$ be two quadratic forms on a real vector space $\RR^n$ of {\it dimension at least $3$}.
Assume that there is no vector $v$ which is simultaneously isotropic for $S$ and $T$, i.e. such that $S(v,v)=T(v,v)=0$.
Then there is a basis in which both $S$ and $T$ are diagonal.

The proof is the following.
Consider the map from the unit sphere $\SSS^{n-1}$ to $\CC \backslash \{0\}$ sending $v$ to $S(v,v) + i T(v,v)$.
Since $n\geqslant    3$, the sphere ${\mathbb S}^{n-1}$ is simply connected so that there is a continuous determination of the argument of $S(v,v) + i T(v,v)$ as a continuous function on ${\mathbb S}^{n-1}$ with values in $\RR$.
Hence, there must exist a point on ${\mathbb S}^{n-1}$ for which this argument achieves its maximum.
One shows easily that such a point $v$ is a simultaneous eigenvector for $S$ and $T$.
One then proceeds by induction on $n$.

Even though this is not the purpose of this paper, one should remember that one of the most important motivations for the spectral theorem comes from mechanics.
In the simplest situation, if a rigid body rotates around a fixed point, with some angular velocity $\Omega \in \RR^3$, its kinetic energy is a (positive definite) quadratic form in $\Omega$.
An orthonormal basis in which this form is diagonal defines the three principal inertial axes and the eigenvalues are the {\it principal inertial moments}.
For more information, see Arnold~\cite[Chapter 6]{arnold4}.

This spectral theorem had not only a profound impact on mathematics but also on the Moon!
Lagrange, Cauchy, Jacobi and Sylvester have their craters on the Moon.
\begin{figure}[ht]
\centering
\begin{minipage}[b]{0.45\linewidth}
\includegraphics[height=.9\textwidth]{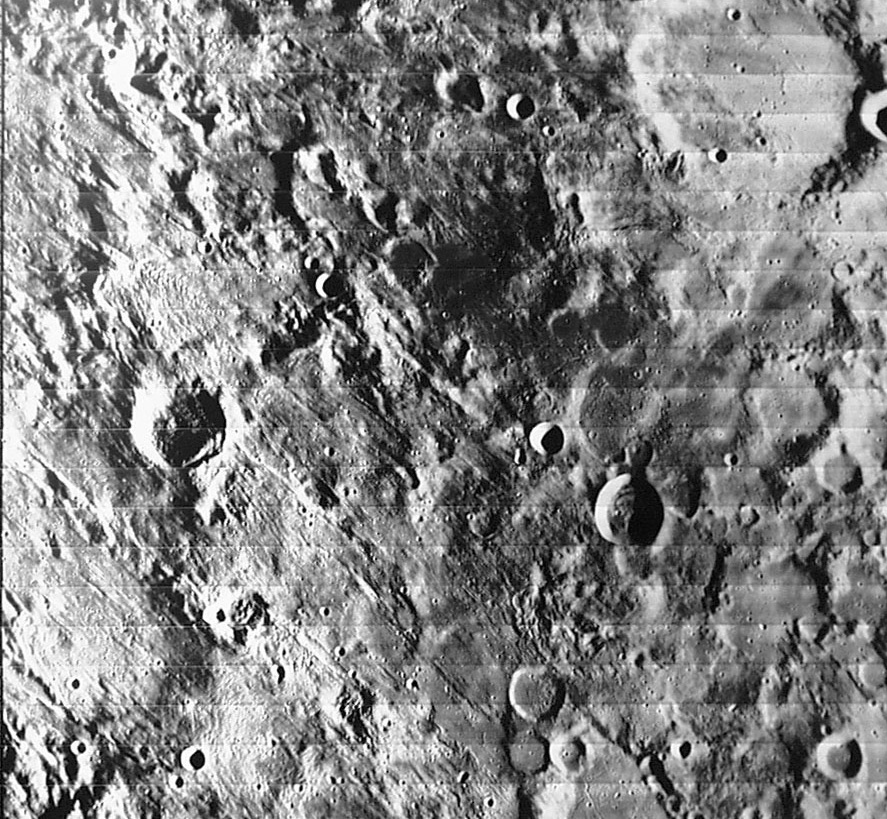}
\center{Lagrange}
\end{minipage}
\hfill
\begin{minipage}[b]{0.45\linewidth}
\includegraphics[height=.9\textwidth]{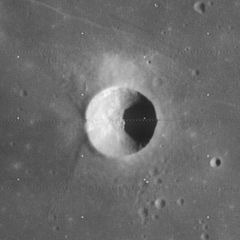}
\center{Cauchy }
\end{minipage}
\end{figure}
\begin{figure}[ht]
\centering
\begin{minipage}[b]{0.45\linewidth}
\includegraphics[height=.9\textwidth]{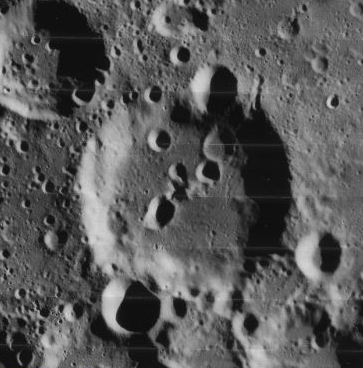}
\center{Jacobi}
\end{minipage}
\hfill
\begin{minipage}[b]{0.45\linewidth}
\includegraphics[height=.9\textwidth]{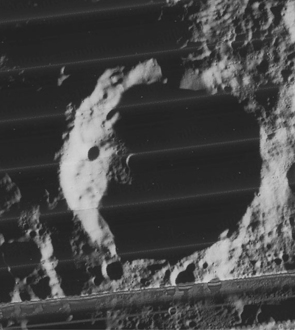}
\center{Sylvester }
\end{minipage}
\end{figure}

\subsection{Tridiagonal matrices, continued fractions and signatures}\label{tri}

In 1853 Sylvester~\cite{sylvester1853} established a link between the 1829 theorem of Sturm on
the number of real roots of a real polynomial, the signature of a tridiagonal matrix and improper continued fractions. We shall now describe the connections between tridiagonal matrices, continued fractions and signatures, going on to discuss the links with Sturm's theorem in section~\ref{sturm} below.

\begin{definition}
The {\it tridiagonal symmetric $n \times n$ matrix} of $\chi=(\chi_1,\chi_2,\ldots,\chi_n)\in \RR^n$ is
$${\rm Tri}(\chi)=\begin{pmatrix}
\chi_1 & 1 & 0 & \dots & 0 & 0 \\
1 & \chi_2 & 1 & \dots & 0 & 0 \\
0 & 1& \chi_3 & \dots & 0 & 0 \\
\vdots & \vdots & \vdots & \ddots & \vdots & \vdots \\
0 & 0 & 0 & \dots  & \chi_{n-1} & 1\\
0 & 0 & 0 & \dots & 1 & \chi_n
\end{pmatrix}.$$
\end{definition}

The study of tridiagonal matrices was initiated by Jacobi~\cite{jacobi1}, who proved that every symmetric matrix over $\ZZ$ is linearly congruent (over $\ZZ$)
to a tridiagonal symmetric matrix. Tridiagonal matrices are ubiquitous in mathematics,  in recurrences, continued fractions, Sturm theory,  numerical analysis, orthogonal polynomials, integrable systems and in the Hirzebruch-Jung resolution of singularities.

\begin{definition} A vector $\chi=(\chi_1,\chi_2,\ldots,\chi_n)\in \RR^n$ is {\it regular} if
${\rm Tri}(\chi)$ is regular in the sense of Definition~\ref{regular}, that is if
$$\mu_k({\rm Tri}(\chi))\neq 0 \in \RR~(1 \leqslant k \leqslant n)~.$$
\end{definition}

\begin{proposition} \label{signtri0}
Given a regular $\chi=(\chi_1,\chi_2,\ldots,\chi_n)\in \RR^n$ let $\mu=(\mu_0,\mu_1,\dots,\mu_n)$ with $\mu_0=1$
and
$$\mu_k~=~\mu_k({\rm Tri}(\chi))~=~{\rm det}({\rm Tri}(\chi_1,\chi_2,\dots,\chi_k))$$
the $k$th minor of ${\rm Tri}(\chi)$. The signature of ${\rm Tri}(\chi)$ is
$$\begin{array}{ll}
\tau({\rm Tri}(\chi))&=~\sum\limits^n_{k=1}{\rm sign}(\mu_k/\mu_{k-1})\\[1ex]
&=~n-2\,{\rm var}(\mu) \in \{-n,-n+1,\dots,n-1,n\}~.
\end{array}$$
\end{proposition} 
\begin{proof} Immediate from the Sylvester-Jacobi-Gundelfinger-Frobenius Theorem \ref{sjgf}. More directly, as in the proof write
 $${\rm Tri}(\chi_1,\chi_2,\dots,\chi_n)~=~\begin{pmatrix} S & v^* \\ v & w \end{pmatrix}$$
 with $S={\rm Tri}(\chi_1,\chi_2,\dots,\chi_{n-1})$, $v=(0 \dots 0~1)$,
 $w=\chi_n$. Expanding the determinant $\mu_n={\rm det}({\rm Tri}(\chi_1,\chi_2,\dots,\chi_n))$ by the last  column gives
 $$\mu_n~=~\chi_n \mu_{n-1}-\mu_{n-2}~(\mu_{-1}=0)$$
so that 
 $$w-vS^{-1}v^*~=~\chi_n- \mu_{n-2}/\mu_{n-1}~=~\mu_n/\mu_{n-1}~.$$
The invertible $n \times n$ matrix
 $$A_n~=~\begin{pmatrix} I_{n-1} & -S^{-1} v^* \\  0 & 1 \end{pmatrix}$$
 is such that
 $$A_n^*{\rm Tri}(\chi_1,\chi_2,\dots,\chi_n)A_n~=~\begin{pmatrix}
 {\rm Tri}(\chi_1,\chi_2,\dots,\chi_{n-1}) & 0 \\ 0 & \mu_n/\mu_{n-1}
 \end{pmatrix}~.$$
with  
 $$-S^{-1}v^*~=~\begin{pmatrix} (-1)^{n-1}\mu_0/\mu_{n-1} \\
 (-1)^{n-2}\mu_1/\mu_{n-1} \\
 \vdots \\
-\mu_{n-2}/\mu_{n-1}
\end{pmatrix}$$
(cf. Usmani~\cite{usmani}). The product invertible $n \times n$  matrix
$$\begin{array}{ll}
A&=~A_n(A_{n-1} \oplus I_1) \dots (A_1 \oplus I_{n-1})\\[1ex]
&=~\begin{pmatrix} 
1 & -\mu_0/\mu_1 & \mu_0/\mu_2 & \dots & (-1)^{n-1}\mu_0/\mu_{n-1}\\ 0 & 1 & -\mu_1/\mu_2 & \dots & (-1)^{n-2}\mu_1/\mu_{n-1}\\
0 & 0 & 1 & \dots & (-1)^{n-3}\mu_2/\mu_{n-1} \\
\vdots & \vdots & \vdots & \ddots & \vdots \\
0 & 0 & 0 &\dots & 1 
\end{pmatrix}
\end{array}$$
is such that
$$A^*{\rm Tri}(\chi_1,\chi_2,\dots,\chi_n)A~=~
\begin{pmatrix} \mu_1/\mu_0 & 0 & \dots & 0\\
0 & \mu_2/\mu_1 & \dots & 0 \\
\vdots & \vdots & \ddots & \vdots \\
0 & 0 & \dots &\mu_n/\mu_{n-1} \end{pmatrix}~.$$
Now apply the Sylvester Law of Inertia \ref{sylvester}.
\end{proof}

The {\it proper continued fraction} of a vector
$\chi=(\chi_1,\chi_2,\dots,\chi_n) \in \RR^n$ is the real number
$$\chi_1+\dfrac{1}{\chi_2+\dfrac{1}{\chi_3+\ddots\lower10pt\hbox{$+\dfrac{1}{\chi_n}$} }} \in \RR$$
assuming there are no divisions by 0.

Continued fractions have a long and distinguished history -- we refer to Karpenkov~\cite{karpenkov} for a recent account.

However, in dealing  with  signatures we need to be more concerned with improper continued fractions, following Sylvester.
The {\it improper continued fraction} of  a vector $\chi=(\chi_1,\chi_2,\dots,\chi_n) \in \RR^n$
is the real number
$$[\chi_1,\chi_2,\dots,\chi_n]=\chi_1-\dfrac{1}{\chi_2-\dfrac{1}{\chi_3-\ddots\lower10pt\hbox{$-\dfrac{1}{\chi_n}$} }} \in \RR$$
assuming there are no divisions by 0.

\begin{definition} The {\it reverse} of a vector $\chi=(\chi_1,\chi_2,\dots,\chi_n) \in \RR^n$ is the
vector $\chi^*=(\chi^*_1,\chi^*_2,\dots,\chi^*_n) \in \RR^n$ with
$$\chi^*_k~=~\chi_{n-k+1}~( 1 \leqslant k \leqslant n)~.$$
\end{definition}

The relationship between the tridiagonal matrices ${\rm Tri}(\chi)$ 
and ${\rm Tri}(\chi^*)$ was the key to Sylvester's reformulation of Sturm's 
root-counting formula in terms of signatures.

\begin{proposition} \label{minor}
Let $\chi=(\chi_1,\chi_2,\dots,\chi_n) \in \RR^n$, and let
$$\begin{array}{l}
\mu_k~=~\mu_k({\rm Tri}(\chi))~=~{\rm det}({\rm Tri}(\chi_1,\chi_2,\dots,\chi_k))\\[1ex]
\mu^*_k~=~\mu_k({\rm Tri}(\chi^*))~=~{\rm det}({\rm Tri}(\chi_n,\chi_{n-1},\dots,\chi_{n-k+1}))~,
\end{array}$$
so that
$$\chi_k \mu_{k-1}~=~\mu_k+\mu_{k-2}~,~\chi^*_k \mu^*_{k-1}~=~\mu^*_k+\mu^*_{k-2}~(1 \leqslant k \leqslant n)$$
with $\mu_0=\mu^*_0=1$, $\mu_{-1}=\mu^*_{-1}=0$.\\
1. The improper continued fractions $[\chi_1,\chi_2,\dots,\chi_k]$  $(1 \leqslant k \leqslant n)$ are well-defined if and only if
$${\rm det}({\rm Tri}(\chi_2,\chi_3,\dots,\chi_k)) \neq 0~(2 \leqslant k \leqslant n)~,$$
in which case 
$$[\chi_1,\chi_2,\dots,\chi_k]~=~
\dfrac{\mu_k}{{\rm det}({\rm Tri}(\chi_2,\chi_3,\dots,\chi_k))}~.$$
Similarly for the reverse $\chi^*\in \RR^n$, with
$$[\chi_n,\chi_{n-1},\dots,\chi_{n-k+1}]~=~
\dfrac{\mu^*_k}{{\rm det}({\rm Tri}(\chi_{n-1},\chi_{n-2},\dots,\chi_{n-k+1}))}~.$$
2. $\chi$ is regular if and only if
$$\chi_k \neq \mu_{k-2}/ \mu_{k-1}\in \RR~(1 \leqslant k \leqslant n),$$
if and only if each  $[\chi_k,\chi_{k-1},\dots,\chi_1]$ $(1 \leqslant k \leqslant n)$ is well-defined. If these conditions are satisfied
$$\begin{array}{ll}
\mu_k/\mu_{k-1}&=\dfrac{{\rm det}({\rm Tri}(\chi_1,\chi_2,\dots,\chi_k))}
{{\rm det}({\rm Tri}(\chi_1,\chi_2,\dots,\chi_{k-1}))}\\[1ex]
&= [\chi_k,\chi_{k-1},\dots,\chi_1] \neq 0 \in \RR~(1 \leqslant k \leqslant n).
\end{array}$$
Similarly for $\chi^*$.\\
\end{proposition}
\begin{proof} 1. By induction on $k$, noting that
$$\begin{array}{l}
[\chi_1,\chi_2,\dots,\chi_k]~=~\chi_1-[\chi_2,\chi_3,\dots,\chi_k]^{-1}~,\\[1ex]
{\rm det}({\rm Tri}(\chi_1,\chi_2,\dots,\chi_k))\\[1ex]
\hskip75pt =~
\chi_1{\rm det}({\rm Tri}(\chi_2,\chi_3,\dots,\chi_k))-{\rm det}({\rm Tri}(\chi_3,\chi_4,\dots,\chi_k))~.
\end{array}$$
2. Immediate from the identity
$$\chi_k - \mu_{k-2}/\mu_{k-1}~=~\mu_k/\mu_{k-1}~=~[\chi_k,\chi_{k-1},\dots,\chi_1]~.$$
\end{proof}

The connection between tridiagonal matrices, continued fractions and variations was established by:

\begin{theorem} [Duality Theorem, Sylvester \cite{sylvester1853}]\label{sylvestertheorem}
Let $\chi=(\chi_1,\chi_2,\dots,\chi_n)\in \RR^n$ and its reverse  $\chi^*=(\chi_n,\chi_{n-1},\dots,\chi_1) \in \RR^n$  be regular.
The tridiagonal symmetric matrices 
$${\rm Tri}(\chi)=\begin{pmatrix}
\chi_1 & 1 & \dots &  0 \\
1 & \chi_2 & \dots & 0 \\
\vdots & \vdots & \ddots &  \vdots \\
0 & 0 & \dots &  \chi_n
\end{pmatrix}~,~{\rm Tri}(\chi^*)=\begin{pmatrix}
\chi_n & 1 &  \dots &  0 \\
1 & \chi_{n-1} &  \dots &  0 \\
\vdots & \vdots & \ddots &  \vdots \\
0 & 0 & \dots & \chi_1
\end{pmatrix}$$
are orthogonally congruent, and have the same signatures
$$\tau({\rm Tri}(\chi))~=~\tau({\rm Tri}(\chi^*)) \in \ZZ~.$$
Furthermore, the variations of the sequences $\mu=(\mu_0,\mu_1,\dots,\mu_n)$, $\mu^*=(\mu^*_0,\mu^*_1,\allowbreak\dots,\mu^*_n)$ of minors 
$$\mu_k~=~{\rm det}({\rm Tri}(\chi_1,\chi_2,\dots,\chi_k))~,~
\mu^*_k~=~{\rm det}({\rm Tri}(\chi_n,\chi_{n-1},\dots,\chi_{n-k+1}))$$ 
$($with $\mu_0=\mu^*_0=1$$)$ are the same
$${\rm var}(\mu)~=~(n-\tau({\rm Tri}(\chi)))/2~=~(n-\tau({\rm Tri}(\chi^*)))/2~=~{\rm var}(\mu^*)$$
and
$$\begin{array}{ll}
\tau({\rm Tri}(\chi))&=\tau({\rm Tri}(\chi^*))\\[1ex]
&=\sum\limits^n_{k=1}{\rm sign}(\mu_k/\mu_{k-1})=\sum\limits^n_{k=1}{\rm sign}([\chi_k,\chi_{k-1},\dots,\chi_1])\\
&=\sum\limits^n_{k=1}{\rm sign}(\mu^*_k/\mu^*_{k-1})=\sum\limits^n_{k=1}{\rm sign}([\chi_{n-k+1},\chi_{n-k+2},\dots,\chi_n])\\
&=~n-2\,{\rm var}(\mu)~=~n-2\,{\rm var}(\mu^*) \in   \{-n,-n+1,\dots,n\}
\end{array}$$
\end{theorem}
\begin{proof} The orthogonal $n \times n$ matrix
$$J=\begin{pmatrix} 0 & 0 & \dots & 0 & 1 \\
0 & 0 &  \dots & 1 & 0 \\
\vdots & \vdots & \ddots & \vdots & \vdots \\
0 & 1 & \dots & 0 & 0 \\
1 & 0 & \dots & 0 & 0 \end{pmatrix}$$
is such that
$${\rm Tri}(\chi^*)=J^*{\rm Tri}(\chi)J~.$$
Thus ${\rm Tri}(\chi)$ and ${\rm Tri}(\chi^*)$ are orthogonally congruent, the eigenvalues are the same by the Spectral Theorem \ref{spectral}, and hence so are the signatures. Finally, apply the identities given by Proposition \ref{signtri0}
$$\begin{array}{ll}
\tau({\rm Tri}(\chi))&=\sum\limits^n_{k=1}{\rm sign}(\mu_k/\mu_{k-1})=\sum\limits^n_{k=1}{\rm sign}([\chi_k,\chi_{k-1},\dots,\chi_1])\\
&=~n-2\,{\rm var}(\mu)\in   \{-n,-n+1,\dots,n\}~.
\end{array}$$
\end{proof}

Note that the signature $\tau({\rm Tri}(\chi))$ is  {\it not} invariant under the permutation of the indices in $\chi$. For  any permutation $\sigma\in \Sigma_n$ of $\{1,2, \dots, n\}$ let
$$\chi_\sigma=(\chi_{\sigma(1)},\chi_{\sigma(2)},\dots,\chi_{\sigma(n)}) \in \RR^n.$$
In general, ${\rm Tri}(\chi_{\sigma})$ is not linearly congruent to ${\rm Tri}(\chi)$,
but it is the case for $\sigma(k)=n-k+1$.

Here is how the  {\it signaletic equivalence} ${\rm var}(\mu)={\rm var}(\mu^*)$ 
of the Duality Theorem \ref{sylvestertheorem} appeared in Sylvester~\cite{sylvester1853} 
\begin{center}
\includegraphics[width= .9\textwidth]{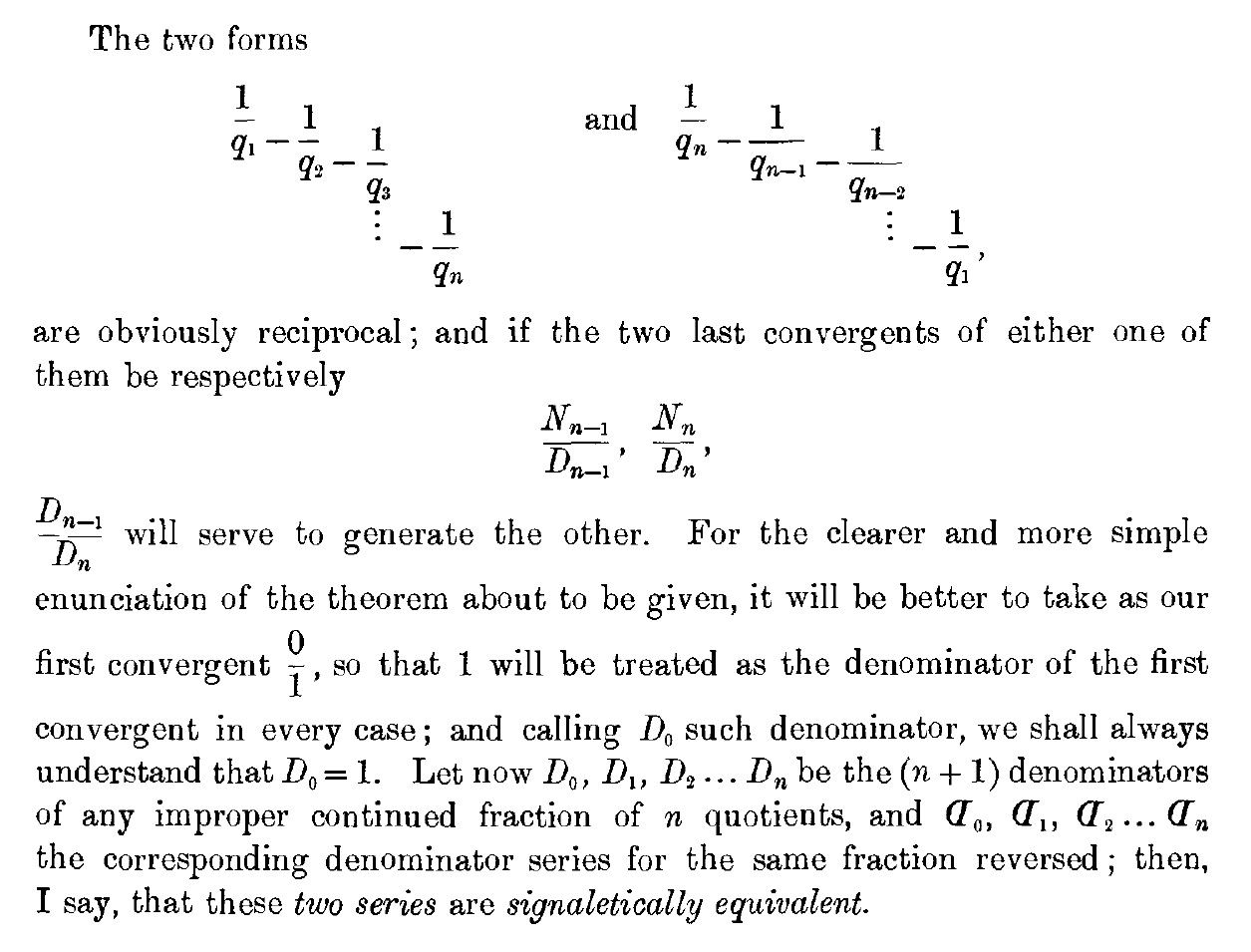}
\end{center}

Since we adopted a different convention for an improper continued fraction, Sylvester's denominators
are our numerators. Specifically, to translate from the language of ~\cite{sylvester1853} into that of the Duality Theorem
\ref{sylvestertheorem} set $(\chi_1,\chi_2,\dots,\chi_n)=(q_1,q_2,\dots,q_n)$ and note that
$$\begin{array}{l}
D_k~=~\mu_k~=~\hbox{\rm numerator of}~[q_1,q_2,\dots,q_k]\\[1ex]
\hskip75pt =\hbox{\rm denominator of}~ [q_1,q_2,\dots,q_n]^{-1}~,\\[1ex]
 \raisebox{\depth}{\scalebox{-1}[-1]{$D$}}_{\hskip-1pt k}~=~\mu^*_k~=~\hbox{\rm numerator of}~[q_n,q_{n-1},\dots,q_{n-k+1}]\\[1ex]
\hskip75pt =\hbox{\rm denominator of}~[q_n,q_{n-1},\dots,q_{n-k+1}]^{-1}~.
\end{array}$$

Sylvester found the invariance ${\rm var}(\mu)={\rm var}(\mu^*)$
remarkable and even worked out a proof ``by hand'' when $n=3$.
Concretely, it means that if $A,B,C$ are three real numbers, the following two sequences have the same number of sign variations 
$$\begin{array}{l}
1,A,BA-1,CBA-C-A\\[1ex]
1,C,BC-1,ABC-A-C.
\end{array}
$$
We leave  this challenging exercise for the reader.

And here is an extract from~\cite{sylvester1853} celebrating the
moment of conception of the identity ${\rm var}(\mu)={\rm var}(\mu^*)$:
\begin{center}
\includegraphics[width=.9\textwidth]{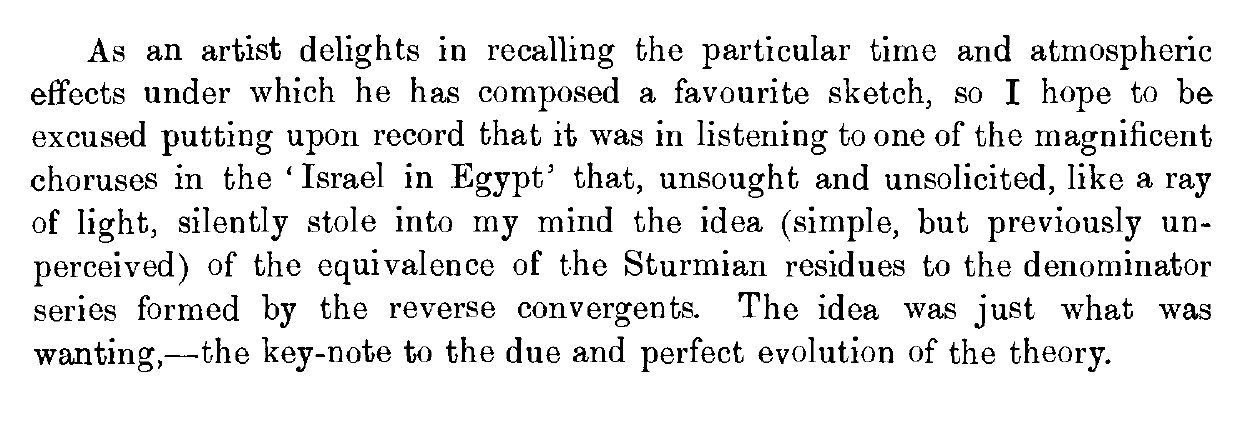}
\end{center}
Sylvester's Duality Theorem~\ref{sylvestertheorem} allowed him to interpret the sign variations provided by the Euclidean algorithm in the original proof of Sturm's theorem by signatures, which we explore further in the next section. The key to this interpretation was the observation of the invariance of $\tau({\rm Tri}(\chi))$ under the permutation $\chi=(\chi_1,\chi_2,\dots,\chi_n) \to \chi^*=(\chi_n,\chi_{n-1},\dots,\chi_1)$.

We now recall the (well-known) connection between proper continued fractions and the Euclidean algorithm:

\begin{example} \label{euclid1}
Given integers $p_0 \geqslant  p_1\geqslant 1$  let $q_1,q_2,\dots,q_n \geqslant 1$ be the successive quotients and $p_2,\dots,p_n \geqslant 1$ the successive remainders in the iterations of the Euclidean algorithm
$$\begin{array}{l}
p_{k-1}=p_kq_k+p_{k+1} \in \ZZ~(1 \leqslant k \leqslant n)~,~0 \leqslant  p_{k+1} <  p_k ,~q_k=\lfloor p_{k-1}/p_k \rfloor,\\
p_n =~\text{greatest common divisor}(p_0,p_1),~p_{n+1}=0.
\end{array}
$$
The vectors $(p_0/p_1,p_1/p_2,\dots,p_{n-1}/p_n)\in \QQ^n$,  $(q_1,q_2,\dots,q_n)\in \ZZ^n$ determine each other by the proper continued fraction
$$p_{k-1}/p_k=q_k+\dfrac{1}{q_{k+1}+\dfrac{1}{q_{k+2}+\ddots\lower10pt\hbox{$+\dfrac{1}{q_n}$} }} \in \QQ$$
and the recurrence
$$q_k=p_{k-1}/p_k - p_{k+1}/p_k \in \ZZ~(1 \leqslant k \leqslant n)~.$$
In terms of matrices, the recurrence is
$$\begin{pmatrix} p_k\\ p_{k+1} \end{pmatrix}=
\begin{pmatrix} 0 & 1 \\ 1  & -q_k \end{pmatrix}
\begin{pmatrix} 0 & 1 \\ 1  & -q_{k-1} \end{pmatrix}
\dots
\begin{pmatrix} 0 & 1 \\ 1  & -q_1 \end{pmatrix}
\begin{pmatrix} p_0 \\ p_1 \end{pmatrix}.
$$
\end{example}

\begin{proposition} \label{recur} {\rm  Let $\chi=(\chi_1,\dots,\chi_n) \in \RR^n$, $\mu=(\mu_0,\mu_1,\mu_2,\dots,\mu_n) \in \RR^{n+1}$
be vectors satisfying the recurrence
$$\mu_k=\mu_{k-1}\chi_k-\mu_{k-2}~(1 \leqslant k \leqslant n)~{\rm with}~\mu_{-1}=0,\,\mu_0 \neq 0 \in \RR.$$
The continued fractions $[\chi_k,\chi_{k-1},\dots,\chi_1]$ are well-defined if and only if $\mu_k \neq 0\in \RR$, in which case there are identities
$$\mu_k/\mu_{k-1}=[\chi_k,\chi_{k-1},\dots,\chi_1] \in \RR~(1 \leqslant k \leqslant n).$$
The vectors $(\mu_1/\mu_0,\mu_2/\mu_1,\dots,\mu_n/\mu_{n-1})\in \RR^n$,  $(\chi_1,\chi_2,\dots,\chi_n)\in \RR^n$ determine each other by
the improper continued fraction
$$\mu_k/\mu_{k-1}=[\chi_k,\chi_{k-1},\dots,\chi_1]$$
and the recurrence
$$\chi_k=\mu_k/\mu_{k-1} +\mu_{k-2}/\mu_{k-1} \in \RR~(1 \leqslant k \leqslant n)~.$$}
\end{proposition}
\begin{proof} By induction on $n$, using the identities
$$\begin{array}{l}
\mu_1/\mu_0=\chi_1,~\mu_2/\mu_0=\chi_2\chi_1-1=[\chi_2,\chi_1][\chi_1],\\
1+[\chi_k,\chi_{k-1},\dots,\chi_1][\chi_{k-1},\chi_{k-2},\dots,\chi_1]=[\chi_{k-1},\chi_{k-2},\dots,\chi_1]\chi_k~(2 \leqslant k \leqslant n).
\end{array}$$
\end{proof}

In terms of matrices, the recurrence in Proposition \ref{recur} is
$$\begin{pmatrix} \mu_k \\ \mu_{k-1} \end{pmatrix}=
\begin{pmatrix} \chi_k & -1 \\ 1  & 0 \end{pmatrix}
\begin{pmatrix} \chi_{k-1} & -1 \\ 1  & 0 \end{pmatrix}
\dots
\begin{pmatrix} \chi_1 & -1 \\ 1  & 0 \end{pmatrix}
\begin{pmatrix} \mu_0 \\ 0 \end{pmatrix}.
$$

\begin{example} \label{euclid2}
Given integers $p_0,p_1 \in \ZZ\backslash \{0\}$ let $q_1,q_2,\dots,q_n\in \ZZ \backslash \{0\}$ be the successive quotients and $p_2,\dots,p_n \in \ZZ\backslash \{0\}$ the successive remainders in the iterations of the Euclidean algorithm modified by a change of sign\footnote{In the conventional Euclidean algorithm (Example \ref{euclid1}) it would have been $p_{k+1}=-(p_kq_k-p_{k-1})$, with $q_k$ the quotient when dividing $p_{k-1}$ by $p_k$ and $p_{k+1}$ the remainder, i.e. $p_{k-1}/p_k=q_k+p_{k+1}/p_k$.}
$$\begin{array}{l}
p_{k+1}=p_kq_k-p_{k-1}~(1 \leqslant k \leqslant n),~0 \leqslant \vert p_{k+1} \vert < \vert p_k \vert~\\
\vert p_n\vert~ =~\text{greatest common divisor}(\vert p_0\vert ,\vert p_1\vert ),~p_{n+1}=0.
\end{array}
$$
In terms of matrices, the recurrence is
$$
\begin{array}{ll}
\begin{pmatrix} p_k\\ p_{k+1} \end{pmatrix}&=
\begin{pmatrix} 0 & 1 \\ -1  & q_k \end{pmatrix}
\begin{pmatrix} 0 & 1 \\ -1  & q_{k-1} \end{pmatrix}
\dots
\begin{pmatrix} 0 & 1 \\ -1  & q_1 \end{pmatrix}
\begin{pmatrix} p_0 \\ p_1 \end{pmatrix}\\[2ex]
&=
\begin{pmatrix}
-{\rm det}({\rm Tri}(q_2,q_3,\dots,q_{k-1}))&
{\rm det}({\rm Tri}(q_1,q_2,\dots,q_{k-1}))\\
-{\rm det}({\rm Tri}(q_2,q_3,\dots,q_k))&
{\rm det}({\rm Tri}(q_1,q_2,\dots,q_k))
\end{pmatrix}
\begin{pmatrix} p_0  \\ p_1 \end{pmatrix}
\end{array}
$$
with inverse
$$\begin{array}{ll}
\begin{pmatrix} p_0 \\ p_1 \end{pmatrix}&=
\begin{pmatrix} q_1 & -1 \\ 1  & 0 \end{pmatrix}
\begin{pmatrix}q_2 & -1 \\ 1  & 0 \end{pmatrix}
\dots
\begin{pmatrix} q_k  & -1 \\ 1 & 0  \end{pmatrix}
\begin{pmatrix} p_k  \\ p_{k+1} \end{pmatrix}\\[2ex]
&=
\begin{pmatrix}
{\rm det}({\rm Tri}(q_1,q_2,\dots,q_k))&
-{\rm det}({\rm Tri}(q_1,q_2,\dots,q_{k-1}))\\
{\rm det}({\rm Tri}(q_2,q_3,\dots,q_k))&
-{\rm det}({\rm Tri}(q_2,q_3,\dots,q_{k-1}))
\end{pmatrix}
\begin{pmatrix} p_k  \\ p_{k+1} \end{pmatrix}
\end{array}$$
The vectors
$$\begin{array}{l}
\chi=(\chi_1,\chi_2,\dots,\chi_n)=(q_n,q_{n-1},\dots,q_1) \in \ZZ^n,\\
\mu=(\mu_0,\mu_1,\dots,\mu_n)=(p_n,p_{n-1},\dots,p_0) \in \ZZ^{n+1}
\end{array}
$$
satisfy the hypothesis of Proposition~\ref{recur}, so that
$$p_k/p_{k+1}=[q_{k+1},q_{k+2},\dots,q_n] \in \QQ~(0 \leqslant k \leqslant n-1).$$
\end{example}

\begin{proposition} \label{tri2} \leavevmode
1. For any $\chi=(\chi_1,\chi_2,\dots,\chi_n) \in \RR^n$ the principal minors
$\mu_k=\mu_k({\rm Tri}(\chi)) \in \RR$ satisfy the recurrence of Proposition~\ref{recur}
$$\mu_{k-2}+\mu_k=\chi_k\mu_{k-1}~(1 \leqslant k \leqslant n)$$
with  $\mu_{-1}=0$, $\mu_0 =1$. \\
2. For any $p=(p_0,p_1,\dots,p_n) \in (\RR\backslash \{0\})^{n+1}$, $q=(q_1,q_2,\dots,q_n) \in (\RR\backslash \{0\})^n$ satisfying
$$\begin{array}{l}
p_{k+1}=p_kq_k-p_{k-1}~(1 \leqslant k \leqslant n),~p_{n+1}=0
\end{array}
$$
the vectors
$$\begin{array}{l}
\chi=(\chi_1,\chi_2,\dots,\chi_n)=(q_n,q_{n-1},\dots,q_1) \in (\RR\backslash \{0\})^n,\\[1ex]
\mu=(\mu_0,\mu_1,\dots,\mu_n)=(1,p_{n-1}/p_n,\dots,p_0/p_n) \in (\RR\backslash \{0\})^{n+1}
\end{array}
$$
satisfy the recurrences of 1., with ${\rm Tri}(\chi)$ regular and 
$$\begin{array}{ll}
\mu_k&=\mu_k({\rm Tri}(\chi))={\rm det}({\rm Tri}(\chi_1,\chi_2,\dots,\chi_k))\\[1ex]
&={\rm det}({\rm Tri}(q_n,q_{n-1},\dots,q_{n-k+1}))=p_{n-k}/p_n
\end{array}$$ 
such that
$$\begin{array}{l}
\mu_{n-k}/\mu_{n-k-1}=[\chi_{n-k},\chi_{n-k-1},\dots,\chi_1]\\[1ex]
\hphantom{\mu_{n-k}/\mu_{n-k-1}}=[q_{k+1},q_{k+2},\dots,q_n]\\[1ex]
\hphantom{\mu_{n-k}/\mu_{n-k-1}}=\dfrac{{\rm det}({\rm Tri}(q_{k+1},q_{k+2},\dots,q_n))}{{\rm det}({\rm Tri}(q_{k+2},q_{k+3},\dots,q_n))}\\[2ex]
\hphantom{\mu_{n-k}/\mu_{n-k-1}}=p_k/p_{k+1} \neq 0 \in \RR~(0 \leqslant k \leqslant n-1).
\end{array}$$
In particular,
$$\begin{array}{ll}
\mu_n&=\mu_n({\rm Tri}(\chi))={\rm det}({\rm Tri}(\chi_1,\chi_2,\dots,\chi_n))\\[1ex]
&={\rm det}({\rm Tri}(q_n,q_{n-1},\dots,q_1))={\rm det}({\rm Tri}(q_1,q_2,\dots,q_n))\\[1ex]
&=p_0/p_n.
\end{array}$$
3. For $p,q$ as in 2. define
$$p^*_k~=~\begin{cases} 1&\hbox{\it if}~k=0 \\
{\rm det}({\rm Tri}(q_1,q_2,\dots,q_k))&\hbox{\it if}~ 1 \leqslant k \leqslant n.
\end{cases}$$
Then
$$[q_1,q_2,\dots,q_k]~=~\dfrac{p^*_k}{{\rm det}({\rm Tri}(q_2,q_3,\dots,q_k))}$$
and by  the Sylvester-Jacobi-Gundelfinger-Frobenius Theorem \ref{sjgf} and the
Sylvester Duality Theorem \ref{sylvestertheorem}
$$\tau({\rm Tri}(q))~=~n-2\,{\rm var}(p_0,p_1,\dots,p_n)~=~
n-2\,{\rm var}(p_0^*,p_1^*,\dots,p^*_n)~.$$
\end{proposition}

\subsection{Sturm's theorem and its reformulation by Sylvester}\label{sturm}

\bigskip
\begin{center}
\includegraphics[width=.6 \textwidth]{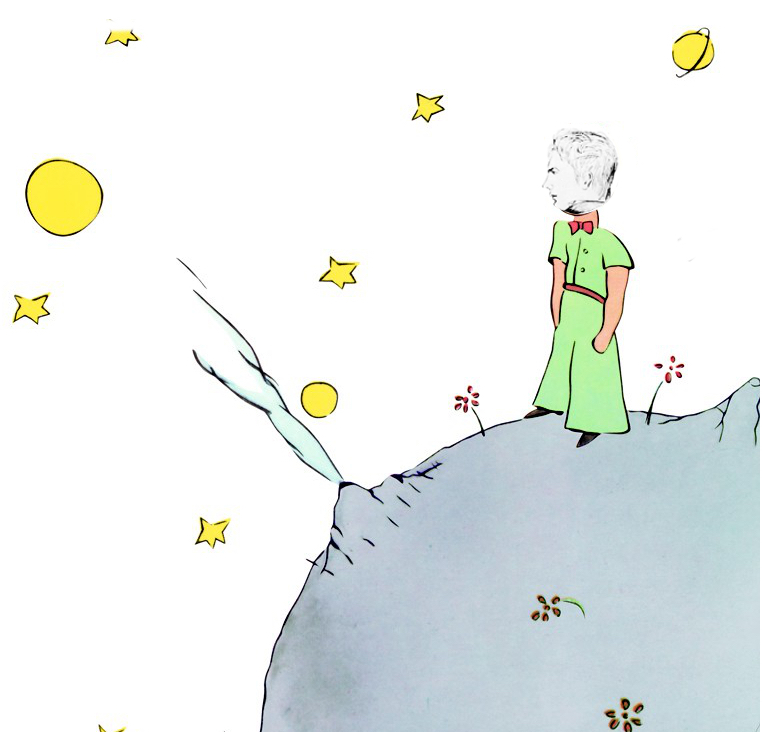}

Jacques Charles Fran\c{c}ois {\bf Sturm}, ForMemRS (1803-1855)

on ``his'' asteroid ``Sturm 31043'', almost 4 million km from the Sun

(Artist's impression)

\end{center}
\bigskip

Many contemporary mathematicians have forgotten the numerical aspect of theoretical algebra.
If a polynomial with real coefficients is given, how can one determine \emph{in practice}  its roots with a given accuracy?
Modern computers give us the feeling that it suffices to type the command ``Solve $P(x)=0$'' to get an immediate answer.
As a matter of fact, treatises on Algebra, at least until the end of the nineteenth century, placed a strong emphasis on this numerical problem.
For instance, the classical ``Cours d'alg\`ebre sup\'erieure'', by J.A.~Serret~\cite{serret}, dated 1877, is one of the first textbooks with a thorough presentation of Galois theory. It contains several chapters on the numerical aspect and splits the topics in two distinct problems.
The first is the \emph{separation of roots}: one has to count the number of roots in a given interval, in order to locate intervals containing a single root.
The second consists of various numerical methods enabling to shrink such an interval to any desirable length.
Concerning the separation problem, there is no doubt that the most impressive theorem is due to Sturm.

\begin{quote}
{\it L'alg\`ebre offrait une lacune regrettable, mais cette lacune se trouva combl\'ee de la mani\`ere la plus heureuse par le fameux th\'eor\`eme de Sturm. Ce grand g\'eom\`etre communiqua \`a l'Acad\'emie des Sciences, en 1829, la d\'emonstration de son th\'eor\`eme qui constitue l'une des plus brillantes d\'ecouvertes dont se soit enrichie l'Analyse math\'ematique~\cite{serret}.}
\end{quote}

\begin{quote}
{\it Le th\'eor\`eme de Sturm a eu le bonheur de devenir imm\'ediatement classique et de prendre dans l'enseignement une place qu'il conservera toujours~\cite[page 291]{hermite1890}.}
\end{quote}

Amazingly, one of the ``most brilliant discoveries in Analysis'' is only familiar today to a minority of mathematicians.

\begin{definition} \label{sturmfunctions}
{\rm Let $P(X) \in \RR[X]$ be a degree $n$ polynomial.
\begin{enumerate}
\item The {\it Sturm functions} of $P(X)$ are the sequences of polynomials
$$P_*(X)=(P_0(X),P_1(X),\dots,P_n(X)),~Q_*(X)=(Q_1(X),Q_2(X),\dots,Q_n(X))$$
occurring as the successive remainders and quotients in the iterations of the Euclidean algorithm modified by a change of sign\footnote{cf. Example \ref{euclid2}. Again, in the conventional Euclidean algorithm (Example \ref{euclid1}) it would have been $P_{k+1}(X)=-P_k(X)Q_k(X)+P_{k-1}(X)$, with $Q_k(X)$ the quotient when dividing $P_{k-1}(X)$ by $P_k(X)$ and $P_{k+1}(X)$ the remainder, i.e. $P_{k-1}(X)/P_k(X)=Q_k(X)+P_{k+1}(X)/P_k(X)$.}
$$\begin{array}{l}
P_0(X)=P(X),~P_1(X)=P'(X),\\
P_{k+1}(X)=P_k(X)Q_k(X)-P_{k-1}(X)~(1 \leqslant k \leqslant n),\\
\text{deg}(P_{k+1}(X)) < \text{deg}(P_k(X)),\\
P_n(X)=\text{constant},~ P_{n+1}(X)=0.
\end{array}$$
\item The polynomial $P(X)$ is {\it regular} if the real roots of the polynomials $P_0(X),P_1(X),\allowbreak
\dots,P_n(X)$ are all distinct and non-zero, with $P_n(X) \in \RR \subset \RR[X]$ constant.
In particular $P(X)$ has no repeated roots,
${\rm deg}(P_k(X))=n-k$, and ${\rm deg}(Q_k(X))=1$.
Call $a \in \RR$ {\it regular} if each of $P_0(a),P_1(a),\dots,\allowbreak P_n(a) \in \RR$ is non-zero, in which case the {\it variation} of $a$ is defined to be the number of sign changes in this sequence
$$\text{var}(a)=\text{var}(P_0(a),P_1(a),\dots,P_n(a)) \in \{0,1,\dots,n\}$$
(cf. Definition~\ref{var}).
\end{enumerate}
}
\end{definition}

\begin{theorem}[Sturm 1829,~\cite{sturm}] \label{sturmtheorem}
For a regular degree $n$ polynomial $P(X) \in \RR[X]$ and regular $a,b \in \RR$ with $a <b $ the number of real roots of $P(X) \in \RR[X]$ contained in $[a,b]$ is equal to
${\rm var}(a)-{\rm var}(b) \in \{0,1,\dots,n\}.$
\end{theorem}

\begin{proof} The proof is clever, but not difficult.
The key point is that
$$P_{k+1}(X)+ P_{k-1}(X)=P_k(X)Q_k(X)$$
so that for every real root $x$ of $P_k(X)$, the signs of
$P_{k+1}(x)\neq 0$ and $P_{k-1}(x)\neq 0$ are different $(1 \leqslant k \leqslant n-1)$.

Consider $\text{var}(x)$ as a function of $x \in [a,b]$.
Strictly speaking, it is only defined on the set $U \subset [a,b]$ of regular $x$ which is the complement of the finite set of real roots of some $P_k(X)$ ($0 \leqslant k \leqslant {n-1}$).
On each connected component of $U$, $\text{var}$ is obviously constant.
We therefore have to describe the jumps of $\text{var}(x)$ when $x$ crosses a real root $x_0$ of some $P_k(X)$, with $0 \leqslant k \leqslant n-1$.
Let us compare $\text{var} (x_0-  \epsilon)$ and $\text{var}(x_0 + \epsilon)$ for $\epsilon>0$ small enough.
When $x$ goes from $x_0-\epsilon$ to $x_0+\epsilon$, all components of
$$(P_0(x),P_1(x),\dots,P_n(x))$$
keep the same sign, except the one corresponding to the polynomial $P_k$ for which $P_k(x_0)=0$.

If $1\leqslant k \leqslant n-1$, the signs of the two adjacent components, $P_{k-1}(x),P_{k+1}(x)$ are different.
Therefore, independently of the sign of $P_k(x\pm \epsilon)$, we have $\text{var} (P_{k-1}(x\pm \epsilon), P_{k}(x\pm \epsilon), P_{k+1}(x\pm \epsilon))= 1$.
In particular $\text{var}(x_0-\epsilon) = \text{var}(x_0+ \epsilon)$ so that $\text{var}(x)$ is constant when $x$ crosses $x_0$.

Therefore there can only be a jump in $\text{var}(x)$ when $x$ crosses a root $x_0$ of $P=P_0$.
There are two possibilities: either $P$ is increasing at the real root $x_0$, in which case $f(x_1)>0$ and $P'(x_1)>0$, or $P$ is decreasing at the root $x_0$, in which case $f(x_1)<0$ and $P'(x_1)<0$.
In both cases $\text{var}(x_0-\epsilon)-\text{var}(x_0+\epsilon)=1$.
\end{proof}

We have stated Sturm's theorem in the generic case of a regular polynomial.
We leave to the reader the formulation and proof of the general case, including possible multiple roots.

Let us work out two simple examples.

\begin{example} 1. For $b \in \RR$ let $P(X)=X+b \in \RR[X]$, so that
$$P_0(X)=X+b, P_1(X)=1, Q_1(X)=X+b.$$
Then $P(X)$ is regular, and any $ x \neq -b \in \RR$ is regular. The variation
$${\rm var}(x)={\rm var}(x+b,1)=\begin{cases} {\rm var}(-,+)=1&\text{if $x<-b$} \\ {\rm var}(+,+)=0&\text{if $x>-b$}\end{cases}$$
decreases by 1 as $x \in \RR$ jumps across the real root $-b$ (from below) of $P_0(X)$.\\
2. For $b,c \in \RR$ with $b^2/4-c \neq 0$ let $P(X)=X^2+bX+c \in \RR[X]$, so that
$$\begin{array}{l}
P_0(X)=X^2+bX+c,~P_1(X)=2X+b,~P_2(X)=b^2/4-c,\\
Q_1(X)=(2X+b)/4,~Q_2(X)=(2X+b)/(b^2/4-c).
\end{array}$$
We have that $P_0(X)$ has either 0 or 2 real roots, is regular, and that any $x \in \RR$ with $x^2+bx+c \neq 0$ and $2x+b \neq 0$ is regular, with
$${\rm var}(x)={\rm var}(x^2+bx+c,2x+b,b^2/4-c)=1.$$
If $b^2/4-c>0$ there are two real roots, $x_1<x_2$ say, the variation
$${\rm var}(x)~=~\begin{cases}
{\rm var}(+,-,+)=2&{\rm if}~x<x_1~,\\[1ex]
{\rm var}(-,-,+)=1&{\rm if}~x_1<x < -b/2~,\\[1ex]
{\rm var}(-,+,+)=1&{\rm if}~-b/2<x < x_2~,\\[1ex]
{\rm var}(+,+,+)=0&{\rm if}~x_2<x
\end{cases}$$
decreases by 1 as $x \in \RR$ jumps across a real root $x$ of $P_0(X)$, and remains constant as it jumps across the root $-b/2$ of $P_1(X)$.\\
If $b^2/4-c<0$ there are no real roots, and the variation
$${\rm var}(x)~=~\begin{cases}
{\rm var}(+,-,-)=1&{\rm if}~x < -b/2~,\\[1ex]
{\rm var}(+,+,-)=1&{\rm if}~x>-b/2
\end{cases}$$
remains constant as it jumps across the root $-b/2$ of $P_1(X)$.
\end{example}

In 1853 Sylvester~\cite{sylvester1853} established a link between continued fractions, signatures of tridiagonal matrices and  Sturm's theorem.
\begin{center}
\includegraphics[width=.8\textwidth]{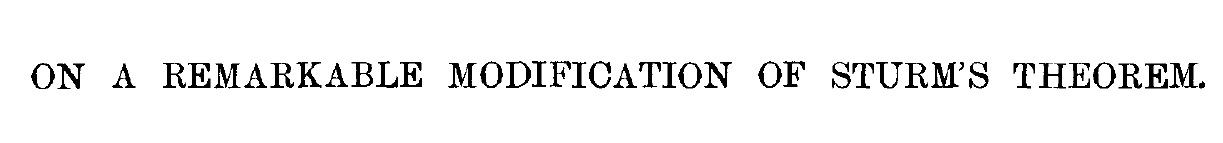}
\end{center}
He noticed that one can use the dictionary between continuous fractions and tridiagonal matrices in the context of polynomials instead of integers, as follows. The Sturm functions  $P_k(X),Q_k(X)$ 
(Definition \ref{sturmfunctions}) are such that
$$\begin{array}{l}
\dfrac{P_k(X)}{P_n(X)}~=~{\rm det}({\rm Tri}(Q_{k+1}(X),Q_{k+2}(X),\dots,Q_n(X)))~,\\[1ex]
\dfrac{P_k(X)}{P_{k+1}(X)}~=~[Q_{k+1}(X),Q_{k+2}(X),\dots,Q_n(X)] \in \RR(X)~(0 \leqslant k \leqslant n-1)
\end{array}
$$
with $P_k(X)$ the numerator in the expression of the $k$th reverse convergent of the improper  continued fraction $[Q_1(X),Q_2(X),\dots,Q_n(X)]\in \RR(X)$ as a quotient of polynomials in $\RR[X]$ in the function field $\RR(X)=(\RR[X]\backslash \{0\})^{-1}\RR[X]$.
By Theorem \ref{sjgf} and Proposition \ref{tri2} the polynomials defined by
$$P^*_k(X)~=~\begin{cases}
1&{\rm if}~k=0\\[1ex]
{\rm det}({\rm Tri}(Q_1(X),Q_2(X),\dots,Q_k(X)))&{\rm if}~1 \leqslant k \leqslant  n
\end{cases}$$
are such that for any regular $a \in \RR$ the signature of the symmetric matrix ${\rm Tri}(Q(a))$ is
$$\tau({\rm Tri}(Q(a)))~=~n-2\,{\rm var}(P^*_0(a),P^*_1(a),\dots,P^*_n(a))~.$$
The polynomials $P^*_1(X),P^*_2(X),\dots,P^*_n(X)$ are the numerators in the expressions of the convergents of $[Q_1(X),Q_2(X),\allowbreak\dots,Q_n(X)]$ as quotients
$$[Q_1(X),Q_2(X),\dots,Q_k(X)]~=~\dfrac{P^*_k(X)}
{{\rm det}({\rm Tri}(Q_2(X),Q_3(X),\dots,Q_k(X)))} \in \RR(X)~.$$

\begin{theorem}[Sylvester 1853] \label{sylvestertheorem2}
For any regular degree $n$ polynomial
$P(X) \in \RR[X]$ with Sturm functions
$(P_0(X),P_1(X),\dots,P_n(X))$, $(Q_1(X),Q_2(X),\dots,\allowbreak Q_n(X))$ define
the symmetric tridiagonal $n \times n$ matrix with entries in $\RR[X]$
$${\rm Tri}(Q)(X)=\begin{pmatrix}
Q_1(X) & 1 & 0 & \dots & 0 & 0 \\
1 & Q_2(X) & 1 & \dots & 0 & 0 \\
0 & 1& Q_3(X) & \dots & 0 & 0 \\
\vdots & \vdots & \vdots & \ddots & \vdots & \vdots \\
0 & 0 & 0 & \dots  & Q_{n-1}(X) & 1\\
0 & 0 & 0 & \dots & 1 & Q_n(X)
\end{pmatrix}.$$
For any regular $a \in \RR$ the signature of the symmetric tridiagonal $n \times n$ matrix ${\rm Tri}(Q(a))$ with entries in $\RR$ is
$$\tau({\rm Tri}(Q)(a))=n-2\,{\rm var}(a) \in \{-n,-n+1,\dots,n\}$$
with ${\rm var}(a)={\rm var}(P_0(a),P_1(a),\dots,P_n(a))$.
\end{theorem}
\begin{proof} 
For any regular $a \in \RR$ by the Sylvester-Jacobi-Gundelfinger-Frobenius Theorem \ref{sjgf} and
the `signaletic equivalence' of Sylvester's Duality Theorem \ref{sylvestertheorem} (cf. Proposition \ref{tri2})
$$\begin{array}{ll}
{\rm var}(a)&=~\text{var}(P_0(a),P_1(a),\dots,P_n(a))~(\hbox{as used in Theorem \ref{sturmtheorem}})\\[1ex]
&=~{\rm var}(P^*_0(a),P^*_1(a)\dots,P^*_n(a))\\[1ex]
&=~(n-\tau({\rm Tri}(Q(a)))/2~(\hbox
{\rm since $\mu_k({\rm Tri}(Q(a))=P^*_k(a)$})~.
\end{array}$$
 \end{proof}

 Theorem \ref{sylvestertheorem2} enabled Sylvester to recast Sturm's theorem in terms of signatures of tridiagonal matrices whose entries are polynomials in one real variable.
The main point is to transform Sturm's theorem into a purely algebraic fact, quite independent of the topology of the field of real numbers.
Starting from a polynomial, one constructs a canonical symmetric matrix whose signature contains the relevant information on the number of real roots.

\begin{corollary}[of Sturm's Theorem~\ref{sturmtheorem} and Sylvester's Theorem~\ref{sylvestertheorem2}]\label{sylvester4}
For any regular degree $n$ polynomial $P(X)$ and regular $a,b \in \RR$ with $a<b$ the number of roots in the interval $[a,b]$ is
$${\rm var}(a)-{\rm var}(b)=(\tau({\rm Tri}(Q(b)))-\tau({\rm Tri}(Q(a))))/2.$$
\end{corollary}
\begin{proof}
The number of real roots of $P(X) \in \RR[X]$ contained in $[a,b]$ is equal to
$$\begin{array}{ll} \text{var}(a)-\text{var}(b)&=~(n-\tau({\rm Tri}(Q(a))))/2 -(n-\tau({\rm Tri}(Q(b))))/2\\
 &=~(\tau({\rm Tri}(Q(b)))-\tau({\rm Tri}(Q(a))))/2 \in \{0,1,\dots,n\}.\end{array}$$
\end{proof}

This new formulation of Sturm's theorem is indeed more algebraic, in the spirit of Sylvester's approach of algebra.
However, strictly speaking the proof of~\ref{sylvestertheorem2} ultimately relies on Sturm's original proof.
Later, we shall reformulate again Sturm's theorem in a completely algebraic way, valid over any field, and independent of Sturm's proof.
 
The book of Barge and Lannes~\cite{bargelannes} is a far-reaching generalization of these remarks, offering algebraic connections between Sturm sequences, the signatures of tridiagonal matrices and Bott periodicity.

Sylvester wrote magisterially of the relationship between algebra and geometry (which we do not wholly share):

\begin{quote}
{\it Aspiring to these wide generalizations, the analysis of quadratic functions soars to a pitch from whence it may look proudly down on the feeble and vain attempts of geometry proper to rise to its level or to emulate it in its flights.} (1850)
\end{quote}

\bigskip
\begin{center}
\includegraphics[width=.3 \textwidth]{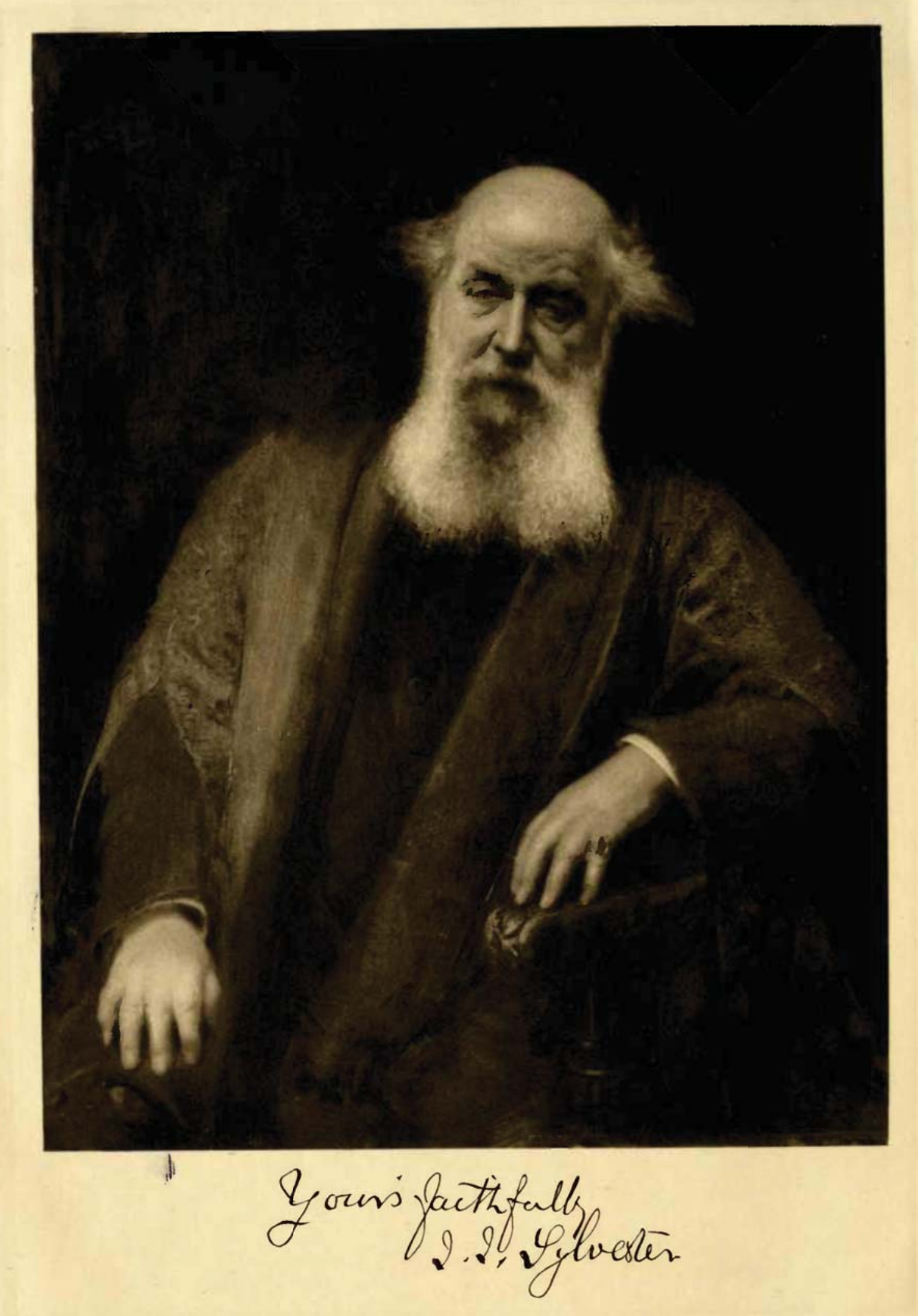}

J.J.Sylvester

Savilian Professor of Geometry, Oxford,

1883-1894
\end{center}
\bigskip

\subsection{Hermite}

\begin{center}
\includegraphics[width=.4 \textwidth]{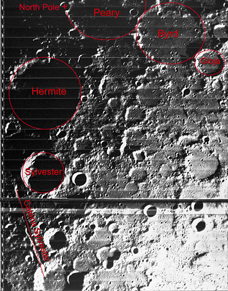}
\hspace*{30pt}
\includegraphics[width=.325 \textwidth]{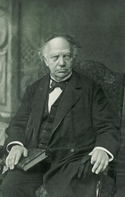}

Charles Hermite (1822 --1901)

\end{center}

\bigskip

Hermite's crater is next to Sylvester's \footnote{et plus grand ! Note d'un des auteurs.}!

There are many other ways of connecting the number of roots of polynomials to the signature of symmetric matrices.
We shall mention only two of them.

A polynomial $P(X) \in \RR[X]$ of degree $n$ with simple roots gives rise to an $n$-dimensional $\RR$-algebra $A_{P(X)}$ of dimension $n$, namely $\RR[X]/P(X)$.
Denote by $\alpha_1,\dots, \alpha_n$ the roots $P(X)$ (real or complex).
Any real $\alpha_i$ gives rise to a linear map 
$$
u\in \RR[X]/P(X) \mapsto u(\alpha_i) \in \RR
$$
and any complex root to a linear map 
$$
u\in \RR[X]/P(X) \mapsto u(\alpha_i) \in \CC.
$$
Since complex roots come in complex conjugate pairs, one gets that the algebra $A_{P(X)}=\RR[X]/P(X)$ is the direct sum of as many copies of $\RR$ as there are real roots and of as many copies of $\CC$ as pairs of complex conjugate pairs of roots.

The multiplication by an element $u$ in $A_{P(X)}$ defines a linear endomorphism $\overline{u}$ of $A_{P(X)}$.
Therefore, $A_{P(X)}$ is canonically equipped with a quadratic form $\phi_{P(X)}$ given by $\phi_{P(X)}(u)= {\rm trace}  (\overline{u}^2)$.
Clearly, any real root gives rise to a $+$ sign in the signature of $\phi_{P(X)}$ and every pair of complex roots gives a $(+,-)$ contribution.
Hence the signature of $\phi_{P(X)}$ is number of real roots.

To compute explicitly the signature of $\phi_{P(X)}$, one can use the natural basis of $A_{P(X)}$ given by $1, X, X^2, ..., X^{n-1}$.
In this basis, the matrix of multiplication by $X$ in $A_{P(X)}$ is the companion matrix $C(P)$ of $P$ and its trace is the sum $\sigma_1$ of the roots of $P$.
In the same way the trace of the multiplication by $X^k$ is the trace of $C(P)^k$, i.e. the sum $\sigma_k$ of the $k$-th powers of the roots of $P$.

Specifically, if $P(X)=X^n-a_{n-1}X^{n-1}-\dots -a_0$, the companion matrix is

$$C(P)=\begin{pmatrix}
0 		& 0    	& 0			& \dots 	& 0 		& a_0 	\\
1		& 0		& 0			& \dots 	& 0		&a_1 	\\
0 		& 1		& 0			& \dots 	& 0		& a_2	\\
\vdots      	& \vdots 	&\vdots 		& \ddots 	& \vdots 	& \vdots 	\\
0		& 0		& 0			& \dots  	& 0		& a_{n-2}	\\
0 		& 0		& 0			& \dots 	&1		& a_{n-1}
\end{pmatrix}
$$
with characteristic polynomial ${\rm det}(XI_n-C(P))=P(X)$.
The matrix of the form on $A_{P(X)}$ in this basis is given by $\phi_{P(X)}(X^i,X^j) =  {\rm trace} (C(P)^{i+j})= \sigma_{i+j}$.

We therefore get the following theorem.

\begin{theorem}[Jacobi-Hermite] \label{hermite}
Let $P(X)$ be a  monic polynomial  of degree $n$ in $\RR[X]$ with distinct roots.
The number of its real roots is equal to the signature of the invertible symmetric $n \times n$
{\rm Jacobi-Hermite matrix}

$$S(P)=\begin{pmatrix}
\sigma_0     	&  \sigma_1    	& \sigma_2 	& \dots 	& \sigma_{n-2} 		& \sigma_{n-1} 	\\
\sigma_1 		& \sigma_2	& \sigma_3 	& \dots 	& \sigma_{n-1} 		& \sigma_n 	\\
\sigma_2 		& \sigma_3 	& \sigma_4 	& \dots 	& \sigma_n 		& \sigma_{n+1} \\
\vdots      		& \vdots 		& \vdots 		& \ddots 	& \vdots 			& \vdots 		\\
\sigma_{n-2} 	& \sigma_{n-1} 	& \sigma_{n}	& \dots  	& \sigma_{2n-4}	& \sigma_{2n-3}\\
\sigma_{n-1} 	& \sigma_n 	& \sigma_{n+1} & \dots 	& \sigma_{2n-3} 	& \sigma_{2n-2}
\end{pmatrix}$$
where $\sigma_k$ denotes the sum of the $k$-th powers of all the roots of $P$. 
\end{theorem}

[See  section \ref{localize} for the significance of the relationship
$S(P)C(P)=C(P)^*S(P)$ between $S(P)$ and $C(P)$.]

As a very simple example let us compute the Jacobi-Hermite matrix for a degree two polynomial $P(X)=X^2+bX+c$.
The companion matrix is
$$C(P)=\begin{pmatrix}
 0 & -c \\
1 &  -b
\end{pmatrix}
$$
so that  ${\rm trace}(C(P)^0)= 2$, ${\rm trace}(C(P)) = -b$ and ${\rm trace}(C(P)^2) = b^2-2c$.
The corresponding Jacobi-Hermite matrix is therefore
$$
S(P)=\begin{pmatrix}
2 & -b  \\
-b &  b^2-2c
\end{pmatrix}
$$
whose signature is $0$ if $b^2-4c<0$ and $2$ if $b^2-4c>0$, confirming high school algebra!

Here is another way of stating the same theorem, closer to Hermite's formulation. 
We still denote by $\alpha_1,\dots, \alpha_n$ the (real or complex) roots of $P(X)$ and we consider the following quadratic form on $\RR^n$:
$$
q(x_1, \dots, x_n)= \sum_{k=1}^n (x_1+ \alpha_k x_2+ \dots + \alpha_k^{n-1} x_n)^2.
$$
Since complex roots come in complex conjugate pairs, this is indeed a real quadratic form.
For the same reason as before, its signature is equal to the number of real roots.
Expanding the sum, one finds 
$$
q(x_1, \dots, x_n)= \sum_{i,j=1}^n \sigma_{i+j-2} x_ix_j
$$
This is Theorem~\ref{sjgf} expressed in more down to earth terminology.

Note that the $n$ linear forms $l_k$ ($1\leqslant k \leqslant n$)
$$
(x_1, \dots, x_n) \in \CC^n \mapsto x_1+ \alpha_k x_2+ \dots + \alpha_k^{n-1} x_n\in \CC
$$ 
are linearly independent since the determinant of the matrix $\alpha_k^{i-1}$ ($1\leqslant i,k \leqslant n$) is the Vandermonde determinant $\prod_{k_1<k_2} (\alpha_{k_1}-\alpha_{k_2})$ which is not zero since we assume that all roots are distinct. 
This means that the matrix $S$ is invertible.

One can easily modify this idea in order to compute the number of real roots which are less than some real number $t$.
Consider the (real) quadratic form 
$$
q_t(x_1, \dots, x_n)= \sum_{k=1}^n (t-\alpha_k)(x_1+ \alpha_k x_2+ \dots + \alpha_k^{n-1} x_n)^2.
$$
Obviously its signature is the number of real roots $<t$ minus the number of real roots $>t$. 
Expanding the sum, one finds 
$$
q_t(x_1, \dots, x_n)= \sum_{i,j=1}^n (t \sigma_{i+j-2} - \sigma_{i+j-1}) x_ix_j.
$$
In other words, the symmetric matrix
$$S(P)C(P)=\begin{pmatrix}
\sigma_1     	& \sigma_2    	& \sigma_3 	& \dots 	& \sigma_{n-1} 		& \sigma_{n} 	\\
\sigma_2 		& \sigma_3	& \sigma_4 	& \dots 	& \sigma_{n} 		& \sigma_{n+1}	\\
\sigma_3 		& \sigma_4 	& \sigma_5 	& \dots 	& \sigma_{n+1} 	& \sigma_{n+2} \\
\vdots      		& \vdots 		& \vdots 		& \ddots 	& \vdots 			& \vdots 		\\
\sigma_{n-1} 	& \sigma_{n} 	& \sigma_{n+1}	& \dots  	& \sigma_{2n-3}	& \sigma_{2n-2}\\
\sigma_{n} 	& \sigma_{n+1} & \sigma_{n+2} & \dots 	& \sigma_{2n-2} 		& \sigma_{2n-1}
\end{pmatrix}$$
has the following property
\begin{theorem}[Jacobi-Hermite] \label{hermite2}
Let $P(X)$ be a  monic polynomial of degree $n$ in $\RR[X]$ with distinct roots.
Let $p_+(t), p_-(t)$ be the number of real roots which are $>t$ and $<t$.
Then, if $t\in \RR$ is not a real root, the signature of the symmetric $n \times n$ matrix $S(P)(tI_n-C(P))$ is $p_-(t)-p_+(t)$.
\end{theorem}

Note that this theorem can be genuinely called ``algebraic'' in the sense that it provides an algebraic count of real roots without referring to Sturm.
Indeed, we remind that $\sigma_i$ is the trace of $C(P)^i$ and can therefore be computed easily from the coefficients of $P$.

The story of the previous theorem is rather tortuous.

In a first step, Sylvester tried to understand each polynomial in the Sturm sequence $P_i$ of a given polynomial $P=P_0$, directly as a function of the roots of $P$.
His paper, dated 1839 and published in 1841, contains the following nice formulas but no proofs~\cite{sylvester1841}.
If $\alpha_1, \dots, \alpha_n$ are the (real or complex) roots of $P$ then
$$
\begin{array}{l}
\dfrac{P_1}{P_0} = \sum\limits_{k=1}^n \dfrac{1}{X-\alpha_k}\\[2ex]
\dfrac{P_i}{P_0} = c_i \sum\limits_{k_1< \dots < k_i} \dfrac{\prod_{a<b ; a,b \in \{k_1,\dots, k_i\}}(\alpha_a-\alpha_b)^2}{(X-\alpha_{k_1})\dots (X-\alpha_{k_i})}
\end{array}
$$
where $c_i$ is some {\it positive constant}.
For instance, the last polynomial is equal to 
$$P_n= c_n \prod_{1\leqslant j < k \leqslant n}(\alpha_j-\alpha_k)^2.$$
This is not a surprise since $P_n$ is a degree 0 polynomial, the result of the modified Euclidean algorithm computing the g.c.d of $P$ and $P'$ (Definition \ref{sturmfunctions}), in other words, the {\it discriminant} of $P$ multiplied by $c_n$. En passant, note that $P_n>0$ if and only if the number of pairs of complex conjugate non real roots is even.\label{odd}

Proofs were provided by Sturm himself in 1842~\cite{sturm1842} (and  expanded by Sylvester in~\cite{sylvester1853b} in a long memoir in 1853).
In 1847, Borchardt reformulated this Sylvester-Sturm theorem, showing the link with the $\sigma_k$~\cite{borchardt}.
Later, in 1857 Borchardt explained that soon after publishing his paper, he received a letter from Jacobi showing that the idea of the matrices $\sigma_k$ and their use in the count of real roots was known to him.
There is however no publication by Jacobi on this topic.

As a matter of fact, the goal of the 1847 paper by Borchardt was to provide  a new proof of the spectral theorem that we now describe.

The $n^2$-dimensional $\RR$-algebra $M_n(\RR)$ of $n\times n$ matrices is equipped with a canonical quadratic form: $\Phi(M) = {\rm trace} (M^2)$. 
Notice that the trace of the square of a symmetric matrix is the sum of squares of all coefficients, and is therefore positive definite.
The restriction of $\Phi(M)$ to any subspace of symmetric matrices is therefore also a positive definite quadratic form.

Let $S$ be a symmetric matrix in $\RR$ and let $P(X)={\rm det}(X I_n-S) \in \RR[X]$ be its characteristic polynomial.
Assume that all roots of $P(X)$ are distinct.
The Cayley-Hamilton theorem implies that the algebra $A_{P(X)}=\RR[X]/P(X)$ embeds as an $n$-dimensional subspace in $M_n(\RR)$, sending $X$ to $S$. 
As $S$ is symmetric, the quadratic form $\phi_{P(X)}$ is the restriction of $\Phi$ to this subspace,
so we conclude that $\phi_{P(X)}$ is positive definite, and Theorem~\ref{hermite} gives that all roots of $P(X)$ are real.
This is the ninth proof of the spectral theorem~\cite{borchardt}\label{eighth}.

By Theorem~\ref{sjgf} the symmetric matrix $S(P)$ has signature $n$ if and only if $\mu_k(S(P)) >0$ for $k=1,2,\dots,n$.
In view of the identities $\sigma_k = \text{trace} (C(P)^k)$ this is an explicit set of $n$ inequalities in the coefficients of the polynomial $P(X)$ which are satisfied if and only if all the roots of $P(X)$ are real. 
In particular,  the property for a polynomial $P(X)$ to have all its roots real can be expressed by the positivity of $n$ polynomials in the coefficients of $P(X)$.
Borchardt's proof shows much more: that these polynomials are actually sums of squares of polynomials.
According to Serret, this proof is ``la plus remarquable'' of the spectral theorem~\cite{serret}.

The modern proof that we provided is certainly not original and can be found for example in Gantmacher~\cite{gantmacher}.

Let us look at the simplest example where $n=2$.
The characteristic polynomial of a non-zero symmetric $2 \times 2$ matrix
$$
S=
\begin{pmatrix}
a & b  \\
b &  c
\end{pmatrix}
$$
is $P(X)={\rm det}(XI_2-S)=X^2-(a+c)X+ac-b^2$, with companion and Hermite matrices
$$C(P)=\begin{pmatrix}
0 & b^2-ac  \\
1 &  a+c
\end{pmatrix}~;~S(P)=\begin{pmatrix}
2 & a+c  \\
a+c &  a^2+2b^2+c^2
\end{pmatrix}
$$
such that ${\rm det}(S(P))=(a-c)^2+4b^2>0$.
The paper by Borchardt gives the computation for $n=3$.

Much more on these kinds of results can be found in~\cite{gantmacher,kreinneimark,fuhrmann}.

One finds in~\cite[page 61]{dieudonne} a theorem due to Sylvester and Hermite, also associating in a constructive way a symmetric matrix to a polynomial such that the number of real roots is given by the signature of the matrix. 

The quotient $(P(X)P'(Y)-P(Y)P'(X))/(X-Y)$ is obviously a symmetric polynomial of the form $\sum a_{i,j}X^iY^j$ and one associates the matrix $a_{i,j}$ to the polynomial $P$.
This is called the {\it Bezoutian} symmetric matrix associated to $P$.
See~\cite{helmkefuhrmann,wimmer} for the history of this concept, which applies to any pair of polynomials $P_0(X),P_1(X) \in \RR[X]$, not just $(P_0(X),P_1(X))=(P(X),P'(X))$.
It turns out that the signature of $a_{i,j}$ is also equal to the number of real roots of $P$.
The connection with Jacobi-Hermite theorem~\ref{hermite} is made clear for instance in~\cite{quarez}.
It turns out that the Bezoutian matrix is congruent to the Jacobi-Hermite matrix.
Recall that the Jacobi-Hermite matrix is the matrix of the bilinear form $\phi_{P(X)}$ on $A_{P(X)} =  \RR[X]/P(X)$ in the basis $\{1,X, \dots, X^{n-1}\}$.
It is not difficult to write down the dual basis of $A_{P(X)}$ with respect to this bilinear form. 
One finds the so-called Horner polynomials and the matrix of $\phi_{P(X)}$ is precisely the Bezoutian.
Modern developments can be found in~\cite{helmke,cazanave}.

Finally, in~\cite{hermitecours}, Hermite computes the signs of the principal minors of the Bezoutian matrix associated to a polynomial $P(X)$.
Using Sylvester's formulas for the Sturm sequence $P_i$ of $P$, he finds that these minors coincide with $P_i$. 
In other words, the Jacobi-Hermite-Sylvester algebraic Bezoutian approach gives a new proof of Sturm's theorem, this time fully algebraic.
La boucle est boucl\'ee.

We do not provide more details since this is not related to our general theme: the connection between signature and topology.
See Basu, Pollack and Roy~\cite[4.3]{basupollackroy} for the applications of quadratic forms and matrices to the counting of polynomial roots in real algebraic geometry.

For historical comments on the ``once famous'' Sturm's theorem, we strongly recommend~\cite{sinaceur, sinaceur2, sinaceur3}.

\subsection{Witt groups of fields}

\bigskip
\begin{center}
\includegraphics[width=.6 \textwidth]{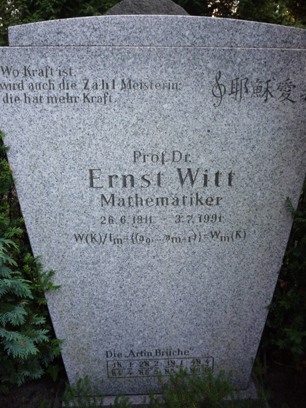}

The grave of Ernst Witt in the Nienstedtener Friedhof, Hamburg

``Wo Kraft ist, wird auch die Zahl Meisterin:
die hat mehr Kraft''

(F. Nietzsche)
\end{center}
\bigskip

In the previous sections, we studied ``classical'' quadratic forms, defined on finite dimensional {\it real} vector spaces.
By considering Sylvester's application of forms to Sturm's theorem we were naturally led to  discuss symmetric matrices whose entries are polynomials in $\RR[X]$ or more generally in $\RR(X)$, the field or rational functions in one variable $X$, with real coefficients.
It is therefore very natural to generalize as much as possible our considerations to symmetric matrices with entries in a general field $K$, or more generally in a commutative, or even non commutative, ring.

In this subsection, we will begin with the most elementary case of a field $K$ of characteristic different from $2$, such as the special case of $K=\RR(X)$.
We will propose first an elementary down to earth approach and only mention more general results, referring to standard textbooks for a deeper description.
The book of Lam~\cite{lam} is a good general reference.

\bigskip

A {\it symmetric form $(V,\phi)$  over $K$} is a finite dimensional $K$-vector space $V$ with a symmetric bilinear pairing $\phi:V \times V \to K$.
A symmetric form is essentially the same as $V$ together with a quadratic function $Q:V \to K$ such that
\begin{itemize}
\item   $Q(av)=a^2Q(v) $ for $a \in K$, $v \in V$,
\item the function $\phi:(v_1,v_2) \in V \times V \mapsto Q(v_1+v_2)-Q(v_1)-Q(v_2) \in K$ is bilinear.
\end{itemize}
Given $(V,\phi)$, define such a $Q$ by
$$Q(v) =\phi(v,v)/2\in K~(v\in V).$$
A basis of $V$ identifies a symmetric form $(V,\phi)$ with a symmetric matrix $S$ with entries in $K$, such as we studied previously for $K=\RR$.
There is a natural notion of {\it isomorphism} between symmetric forms and the goal of the theory is to obtain invariants of the  isomorphism classes of symmetric forms over $K$.
For $K=\RR$ isomorphism is the same as the linear congruence we considered in section~\ref{linear}.

The {\it orthogonal subspace} of a subspace $U \subseteq V$ in a symmetric form $(V,\phi)$ is the subspace
$$U^{\perp}=\{v \in V\,\vert\, \phi(v,w)=0~\text{for all}~w\in U\}.$$
The {\it radical} of $(V,\phi)$ is  $V^{\perp}\subseteq V$.
A symmetric form is {\it nonsingular} if its radical is trivial.
For any $(V,\phi)$  there is defined a nonsingular symmetric form $(V / V^{\perp},[\phi])$.
A subspace $U \subseteq V$ is called {\it isotropic} for $(V,\phi)$ if $\phi(v_1,v_2)=0$ for all $v_1,v_2 \in U$.
A symmetric form is {\it anisotropic} if it contains no non trivial isotropic subspaces.

A nonsingular symmetric form $(V,\phi)$ is called {\it hyperbolic} if it is the direct sum of two isotropic subspaces $U$ and $U'$, in which case $\phi$ restricts to an isomorphism
$$v_1 \in U' \mapsto (v_2 \mapsto \phi(v_1,v_2)) \in U^*={\rm Hom}_K(U,K).$$
The {\it standard hyperbolic form} is defined for any finite-dimensional $K$-vector space $U$ by
$$H(U)=(U \oplus U^*,\phi)$$
with
$$\phi((v_1,f_1),(v_2,f_2))=f_1(v_2)+f_2(v_1) \in K.$$
For any basis on $U$ and the dual basis on $U^*$ the matrix of $\phi$ is
$\begin{pmatrix} 0 & I_n \\ I_n & 0 \end{pmatrix}$ with $n={\rm dim}_KU$.
A nonsingular symmetric form over $K$ is hyperbolic if and only if it is isomorphic to $H(U)$ for some $U$.

The {\it direct sum} of symmetric forms $(V_1,\phi_1)$, $(V_2,\phi_2)$ is the symmetric form $(V_1 \oplus V_2,\phi_1 \oplus \phi_2)$, with
$$(\phi_1\oplus \phi_2)((v_1,v_2),(w_1,w_2))=\phi_1(v_1,w_1)+\phi_2(v_2,w_2) \in K.$$

\begin{theorem}[Witt~\cite{witt}]  \label{witt}
A symmetric form $(V,\phi)$ over any field $K$ (of characteristic $\neq 2$) is isomorphic to a sum  $(V_a,\phi_a)\oplus (V_h,\phi_h) \oplus (V_0,\phi_0)$ where:
\begin{itemize}
\item $(V_a,\phi_a)$ is anisotropic,
\item $(V_h,\phi_h)$ is hyperbolic,
\item $(V_0,\phi_0)$ is trivial, i.e $\phi_0=0$.
\end{itemize}
Moreover, this decomposition is unique, up to isomorphism.
\end{theorem}

Since hyperbolic and trivial forms are classified by dimension, only the description of the anisotropic case is relevant.

 \begin{definition}
 A {\it diagonalization} of a symmetric form $(V,\phi)$ over a field $K$ is an isomorphism $(V,\phi) \cong \bigoplus\limits^n_{k=1}(K,a_k)$  with $a_k \in K$.
 \end{definition}

Recall that we have already noticed the following fundamental fact~\ref{linear}.

\begin{proposition} \label{diag}
Every symmetric form $(V,\phi)$ over a field  $K$  of characteristic $\neq 2$ can be diagonalized.\\
1. For an anisotropic $(V,\phi)$ the symmetric matrix $S$ with respect to any basis is regular, and the diagonalization is of the type
$$(V,\phi) \cong \bigoplus^n_{k=1}(K,\mu_k(S)/\mu_{k-1}(S)).$$
2. For a hyperbolic $(V,\phi)$ the dimension $n={\rm dim}_KV$ is even, and there is a diagonalization of the type
$$(V,\phi)\cong \bigoplus\limits^{n/2}_{k=1} (K\oplus K,1 \oplus -1).$$
\end{proposition}

\begin{proof} In view of Theorem~\ref{witt} we can consider separately the
anisotropic and hyperbolic cases.\\
1. For an anisotropic $(V,\phi)$ the matrix $S$ is regular, i.e. the principal minors $\mu_k(S)\in K$  ($1 \leqslant k \leqslant n$) are non-zero.
Proceeding as in the proof of Theorem~\ref{sjgf} we have that $(V,\phi)$ is isomorphic to
$\oplus^n_{k=1}(K,\mu_k(S)/\mu_{k-1}(S))$.\\
2. For the hyperbolic case it is enough to consider the 2-dimensional case
$$(V,\phi)=~(K \oplus K,\begin{pmatrix} 0 &  \lambda \\ \lambda & a \end{pmatrix})$$
for which there is defined an isomorphism
$$\begin{pmatrix} \lambda & (a+1)/2 \\ \lambda  & (a-1)/2 \end{pmatrix} : (V,\phi) \to (K\oplus K,1\oplus -1).$$
\end{proof}

The set ${\rm Sym}(K)$ of isomorphism classes of nonsingular symmetric forms over $K$, equipped with  $\oplus$, is an abelian monoid.
Denote by $GW(K)$ the corresponding Grothendieck group, i.e. the group of formal differences of symmetric forms.
More precisely, it is the quotient of  ${\rm Sym}(K) \times {\rm Sym}(K)$ by the equivalence relation $(s_1,s_2) \sim (t_1,t_2)$ if $s_1\oplus t_2$ is isomorphic to $s_2 \oplus t_1$.
Note that the natural map
$$(V,\phi)\in {\rm Sym}(K) \mapsto  ((V,\phi),0) \in GW(K) $$
is injective so that there is no loss of information when we go from isomorphism classes to elements of $GW(K)$.

\begin{definition} The {\it Witt group} of $K$ is the quotient
$$W(K)=~GW(K)/\{\text{subgroup generated by the hyperbolic forms}\}.$$
\end{definition}

 Every nonsingular symmetric form $(V,\phi)$ is diagonalizable by Proposition \ref{diag}, with
$(V,\phi) \cong \sum\limits^n_{k=1}(K,a_k)$ $(a_k \in K^{\bullet})$, so that
$$[V,\phi]=\sum\limits^n_{k=1}(K,a_k) \in W(K)$$
and $W(K)$ is generated by the 1-dimensional symmetric forms $(K,a)$ $(a \in K^{\bullet}$).

It follows from Witt's theorem that the following conditions on two anisotropic nonsingular symmetric forms $s_1=(V_1,\phi_1)$, $s_2=(V_2,\phi_2)$  are equivalent:
\begin{itemize}
\item $s_1$ is isomorphic to $s_2$,
\item $s_1\oplus t$ is isomorphic to $s_2 \oplus t$, for any nonsingular symmetric form $t$,
\item $s_1=s_2 \in W(K)$.
\end{itemize}
Therefore $W(K)$ is indeed the key object in the understanding of nonsingular symmetric forms over $K$.

\begin{definition} A \emph{lagrangian} $L$ of a nonsingular symmetric form $(V,\phi)$ over $F$ is a subspace $L \subseteq V$ such that $L=L^{\perp}$, with $L^{\perp}=\{x \in V\,\vert\,\phi(x,y)=0 \in F~\hbox{\rm for all}~y \in L\}$.
\end{definition}

\begin{remark}\label{remark1}
The following conditions on a nonsingular symmetric form $(V,\phi)$ over $K$ (of characteristic $\neq 2$) are equivalent.
\begin{itemize}
\item[1.] $(V,\phi)$ is hyperbolic,
\item[2.]  $(V,\phi)$ admits a lagrangian,
\item[3.]  $(V,\phi)=0 \in W(K)$.
\end{itemize}
In order to prove 3. $\Longrightarrow$ 2. one can appeal to
Witt's Theorem~\ref{witt}, or else note that if $(V',\phi')$ is a nonsingular symmetric form over $K$ with lagrangian $L'$
and $L$ is a lagrangian of $(V,\phi)\oplus (V',\phi')$ then $V \cap (L+L')$ is a lagrangian of $(V,\phi)$.
\end{remark}

Before we go on, let us describe two trivial examples.

We know the structure of nonsingular symmetric forms over $\RR$ from Theorem~\ref{century}: every such form $(V,\phi)$ is isomorphic to $\oplus^n_{k=1}(\RR,a_k)$ with $a_k= \pm 1$.
Since $(\RR,1) \oplus (\RR,-1)$ is a hyperbolic form, we conclude that $W(\RR)$ is isomorphic to $\ZZ$, with the function
$$\tau= {\rm signature}:~(V,\phi)=\oplus^n_{k=1}(\RR,a_k) \in W(\RR) \mapsto  \tau(V,\phi) \in  \ZZ$$
an isomorphism.

We know that every nonsingular symmetric form over the complex numbers $\CC$ is isomorphic to $\oplus^n_{k=1} (\CC,1)$, with $(\CC,1) \oplus (\CC,1)$ hyperbolic.
The function
$${\rm dim}:~(V,\phi) \in W(\CC) \mapsto {\rm dim}_\CC(V)\bmod\,2  \in \ZZ/ 2 \ZZ $$
is an isomorphism.

In order to describe more interesting examples, we need to develop the basic features of Witt theory.

Let $K^{\bullet}=K\backslash \{0\}$ the multiplicative group of units of $K$.
Note that  every element of $K^{\bullet}/(K^{\bullet})^2$ has order $2$, with $(K^{\bullet})^2=\{a^2 \,\vert\, a\in K^{\bullet}\} \subseteq K^{\bullet}$ the subgroup of squares.
Note also that the group $W(K)$ is denoted additively and  $K^{\bullet}/(K^{\bullet})^2$ multiplicatively.

First, we observe that there are two natural maps defined on $W(K)$.

\begin{definition}\label{wittdisc} \leavevmode

1. The {\it dimension map}
$${\rm dim} : (V,\phi) \in W(K) \mapsto {\rm dim}_KV\,\bmod \,2 \in \ZZ/ 2 \ZZ $$
is well-defined for any $K$, since hyperbolic forms have even dimensions.

2. The {\it discriminant map}
$${\rm disc} : (V,\phi) \in W(K) \mapsto (-1)^{r(r-1)/2} {\rm det}(\phi) \in  K^{\bullet}/(K^{\bullet})^2$$
is well-defined, with $r={\rm dim}_KV$.
\end{definition}

The $0$ element of $W(K)$ is represented by the unique symmetric form on the $0$-dimensional vector space and the associated $0\times 0$ matrix has discriminant $1$, in line with the convention of Sylvester \cite[p.\,616]{sylvester1853} that a $0 \times 0$ matrix has determinant 1.

See Milnor and Husemoller~\cite[p.\,77]{milnorhusemoller} for the basic properties of the discriminant. In particular, note that
$$\begin{array}{l}
{\rm disc}(V\oplus V',\phi \oplus \phi')~=~(-1)^{rr'}{\rm disc}(V,\phi)\,
{\rm disc}(V',\phi') ~,\\[1ex]
{\rm disc}([V,\phi]-[V',\phi'])~=~(-1)^{(r+r')(r+r'-1)/2 + r'}{\rm det}(\phi)/{\rm det}(\phi') \in K^{\bullet}/(K^{\bullet})^2~.
\end{array}$$
The function ${\rm disc}:W(K) \to K^{\bullet}/(K^{\bullet})^2$ is {\it not}
a homomorphism of abelian groups: for example
$${\rm disc}(K,1)~=~1~,~{\rm disc}(K \oplus K,1 \oplus 1)~=~-1 \in
K^{\bullet}/(K^{\bullet})^2~.$$
However, the restriction of the discriminant to the `fundamental ideal'
$I(K)={\rm ker}({\rm dim}:W(K) \to \ZZ/2\ZZ) \subseteq W(K)$
$$[V,\phi]-[V',\phi'] \in I(K) \mapsto {\rm disc}([V,\phi]-[V',\phi'] )=
{\rm det}(\phi)/{\rm det}(\phi')   \in K^\bullet/(K^\bullet)^2$$
is a homomorphism. 

\bigskip

We can now give Witt's description of $W(K)$ by generators and relations, for any field $K$ of characteristic $\neq 2$.

\begin{theorem}[Witt~\cite{witt}]
 \label{generators} The Witt group $W(K)$ is isomorphic to $\ZZ[K^{\bullet}]/N$ with $N \subseteq \ZZ[K^{\bullet}]$ the subgroup generated by elements of the type
 $$[x]-[xy^2]~,~[x]+[-x]~,~[x]+[y]-[x+y]-[xy(x+y)^{-1}]$$
for $x,y \in K^{\bullet}$ with $x+y \in K^{\bullet}$.
The map
$$(V,\phi)=\bigoplus\limits^n_{i=1}(K,a_i)  \in W(K) \mapsto  \sum^n_{i=1}a_i \in \ZZ[K^{\bullet}]/N
$$
is an isomorphism.
\end{theorem}
\begin{proof}
We know that any anisotropic nonsingular symmetric form over $K$ can be diagonalized, so that the map under consideration is well-defined and onto.
The kernel will be identified with $N$ by describing the relations in
$W(K)$ between these elements $[x]$.

The first obvious relation is $[x]= [xy^2]$ for every $y \in K^{\bullet}$.
This corresponds to the change of the basis element in a 1-dimensional vector space.

Moreover, one has $[-x]= -[x]$ as follows from the observation that $[x]+[-x]$ is hyperbolic.

For any $x,y\in K^{\bullet}$ with $x+y \in K^{\bullet}$ the invertible $2 \times 2$ matrices defined by
$$
S=\begin{pmatrix} x & 0 \\ 0 & y \end{pmatrix}
\quad; \quad
P = \begin{pmatrix} 1 & -y(x+y)^{-1} \\ 1 & x(x+y)^{-1} \end{pmatrix}
$$
are such that
$$
P^*SP=\begin{pmatrix} x+y & 0 \\ 0&  xy(x+y)^{-1}   \end{pmatrix}~,
$$
where we write $P^*$ for the transpose $P^T$.

It follows that
$$[x]+[y]= [x+y]+ [ xy(x+y)^{-1}] \in W(K).$$

Denote for a moment $\overline{W}(K)=\ZZ[K^{\bullet}]/N$ the abstract abelian group generated by all $[x]$'s subject to these relations. We have just shown that the function sending a formal linear combination in $K^{\bullet}$ (possibly with repetitions) to a diagonal symmetric form
$$\sum\limits^n_{i=1}a_i \in \overline{W}(K) \mapsto \sum\limits^n_{i=1}[K,a_i] \in W(K)$$
is well-defined. The inverse function sends the Witt class $[V,\phi]$ of a symmetric form $(V,\phi)$ to any diagonalization $(V,\phi)\cong \sum\limits^n_{i=1}(K,a_i)$  
$$[V,\phi] \in W(K) \mapsto \sum\limits^n_{i=1}a_i\in \overline{W}(K)~.$$
In order to check that this is well-defined, we have to consider two diagonalizations  $(V,\phi)\cong \sum\limits^n_{i=1}(K,a_i)\cong \sum\limits^n_{i=1}(K,a'_i)$, and prove that
$\sum\limits^n_{i=1}(a_i-a'_i) \in N \subseteq \ZZ[K^{\bullet}]$. The diagonalizations are related by an element $P \in GL(n,K)$
such that 
$$P^*\text{diag}(a_1,a_2,\dots,a_n)P~=~\text{diag}(a'_1,a'_2,\dots,a'_n)~.$$
The group $\text{GL}(n,K)$ is generated by matrices which act non trivially only on two elements of a vector basis (i.e. the elementary operations of Gaussian elimination), so only the case $n=2$ need be considered. For further details see~Lam~\cite[Thm.\,II.4.1]{lam}.
\end{proof}

Although the expression for $W(K)$ given by Theorem~\ref{generators} is not very useful for an arbitrary field $K$, it is adequate for $K$ in which it is known which elements of  $K^{\bullet}$ are squares:

If $K=\RR$ then $x \in \RR^{\bullet}$ is a square if and only if $x>0$, and the signed augmentation
$$[x] \in W(\RR)=\ZZ[\RR^*]/N \mapsto {\rm sign}(x) \in\ZZ$$
is an isomorphism - evidently the signature!

If $K=\CC$ then every $z \in \CC^{\bullet}$ is a square, so that $[z]=[1] \in N$.
It now follows from $([z]+[i^2z])+([z]-[-z])=2[z]=0 \in N$ that the $\bmod\,2$ augmentation
$$[z] \in W(\CC)=\ZZ[\CC^{\bullet}]/N \mapsto 1 \in \ZZ/2\ZZ$$
is an isomorphism - evidently the dimension $\bmod\, 2$!

See section~\ref{order} below for the examples
$K=\RR(X),\QQ,\FF_p$.
The case $K=\RR(X)$ is particularly relevant to the Witt group interpretation of the theorems of Sturm and Sylvester.

\subsection{The Witt groups of the function field $\RR(X)$, ordered fields and $\QQ$.} \label{order}

It is our goal to give an explicit description of the Witt group $W(\RR(X))$ of the field $\RR(X)$ of rational functions $P(X)/Q(X)$ ($P(X),Q(X)\in \RR[X]$, $Q(X)\neq 0)$.
The inclusion $i:\RR \subset \RR(X)$ induces a morphism $i:W(\RR)=\ZZ \to W(\RR(X))$ ( a split injection, in fact) with image the Witt classes of quadratic forms with constant coefficients.
We shall establish an isomorphism 
$$
{\rm coker}(i:\ZZ \to W(\RR(X)))\cong \ZZ[\RR] \oplus \ZZ/2\ZZ[{\cal H}]
$$
with $\ZZ[\RR]$ the free abelian group generated by $\RR$, and similarly for $\ZZ/2\ZZ[{\cal H}]$ where ${\cal H}$ denotes the upper half plane $\{z \in \CC \vert \Im(z) >0 \}$.
This will allow a modern presentation of Sturm and Sylvester's theorems.

Consider an invertible symmetric matrix $S(X)$ over $\RR(X)$.
Its determinant is a non-zero rational function ${\rm det}(S(X))\in \RR(X)^{\bullet}$, with a finite number of zeroes and poles.
Therefore, when $x\in \RR$ is away from some finite set $\{x_1, \dots, x_k\}$ of real numbers, the matrix $S(x)$ is an invertible symmetric matrix
in $\RR$, for which we can compute the signature $\tau(S(x))\in \ZZ$.
When $x$ is a zero or a pole of the determinant of $S(X)$, the value of $\tau(S(x))$ is undefined.

This suggests the following definition.
\begin{definition}
1. A \emph{sign-function} is a map
$\RR \backslash \{x_1,\dots,x_k\} \to \ZZ$  whose domain of definition is the complement of a finite set $\{x_1, \dots, x_k\} \subset \RR$ and which is constant on each connected component of the complement of this finite set.

2. Two sign-functions are equivalent if they agree outside of a finite set.
\end{definition}

We denote the additive group of equivalence classes of sign-functions by ${\mathcal Sign}$.
A symmetric matrix $S(X)$ in $\RR(X)$ determines a sign-function
$${\rm sign}(S(X)):x \in \RR\backslash\{\text{roots and poles of {\rm det}($S(X)$)}\} \mapsto \tau(S(x)) \in  \ZZ.$$
Note that if $S_2(X)= V(X)^*S_1(X)V(X)$, i.e. if $S_1(X)$ and $S_2(X)$ are linearly congruent under some invertible matrix $V(X)$ with coefficients in $\RR(X)$, one can  conclude that the signatures of $S_1(x)$ and $S_2(x)$ coincide outside of the roots and poles of the determinant of $V(X)$ so that two linearly congruent matrices do define equivalent sign-functions.

Clearly the sign-function associated to a direct sum of two matrices is the sum of the sign-functions, and the sign-function of a hyperbolic matrix is obviously $0$.
Therefore, we have a well defined map:
$$
{\rm sign} : W(\RR(X)) \to {\mathcal Sign}.
$$

{\it More precisely, we will show that {\rm sign} is injective when restricted to forms with discriminant $1$}.

Before proving that, we will give a better description of the group ${\mathcal Sign}$.
Regard $\ZZ[\RR]$ as the group of functions $\RR \to \ZZ$ which vanish outside of a finite set.

\begin{lemma}
There is an exact sequence
$$\xymatrix{0 \ar[r] &\ZZ \ar[r] & {\mathcal Sign} \ar[r]^-{\di{\partial}} & \ZZ[\RR] \ar[r] & 0.}
$$
The injection $\ZZ \to {\mathcal Sign}$ splits in two natural ways, by 
$$\sigma \in {\mathcal Sign} \mapsto  \sigma(\pm \infty)
~=~\lim\limits_{x \to \infty} \sigma(\pm x)~.$$ 
\end{lemma}

\begin{proof}
Given a sign-function $\sigma$, one can define its ``derivative'', i.e. the function
$$
\partial \sigma (a) = \lim_{\epsilon \rightarrow 0^+} \big(\sigma (a+ \epsilon) - \sigma(a- \epsilon)\big)
$$
which is clearly an element of $\ZZ[\RR]$. For $\sigma:\RR \backslash 
\{x_1,x_2,\dots,x_n\} \to \ZZ$ with $x_1<x_2<\dots<x_n$ let
$$y_k~=~(x_k+x_{k+1})/2~(1 \leqslant k \leqslant n-1)~,$$
so that
$$\partial \sigma~=~ (\sigma(y_1)-\sigma(-\infty)).x_1+
(\sigma(y_2)-\sigma(y_1)).x_2+\dots+(\sigma(\infty)-\sigma(y_{n-1})).x_n \in
\ZZ[\RR]~.$$
The kernel of $\partial$ consists of constant functions.
One shows that $\partial$ is onto by ``integration''.
\end{proof}

Recall that the signature of any nonsingular symmetric form over $\RR$ is always congruent modulo two to the dimension of the underlying vector space, so that if $\sigma=\text{sign}(S(X)) \in {\mathcal Sign}$
$$\sigma(-\infty)~\equiv~\sigma(y_1)~\equiv~\sigma(y_2)~\equiv~\dots~
\equiv~ \sigma(y_n)~\equiv~\sigma(\infty)~\bmod \,2$$
and the image of the ``derivative''
$$
\partial \circ {\rm sign} : W(\RR(X)) \to \ZZ[\RR]
$$
is divisible by 2. We shall denote the ``jump'' $( \partial \circ  {\rm sign})/2$ by
$$
\tau_{\RR} : W(\RR(X)) \to \ZZ[\RR]
$$
and call it the {\it total signature}. 

A non-zero polynomial $F(X) \neq 0 \in \RR[X]$ is an invertible symmetric $1 \times 1$ matrix over $\RR(X)$, corresponding an element $[F(X)] \in W(\RR(X))$.
Of course, ${\rm sign} ([F(X)])$ is simply the usual sign of $F(x)$, equal to $\pm 1$ outside of the set of roots of $F$.
As for $\tau_{\RR}([F(X)])$, it is non zero precisely at the values of $x$ at which $F(x)$ changes sign, from $-$ to $+$ or from $+$ to $-$, with value $\pm 1$ accordingly. For any $a \in \RR$ 
$$F(X)~=~(X-a)^NG(X)$$
with $N \geqslant 0$, $G(a) \neq 0$, and the jump of the sign of $F(X)$ at $a$ is
$$(\lim\limits_{x \to a^+}{\rm sign}(F(x))-
\lim\limits_{x \to a^-}{\rm sign}(F(x)))/2~=~
\begin{cases}
{\rm sign}(G(a))&\hbox{\rm if $N$ is odd}\\
0&\hbox{\rm if $N$ is even}~.
\end{cases}$$
In the special case $N=1$ we have that $a \in \RR$ is a simple (i.e. non-multiple) root of $F(X)$ with 
$$F(a)~=~0~,~G(X)~=~\dfrac{F(X)-F(a)}{X-a}~,~G(a)~=~F'(a)~,$$
so that the jump is $1$ (resp. $-1$) if and only if $F(X)$ is increasing (resp. decreasing) at $X=a$, if and only if $F'(a)>0$ (resp. $F'(a)<0$). Thus if $F(X)$ has $n$ distinct real roots $a_1,a_2,\dots,a_n \in \RR$ the total signature is
$$\tau_{\RR}([F(X)])~=~\sum\limits^n_{i=1}{\rm sign}(F'(a_i)).a_i \in \ZZ[\RR]~.$$

We will show that two invertible symmetric $n \times n$ matrices in $\RR(X)$ are linearly congruent if and only if they have the same discriminant and the same sign-function.

We now describe the discriminant (\ref{wittdisc}) of a nonsingular symmetric form over $\RR(X)$.

Irreducible polynomials in $\RR[X]$ have degree $1$ or $2$.
Every $a \in \RR$ determines an irreducible polynomial $X-a \in \RR[X]$ with root $a \in \RR$.
Every $z \in {\cal H}$ defines an irreducible real polynomial of degree 2:
$$
P_z(X) = X^2 - 2 \Re(z) X + \vert z \vert ^2.
$$
with complex roots $z$ and $\bar{z}$.

Any non-zero polynomial in $\RR[X]$ can be written in a unique way as a product of irreducible polynomials, in the form
$$
F(X)=\lambda \prod_{i=1}^n (X-a_i)^{n_i} \prod_{j=1}^m P_{z_j}(X)^{m_j}
$$
where $\lambda \neq 0,a_i\in \RR$, $z_j \in {\cal H}$, and $n_i,m_j>0$.
The total signature is
$$\tau_\RR(F(X))~=~\sum\limits^n_{i=1,n_i~{\rm odd}}{\rm sign}
\bigg(\lambda  \prod_{j=1,\,j \neq i}^n (a_j-a_i)\bigg).a_i \in \ZZ[\RR]$$
and
$$\sigma(F(X))(\infty)~=~\text{sign}(\lambda)~,~
\sigma(F(X))(-\infty)~=~\text{sign}(\lambda) (-1)^N \in \{\pm 1\} \subset \ZZ$$
with $N$ the number of odd $n_i$'s, i.e. the number of terms in $\tau_\RR(F(X))$. 
A rational function $F(X) \in \RR(X)$ can be written in the same way, if one agrees that $n_i$ and $m_j$ can be positive or negative.
Such a function $F$ is a square if and only if $\lambda >0$ and all $n_i$ and $m_j$ are even numbers.
It follows that we have a very simple description
$$
\RR(X)^{\bullet}/(\RR(X)^{\bullet})^2 \cong \ZZ/ 2 \ZZ \oplus \ZZ / 2 \ZZ[\RR] \oplus \ZZ/2\ZZ[{\cal H}]
$$
and the discriminant map has three components
$${\rm disc}=({\rm disc}_1,{\rm disc}_\RR,{\rm disc}_{\cal H}):
W(\RR(X)) \to \RR(X)^{\bullet}/(\RR(X)^{\bullet})^2.$$

We can now give a precise description of the Witt group $W(\RR(X))$.
As before, if $P(X)/Q(X) \in \RR(X)^{\bullet}$ we denote by $[P(X)/Q(X)]$ the element of $W(\RR(X))$ defined by the $1\times 1$ matrix with the single element $P(X)/Q(X)$.

\begin{theorem} \label{wittrational}  The following map is an isomorphism
$$
p: \ZZ [1] \oplus \ZZ[\RR] \oplus \ZZ/2\ZZ[{\cal H}] \to W(\RR(X))
$$
$$
(n,n_a,m_z) \mapsto n [1] + \sum\limits_{a\in \RR}n_a ([X-a]-[1])  + \sum\limits_{z \in {\cal H}} m_z ([P_z(X)]-[1])
$$
with inverse
$$p^{-1}(S(X))~=~(\tau(S(\infty)),\tau_{\RR}(S(X)),{\rm disc}_{\cal H}(S(X)))~.$$
\end{theorem}

In order to check that $p$ is well-defined, note that 
$$p(0,0,2m_z)~=~\sum\limits_{z \in {\cal H}}2m_z([P_z(X)]-[1])~=~0\in W(\RR(X))$$
 by virtue of the isomorphism of nonsingular symmetric forms over $\RR(X)$
$$\begin{array}{l}
\begin{pmatrix} X-\Re(z) & \Im(z) \\
-\Im(z) & X-\Re(z) \end{pmatrix}~:\\[2ex]
(\RR(X) \oplus \RR(X),
\begin{pmatrix} P_z(X) & 0 \\
0 & P_z(X) \end{pmatrix}) \to (\RR(X)\oplus \RR(X),\begin{pmatrix} 1 & 0 \\
0 & 1 \end{pmatrix})~.
\end{array}$$

It is easy to show that $p$ is {\it injective}.
Indeed,

\begin{itemize}
\item ${\rm sign}([X-a])$ is equal to $-1$ for $x<a$ and $1$ for $x>a$ so that $\tau_{\RR} ([X-a])$ is the function equal to $1$ for $x=a$ and $0$ for $x\neq a$.
\item ${\rm sign} ([1])$ and ${\rm sign}([P_z(X)])$ are  constant functions equal to $1$ so that $\tau_{\RR}([1]) = \tau_{\RR}(P_z(X))=0$.
\item $\text{disc} (p(n, n_a, m_z)) = ( n,n_a,m_z) \in \ZZ/ 2 \ZZ \oplus \ZZ/2\ZZ[\RR] \oplus \ZZ/2\ZZ[{\cal H}]$.
\end{itemize}
Any element of the kernel of $p$ has therefore $n_a=0$.
It then follows that $n=0$ and finally $m_z=0$.

\medskip

It remains to show that $p$ is {\it onto}.
In other words, we have to show that $W(\RR (X))$ is generated by $[1], [X-a],[P_z(X)]$.
We already know that elements of the form $[F(X)]$ generate $W(\RR(X))$.

Note that:

\begin{itemize}
\item $[{P(X)}/{Q(X)}]=[P(X)Q(X)]$ so that the $[F(X)]$'s with $F$ polynomial in $\RR[X]$ generate $W(\RR(X))$.
\item Modulo squares, it is enough to consider the $F(X)$ of the form
$$F(X)=\pm \prod_{i=1}^n (X-a_i) \prod_{j=1}^m P_{z_j}(X)$$
with distinct $a_i,z_j$, so that  $n_i=m_j=1$.
\end{itemize}

We can now show that $1$, $[X-a]$ and $[P_z(X)]$  generate $W(K)$.
The proof will be by induction.
Let us first consider products of two factors.
Let $P(X)$ and $Q(X)$ be two irreducible polynomials, each of the form $(X-a)$ or $P_z(X)$.
It is easy to check that one can find two non zero real numbers $\lambda,\mu$ such that $\lambda P(X) + \mu Q(X)$ is a square in $\RR[X]$.
Therefore we have, in $W(K(X))$:

$$
[\lambda P] + [\mu Q] = [\lambda P + \mu Q] + \left[ \frac{ \lambda \mu P  Q}{\lambda P + \mu Q}\right].
$$
so that
$$
\pm [P] \pm [Q] =  [1]\pm [PQ].
$$
So, the product of two irreducible polynomials $[PQ]$ can be expressed as a combination of those polynomials $[P],[Q]$, and $[1]$.
By induction, we can therefore express any $[F(X)]$ as a combination of the elements $1$, $[X-a]$ and $[P_z(X)]$.

This ends the proof of Theorem~\ref{wittrational} giving a full description of $W(\RR(X)$. Let us reformulate it in the following way:

\begin{corollary}\label{split}
The Witt group of $\RR(X)$ fits into a split exact sequence
$$\xymatrix{0 \ar[r] &  \ZZ \ar[r]^-{\di{i}} & }
\xymatrix@C+30pt{W(\RR(X)) \ar[r]^-{\di{\tau_{\RR} \oplus {\rm disc}_{\cal H}}} &}
\xymatrix{\ZZ[\RR] \oplus \ZZ/2 \,\ZZ[{\cal H}] \ar[r]& 0}
$$
where $i$ is induced by the embedding $i:\RR \subset \RR[X]$, split by
$$W(\RR(X)) \to \ZZ~;~S(X) \mapsto \tau(S(\infty))~.$$
\end{corollary}

The Witt class of the monic polynomial 
$$P(X)~=~(X-x_1)(X-x_2)\dots (X-x_n)P_{z_1}(X)P_{z_2}(X)
\dots P_{z_m}(X)$$
with $n$ distinct real roots $x_1,x_2,\dots,x_n \in \RR$ and 
$2m$ distinct complex roots
$\{z_1,z_2,\dots,z_m\} \cup \{\bar{z}_1,\bar{z}_2,\dots,\bar{z}_m\}
\in {\cal H}\cup \overline{\cal H}$ is given by
$$\begin{array}{ll}
[P(X)]&=~[1] + \sum\limits^n_{i=1}{\rm sign}(P'(x_i))([X-x_i]-[1])
+ \sum\limits^m_{j=1}([P_{z_j}(X)]-[1]) \\[1ex]
&\in W(\RR(X))~=~\ZZ \oplus \ZZ[\RR] \oplus \ZZ/2\ZZ[{\cal H}]~.
\end{array}$$
Note that we already showed that the injection $i: \ZZ \rightarrow W(\RR(X))$ splits in two ways, using the limit of ${\rm sign}$ at $\pm\infty$, with
$$\sigma(P(X))(\infty)~=~1~,~\sigma(P(X))(-\infty)~=~1-2\sum\limits^n_{i=1}\text{sign}(P'(x_i))~=~(-1)^n~.$$
\indent Again, the following is just a reformulation of what we proved so far.

\begin{corollary}
Let $S(X)$ be a symmetric $n \times n$ matrix in $\RR(X)$ with non vanishing determinant, so that $(\RR(X)^n,S(X))$ is a nonsingular symmetric form over $\RR(X)$ with Witt class $[\RR(X)^n,S(X)] \in W(\RR(X))$ and total signature
$$\tau_{\RR}([S(X)])=\sum\limits_{a\in \RR}\tau_\RR([S(X)])(a). a  \in \ZZ[\RR].$$
For any $a \in \RR$ which is a zero or a pole of ${\rm det}(S(X)) \in \RR(X)^{\bullet}$
$$\tau_{\RR}([S(X)])(a) =\varinjlim_{\epsilon} (\tau(S(a+\epsilon))-\tau(S(a-\epsilon)))/2$$
is half the jump across $x=a$ of the signature of the nonsingular symmetric form $S(x)$ over $\RR$. If $a \in \RR$ is not a zero or a pole of
${\rm det}(S(X))\in \RR(X)^{\bullet}$ then $\tau_{\RR}([S(X])(a)=0 \in \ZZ$.
\end{corollary}

Now that we have a complete understanding of the Witt group over $\RR(X)$, we can restate Sylvester's reformulation of Sturm's theorem, as an equality between elements in a Witt group, using the following diagonalization.

\begin{proposition} Let $P(X) \in \RR[X]$ be a degree $n$ regular polynomial  with Sturm functions 
$$P(X)=(P_0(X),P_1(X),\dots,P_n(X))~,~Q(X)=(Q_1(X),Q_2(X),\dots,Q_n(X))~.$$
The tridiagonal symmetric matrix ${\rm Tri}(Q(X))$
is linearly congruent in $\RR(X)$ to the diagonal matrix with entries
$$\dfrac{P_{k-1}(X)}{P_k(X)}~=~
\dfrac{{\rm det}({\rm Tri}(Q_1(X),Q_2(X),\dots,Q_{n-k+1}(X))}
{{\rm det}({\rm Tri}(Q_2(X),Q_3(X),\dots,Q_{n-k+1}(X))} \in \RR(X)~
(1 \leqslant k \leqslant n)~.$$
The  Witt class is thus
$$\begin{array}{ll}
[{\rm Tri}(Q(X))]&=\sum\limits_{k=1}^n[P_{k-1}(X)/P_k(X)]\\
&=\sum\limits_{k=1}^n[P_{k-1}(X)P_k(X)] \in W(\RR(X))~.
\end{array}$$
\end{proposition}
\begin{proof} Let $Q^*(X)=(Q_n(X),Q_{n-1}(X),\dots,Q_1(X))$. 
As in the proof of Theorem \ref{sylvestertheorem} define the $n \times n$
orthogonal matrix
$$J~=~\begin{pmatrix} 0 & 0 & \dots & 0 & 1 \\
0 & 0 & \dots & 1 & 0 \\
\vdots & \vdots & \ddots & \vdots & \vdots \\
0 & 1 & \dots & 0 & 0 \\
1 &  0 & \dots & 0 & 0 
\end{pmatrix}$$
such that
$${\rm Tri}(Q^*(X))~=~J^*{\rm Tri}(Q(X))J~.$$
As in the proof of Proposition \ref{signtri0} define the
invertible $n \times n$ matrix
$$A~=~ \begin{pmatrix} 
1 & -P_n(X)/P_{n-1}(X) & P_n(X)/P_{n-2}(X) & \dots & (-1)^{n-1}P_n(X)/P_1(X)\\ 
0 & 1 & -P_{n-1}(X)/P_{n-2}(X) & \dots & (-1)^{n-2}P_{n-1}(X)/P_1(X)\\
0 & 0 & 1 & \dots & (-1)^{n-3}P_{n-2}(X)/P_1(X)\\
\vdots & \vdots & \vdots & \ddots & \vdots \\
0 & 0 & 0 &\dots & 1 
\end{pmatrix}$$
such that
$$A^*{\rm Tri}(Q^*(X))A~=~
\begin{pmatrix} P_{n-1}(X)/P_n(X) & 0 & \dots & 0 \\
0 & P_{n-2}(X)/P_{n-1}(X)  &  \dots & 0\\
\vdots & \vdots &  \ddots & \vdots \\
0& 0 & \dots & P_0(X)/P_1(X)\end{pmatrix}$$
(which can also be deduced from the proof of Theorem~\ref{sjgf}).
\end{proof}

Let us evaluate
$$
\tau_{\RR}([{\rm Tri}(Q(X))])= \sum_{k=1}^n \tau_{\RR} (P_{k-1}(X)P_k(X)) \in \ZZ[\RR].
$$
We have to add signs $\pm$ for each root of one of the polynomials $P_{k-1}(X)P_k(X)$, depending on whether $P_{k-1}(X)P_k(X)$ is increasing or decreasing in the neighborhood of this root.
Note that
\begin{itemize}
\item Any root $x \in \RR$ of $P_k(X)$  for $1\leqslant k \leqslant n-1$ appears twice in $P_0(X)P_1(X)$, $P_1(X)P_2(X),\dots,P_{n-1}(X)P_n(X)$.
\item $P_{k-1}(x)$ and $P_{k+1}(x)$ have opposite signs at any root $x \in \RR$ of $P_k(X)$ since $P_{k-1}(X)+P_{k+1}(X) = P_k(X)Q_k(X)$. Therefore the two signs associated to the two occurrences of some root of $P_k(X)$ cancel ($1\leqslant k \leqslant n-1$).
\item The product $P_0(X)P_1(X)=P(X)P'(X)$ is increasing in the neighborhood of any root of $P_0(X)=P(X)$ (since the derivative of $P(X)P'(X)$ is $P'(X)^2 + P(X)P''(X)$).
\item $P_n(X)$ is a non zero constant so does not vanish!
\end{itemize}
(This is essentially the proof of Sturm's Theorem~\ref{sturmtheorem}).
It follows that
$$
\tau_{\RR}([{\rm Tri}(Q(X))])(x)=
\begin{cases}
0 & \text{if $P(x)\neq 0$}\\
+1&\text{if $P(x)=0$}
\end{cases}
$$
so that the total signature $\tau_{\RR}([{\rm Tri}(Q(X))]) \in \ZZ[\RR]$ is indeed counting the roots of $P$ (with a $+$ sign).

\begin{corollary}\label{preceding}
For any $a<b \in \RR$ the composite
$$\xymatrix@C+10pt{W(\RR(X)) \ar[r]^-{\di{\tau_{\RR}}}  &}\xymatrix@C+45pt{\ZZ[\RR]
\ar[r]^-{\di{\rm restriction}}  & \ZZ[[a,b]] \ar[r]^-{\di{\rm augmentation}} & \ZZ}$$
 sends $({\rm Tri}(Q(X)))$ to
 $$\begin{array}{ll}
 \{\text{no. of real roots of $P(X)$ in $[a,b]$}\}&=(\tau({\rm Tri}(Q(b)))-\tau({\rm Tri}(Q(a))))/2  \\
 & ={\rm var}(a)-{\rm var}(b) \in \{0,1,\dots,n\}.\qed
 \end{array}
 $$
\end{corollary}

For $z\in {\cal H}$, it is easy to evaluate
$$
disc_{\cal H}([{\rm Tri}(Q(X))])(z)= \sum_{k=1}^n disc_{\cal H}(P_k(X)P_{k-1}(X))(z) \in \ZZ/ 2 \ZZ.
$$
As before, since every $P_i$ occurs twice in $P_0P_1, P_1P_2, \dots, P_{n-1}P_n$, except $P=P_0$ and $P_n$, we have:
$$
disc_{\cal H}([{\rm Tri}(Q(X))])(z)=
\begin{cases}
0 & \text{if $P_z$ does not divide $P$.}\\
+1&\text{if $P_z$ divides $P$}.
\end{cases}
$$

We can now formulate our next to last reformulation of Sturm-Sylvester's Theorem (Corollary~\ref{sylvester4}).

\begin{theorem}\label{nexttolast} Let $P\in \RR[X]$ be a monic degree $n+2m$  polynomial with distinct roots.
Let $x_1,x_2, \dots, x_n$ be its real roots and $z_1, \dots , z_m$ its complex roots in ${\cal H}$. Let $(P_0,P_1,\dots,P_{n+2m})$, $(Q_1,Q_2,\dots,Q_{n+2m})$ be the Sturm functions, and let ${\rm Tri}(Q)$ be the associated tridiagonal matrix.
We have the following equality in $W(\RR(X))$:
$$\begin{array}{ll}
[{\rm Tri}(Q)]&
= [X-x_1]+ [X-x_2]+ \dots +[X-x_n]+ [P_{z_1}] + \dots + [P_{z_m}] -m[1]\\[1ex]
&=n[1]+\sum\limits^n_{i=1}([X-x_i]-[1])+\sum\limits^m_{j=1}
([P_{z_j}(X)]-[1])\\[1ex]
&\in W(\RR(X))~=~\ZZ \oplus \ZZ[\RR] \oplus \ZZ/2\ZZ[{\cal H}].
\end{array}
$$
\end{theorem}

\begin{proof}
We already checked that the two sides have the same image by $\tau_{\RR} \oplus {\rm disc}_{\cal H}$.
Their difference is therefore Witt equivalent to a matrix with constant real coefficients.
It is therefore sufficient to check that the signatures of both sides agree for large positive values of $x$.
Let us write $Q_k(X)$ in the form $a_kX+b_k$ ($1 \leqslant k \leqslant n+2m$).
The signature $s$ of ${\rm Tri}(Q)(x)$ for large positive $x$ is the signature of the diagonal matrix with coefficients $a_k$, i.e. is the number of positive $a_k$ minus the number of negative $a_k$.
The signature of ${\rm Tri}(Q)(x)$ for large negative $x$ is the signature of the diagonal matrix with coefficients $-a_k$, so it is the opposite of the signature for $x$ large and positive. 
We know that the differences of signatures of ${\rm Tri}(Q)(x)$ between $x=+\infty$ and $- \infty$ is twice the number of real roots of $P$, i.e. $n$.
Therefore the signature of ${\rm Tri}(Q)(x)$ for $x=+\infty$ (resp. $-\infty$) is equal to $+n$ (resp. $-n$). 
Therefore, the two sides of the equality of the theorem have the same signature for $x$ large.
\end{proof}

\bigskip

\begin{remark}
Our description of the sign-function $sign$ on $W(\RR(X))$ can be expressed in a slightly different language.

An {\it ordering} ${\cal O}$ of a field $K$ is a decomposition as a disjoint union $$K=K_+ \sqcup \{0\} \sqcup K_-$$
such that
\begin{itemize}
\item $x \in K_-$ if and only if $-x\in K_+$.
\item The sum and product of two elements of $K_+$ belongs to $K_+$.
\end{itemize}
\item The {\it ${\cal O}$-sign} of an element $a \in K$ is
$${\rm sign}_{\cal O}(a)=\begin{cases} +1&{\rm if}~x \in K_+ \\
\hspace*{10pt}0&{\rm if}~x=0\\
-1&{\rm if}~x \in K_-
\end{cases}$$
and satisfies the usual product rules for signs.

The  proof of Sylvester's Law of Inertia~\ref{sylvester}, with no modification, enables us to define the ${\cal O}$-signature of  a symmetric $(V,\phi)$ using any of the  diagonalization $(V,\phi) \cong \oplus_{k=1}^n (K,a_k)$ given by Proposition~\ref{diag}.
The number is independent of the diagonalization, since a square $b^2\in K^{\bullet}$ has ${\rm sign}_{\cal O}(b^2)=+1$ for every ordering ${\cal O}$.

The {\it ${\cal O}$-signature} of a nonsingular symmetric form $(V,\phi)$ over $K$ is
$$\tau_{{\cal O}}(V,\phi)=\sum\limits^n_{k=1}{\rm sign}_{\cal O}(a_k) \in \ZZ$$
for any diagonalization $(V,\phi) \cong \oplus^n_{k=1}(K,a_k)$.
In our new language, any ordering ${\cal O}$ on $K$ yields a signature homomorphism (also congruent modulo 2 to the dimension)
$$\tau_{\cal O}: \sum\limits_{k=1}^n(K,a_k) \in W(K) \mapsto \sum\limits^n_{k=1}{\rm sign}_{\cal O}(a_k) \in \ZZ.$$
In the special case of $K=\RR(X)$, we leave to the reader the description of orderings and to relate the corresponding ${\cal O}$-signatures to the map ${\rm sign} : W(\RR(X)) \to {\mathcal Sign}$ that we constructed above.

Pfister~\cite{pfister} proved the great theorem: for any field $K$ of characteristic $\neq 2$ the intersection of all kernels of the homomorphisms $\tau_{\cal O}$ with ${\cal O}$ running over all the orderings of $K$ is torsion.
In other words, for  every $(V,\phi) \in W(K)$ of infinite order, there is some ordering ${\cal O}$ such that  $\tau_{{\cal O}}(V,\phi) \neq 0$.

\end{remark}

\bigskip

The description that we gave of $W(\RR(X))$ ``by hand'' can be greatly generalized and conceptualized.
We refer to~\cite[Chapter 6]{lam} for a detailed exposition and we limit ourselves to a quick sketch since we shall not use this generalization.

Let $K$ be any field (of characteristic different from 2) and let us compute $W(K(X))$.
Any polynomial in $K[X]$ splits as the product of a constant $c$ and irreducible monic polynomials.
Let $\pi\in K[X]$ be such an irreducible monic polynomial.
The quotient $K[X]/\pi$ is the {\it residue field} at $\pi$.

As usual in commutative algebra, one can construct the {\it localization of $K[X]$ at $\pi$}, inverting all elements which are not divisible by $\pi$.
This produces a ring $K_{\pi}[X]=S_{\pi}^{-1}K[X]$ whose field of fractions is $K(X)$,
with $S_{\pi}=\{\pi^k \,\vert\, k \geqslant 0\}$.

For instance, when $K=\RR$ and $\pi = X-a$, the ring $K_{\pi}[X]$ consists of rational functions $P(X)/Q(X)$ which are finite in a neighborhood of $a$, i.e. such that $Q(a)\neq 0$.
So, the passage from $K[X]$ to $K_{\pi}[X]$ is indeed a ``localization around $a$''.

$K(X)$ is equipped with a discrete valuation  $v_{\pi}$ with values in $\ZZ \cup \{ \infty \}$ such that
$v_{\pi} (\pi^n P(X)/Q(X)) = n $ if $P,Q$ are coprime to $\pi$.
The subring consisting of elements with non negative valuation is by definition the local ring $K_{\pi}[X]$.

The norm $\exp( - v_{\pi})$ defines a topology on $K(X)$ which is not complete.
Denote its completion by $K_{\pi}(X)$.

For instance, when $K=\RR$ and $\pi = X-a$, the residue field $K[X]/\pi$ is isomorphic to $\RR$.
As for the field $K_{\pi}(X)$, it  consists of Laurent series at $a$, i.e. formal sums
$$
\sum_{k=v}^{+\infty} \alpha_k(X-a)^k
$$
with $\alpha_k\in \RR$ and $\alpha_v\neq 0$.
The valuation of such an element is $v\in \ZZ$.
Note that series whose valuation is $0$ are precisely the invertible elements of the local ring $K_{\pi}[X]$, i.e. of the ring of elements with non negative valuation.
In general, the field $K_{\pi}(X)$ consists of series $\sum_{k=v}^{+\infty} \alpha_k\pi^k$ where the coefficients $\alpha_k$ are now in the residue field $K[X]/\pi$.

We define a ``residue map''
$$
{\rm res}_{\pi} : W(K(X)) \to W(K[X]/\pi).
$$
The construction is easy as a composition of two natural maps.

The first, $W(K(X)) \to W(K_{\pi}(X))$,  is simply induced by the embedding of $K(X)$  in its completion $K_{\pi}(X)$.

The second goes from $W(K_{\pi}(X))$ to the Witt group of its residue field $W(K[X]/\pi)$.
Denote by $U$ the group of unit elements of $K_{\pi}[X]$ (with valuation $0$).
Modulo squares, a non zero element $q$ of  $K_{\pi}(X)$ can be written in the form $u$ or $\pi u$ for some unit
$$
u= \sum_{k=0}^{+\infty} \alpha_k\pi^k
$$
with $\alpha_0 \in K[X]/\pi$ different from $0$.
One sets ${\rm res}_{\pi} ([u]) =0$ and ${\rm res}_{\pi} ([\pi u]) = [\alpha_0] \in W(K[X]/\pi)$.

In the case $K= \RR$, and $\pi= X-a$ the residue map ${\rm res}_{a} : W(K(\RR)) \to W(\RR)$ is such that
${\rm res}_{a}(\alpha_{-1} (X-a)^{-1})$ is $+1$ or $-1$ according to the sign of $\alpha_{-1}$ and this is coherent with the classical concept of residue in complex analysis.

In effect, Milnor (\cite[Chapter IV]{milnorhusemoller}) established the following exact sequence for any field $K$ of characteristic $\neq 2$, generalizing our computation for $K = \RR$
$$
0 \to W(K) \xrightarrow{i}  W( K(X)) \xrightarrow{p} \bigoplus_{ \pi \in \RR[X]~ {\rm prime} } W(K[X]/\pi) \to 0.
$$

The proof follows the same line as the special case that we examined in detail.

We can at last give our final reformulation of Sturm-Sylvester's theorem, valid in any field (always of characteristic different from 2).
Let $P$ be a monic separable element of $K[X]$ (i.e. with distinct roots in the algebraic closure of $K$). 
Let $P = \pi_1 \dots \pi_k$ be its decomposition as a product of prime polynomials in $K[X]$.
We still denote by ${\rm Tri}(Q)$ the tridiagonal matrix associated to $P$ by Sturm's algorithm.
The following theorem gives explicitly the image $p([{\rm Tri}(Q)])$ in $\bigoplus_{ \pi \in \RR[X]~ {\rm prime} } W(K[X]/\pi)$, component wise.

\begin{theorem} 
If $\pi$ is prime in $K[X]$, we have
$$
res_{\pi} ([{\rm Tri}(Q(X))]) = 
\begin{cases}
0  \in W(K[X]/\pi)  & {\rm if}~\pi \text{ is not a divisor of } P \\
[\pi'] \in W(K[X]/\pi) & {\rm if}~\pi   \text{ is  a divisor of } P,
\end{cases}
$$
where $\pi'$ denotes the derivative of $\pi$.
\end{theorem}

In the case $K=\RR$, and $\pi = X-a$, we have $\pi' =1 $ and 
$$W(\RR[X]/(X-a))=W(\RR) = \ZZ.$$
When $\pi=P_z$, the derivative $\pi'$ is the non trivial element in $W(\RR[X]/P_z(X))=W(\CC) = \ZZ/2\ZZ$.
Therefore, we do recover the previous version~\ref{nexttolast}. 

\begin{proof} 
We know from the discussion preceding Corollary \ref{preceding} (which applies to an arbitrary $K$) that 
$$[{\rm Tri}(Q(X))]=\sum^n_{k=1}[P_k(X)P_{k-1}(X)] \in W(K(X))$$
where as usual $P_k$ denotes the Sturm's functions.
Choose some prime element $\pi$ in $K[X]$. 
We have to show that $\sum^n_{k=1}[P_kP_{k-1}]$  and $[\pi_1\pi_1']+ \dots + [\pi_k \pi_k']$ have the same residue at $\pi$, i.e. the same image by ${\rm res}_{\pi}$.
This residue can only be non trivial if $\pi$ is a divisor of one of the $P_k$.
Note that by the recurrence formula defining the $P_k$'s, $\pi$ cannot divide two consecutive $P_k$ since that would imply that it divides all the $P_k$ and in particular $P_n$ which is a non vanishing constant.
We first compute 
$$
{\rm res}_{\pi} (\sum^n_{k=1}[P_kP_{k-1}] )
$$
when $\pi$ divides $P_k$ with $0<k<n$ so that $P_k= \pi R$ for some polynomial $R\in K[X]$. 
Then the previous sum reduces to 
$$
{\rm res}_{\pi} ([P_kP_{k-1}] + [P_{k+1}P_{k}]).
$$
By the above definition of the residue:
$$
{\rm res}_{\pi} ([P_kP_{k-1}]) = P_kP_{k-1}/ \pi = RP_{k-1} \in W(K[X]/ \pi)
$$
and similarly 
$$
{\rm res}_{\pi} ([P_{k+1}P_{k}]) = P_{k+1}P_{k}/ \pi = P_{k+1}R \in W(K[X]/ \pi)
$$
Since $P_{k-1}+P_{k+1}= P_k Q_k$, it follows that $P_{k-1}+P_{k+1}$ is divisible by $\pi$ and therefore vanishes in $K[X]/\pi$.
Hence the sum of these two residues vanishes.
Therefore the only non trivial residues are associated to the prime divisors of $P_0$, i.e. to one of the $\pi_i$ ($1\leqslant i \leqslant k$).

We have to compute ${\rm res}_{\pi_i} ([P_0P_1]) \in W(K[X]/\pi_i)$. 
Write $P=P_0$ as $\pi_i A$ where $A\in K[X]$ is the product of the $\pi_j$ with $j \neq i$.
Since $P_1= P_0'$, we have $P_1= \pi_i A' + \pi_i' A$ and 
$$
P_0P_1= \pi_i A (\pi_i A' + \pi_i' A)= \pi_i \pi_i' A^2 + \pi_i^2 AA'
$$
In particular the $\pi_i$-valuation of $P_0P_1$ is 1 and one has
$$
{\rm res}_{\pi_i} ([P_0P_1]) = [\pi_i' A^2+ \pi_i AA'] = [\pi_i'] \in W(K[X]/\pi_i).
$$
In other words
$$
res_{\pi_i}([{\rm Tri}(Q)]) =  [\pi_i'] \in W(K[X]/\pi_i).
$$
\end{proof}

The case of the rationals $\QQ$ is another example where this general technique applies.
Prime ideals of $\ZZ$ correspond to prime numbers $p$.
The residue fields are $\FF_p$.
The completion of the field of fractions $\QQ$ of $\ZZ$ for the $p$-adic valuation is the field $\QQ_p$ of $p$-adic numbers.
The completion of $\QQ$ with the archimedean valuation is of course $\RR$, with Witt group $\ZZ$.
The residue map goes from $W(\QQ)$ to $W(\FF_p)$ and the corresponding exact sequence is now

$$0 \rightarrow \ZZ \rightarrow W(\QQ )\rightarrow \ZZ /2\oplus \bigoplus\limits _{p~{\rm odd}}W(\FF_p)\rightarrow 0.$$

See Milnor and Husemoller~\cite[Chapter III]{milnorhusemoller} for many computations of the symmetric Witt group $W(K)$.
In particular, the Witt group of a finite field $\FF_q$ with $q$ odd (i.e. a power of an odd prime $p$) is
$$W(\FF_q)~=~\begin{cases} 
(\ZZ/2\ZZ)[\FF_q^{\bullet}/(\FF_q^{\bullet})^2]&{\rm if}~q \equiv 1(\bmod\,4)\\[1ex]
\ZZ/4\ZZ&{\rm if}~q\equiv 3(\bmod\,4)
\end{cases}$$
in accord with the Legendre symbol for deciding if $-1$ is a square $\bmod\,p$
$$\bigg(\dfrac{-1}{p}\bigg)=\begin{cases} 1&\text{if $p\equiv 1(\bmod\,4)$}\\
-1&\text{if $p\equiv 3(\bmod\,4)$}~.
\end{cases}$$
In the Appendix (Section \ref{appendix}) we shall recall the definition of the Witt group $W(R)$ of $R$-valued nonsingular symmetric forms on f.g. free $R$-modules, for any commutative ring $R$, and describe the localization exact sequence relating $W(R)$ and $W(S^{-1}R)$  for the localization $S^{-1}R$ inverting a multiplicative subset $S \subset R$ of non-zero divisors
$$\xymatrix{\dots \ar[r] & W(R) \ar[r] & W(S^{-1}R) \ar[r]^-{\di{\partial}} & W(R,S) \ar[r] &\dots}$$
with $W(R,S)$ the Witt group of $S^{-1}R/R$-valued nonsingular symmetric forms on $S$-torsion $R$-modules of homological dimension 1.   
The Witt group $W(\QQ)$ is given by $(R,S)=(\ZZ,\ZZ\backslash \{0\})$ with $S^{-1}R=\QQ$, as detailed in \cite{milnorhusemoller}.
We shall relate the above computation of $W(K(X))$ to the case $(R,S)=(K[X],K[X]\backslash \{0\})$ with $S^{-1}R=K(X)$. 
In particular, for each prime $p \geqslant 2$ there is defined a multiplicative subset
$S_p=\{p^k\,\vert\, k \geqslant 0\} \subset \ZZ$ with $S_p^{-1}\ZZ=\ZZ[1/p]$,  and
$$W(\ZZ,\ZZ\backslash \{0\})~=~\bigoplus\limits_{p\geqslant 2} W(\ZZ,S_p)~=~\bigoplus\limits_{p \geqslant 2}W(\FF_p)$$
with
$$W(\FF_p)~=~\begin{cases}
\ZZ/2\ZZ&{\rm if}~p=2\\[1ex]
\ZZ/2\ZZ\oplus \ZZ/2\ZZ&{\rm if}~p \equiv 1(\bmod\, 4)\\[1ex]
\ZZ/4\ZZ&{\rm if}~p \equiv 3(\bmod\, 4).
\end{cases}$$

\newpage

\section{Topology}\label{topology}

$$
\includegraphics[width=.9\linewidth]{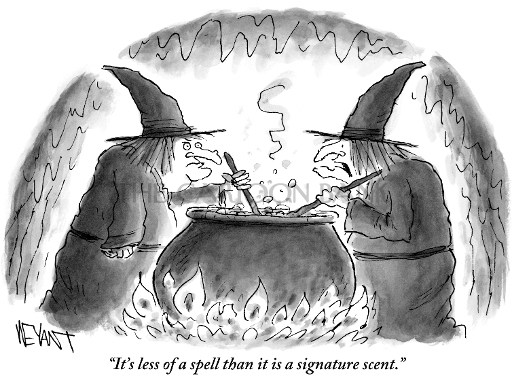}
$$

\subsection{Even-dimensional manifolds}

We shall only consider {\it compact oriented smooth manifolds, possibly with a non empty boundary}, and in the first instance work with homology and cohomology with coefficients in a commutative ring $R$.

Recall that for any commutative ring $R$  the {\it cup product} equips the $R$-coefficient cohomology of any space $X$
$$H^{\star}(X;R) =\bigoplus\limits_{k\geqslant   0} H^k(X;R)$$
with a canonical structure of a graded associative and commutative ring.
Commutativity should be understood in the graded sense.

For a {\it closed} (i.e. without boundary) oriented $m$-dimensional manifold $M$ there is a {\it Poincar\'e duality}, which can be expressed as an isomorphism between $H^k(M;R)$ and $H_{m-k}(M;R)$.
Using this duality, the ring structure on cohomology then transforms to an {\it intersection product} between homology classes:
$$\phi_M:~H_k(M;R)\times H_{\ell}(M;R) \to H_{k+\ell-m}(M;R).$$
The intersection in homology is also well defined for a manifold $M$ with boundary $\partial M$, using the Poincar\'e-Lefschetz duality $H^k(M;R) \cong H_{m-k}(M,\partial M;R)$. For $k+\ell=m$ composition with the augmentation map
$H_0(M;R) \to R$ gives a bilinear pairing
$$\phi_M:~H_k(M;R) \times H_{\ell}(M;R) \to H_0(M;R) \to R$$
such that
$$\phi_M(v,u)=(-1)^{k\ell}\phi_M(u,v) \in R$$
for $u \in H_k(M;R)$ and $v \in H_{\ell}(M;R)$.

\begin{center}
\includegraphics[width=.35 \textwidth]{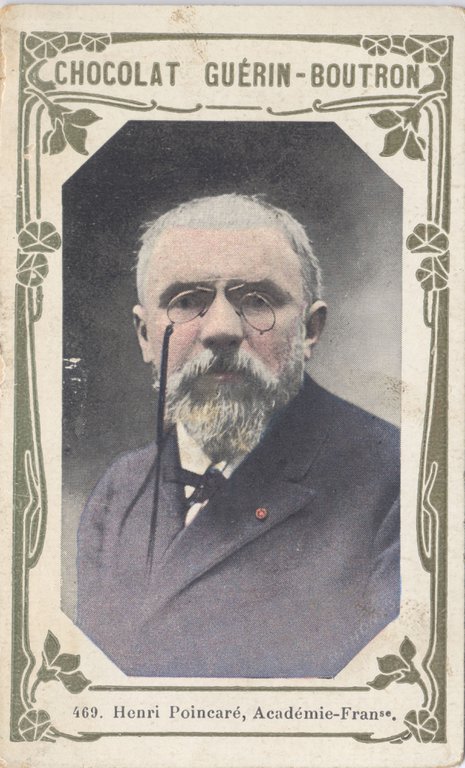}

Henri Poincar\'e (1854--1912)
\end{center}

We shall be particularly concerned with the cases $R=\ZZ,\RR$.
In this case, the homology abelian groups $H_k(M;\ZZ)$ are f.g., and the homology vector spaces $H_k(M;\RR)$ are finite dimensional.

\begin{remark}
A smooth oriented submanifold $L^\ell \subset M$ represents a $\ZZ$-coefficient homology class $[L] \in H_\ell(M, \ZZ)$.
For submanifolds $K^k,L^\ell \subset M^m$ the algebraic intersection $\phi_M([K],[L]) \in H_{k+\ell-m}(M;\ZZ)$ has a geometric description:  by a small deformation, one can assume that the two submanifolds are transverse and one considers the $\ZZ$-coefficient homology class of the oriented intersection. If $K^k,L^{\ell}  \subset M^m$ are transverse then $K \cap L \subseteq M$ is a $(k+\ell-m)$-dimensional submanifold (understood to be $\emptyset$ if $k+\ell<m$) and
$$
\phi_M([K],[L])=[K \cap L] \in H_{k+\ell-m}(M;\ZZ).
$$
For $k+\ell=m$ the intersection $K \cap L$ is a finite set of points, and the augmentation $H_0(M;\ZZ) \to \ZZ$ (an isomorphism for connected $M$) sends $\phi_M([K],[L]) \in H_0(M;\ZZ)$ to the number of points, with a sign given by the orientations.
\end{remark}

Therefore, restricting to the case of the middle dimension, there is a canonical $\pm$-symmetric form associated to an even dimensional manifold.

\begin{definition} For any  $2\ell$-dimensional manifold with boundary $(M,\partial M)$  there is defined a $(-1)^\ell$-symmetric pairing
$$
\phi_M:~H_\ell(M;R) \times H_{\ell}(M;R) \to R.
$$
\end{definition}

If $H_\ell(M;R)$ is a f.g. free $R$-module this is an example of a $(-1)^{\ell}$-symmetric form over $R$, an obvious generalization of the symmetric forms over a field $F$ we were considering in section~\ref{algebra}.

\begin{definition} Let $\epsilon = \pm 1$.

1. A {\it $\epsilon$-symmetric form over R} $(V,\phi)$ is a f.g. free $R$-module $V$ with a bilinear pairing $\phi:V \times V \to R$ such that $\phi(v,w)=\epsilon\phi(w,v)$.

2. A form $(V,\phi)$ is {\it nonsingular} if the $R$-module morphism $\phi:V \to V^*={\rm Hom}_R(V,R)$ given by $\phi(v) (w)=\phi(v,w)$ is an isomorphism.

3. An {\it isomorphism} $f:(V,\phi) \to (V',\phi')$ of $\epsilon$-symmetric forms over $R$ is an $R$-module isomorphism $f:V \to V'$ such that $\phi'(f(v) ,f(w))=\phi(v,w)$.
\end{definition}

(In the Appendix~\ref{appendix} we shall even be considering forms over a ring with involution).

Let $F_{\ell}(M;\ZZ)$ denote the quotient of  $H_{\ell}(M;\ZZ)$ by its torsion.
This is a f.g. free $\ZZ$-module which is also equipped with the $(-1)^{\ell}$-symmetric intersection form $\phi_M$.
The {\it  intersection matrix} of $(M,\partial M)$ is the  integral $(-1)^{\ell}$-symmetric  $n \times n$ matrix
$$S_M=\big(\,\phi_M(b_i,b_j)\in \ZZ\,\big)$$
with respect to some basis $\{b_1,b_2,\dots,b_n\}\subset F_{\ell}(M;\ZZ)$.
The intersection matrix is invertible for closed $M$ or if $\partial M=\SSS^{2\ell-1}$, in which case it has determinant $\pm 1$ and the form $(F_\ell(M;\ZZ),\phi_M)$ over $\ZZ$ is nonsingular.

In 1923, Hermann Weyl~\cite{weylsign} defined the {\it  signature}  of an oriented $4k$-dimensional manifold with boundary $(M,\partial M)$ as the signature of the intersection symmetric $n \times n$ matrix $S_M$
$$
\tau(M)=\tau(S_M) \in \ZZ.
$$
See Eckmann~\cite{eckmann} for the backstory of Weyl's paper.

This is Weyl's signature:
$$
\includegraphics[width=.8\linewidth]{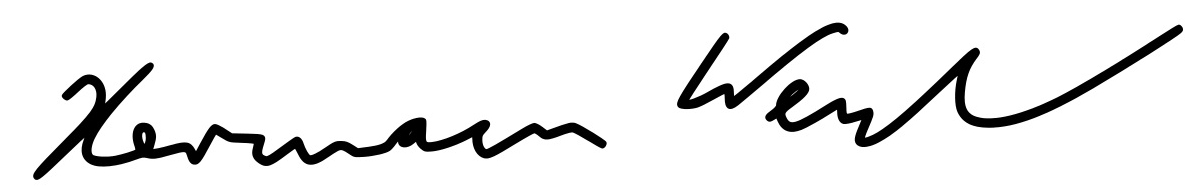}
$$

Thom~\cite{thom} proved that any homology class in $H_{\ell}(M;\ZZ)$ has an integral multiple which is represented by an $\ell$-dimensional submanifold $P^{\ell} \subset M$.
Thus for any basis $\{b_1,b_2,\dots,b_n\} \subset F_{\ell}(M;\ZZ)$ there exist non zero integers $k_1,k_2,\dots,k_n$ such that $k_i b_i=[P_i] \subset F_{\ell}(M;\ZZ)$ for  some $\ell$-dimensional submanifolds $P_i^{\ell} \subset M$ in general position.
The intersection number $\phi_M(b_i,b_j)$ is therefore equal to $\phi_M([P_i],[P_j])/k_ik_j$.
As mentioned earlier, the intersection $\phi_M([P_i],[P_j])=[P_i \cap P_j]$ can be interpreted as the number of signed intersections of $P_i$ and $P_j$ in general position.

The intersection matrix $S_M$ determines the intersection form over $\RR$ since  $\RR\otimes_{\ZZ} F_{\ell}(M;\ZZ) =H_{\ell}(M;\RR).$
Very frequently, we shall forget the integral structure on $H_\ell(M;\ZZ) $ and only consider the $(-1)^{\ell}$-symmetric form $\phi_M$ on the vector space $H_\ell(M;\RR)$.
Its radical is
$$\begin{array}{ll}
H_{\ell}(M;\RR)^{\perp}&=~\text{ker}(\phi_M:H_{\ell}(M;\RR) \to H_{\ell}(M;\RR)^*)\\
&=~\text{im}(H_{\ell}(\partial M;\RR)) \subseteq H_{\ell}(M;\RR).
\end{array}
$$

\begin{example}
The $2k$-dimensional complex projective space $\CC\PP^{2k}$ is a closed $4k$-dimensional manifold with the homology $H_{2k}(\CC\PP^{2k};\ZZ)=\ZZ$  generated by the fundamental class of some $\CC\PP^{k}$ linearly embedded in $\CC\PP^{2k}$ with self-intersection 1.
The intersection $1 \times 1$ matrix is the symmetric $1\times 1$ matrix
$$
S_{\CC\PP^{2k}}=(1),
$$
so that the signature is
$$
\tau(\CC\PP^{2k})=1 \in \ZZ.
$$
\end{example}

Note that if one reverses the orientation of $M$, the signature is changed to its opposite, $\tau(-M)=-\tau(M)$.

Note that when the dimension of the manifold has the form $2 \ell= 4 k+2$, the intersection form is skew symmetric and therefore there is no notion of signature.

In order to understand the intersection form on the boundary of an odd-dimensional manifold we recall some standard algebraic notions:

\begin{definition}\leavevmode

1. Let $(V,\phi)$ be an $\epsilon$-symmetric form over $R$,
with $V$ a f.g. free $R$-module and  $\phi=\epsilon \phi^*:V \to V^*=\text{Hom}_R(V,R)$.
Given a submodule $L \subseteq V$ let $L^{\perp}=\{x \in V\,\vert\,\phi(x,y)=0 \in R~\hbox{\rm for all}~y \in L\} \subseteq V$.
The submodule is \emph{isotropic} if $L \subseteq L^{\perp}$.
An isotropic submodule is a \emph{sublagrangian} if $L$ and $L^{\perp}$ are f.g. free direct summands of $V$, and $x \in V \mapsto (y \mapsto \phi(x,y)) \in L^*$ is surjective.
A \emph{lagrangian} $L$ is a sublagrangian such that $L=L^{\perp}$.

2. The \emph{$\epsilon$-symmetric Witt group} $W_{\epsilon}(R)$ is the group of equivalence classes of nonsingular $\epsilon$-symmetric forms over $R$, with $(V,\phi) \sim (V',\phi')$ if and only if $(V,\phi) \oplus (H,\theta)$ is isomorphic
to $(V',\phi')\oplus (H',\theta')$ for some forms $(H,\theta)$, $(H',\theta')$ with lagrangians.
\end{definition}

Of course, if $R=F$ is a field and $\epsilon=1$ this is just the Witt group $W(F)$ already considered in section~\ref{algebra}.

\begin{proposition} \label{lagr}

1. For a $(2\ell+1)$-dimensional manifold with boundary $(N,\partial N)$
the submodule
$$L={\rm ker}(H_{\ell}(\partial N;R) \to H_{\ell}(N;R)) \subset H_{\ell}(\partial N;R)$$
is isotropic with respect to the  $(-1)^{\ell}$-symmetric pairing
$\phi_{\partial N}:H_{\ell}(\partial N;R) \times H_{\ell}(\partial N;R) \to R$.

2. If $R$ is a field then $L$ is a lagrangian of the nonsingular
$(-1)^{\ell}$-symmetric form  $(H_{\ell}(\partial N;R),\phi_{\partial N})$ over $R$.
In particular, for $R=\RR$, $\ell=2k$ this gives $\tau(\partial N)=0 \in \ZZ$.
\end{proposition}
\begin{proof} Consider the commutative diagram with exact rows
$$\xymatrix{H_{\ell+1}(N,\partial N;R) \ar[r] \ar[d] &
H_{\ell}(\partial N;R) \ar[r] \ar[d] & H_{\ell}(N;R) \ar[d] \\
H^{\ell}(N;R) \ar[r] \ar[d] & H^{\ell}(\partial N;R) \ar[d] \ar[r]
& H^{\ell+1}(N,\partial N;R) \ar[d] \\
H_{\ell}(N;R)^* \ar[r] & H_{\ell}(\partial N;R)^* \ar[r] & H_{\ell+1}(N,\partial N;R)^*}$$
\end{proof}

\begin{example}
Consider $M^{2\ell}=\SSS^{\ell} \times \SSS^{\ell}=\partial N$ with $N^{2\ell+1}=\DDD^{\ell+1} \times \SSS^{\ell}$.
The intersection $(-1)^{\ell}$-symmetric form of $M$ is hyperbolic
$$
(H_{\ell}(M;R),\phi_M)=H_{(-1)^{\ell}}(R)=
(R \oplus R,\begin{pmatrix} 0 & 1 \\ (-1)^{\ell} & 0 \end{pmatrix})
$$
with $H_\ell(M;R)=R \oplus R$, generated by the two factors intersecting in one point and each having self intersection 0.
The lagrangian determined by $N$ is
$$
L={\rm ker}(H_{\ell}(M;R) \to H_{\ell}(N;R))=R \oplus 0 \subset H_{\ell}(M;R)=R \oplus R.
$$
\end{example}

By definition, a space $X$ is {\it $k$-connected} for some $k \geqslant   1$ if it is connected and satisfies one of the equivalent conditions:
\begin{itemize}
\item[1.] $\pi_j(X)=0$ for $1 \leqslant j \leqslant k$,
\item[2.] $\pi_1(X)=\{1\}$ and $H_j(X;\ZZ)=0$ for $1 \leqslant  j \leqslant k$.
\end{itemize}
We shall say that an $m$-dimensional manifold with boundary $(M,\partial M)$ is {\it $k$-connected} if $M$ is $k$-connected and $\partial M$ is connected.

For an $(\ell-1)$-connected $2\ell$-dimensional manifold with boundary $(M,\partial M)$ we have that the $\ZZ$-module $H_{\ell}(M;\ZZ)$ is f.g. free, with the natural $\ZZ$-module morphism defined by evaluation
$$
f \in H^{\ell}(M;\ZZ) \mapsto (u \mapsto f(u)) \in H_{\ell}(M;\ZZ)^*
$$
an isomorphism.
The intersection pairing thus defines a $(-1)^{\ell}$-symmetric form $(H_{\ell}(M;\ZZ),\phi_M)$ over $\ZZ$.
Furthermore, $H_j(\partial M;\ZZ)=0$ for $1 \leqslant j \leqslant \ell-2$, and in view of the exact sequence
$$
\xymatrix{0 \ar[r] & H_{\ell}(\partial M;\ZZ) \ar[r]
&H_{\ell}(M;\ZZ) \ar[r]^-{\phi_M} & H_{\ell}(M;\ZZ)^* \ar[r] & H_{\ell-1}(\partial M;\ZZ) \ar[r] &0}
$$
we have that  :
\begin{itemize}
\item[1.] $(H_{\ell}(M;\ZZ),\phi_M)$ is nondegenerate (i.e. ${\rm det}(\phi_M) \neq 0 \in \ZZ$) if and only if $H_{\ell}(\partial M;\ZZ)=0$, if and only if $\partial M$ is a $\QQ$-coefficient homology $(2\ell-1)$-sphere,
\item[2.]  $(H_{\ell}(M;\ZZ),\phi_M)$ is nonsingular if and only if $H_{\ell}(\partial M;\ZZ)=H_{\ell-1}(\partial M;\ZZ)=0$, if and only if $\partial M$ is a $\ZZ$-coefficient homology $(2\ell-1)$-sphere.
\end{itemize}

\subsection{Cobordism}

\begin{definition}
An {\it  $(m+1)$-dimensional cobordism} $(N;M,M')$ is an $(m+1)$-dimensional manifold $N$ with the boundary decomposed as $\partial N=M \sqcup -M'$ for
closed $m$-dimensional manifolds $M$, $M'$, where $-M'$ = $M'$ with the opposite orientation.
\end{definition}

Signature is cobordism invariant:

\begin{theorem}[Thom~\cite{thom}] \label{cobinv}
If $(N;M,M')$ is a $(4k+1)$-dimensional cobordism  then
$$
\tau(M)-\tau(M')=\tau(\partial N)=0 \in \ZZ.
$$
\end{theorem}

\begin{proof} As in Proposition~\ref{lagr} 2., for any field $R$ the subspace
$$
L=\text{ker}(H_{2k}(\partial N;R) \to H_{2k}(N;R)) \subset H_{2k}(\partial N;R)
$$
is a lagrangian of the intersection symmetric form
$$
(H_{2k}(\partial N;R),\phi_{\partial N})=(H_{2k}(M;R),\phi_{M}) \oplus (H_{2k}(M';R),-\phi_{M'}).
$$
For $R=\RR$ we thus have that $\partial N=M \sqcup -M'$ has signature $\tau(\partial N)=0$.
\end{proof}

The set of cobordism classes of {\it closed} $n$-dimensional manifolds is an abelian group $\Omega_n$, with addition by disjoint union, and the cobordism class of the empty manifold $\emptyset$ as the zero element.
Thom~\cite{thom} initiated the computation of $\Omega_n$, starting with the signature
$$
\tau:~\Omega_{4k} \to W(\RR)=\ZZ;~M \mapsto \tau(M)=(H_{2k}(M;\RR),\phi_M)
$$
It turns out that  $\tau$ is an isomorphism for $k=1$ and a surjection for $k\geqslant   2$.
Each $\Omega_n$ is finitely generated.

Recall that every vector bundle over any space $M$ canonically defines characteristic classes, which are cohomology classes in $M$ (Milnor and Stasheff~\cite{milnorstasheff}).
With integral coefficients, this leads to the {\it Pontrjagin classes} $p_k$ which belong to $H^{4k}(M;\ZZ)$.
This applies in particular to the tangent bundle of a smooth closed manifold $M$ and produces the {\it Pontrjagin classes} of the manifold.
If $4k=4k_1+4k_2+\dots+ 4k_i$, the class $p_{k_1}p_{k_2}\dots p_{k_i}$ is in $H^{4k}(M;\ZZ)=\ZZ$ (for connected $M$) and is therefore an integer: one speaks of the {\it Pontrjagin numbers} of the manifold $M$.
They are invariant under cobordism.
For more on cobordism theory, see Milnor and Stasheff~\cite{milnorstasheff}.

The signature is clearly a homotopy invariant.
The Pontrjagin classes are diffeomorphism invariants but not homeomorphism invariants, let alone homotopy invariants.  It was therefore a surprise when in 1953, Hirzebruch~\cite{hirzebruch1953},
\cite{hirzebruch1971} showed that the signature of a closed $4k$-dimensional manifold $M$ is expressed in terms of characteristic numbers. More precisely, there is a formula:
$$
\tau(M)=\int_M L(M)=
\langle L_k(p_1,p_2,\dots,p_k),[M] \rangle \in \ZZ \subset \RR
$$
where $L_k$ is a certain polynomial in the Pontrjagin classes. For instance,

\begin{itemize}
\item $L_1=\frac{1}{3}p_1$
\item $L_2=\frac{1}{45}(7p_2 - p_1^2)$
\item $L_3=\frac{1}{945}(62p_3 -13p_1p_2+2p_1^3)$
\item $L_4=\frac{1}{14175}(381p_4 - 71p_1p_3 - 19p_2^2+22p_1^2p_2 -3p_1^4)$
\end{itemize}

For a {\it Mathematica} code generating the $L_k$'s, see McTague~\cite{mctague}

\medskip

In 1956 Milnor~\cite{milnorexotic} used the failure of the Hirzebruch signature theorem for manifolds {\it with boundary} to detect that certain 7-dimensional manifolds homeomorphic to $\SSS^7$ had exotic (i.e. non-standard) differentiable structures;
 the subsequent work of Kervaire and Milnor~\cite{kervairemilnor} used the signatures of manifolds with boundary as essential tools in the classification of all high-dimensional exotic 
 spheres.\footnote{See the October 2015 issue of the Bulletin of the American Mathematical Society 
 dedicated to the work of John Milnor,  for a collection of materials relating to the history of manifolds in general and exotic spheres in particular.}$\,$In 1965 Novikov~\cite{novikov1965} used signatures of non-compact manifolds to prove that the rational Pontrjagin classes $p_k(M) \in H^{4k}(M;\QQ)$ are homeomorphism invariants.

A different approach to the Hirzebruch signature theorem is due to Atiyah and Singer~\cite{atiyahsinger}.
In the 1960's they proved the ``Atiyah-Singer Index Theorem'' expressing the analytic index of an elliptic operator on a closed manifold in terms of characteristic classes.
The signature is the index of the signature operator: the index theorem in this case recovers the Hirzebruch signature theorem.
The proof of the index theorem is a piece of cake (designed by John Roe and baked by Ida Thompson for Michael Atiyah's 75th birthday in 2004):

$$
\includegraphics[width=.5\linewidth]{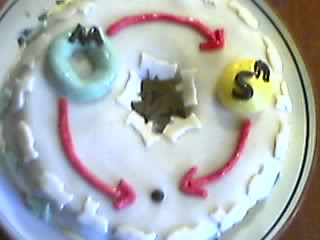}
$$

\subsection{Odd-dimensional manifolds}

Just as forms are the basic algebraic invariants of even-dimensional manifolds, so {\it formations} are the basic algebraic invariants of odd-dimensional manifolds.

\begin{definition} (Ranicki~\cite[p.68]{ranickiexact}) \label{formation}
Let $\epsilon=\pm 1$.

1. An \emph{  $\epsilon$-symmetric formation $(H,\theta;F,G)$ over $R$} is a nonsingular $\epsilon$-symmetric form $(H,\theta)$  over $R$ together with a lagrangian $F$ and a sublagrangian $G$.

2. An \emph{isomorphism} of formations
$$f:~(H,\theta;F,G) \to (H',\theta';F',G')$$
is an isomorphism of forms $f:(H,\theta) \to (H',\theta')$ such that $f(F)=F'$, $f(G)=G'$.
A {\it stable isomorphism} of formations
$$[f]:~(H,\theta;F,G) \to (H',\theta';F',G')$$
is an isomorphism of the type
$$f:~(H,\theta;F,G) \oplus (K,\psi;I,J) \to (H',\theta';F',G') \oplus (K',\psi'; I',J')$$
with $K=I \oplus J$, $K'=I' \oplus J'$.

3. The \emph{boundary} of the formation is the nonsingular $\epsilon$-symmetric form
$$
\partial(H,\theta;F,G)=(G^{\perp}/G,[\theta]).
$$

4.  The formation is \emph{nonsingular} if $\partial(H,\theta;F,G)=0$, i.e. if  $G$ is a lagrangian.

5.  The \emph{boundary} of an $\epsilon$-symmetric form $(V,\phi)$ over $R$ is the nonsingular $(-\epsilon)$-symmetric formation over $R$
$$
\partial (V,\phi)=(H_{-\epsilon}(V);V,\{(v,\phi(v))\,\vert\, v \in V\})
~\text{with}~H_{-\epsilon}(V)=(V \oplus V^*,\begin{pmatrix} 0 & 1 \\ -\epsilon & 0\end{pmatrix}).
$$
\end{definition}

\begin{proposition} \label{stablelagrangian} \leavevmode

1. A nonsingular $\epsilon$-symmetric form $(V,\phi)$ (over any $R$) is such that $(V,\phi)=0 \in W_{\epsilon}(R)$ if and only if $(V,\phi)$ is (isomorphic to) the boundary $\partial(H,\theta;F,G)$ of an $\epsilon$-symmetric formation $(H,\theta;F,G)$ over $R$.

2. If $(H,\theta;F,G)$ is an $\epsilon$-symmetric formation over a field $R$ then
$$
L=((F+G) \cap G^{\perp})/G \subseteq G^{\perp}/G
$$
is a lagrangian of the nonsingular $\epsilon$-symmetric form $\partial(H,\theta;F,G)=(G^{\perp}/G,[\theta])$.

3. An automorphism $A:(H,\theta) \to (H,\theta)$ of a nonsingular $\epsilon$-symmetric form $(H,\theta)$ with a lagrangian $F \subset H$ determines a nonsingular $\epsilon$-symmetric formation $(H,\theta;F,A(F))$.

4. If $1/2 \in R$ for any nonsingular $\epsilon$-symmetric formation $(H,\theta;F,G)$ there exists an automorphism $A:(H,\theta) \to (H,\theta)$ such that $A(F)=G$.
\end{proposition}

\begin{proof}
1. Immediate from the extension of Proposition~\ref{extension} 1. from a field to a ring.

2. Immediate from 1. and Remark~\ref{remark1} 2.

3. Trivial.

4. Immediate from the extension of Proposition~\ref{extension} 2.
\end{proof}

We denote by $\Sigma_n$ the {\it  orientable surface of genus} $n$.
 $$\begin{array}{c}
\includegraphics[width=.25\linewidth]{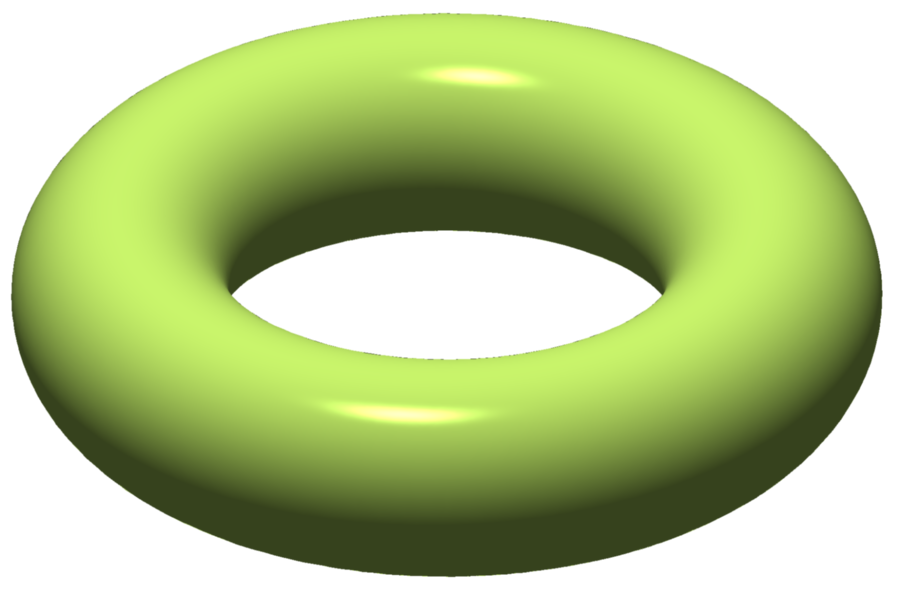}\\
\text{$\Sigma_1$}
\end{array} \hskip7mm
\begin{array}{c}
\includegraphics[width=.25\linewidth]{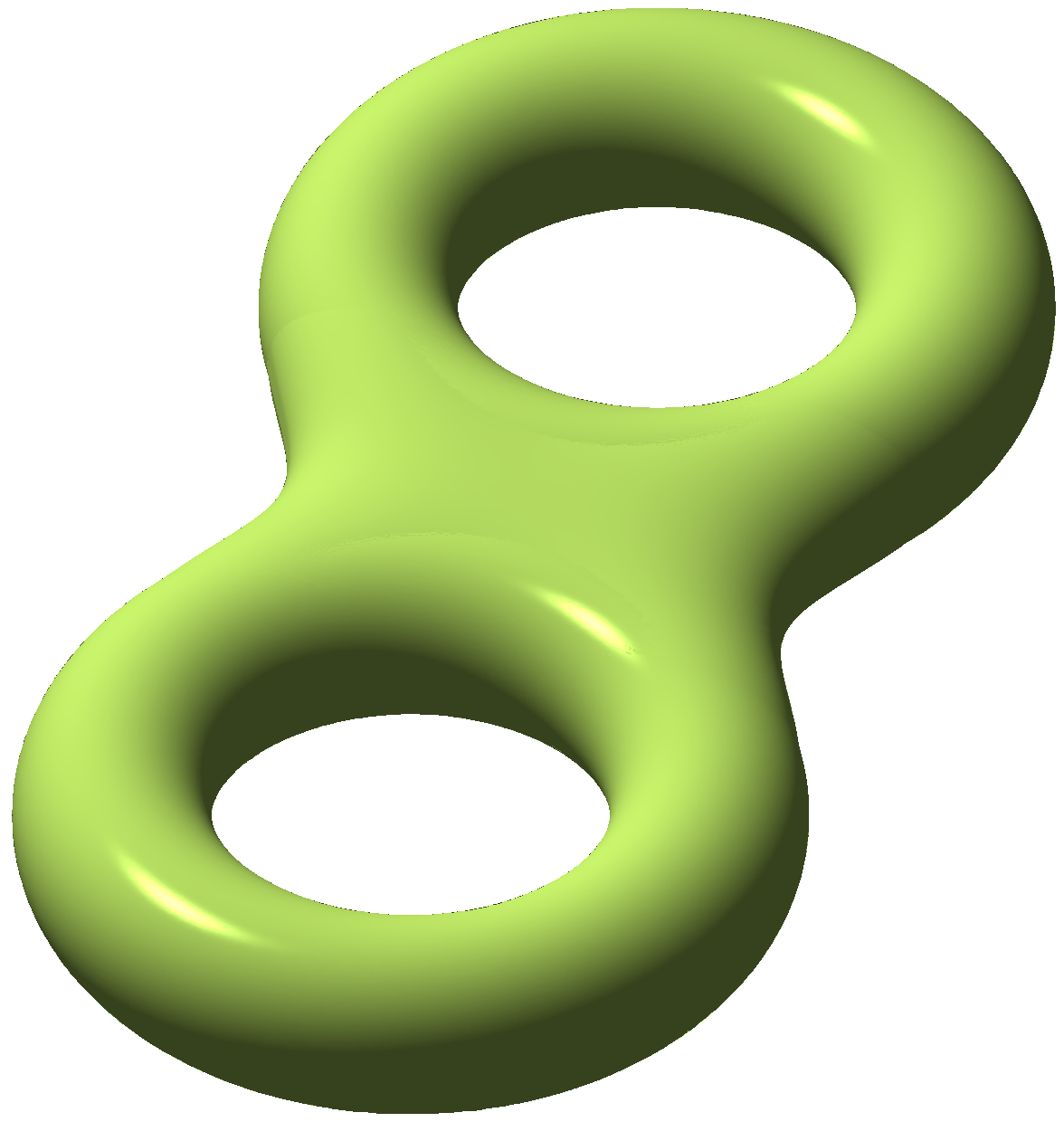}\\
\text{$\Sigma_2$}
\end{array}\hskip7mm
\begin{array}{c}
\includegraphics[width=.25\linewidth]{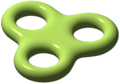}\\
\text{$\Sigma_3$}
\end{array}$$
Note that the intersection  form of $\Sigma_n$ is the standard symplectic form
$$\phi_{\Sigma_n}=
\begin{pmatrix} 0 & I_n \\ -I_n & 0 \end{pmatrix}:~
H_1(\Sigma_n;R)=R^n \oplus R^n \to
H_1(\Sigma_n;R)^*=(R^n)^* \oplus (R^n)^*$$
The {\it  mapping class group} $\Gamma_n=\pi_0(\text{{\rm Aut}}(\Sigma_n))$ is the group of orientation preserving diffeomorphisms $A:\Sigma_n \to \Sigma_n$, modulo isotopy.
The {\it  symplectic group} of a ring $R$
$$
{\rm Sp}(2n,R)=\text{{\rm Aut}}(H_-(R^n))~(n \geqslant   1)
$$
consists of the invertible $2n \times 2n$ matrices $A=(a_{ij})$ ($a_{ij} \in R$) preserving the standard symplectic form on $R^{2n}$:
$$
A^*\begin{pmatrix} 0 & I_n \\ -I_n & 0 \end{pmatrix}A=\begin{pmatrix} 0 & I_n \\ -I_n & 0 \end{pmatrix}.
$$
We shall also use $A$ to denote elements of $\Gamma_n$.
Looking at the action of $A \in \Gamma_n$ on homology, one gets a canonical group homomorphism: $c_n:\Gamma_n \to {\rm Sp}(2n,\ZZ)$ which is an isomorphism for $n=1$ and a surjection for $n \geqslant   2$.
Any $A\in \Gamma_n$ determines a closed 3-dimensional manifold with a Heegaard decomposition
$$
N^3=\#_n \SSS^1 \times \DDD^2 \cup_A \#_n \SSS^1 \times \DDD^2
$$
and hence a nonsingular skew-symmetric formation over $R$ which only depends on $c(A) \in {\rm Sp}(2n,\ZZ)$:
$$
(H,\theta;L_1,L_2)=(H_1(\Sigma_n;R),\phi_{\Sigma_n}; R^n \oplus 0,c_n(A)(R^n \oplus 0)).
$$

Every closed connected 3-dimensional manifold $N$ has  a Heegaard decomposition; two decompositions determine a stable isomorphism of formations.

\begin{remark}\label{lens1}
The first application of the symplectic group in topology is due to Poincar\'e~\cite{poincare1892}: the construction for any $A = \begin{pmatrix} a & b \\ c & d \end{pmatrix} \in {\rm Sp}(2,\ZZ)={\rm SL}(2,\ZZ)$ of the 3-manifold
$$N^3=\SSS^1 \times \DDD^2 \cup _A \SSS^1 \times \DDD^2,$$
which is in fact  the lens space $L(c,a)$.
See section~\ref{lens} below for a somewhat more detailed account of the lens spaces, and the connection with number theory.
\end{remark}

\begin{definition} \label{Heegaard}
A \emph{generalized Heegaard decomposition} of a closed $(\ell-1)$-connected $(2\ell+1)$-dimensional manifold $N^{2\ell+1}$ is an expression as
$$
N=N_1\cup_M N_2
$$
for a closed $(\ell-1)$-connected codimension $1$ submanifold $M^{2\ell} \subset N$ and $(\ell-1)$-connected codimension 0 submanifolds $N_1,N_2 \subset M$ with
$$
M=N_1 \cap N_2=\partial N_1=\partial N_2 \subset N$$
$$
\includegraphics[width=.5\linewidth]{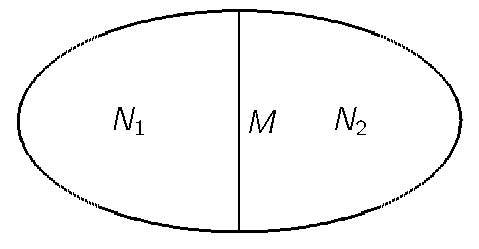}\eqno{}
$$
\end{definition}

\begin{proposition} \label{Hformation}
A generalized Heegaard decomposition $N^{2\ell+1}=N_1 \cup_MN_2$ determines a nonsingular $(-1)^{\ell}$-symmetric formation $(H_{\ell}(M;R),\phi_M;L_1,L_2)$ with the $R$-submodules (assumed to be f.g. free direct summands)
$$
\begin{array}{l}
L_1={\rm ker}(H_{\ell}(M;R) \to H_{\ell}(N_1;R)),\\
L_2={\rm ker}(H_{\ell}(M;R) \to H_{\ell}(N_2;R)) \subset H_{\ell}(M;R)
\end{array}
$$
lagrangians such that
$$
L_1 \cap L_2=H_{\ell+1}(N;R),~H_{\ell}(M;R)/(L_1+L_2)=H_{\ell}(N;R).
$$
\end{proposition}

\begin{proof}
Immediate from the Mayer-Vietoris exact sequence
$$
0 \to H_{\ell+1}(N;R) \to H_{\ell}(M;R)  \to H_{\ell}(N_1;R) \oplus H_{\ell}(N_2;R)\to H_{\ell}(N;R) \to 0.
$$
\end{proof}

Every closed $(\ell-1)$-connected $(2\ell+1)$-dimensional manifold $N$ has a generalized Heegaard decomposition; the formations associated to different decompositions are stably isomorphic.

\begin{remark}\leavevmode

1. For even-dimensional manifolds the intersection $(-1)^{\ell}$-symmetric intersection form defines a function
$$
\begin{array}{l}
\{\hbox{diffeomorphism classes of  $(\ell-1)$-connected $2\ell$-dimensional}\\
\hskip100pt \hbox{manifolds with boundary}\}\\
\to  \{\hbox{isomorphism classes of $(-1)^{\ell}$-symmetric forms over $\ZZ$}\};\\
\hskip100pt (M,\partial M) \mapsto (H_{\ell}(M;\ZZ),\phi_M)
\end{array}
$$
with closed manifolds sent to nonsingular forms.
See Wall~\cite{walleven} for the classification of $(\ell-1)$-connected $2\ell$-dimensional manifolds for $\ell \geqslant   3$.

2. For odd-dimensional manifolds generalized Heegaard decompositions give a corresponding function
$$
\begin{array}{l}
\{\hbox{diffeomorphism classes of  $(\ell-1)$-connected $(2\ell+1)$-dimensional}\\
\hskip100pt \hbox{manifolds with boundary}\}\\
\to  \{\hbox{stable isomorphism classes of $(-1)^{\ell}$-symmetric formations over $\ZZ$}\};\\
\hskip100pt (N,\partial N) \mapsto (H,\theta;L_1,L_2)
\end{array}
$$
with
$$
\begin{array}{l}
L_1 \cap L_2=H_{\ell+1}(N;\ZZ),~K/(L_1+L_2)=H_{\ell}(N;\ZZ),\\
(L_2^{\perp}/L_2,[\theta])=(H_{\ell}(\partial N),\phi_{\partial N}),
\end{array}
$$
and closed manifolds sent to nonsingular formations.
See Wall~\cite{wallodd} for the classification of $(\ell-1)$-connected $(2\ell+1)$-dimensional manifolds for $\ell \geqslant   3$.
\end{remark}

\subsection{The union of manifolds with boundary; Novikov additivity and Wall nonadditivity of the signature}

In Theorem~\ref{nonadd1} below we state the Novikov additivity theorem for the signature of a union of $4k$-dimensional manifolds  glued along components.
We then go on to state in Theorem~\ref{nonadd2} the Wall nonadditivity for the signature of a union of $4k$-dimensional relative cobordisms.

\begin{definition} \label{geometricunion} \leavevmode

1. An $m$-dimensional manifold with boundary $(M,\partial M)$ is a {\it union} of $m$-dimensional manifolds with boundary $(M_1,\partial M_1)$, $(M_2,\partial M_2)$ along a codimension 1 submanifold with boundary $(N,\partial N) \subset (M,\partial M)$ if
$$(M,\partial M)=(M_1,\partial M_1) \cup_{(N,\partial N)} (M_2,\partial M_2),$$
that is if
$$(M,\partial M)=(M_1\cup_NM_2,N_1 \cup_{\partial N}N_2)$$
with
$$\begin{array}{l}
N_1={\rm cl.}(\partial M_1\backslash N),~
N_2={\rm cl.}(\partial M_2\backslash N),\\
\partial M_1~= N \cup_{\partial N} N_1,~\partial M_2=N\cup_{\partial N}N_2,~\partial N=\partial N_1 =\partial N_2
\end{array}$$
$$
\includegraphics[width=.5\linewidth]{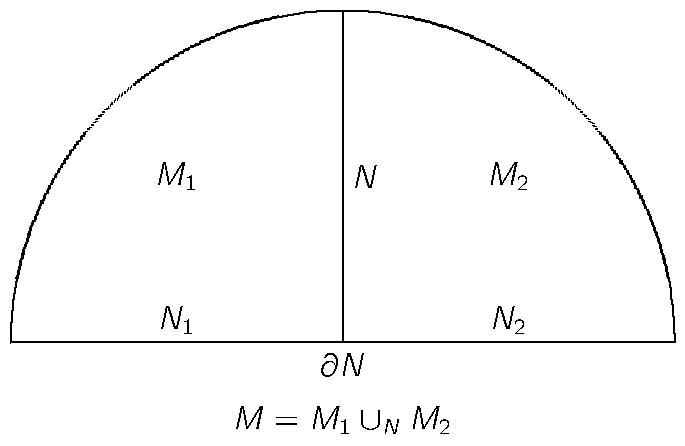}
$$

2. A union is {\it special} if
$$
\partial N=\emptyset,~\partial M_1=N,~
\partial M_2=N \sqcup N_2,~N_1=\emptyset,~N_2=\partial M.
$$

3. The {\it trinity} of the union in 1. is the $m$-dimensional manifold with boundary
$$\begin{array}{ll}
T&=~T(\partial N;N,N_1,N_2)\\
&=~\{\hbox{\rm  neighbourhood of the stratified set $N \cup N_1 \cup N_2 \subset M$}\}\\
&=~\partial N \times \DDD^2 \cup_{\partial N \times (I_0 \sqcup I_1 \sqcup I_2)}(N \times I_0 \sqcup N_1 \times I_1 \sqcup N_2 \times I_2) \subset M
\end{array}$$
with $I_0,I_1,I_2 \subset \partial \DDD^2=\SSS^1$ disjoint closed arcs.
$$
\includegraphics[width=.5\linewidth]{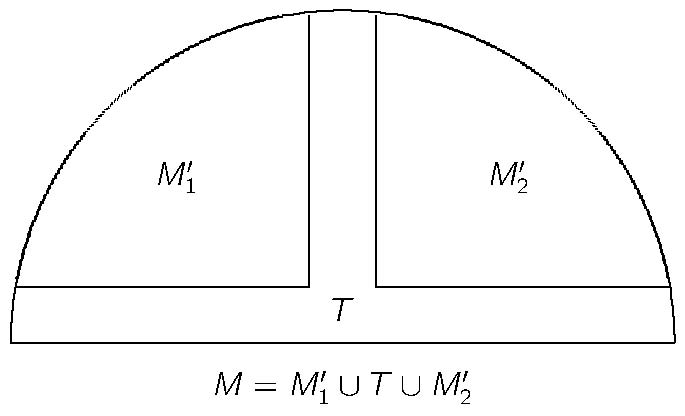}
$$
The closure of the complement
$$
({\rm cl.}(M \backslash T),\partial({\rm cl.}(M \backslash T)))=
(M'_1,\partial M'_1) \sqcup (M'_2,\partial M'_2)
$$
is the disjoint union of two copies $(M'_1,\partial M'_1)$, $(M'_2,\partial M'_2)$ of
$(M_1,\partial M_1)$, $(M_2,\partial M_2)$, with $\partial T=\partial M \sqcup (\partial M'_1 \sqcup \partial M'_2)$.
This gives $(M,\partial M)$ an expression as a special union
$$
(M,\partial M)=(M'_1 \sqcup M'_2,\partial M'_1 \sqcup \partial M'_2)\cup_{(\partial M'_1 \sqcup \partial M'_2,\emptyset)}(T,\partial T).
$$
\end{definition}

\begin{remark}
The trinity terminology is due to Borodzik, Nemethi and Ranicki~\cite{bnr}, which studies relative cobordisms of manifolds with boundary using the method of algebraic surgery theory.
\end{remark}

We have the Novikov additivity theorem for the signature:

\begin{theorem}[Novikov~\cite{novikov1971} 1967 for $\partial N=\emptyset$, Wall~\cite{wall1969} 1969 in general]
\label{nonadd1}
The signature of a $4k$-dimensional manifold with boundary which is a union
$$
(M,\partial M)=(M_1,\partial M_1)\cup_{(N,\partial N)}(M_2,\partial M_2)
$$
is the sum of the signatures of $M_1,M_2$ and the trinity $T$
$$
\tau(M)=\tau(M_1)+\tau(M_2)+\tau(T) \in \ZZ.
$$
In particular, if $\partial N=\emptyset$ (e.g. if the union is special) then $\tau(T)=0$.
\end{theorem}

\begin{proof}
Consider first the case of an expression of a $4k$-dimensional manifold with boundary as a special union
$$
(M,\partial M)=(M_1,\partial M_1) \cup_{(N,\partial N)}(M_2,\partial M_2).
$$
The subforms $(V_1,\phi_1)$, $(V_2,\phi_2)$ of the nonsingular symmetric form over $\RR$
$$
\begin{array}{c}
(V,\phi)=({\rm im}(\phi_M:H_{2k}(M;\RR)\to H_{2k}(M;\RR)^*),\Phi_M)\\
(\Phi_M(\phi_M(u),\phi_M(y))=\phi_M(u,v))
\end{array}
$$
defined by
$$
\begin{array}{l}
(V_1,\phi_1)=({\rm im}(H_{2k}(M_1;\RR) \to H_{2k}(M;\RR)^*),\phi\vert),\\
(V_2,\phi_2)=({\rm im}(H_{2k}(M_2;\RR) \to H_{2k}(M;\RR)^*),\phi\vert) \subseteq (V,\phi)
\end{array}
$$
are such that $V_1^{\perp}=V_2$.
The diagonal
$$
\Delta: u \in (V_1,0) \mapsto (u,u) \in (V_1,-\phi_1) \oplus (V,\phi)
$$
is the inclusion of a sublagrangian with
$$
(\Delta^{\perp}/\Delta,[-\phi_1 \oplus \phi])=(V_2,\phi_2).
$$
By Proposition~\ref{extension} (in the Appendix~\ref{appendix})
$$
-\tau(V_1,\phi_1)+\tau(V,\phi)=\tau(V_2,\phi_2) \in \ZZ
$$
so that
$$
\tau(M)=\tau(V,\phi)=\tau(V_1,\phi_1)+\tau(V_2,\phi_2)=\tau(M_1)+\tau(M_2) \in \ZZ
$$

For the general case note that the union
$$
(M,\partial M)= (M'_1 \sqcup M'_2,\partial M'_1 \sqcup \partial M'_2)
\cup_{(\partial M'_1 \sqcup \partial M'_2,\emptyset)} (T,\partial T)
$$
is special, so that
$$
\tau(M)=\tau(V,\phi)=\tau(M'_1\sqcup M'_2)+\tau(T)~
=~\tau(M_1)+\tau(M_2) +\tau(T) \in \ZZ.
$$

Finally, if $\partial N=\emptyset$ then $T=(N \sqcup N_1 \sqcup N_2) \times [0,1]$ and $\tau(T)=0 \in \ZZ$.
\end{proof}

In 1969 Wall~\cite{wall1969}  computed the signature of the trinity $T=T(\partial N;N,N_1,N_2)$ of  a $4k$-dimensional union  in terms of the `triformation' consisting of the symplectic form over $\RR$
$$
(H,\theta)=(H_{4k-2}(\partial N;\RR),\phi_N)
$$
and the  three lagrangians
$$
\begin{array}{l}
L={\rm ker}(H_{4k-2}(\partial N;\RR) \to H_{4k-2}(N;\RR)),\\
L_1={\rm ker}(H_{4k-2}(\partial N;\RR) \to H_{4k-2}(N_1;\RR)),\\
L_2={\rm ker}(H_{4k-2}(\partial N;\RR) \to H_{4k-2}(N_2;\RR)).
\end{array}
$$
by constructing a symmetric form $W(H,\theta;L,L_1,L_2)$ over $\RR$ such that
$$
\tau(T)=\tau(W(H,\theta;L,L_1,L_2)) \in \ZZ.
$$
This nonadditivity of the signature invariant has subsequently turned out to be a manifestation of the Maslov index, as we shall recount further below.

\begin{definition}[Ranicki,~\cite{ranickiexact}] \label{union} \leavevmode

1.An $\epsilon$-symmetric form $(V,\phi)$ over $R$ is a {\it union} of subforms $(V_1,\phi_1),(V_2,\phi_2) \subseteq (V,\phi)$
$$
(V,\phi)=(V_1,\phi_1) \cup (V_2,\phi_2)
$$
if $\phi(V_1,V_2)=\{0\}$, or equivalently $V_1 \subseteq V_2^{\perp}$,
$V_2 \subseteq  V_1^{\perp}$.

2. The union is {\it special} if $V_1^{\perp}=V_2$. 
\end{definition}

\begin{example}
Let $(M,\partial M)$ be a $2\ell$-dimensional manifold with boundary with $(-1)^{\ell}$-symmetric intersection form $(V,\phi)=(H_{\ell}(M;R),\phi_M)$ over $R$. An expression as a union
$$
(M,\partial M)=(M_1,\partial M_1) \cup_{(N,\partial N)} (M_2,\partial M_2)
$$
determines submodules
$$
V_1={\rm im}(H_{\ell}(M_1;R) \to H_{\ell}(M;R)),~
V_2={\rm im}(H_{\ell}(M_2;R) \to H_{\ell}(M;R)) \subseteq V
$$
such that $\phi(V_1,V_2)=\{0\}$. If $V_1$ and $V_2$ are f.g. free direct summands of $V$ (e.g. if $R$ is a field) then $(V,\phi)$ is the union of the subforms $(V_1,\phi_1)$, $(V_2,\phi_2)$ in the sense of Definition~\ref{union} 1., with $\phi_1=\phi\vert_{V_1}$, $\phi_2=\phi\vert_{V_2}$.
Furthermore if $(M,\partial M)$ is a special geometric union in the sense of  Definition~\ref{geometricunion} 2. and $(V,\phi)$ is a special algebraic union in the sense of  Definition~\ref{union} 2.
\end{example}

\begin{proposition}
For a special union $\epsilon$-symmetric form $(V,\phi)=(V_1,\phi_1) \cup (V_2,\phi_2)$ there is defined an isomorphism of $\epsilon$-symmetric forms
$$
(V_3,\phi_3) \oplus (V_1,\phi_1)~\cong~(V_2,-\phi_2) \oplus (V,\phi)
$$
for some nonsingular $\epsilon$-symmetric form $(V_3,\phi_3)$ with a lagrangian.
\end{proposition}

\begin{proof}
As in the proof of  Theorem~\ref{nonadd1} apply Theorem~\ref{extension} to the diagonal
$$
\Delta: (u,u) \in (V_2,0) \mapsto (u,u) \in (V_2,-\phi_2) \oplus (V,\phi)
$$
which is the inclusion of a sublagrangian with
$$
(\Delta^{\perp}/\Delta,[-\phi_2 \oplus \phi])=(V_1,\phi_1).
$$
\end{proof}

\begin{example}
For a special union symmetric form $(V,\phi)=(V_1,\phi_1) \cup (V_2,\phi_2)$ over $\RR$
$$
\tau(V,\phi)=\tau(V_1,\phi_1)+\tau(V_2,\phi_2) \in \ZZ.
$$
\end{example}

\begin{definition} \label{tri1} \leavevmode

1. A {\it $(-\epsilon)$-symmetric triformation over $R$} $(H,\theta;L_1,L_2,L_3)$  is a nonsingular $(-\epsilon)$-symmetric form $(H,\theta)$ over $R$ together with three lagrangians $L_1,L_2,L_3 \subset H$.

2. The {\it boundary} of the triformation is the nonsingular $(-\epsilon)$-symmetric formation
$$
\partial (H,\theta;L_1,L_2,L_3)= (H,\theta;L_1,L_2) \oplus
(H,\theta;L_2,L_3) \oplus (H,\theta;L_3,L_1).
$$

3. The {\it Wall form} (\cite{wall1969}) $W(H,\theta;L_1,L_2,L_3)=(W,\psi)$ is the $\epsilon$-symmetric pairing
$$
\psi :~W \times W \to R ;~(u,v) \mapsto \psi(u,v)=\epsilon \psi(v,u)
$$
with
$$\begin{array}{l}
W={\rm ker}((j_1~j_2~j_3):L_1 \oplus L_2 \oplus L_3 \to H),\\
j_i={\rm inclusion}:~(L_i,0) \to (H,\theta),\\
\psi:~W \times W \to R;~(u,v)=((u_1,u_2,u_3),(v_1,v_2,v_3)) \mapsto \theta(j_1(u_1),j_2(v_2))
\end{array}.
$$

4. Let $(V,\phi)$ be an $\epsilon$-symmetric form over $R$ which is a union
$$
(V,\phi)=(V_1,\phi_1) \cup (V_2,\phi_2),
$$
such that  $V_1 \subseteq V_2^{\perp}$, $V_2 \subseteq V_1^{\perp}$ are direct summands.
Let
$$
i_1:~(V_1,\phi_1) \to (V,\phi),~i_2:~(V_2,\phi_2) \to (V,\phi)
$$
be the inclusions.
The {\it canonical $(-\epsilon)$-symmetric triformation} $(H,\theta;L_1,L_2,L_3)$ over $R$ is given by
$$\begin{array}{l}
(H,\theta)=
\bigg(\dfrac
{{\rm ker}(\begin{pmatrix}
i_1^*\phi & i_1^* \\ 0 & i_2^* \end{pmatrix}:V \oplus V^* \to V_1^* \oplus V_2^*)}
{{\rm im}(\begin{pmatrix} i_1 & i_2 \\ -\phi i_1 & 0 \end{pmatrix}:
V_1 \oplus V_2 \to V \oplus V^*)},\begin{bmatrix}
0 & 1 \\ -\epsilon & 0 \end{bmatrix} \bigg),\\
\begin{bmatrix} 1 \\ 0 \end{bmatrix}:~
L_1= \dfrac{V_1^{\perp}}{V_2}=\dfrac{{\rm ker}(i_1^*\phi:V \to V_1^*)}{{\rm im}(i_2:V_2 \to V)} \to H,\\
\begin{bmatrix} 0 \\ 1 \end{bmatrix}:~
L_2= \dfrac{V_2^{\perp}}{V_1}=\dfrac{{\rm ker}(i_2^*\phi:V \to V_2^*)}{{\rm im}(i_1:V_1 \to V)} \to H,\\
\begin{bmatrix} 0 \\ 1 \end{bmatrix}:~L_3={\rm ker}(\begin{pmatrix} i_1^* \\ i_2^* \end{pmatrix}:
V^* \to V_1^* \oplus V_2^*) \to H
\end{array}
$$
\end{definition}

\medskip

\begin{example}
1.   Definition~\ref{tri1} 3. associates to any skew-symmetric triformation $(H,\theta;L_1,L_2,L_3)$ over $\RR$ a symmetric form $W(H,\theta;L_1,L_2,L_3)$ over $\RR$ and hence a signature
$$
\tau(H,\theta;L_1,L_2,L_3)= \tau(W(H,\theta;L_1,L_2,L_3))\in \ZZ.
$$

2. For any $\alpha_1,\alpha_2 \in {\rm Sp}(2n,R)$ we have  a skew-symmetric triformation $(H_-(L);\allowbreak L,\alpha_1(L),\allowbreak \alpha_2(L))$ and hence a symmetric form $W(H_-(L);L,\alpha_1(L),\alpha(L_2))$ over $R$. For
$R \subseteq \RR$ the signature
$$
\tau(W(H_-(L);L,\alpha_1(L),\alpha_2(L))) \in W(\RR)=\ZZ
$$
is a cocycle on ${\rm Sp}(2n,R)$ : these are the `Meyer cocycle' for $R=\ZZ$ and the `Maslov cocycle'  for $R=\RR$ -- see section~\ref{maslov} below.
\end{example}

\begin{proposition} \label{formunion} \leavevmode

1. Let $(H,\theta;L_1,L_2,L_3)$ be an $(-\epsilon)$-symmetric triformation over $R$. If $W={\rm ker}(L_1 \oplus L_2 \oplus L_3 \to H)$ is a f.g. free $R$-module (e.g. if $R$ is a field, or if $H=L_1+L_2+L_3$) then the Wall form $(W,\psi)$ is an $\epsilon$-symmetric form over $R$, with the boundary stably isomorphic to the boundary of the triformation
$$
\partial (W,\psi)~\cong~\partial(H,\theta;L_1,L_2,L_3).
$$

2. A union $\epsilon$-symmetric form
$(V,\phi)=(V_1,\phi_1) \cup (V_2,\phi_2)$ over $R$ is such that
there is defined an isomorphism of $\epsilon$-symmetric forms over $R$
$$
(V,\phi)~\cong~((V_1,\phi_1)\oplus (V_2,\phi_2)) \cup W(H,\theta,L_1,L_2,L_3)
$$
with $(H,\theta;L_1,L_2,L_3)$ the canonical $(-\epsilon)$-symmetric
triformation over $R$ given by Definition~\ref{tri1}.

3. For $\epsilon=1$, $R=\RR$
$$
\tau(V,\phi)=\tau(V_1,\phi_1)+\tau(V_2,\phi_2)+\tau(W(H,\theta,L_1,L_2,L_3)) \in \ZZ.
$$
\end{proposition}

\begin{proof} An abstract version of Wall~\cite{wall1969}.
\end{proof}

\begin{example} \label{triex} \leavevmode

1.  Let
$$
(M^{2\ell},\partial M)=(M_1,\partial M_1) \cup_{(N,\partial N)} (M_2,\partial M_2)
$$
be an $(\ell-1)$-connected $2\ell$-dimensional manifold with boundary which is a union of $(\ell-1)$-connected manifolds with boundary $(M_1,\partial M_1)$, $(M_2,\partial M_2)$ along an $(\ell-2)$-connected manifold with boundary $(N,\partial N)$. The $(-1)^{\ell}$-symmetric intersection form $(V,\phi)=(H_{\ell}(M;R),\phi_M)$ is an algebraic union
$$
(V,\phi)=(V_1,\phi_1) \cup (V_2,\phi_2)=(H_{\ell}(M_1;R),\phi_{M_1})\cup (H_{\ell}(M_2;R),\phi_{M_2})
$$
such that the canonical $(-1)^{\ell-1}$-symmetric triformation is determined by $\partial N=\partial N_1=\partial N_2$
$$\begin{array}{l}
(H,\theta;L,L_1,L_2)=(H_{\ell-1}(\partial N;R),\phi_{\partial N};
{\rm ker}(H_{\ell-1}(\partial N;R) \to H_{\ell-1}(N;R)),\\
\hskip25pt
{\rm ker}(H_{\ell-1}(\partial N;R) \to H_{\ell-1}(N_1;R)),
{\rm ker}(H_{\ell-1}(\partial N;R) \to H_{\ell-1}(N_2;R)))
\end{array}$$
The Wall form of the triformation is the $(-1)^{\ell}$-symmetric intersection form of the trinity $T=T(\partial N;N,N_1,N_2)$
$$
W(H,\theta;L,L_1,L_2)=(H_{\ell}(T;R),\phi_T).
$$

2. For $\ell=2k$, $R=\RR$
$$
\begin{array}{ll}
\tau(M)&=~\tau(M_1)+\tau(M_2)+\tau(T)\\
&=~\tau(M_1)+\tau(M_2)+\tau(W(H,\theta;L,L_1,L_2)) \in \ZZ
\end{array}
$$
with $\tau(T)=\tau(W(H,\theta;L,L_1,L_2)) \in \ZZ$ the nonadditivity invariant of Wall~\cite{wall1969}.
\end{example}

Finally, here is the Wall non-additivity theorem in full generality:

\begin{theorem}[Wall~\cite{wall1969}] \label{nonadd2}
The signature of a $4k$-dimensional manifold with boundary which is a union
$$
(M,\partial M)=(M_1,\partial M_1) \cup_{(N,\partial N)} (M_2,\partial M_2)
$$
is given by
$$
\tau(M)=\tau(M_1)+\tau(M_2) +\tau(T) \in \ZZ
$$
with $T=T(\partial N;N,N_1,N_2)$ the trinity of the union.
\end{theorem}

\subsection{Plumbing: from quadratic forms to manifolds}\label{plumbingmanifolds}

The algebraic plumbing of symmetric matrices over $\RR$ (Definition~\ref{plumbing}) extends readily to $\epsilon$-symmetric forms over $\ZZ$:

\begin{definition} \label{plumbing2}
The {\it plumbing} of an $\epsilon$-symmetric form $(V,\phi)$ over $\ZZ$ with respect to $v \in V^*$ and a 1-dimensional $\epsilon$-symmetric form $(\ZZ,w)$ is the  $\epsilon$-symmetric form over $\ZZ$
$$
(V',\phi')=(V\oplus \ZZ,\begin{pmatrix} \phi & \epsilon v^*\\  v & w \end{pmatrix}).
$$
\end{definition}

\begin{proposition}\label{plumbingeffect}
If $(V',\phi')$ is an $\epsilon$-symmetric form over $\ZZ$ which is obtained from an $\epsilon$-symmetric form $(V,\phi)$ over $\ZZ$
with ${\rm det}(\phi) \neq 0 \in \ZZ$ by plumbing with respect to
$v \in V^*$, $(\ZZ,w)$ then up to isomorphism of $\epsilon$-symmetric forms over $\QQ$
$$\QQ \otimes_\ZZ (V',\phi')=\QQ \otimes_\ZZ(V,\phi) \oplus (\QQ,w-\epsilon v\phi^{-1}v^*).$$
\end{proposition}

\begin{proof} Immediate from the isomorphism
$$\begin{pmatrix} 1 & -\epsilon \phi^{-1}v^* \\ 0 & 1 \end{pmatrix}:~
\QQ\otimes_\ZZ(V',\phi') \xymatrix{\ar[r]^-{\cong}&}
\QQ\otimes_\ZZ(V,\phi) \oplus (\QQ,w-\epsilon v\phi^{-1}v^*)$$
\end{proof}

We shall now show that geometric plumbing determines algebraic plumbing of the intersection form over $\ZZ$, and that in certain cases
algebraic plumbing is realized by the geometric plumbing procedure of Milnor.

Denote by ${\rm SO}(n)$ the subgroup of the orthogonal group consisting of matrices of determinant 1.
Recall that any continuous map  $w:\SSS^{\ell-1}\to {\rm SO}(n)$ can be used to construct a $\RR^{n}$-vector bundle over the sphere $\SSS^{\ell}$.
The sphere $\SSS^{\ell}$ can be decomposed as two hemispheres, homeomorphic to the ball $\DDD^{\ell}$ intersecting along the equator, homeomorphic to $\SSS^{\ell-1}$.
One then starts with two trivial bundles $\DDD^{\ell} \times \RR^n$ and one glues them on their boundaries, identifying the point $(u,v) \in \SSS^{\ell-1} \times \RR^n$ in the boundary
of the first copy to $(u,w(u)(v) ) \in \SSS^{\ell-1} \times \RR^n$ in the boundary of the second copy.
The map $w$ is called a {\it clutching map}.
Any oriented vector bundle over a sphere is isomorphic to one obtained by this construction so that
the set of isomorphism classes of vector bundles over $\SSS^{\ell}$ is identified with the homotopy group $\pi_{\ell}({\rm BSO}(n))=\pi_{\ell-1}({\rm SO}(n))$.
The Bott periodicity theorem gave the computation of the stable groups
$$\pi_{\ell-1}({\rm SO})=\varinjlim\limits_{n} \pi_{\ell-1}({\rm SO}(n)),
$$
leading to the computation of the unstable groups (Kervaire~\cite{kervaireunstable}).
The Euler number (defined as the self-intersection of the zero section of the corresponding $\ell$-plane bundle over $\SSS^{\ell}$) defines a homomorphism
$$\chi:  w \in \pi_{\ell-1}({\rm SO}(\ell)) \mapsto \chi(w) \in \ZZ.$$
Note that $\chi(w)=0$ if $\ell$ is odd.
One can also think of $\chi(w)$  as the degree of the map $\SSS^{\ell-1} \to \SSS^{\ell-1}$ sending $u \in \SSS^{\ell-1}$ to the first column of $w(u)$,
and also the Hopf invariant of a map $J(w):\SSS^{2\ell-1} \to \SSS^{\ell}$.

For $\ell=2k$ and $k \neq 1,2,4$ the Euler number is even since the Hopf invariant of any map $\SSS^{4k-1} \to \SSS^{2k}$ is even in these dimensions (Adams~\cite{adams1}).

We now describe a general plumbing construction of $2\ell$-dimensional manifolds with boundary, using surgery.

Start with the following input:
\begin{enumerate}
\item a $2\ell$-dimensional manifold with boundary $(M,\partial M)$,
\item  an embedding $v:(\DDD^{\ell} \times \DDD^{\ell},\SSS^{\ell-1} \times \DDD^{\ell}) \subseteq (M,\partial M)$,
\item  a clutching map $w:\SSS^{\ell-1}\to {\rm SO}(\ell)$ producing a vector bundle $E(w)$ over $\SSS^{\ell}$, using the adjoint map
$$f(w):~\SSS^{\ell-1} \times \DDD^{\ell} \to \SSS^{\ell-1} \times \DDD^{\ell};~(x,y) \mapsto (x,w(x)(y) )~.$$
\end{enumerate}

The {\it  plumbed} $2\ell$-dimensional manifold with boundary is the following output:
$$(M',\partial M')=(M \cup_{f(w)} \DDD^{\ell} \times \DDD^{\ell},\text{cl.}(\partial M \backslash \SSS^{\ell-1} \times \DDD^{\ell}) \cup \DDD^{\ell} \times \SSS^{\ell-1})~.$$

\begin{center}
\includegraphics[width=\linewidth]{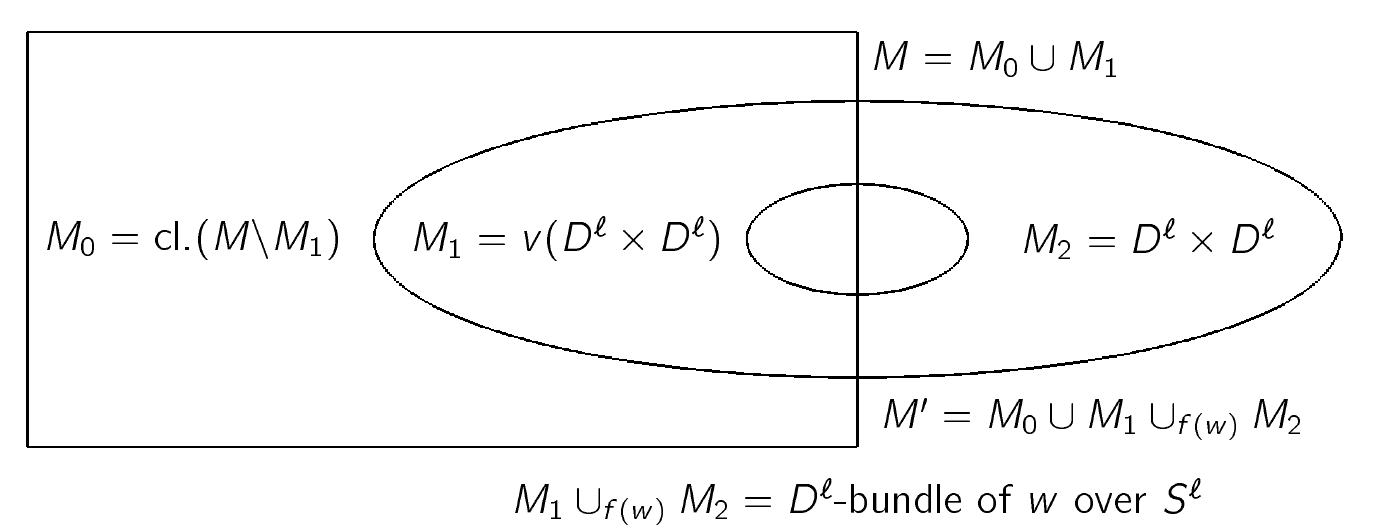}

\includegraphics[width=.5\linewidth]{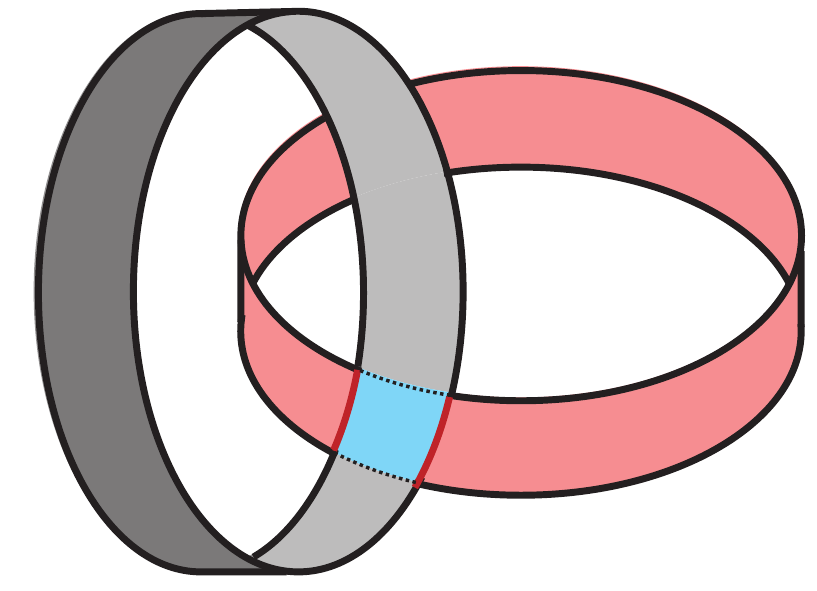}
\end{center}

Let us observe the algebraic effect of geometric plumbing in the highly-connected case of interest:

\begin{proposition} \label{situation}
Let $(M,\partial M)$ be an $(\ell-1)$-connected $2\ell$-dimensional manifold with boundary,  and let $(M',\partial M')$ be the plumbed
manifold with boundary obtained using
$$v:~(\DDD^{\ell} \times \DDD^{\ell},\SSS^{\ell-1} \times \DDD^{\ell}) \subset (M,\partial M),~w:~\SSS^{\ell-1}\to {\rm SO}(\ell).$$

1. $(M',\partial M')$ is $(\ell-1)$-connected with $(-1)^{\ell}$-symmetric intersection form over $\ZZ$ given by algebraic plumbing (Definition~\ref{plumbing2})
$$(H_{\ell}(M';\ZZ),\phi_{M'})=(H_{\ell}(M;\ZZ) \oplus \ZZ,
\begin{pmatrix} \phi_{M} & (-1)^{\ell} v^{\star} \\ v & \chi(w)\end{pmatrix})$$
where
$$v=v[\DDD^{\ell} \times \DDD^{\ell}] \in H_{\ell}(M,\partial M;\ZZ)=H_{\ell}(M;\ZZ)^{\star}$$
and $\chi(w)=(-1)^{\ell}\chi(w)\in \ZZ$ is the Euler number of $w$.

2. If $(H_{\ell}(M;\ZZ),\phi_M)$ is nondegenerate then
$(H_{\ell}(M';\ZZ),\phi_{M'})$ is nondegenerate if and only if $w-(-1)^{\ell}v\phi_M^{-1}v^* \neq 0 \in \QQ$, in which case
$$\QQ\otimes_{\ZZ}(H_{\ell}(M';\ZZ),\phi_{M'})=
\QQ\otimes_{\ZZ}(H_{\ell}(M;\ZZ),\phi_M)\oplus
(\QQ,\chi(w)-(-1)^{\ell}v\phi_M^{-1}v^*).$$
For even $\ell$ it follows that the signatures of $M$, $M'$ are related by $$\tau(M')=\tau(M)+{\rm sign}(\chi(w)-v\phi_M^{-1}v^{\star}) \in \ZZ.$$
\end{proposition}

\begin{remark} In the situation of Proposition~\ref{situation} 2.
with $\ell$ even we have a union as in Definition~\ref{geometricunion}
$$(M',\partial M')=(M,\partial M) \cup_
{( \SSS^{\ell-1}\times \DDD^{\ell},\SSS^{\ell-1} \times \SSS^{\ell-1})}  (\DDD^{\ell} \times \DDD^{\ell},\partial(\DDD^{\ell} \times \DDD^{\ell})).$$
The skew-symmetric triformation over $\QQ$ given by Example~\ref{triex} is
$$(H_{-}(\QQ);
\begin{pmatrix} 1\\ 0 \end{pmatrix} \QQ,
\begin{pmatrix} v\phi_M^{-1}v^* \\ 1\end{pmatrix} \QQ,
\begin{pmatrix} \chi(w) \\ 1 \end{pmatrix} \QQ)$$
with the Wall form
$(W,\psi)=(\QQ,\chi(w)-v\phi_M^{-1}v^{\star})$ of signature
$$\tau(W,\psi)={\rm sign}(\chi(w)-v\phi_M^{-1}v^{\star}) \in \ZZ.$$
\end{remark}

\begin{definition}\label{graphmanifold}
A \emph{graph manifold} is an $(\ell-1)$-connected $2\ell$-dimensional manifold with boundary constructed from $\DDD^{2\ell}$ by the geometric  plumbing of  $n$ $\ell$-plane bundles over $\SSS^{\ell}$, using  a graph with vertices $j=1,2,\dots,n$ and weights $\chi_j \in \pi_{\ell-1}({\rm SO}(\ell))$.
The weights are $\ell$-plane bundles $\chi_j$ over $\SSS^{\ell}$.
\end{definition}

\begin{theorem}[Milnor 1959~\cite{milnorplumb}, Hirzebruch 1961~\cite{hirzebruchneumannkoh}]
Let $\ell \geqslant   2$ and let $S=(s_{ij} \in \ZZ)$  be a $(-1)^{\ell}$-symmetric $n \times n$ matrix, such that for $\ell=2k$ and $k \neq 1,2,4$ the diagonal entries $s_{jj}\in \ZZ$ are even.
Then $S$ is realized by a  graph $2\ell$-dimensional manifold with boundary $(M,\partial M)$ such that
$$
(H_{\ell}(M;\ZZ),\phi_M)=(\ZZ^n,S).
$$
\end{theorem}

If the graph is a tree then for $\ell \geqslant   2$ $M$ is $(\ell-1)$-connected, and for $\ell \geqslant   3$ $M$ and $\partial M$ are both $(\ell-1)$-connected.

\begin{example} Let $q \in \ZZ$. The graph  4-dimensional
manifold realizing the symmetric $1 \times 1$ matrix $q$ 
$$(M,\partial M) = (\DDD^4 \cup \DDD^2 \times \DDD^2,
\SSS^1 \times \DDD^2 \cup_{\begin{pmatrix} 1 & 0 \\ q & 1 \end{pmatrix}}\SSS^1 \times\DDD^2)$$
has boundary a lens space $\partial M=L(q,1)$ -- see Remark~\ref{lens1} and Section~\ref{lens}.
\end{example}

The plumbing construction and graph manifolds were  motivated  by Hirzebruch's 1950's and 1960's work on the resolution of singularities (see section~\ref{lens}
below for one particular class of examples realizing tridiagonal symmetric matrices) and by the
{\it exotic spheres}  $\Sigma^n$ of Milnor~\cite{milnorexotic},~\cite{milnorplumb}.

The celebrated $8 \times 8$ symmetric matrix over $\ZZ$
$$E_8=\begin{pmatrix} 2 & 1 &0 & 0 & 0 & 0 & 0 & 0\\
1 & 2 &1 & 0 & 0 & 0 & 0 & 0 \\
0 & 1 & 2 &1 & 0 & 0 & 0 & 0 \\
0 & 0 & 1 &2 & 1 & 0 & 0 & 0 \\
0 & 0 & 0 &1 & 2 & 1 & 0 & 1 \\
0 & 0 & 0 &0 & 1 & 2 & 1 & 0 \\
0 & 0 & 0 &0 & 0 & 1 & 2 & 0 \\
0 & 0 & 0 &0 & 1 & 0 & 0 & 2 \end{pmatrix}$$
is invertible, positive definite, has even diagonal entries, and signature
$$\tau(E_8)=8 \in \ZZ,$$
corresponding to the exceptional Lie algebra $E_8$ with Dynkin diagram
$$\includegraphics[width=7.5cm]{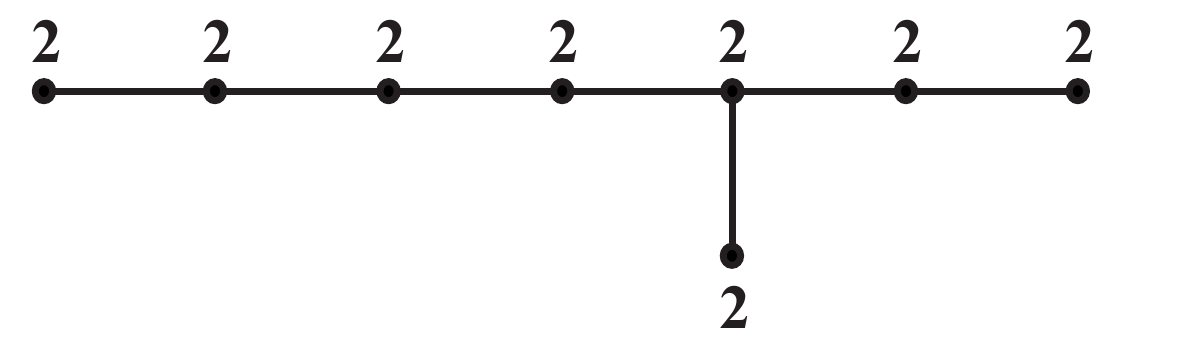}$$
See the website Ranicki~\cite{ranicki8}  for source material on the role of the number 8 in general and the form $E_8$ in particular, in topology, algebra and mathematical physics.
$$
\includegraphics[width=.35\linewidth]{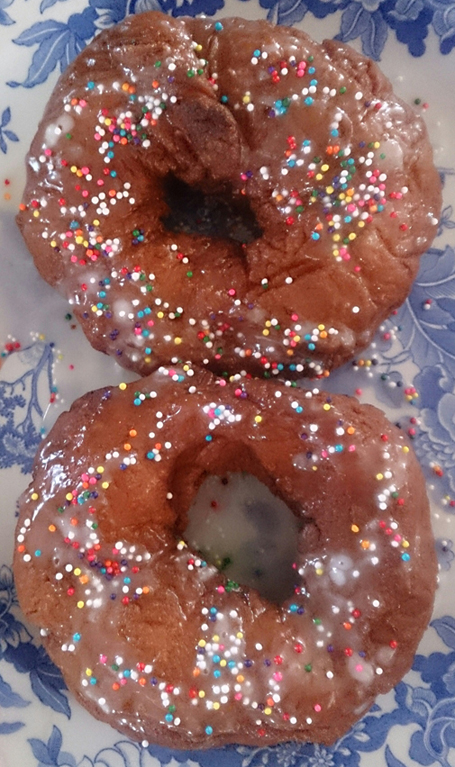}
$$
The 4-dimensional graph manifold $(M,\partial M)$ with intersection matrix $E_8$ has boundary the Poincar\'e homology 3-sphere $\partial M=\Sigma^3$ with $H_*(\Sigma^3)=H_*(\SSS^3)$, $\pi_1(\Sigma^3) \neq \{1\}$.

If  $\Sigma^{2\ell-1}=\partial M$  for a parallelizable $(\ell-1)$-connected $(M^{2\ell},\partial M)$ then for $\ell \geqslant   3$ the differentiable structure on $\partial M$ is detected by the signature defect of $(M,\partial M)$ if $\ell$ is even, and by the Kervaire invariant if $\ell$ is odd.
In particular, for $k \geqslant   2$ the $(2k-1)$-connected $4k$-dimensional graph manifold $(M,\partial M)$ with intersection matrix $E_8$ has boundary the Milnor exotic sphere $\partial M=\Sigma^{4k-1}$ generating the finite cyclic group $bP_{4k}$ of the exotic spheres which bound framed manifolds (Milnor~\cite{milnorexotic}, Kervaire and Milnor~\cite{kervairemilnor}).

Brieskorn~\cite{brieskorn}  proved that all such pairs  $(M^{2\ell},\Sigma^{2\ell-1})$  for $\ell \geqslant   3$ are  fibres of the Milnor fibration~\cite{milnorsingular}
$$
u\in \SSS^{2\ell+1}\backslash V \mapsto  f(u)/\vert f(u)\vert \in \SSS^1
$$
determined by an  isolated singular point of  a complex hypersurface  $V=f^{-1}(0) \subset \CC^{\ell+1}$, with $f$ a polynomial function
$$
f: (u_0,u_1,\dots,u_\ell) \in \CC^{\ell+1}  \mapsto
\sum\limits^{\ell}_{j=0}u_j^{a_j} \in \CC.
$$

\newpage

\section{Knots, links,  braids and signatures}\label{knotslinks}

\subsection{Knots, links, Seifert surfaces and complements}

If $c\geqslant   1$ is an integer, we shall denote by $\SSS^1_c$ the disjoint union of $c$ copies of the circle.
A (classical) $c$-component {\it link} is a smooth embedding $L:\SSS^1_c \hookrightarrow \SSS^3$ of $\SSS^1_c$ in the $3$-sphere.
A {\it knot} is a link with $c=1$.

Two $c$-component links $L_0,L_1$ are considered equivalent if they are {\it isotopic}, i.e. if there is a smooth family $(L_t)_{t\in [0,1]}$ connecting them.
Equivalently, they are isotopic if and only if there is an orientation preserving diffeomorphism of the $3$-sphere sending $L_0$ on $L_1$ (as follows from Cerf's theorem that the group of orientation preserving diffeomorphisms of the 3-sphere is connected).
A link is {\it trivial} if it is isotopic to a link whose image lies in $\RR^2\subset \RR^3 \subset \SSS^3$.

There are many excellent books presenting the present state of knot and link theory, including from an historical perspective.
We recommend in particular the books of Crowell and Fox~\cite{crowellfox}, Sossinsky~\cite{sossinsky}, Kauffman~\cite{kauffman} and Lickorish~\cite{lickorish}.
In this section, we would like to concentrate on a very specific part of this theory relating knots and links to quadratic forms and their signatures.

The first step is actually related to many other aspects of knot theory.
In 1930 Frankl and Pontrjagin~\cite{franklpontrjagin} published the result that every knot $k:\SSS^1 \hookrightarrow \SSS^3$ is the boundary $k(\SSS^1)=\partial F$ of an embedded (oriented) surface $F \hookrightarrow \SSS^3$.
This result was reproved  in 1935 by Seifert~\cite{seifert},  acknowledging the priority of~\cite{franklpontrjagin}.
The advantage of~\cite{seifert} is that it provided a practical algorithm for constructing the Frankl-Pontrjagin surfaces of  a knot $k$  from plane projections $P:\RR^3 \to \RR^2$ (cf. Guillou and Marin~\cite[p.67]{guilloumarin}).  These are called  the {\it Seifert surfaces} of $k$.
A knot is isotopic to the trivial knot if and only if there exists a Seifert surface which is a disc.
The algorithm applies to any link $L: \SSS^1_c \hookrightarrow \RR^3$, and produces a Seifert surface $F \hookrightarrow \RR^3$ such that $\partial F=L(\SSS^1_c)$.
There exists a linear map $P:\RR^3 \to \RR^2$ (many such in fact) such that the image of the composite $PL:\SSS^1_c \to \RR^2$ is a collection of oriented curves with a finite number $\ell$ of transverse double points labelled as over/underpasses.
This is a {\it plane projection} of $L$.
Modify this collection to delete all double points to obtain $n$ {\it Seifert circles} as shown on the following figure, extracted from~\cite[page 21]{sossinsky}.

\begin{center}
\includegraphics[width=.8\linewidth]{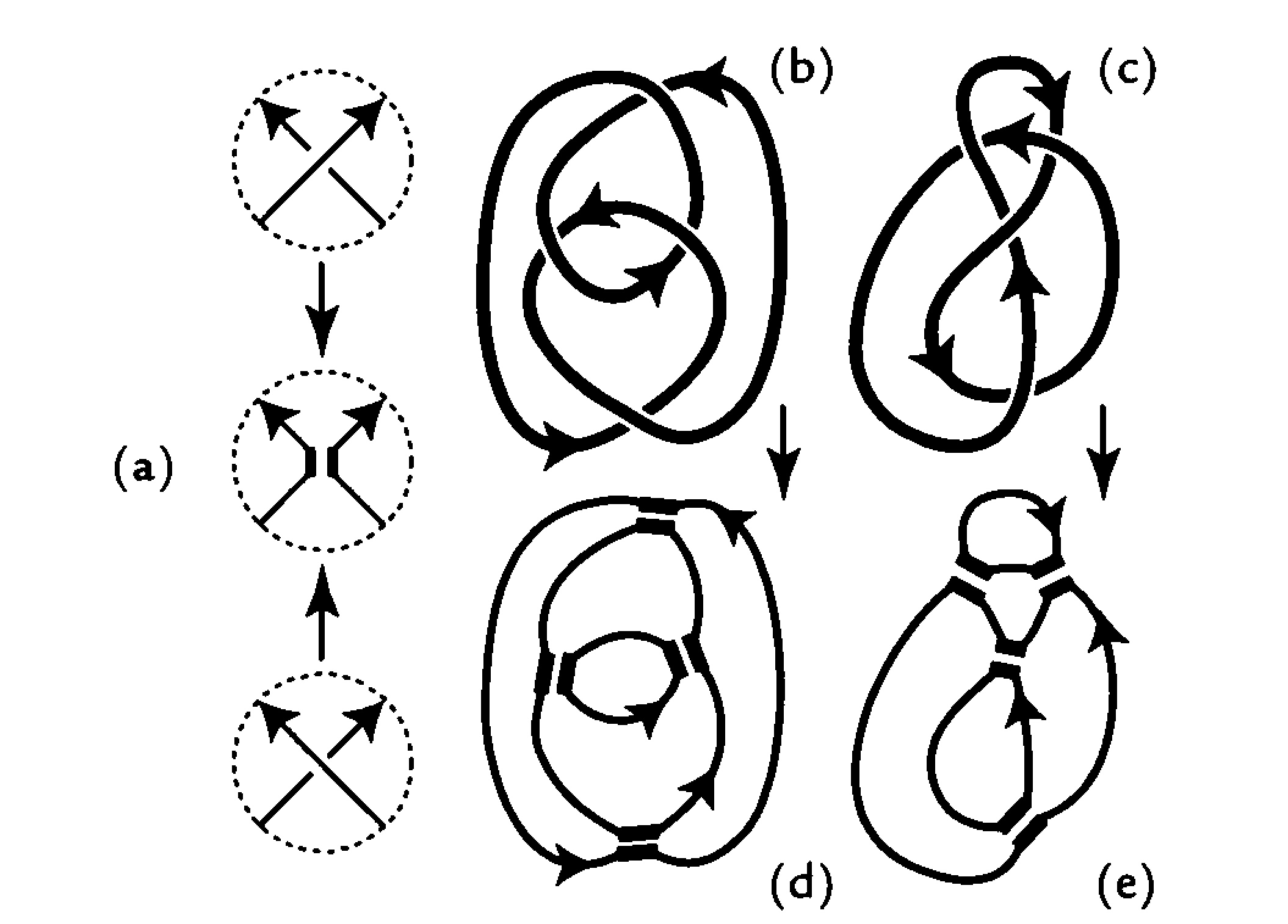}
\end{center}

Each of these curves bounds a disc in the plane.
One can push these discs away from the plane, using the third dimension, in such a way that they are disjoint.
Corresponding to each double point of the projected knot, one glues a twisted band between the corresponding pushed discs.
The result is a Seifert surface with $n$ 0-handles and $\ell$ 1-handles
$$
F=\coprod\limits_n \DDD^2 \cup \coprod\limits_\ell \DDD^1 \times \DDD^1  \subset \RR^3
$$
such that
$$
\partial F=L(\SSS^1_c) \subset \RR^3.
$$
Seifert used this construction to associate a bilinear form to a knot (which also applies to a link).
Consider a Seifert surface $F$ for $k:\SSS^1 \subset \SSS^3$.
If $u$ and $v$ are two homology classes of $H_1(F;\ZZ)$, one can represent them by two closed curves in $F$.
Pushing $u$ in the positive normal direction of $F$, one gets two disjoint curves $u^+$ and $v$ in the 3-sphere and the linking number is therefore well defined:
$${\rm lk}(u^+,v)~=~u^+(\SSS^1) \cap \delta v(\DDD^2) \in \ZZ$$ 
for any extension $\delta v:\DDD^2 \to \SSS^3$ of $v$ which intersects $u^+(\SSS^1) \subset \SSS^3$ transversely.
This is a non symmetric bilinear form called the {\it Seifert form}
$$\Sigma:~H_1(F;\ZZ) \times H_1(F;\ZZ) \to \ZZ.$$
Note that ${\rm lk}(u^+,v)-{\rm lk}(v^+,u)$ is the (skew symmetric) intersection of the two curves $u,v$ on $F$, so that $(H_1(F;\ZZ),\Sigma-\Sigma^*)$ is a symplectic form over $\ZZ$.
See the paper of Collins~\cite{collins} for the relationship between the genus of $F$, $\Sigma$, $c$, $\ell$ and $n$.

As an example, the following picture, extracted from~\cite[page 200]{kauffman} shows a connected Seifert surface for the trefoil knot, a basis $a,b$ for its homology and the Seifert matrix associated to this basis.

\begin{center}
\includegraphics[width=\linewidth]{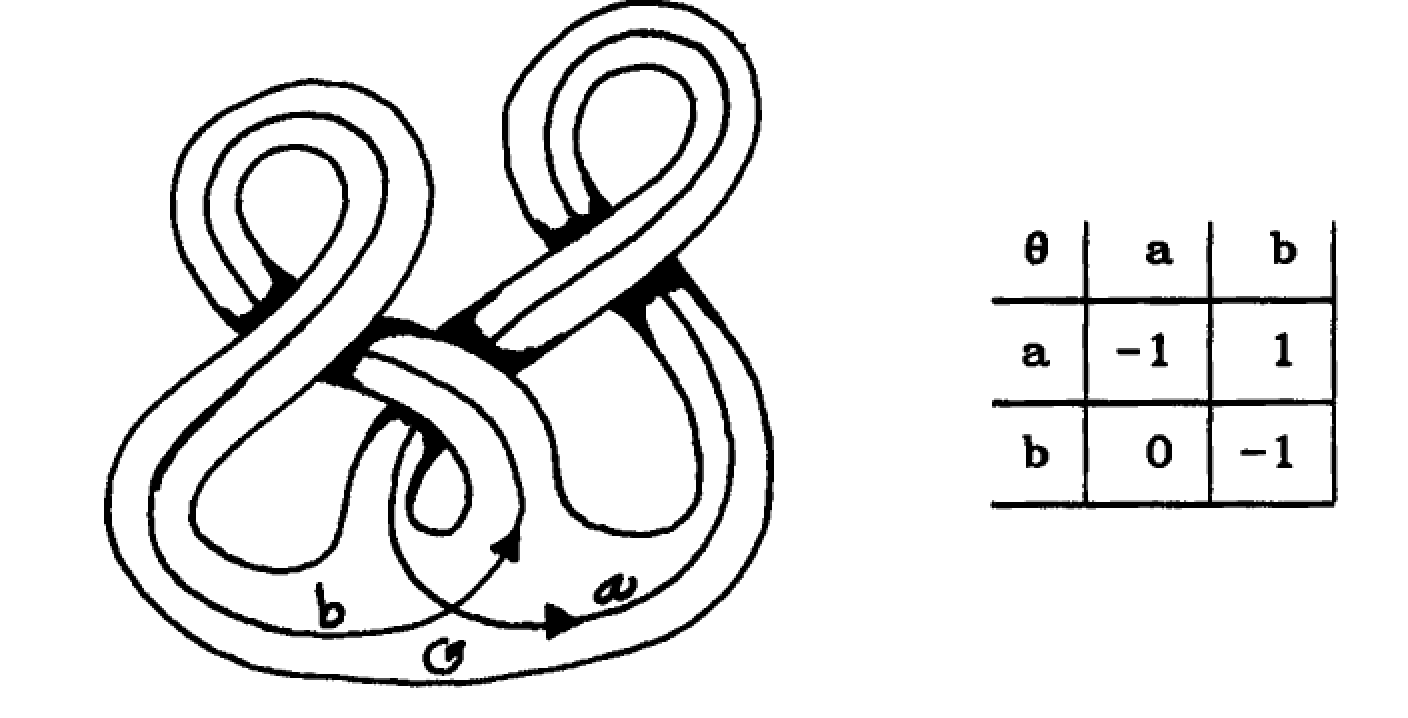}
\end{center}

The main difficulty with this construction is of course that the Seifert surface is not unique, so that this form is not canonically attached to the knot.
Further below we shall describe the effect on the Seifert matrix $\Sigma$ of choosing a different Seifert surface $F$.

It is not hard to prove that one can go from one connected Seifert surface $F_1$ to another one $F_2$ by very simple elementary operations.
The first is simply an isotopy of the ambient sphere.
The second is a $1$-surgery: delete two disjoint discs in the interior of Seifert surface $F$ and connect their boundaries by a tube disjoint from $F$.
The third is the opposite move: choose an embedded tube in $F$, cut it open, and fill the two circles with discs.

One can look at the effect of these operations on the Seifert matrices.
If the Seifert matrix associated to a Seifert surface $F$ is an $n\times n$ matrix $\Sigma$, after a $1$-surgery it is an $(n+2)\times(n+2)$ matrix of the form

$$\Sigma '=\begin{pmatrix}
\Sigma & 0 & 0 \\
0 & 0 & 1\\
\alpha  & 0&  0
\end{pmatrix}.
$$

It was therefore natural to look for invariants of Seifert matrices under this kind of operations which generate the {\it $S$-equivalence} relation, which was introduced by Murasugi~\cite{murasugi}.
Equivalently, one can study invariants of the homology of the canonical infinite cyclic cover of the link exterior.

The {\it exterior} of a link $L:\SSS^1_c \hookrightarrow \SSS^3$ is the 3-dimensional manifold
$X$ with boundary $\partial X=L(\SSS^1_c)\times \SSS^1$ obtained by deleting from the 3-sphere a tubular neighborhood of the link.
The inclusion of the link exterior in the link complement $X \hookrightarrow \SSS^3\backslash L(\SSS^1_c)$ is a homotopy equivalence.
The linking number with $L$ defines a canonical surjection $\pi_1(X) \to \ZZ$ which represents $(1,1,\dots,1) \in H^1(X)$ (which is a direct sum of $c$ copies of $\ZZ$).
The kernel of this epimorphism is the fundamental group
$$\pi_1(\overline{X})={\rm ker}(\pi_1(X) \to \ZZ)$$
of an infinite cyclic covering $\overline{X}$ of $X$.

One way to define this covering is to choose a map $p:X \to \SSS^1$ representing this epimorphism, and to identify $\overline{X}$ with the pullback
$$\overline{X}=\{(x,t) \in X \times \RR \,\vert\, p(x)=e^{2\pi i t} \in \SSS^1\}$$ of the universal cover of the circle.
The homeomorphism
$$z: (x,t) \in \overline{X} \mapsto(x,t+1) \in  \overline{X}$$
is a generating covering translation, and the lift of the projection $p$
$$\overline{p}: (x,t) \in \overline{X} \mapsto t \in \RR$$
is $\ZZ$-equivariant.

Choose $p:X \to \SSS^1$ to be a smooth map whose restriction to the boundary $\partial X=k(\SSS^1) \times \SSS^1$ is the second projection.
The inverse image of a regular value $* \in \SSS^1$ is a Seifert surface $F=p^{-1}(*) \hookrightarrow \SSS^3$ for $k$ with $\partial F=k(\SSS^1) \times \{*\} \hookrightarrow X$.
The embedded surface $F \hookrightarrow X$ has a tubular neighbourhood $F \times [0,1] \hookrightarrow X$.
The covering $\overline{X}$ is of course trivial over $X\backslash F$.
The closure of one component of $\bar{p}^{-1}(X\backslash F) \hookrightarrow \overline{X}$ is a compact manifold $X_F$ whose boundary consists of two copies of $F$, one being the image of the other by $z$.
Equivalently, $X_F={\rm closure}(X \backslash (F \times [0,1])) \hookrightarrow X$.
Note that $X_F$ is a fundamental domain for the  covering $\overline{X} \to X$.

$$
\includegraphics[width=\linewidth]{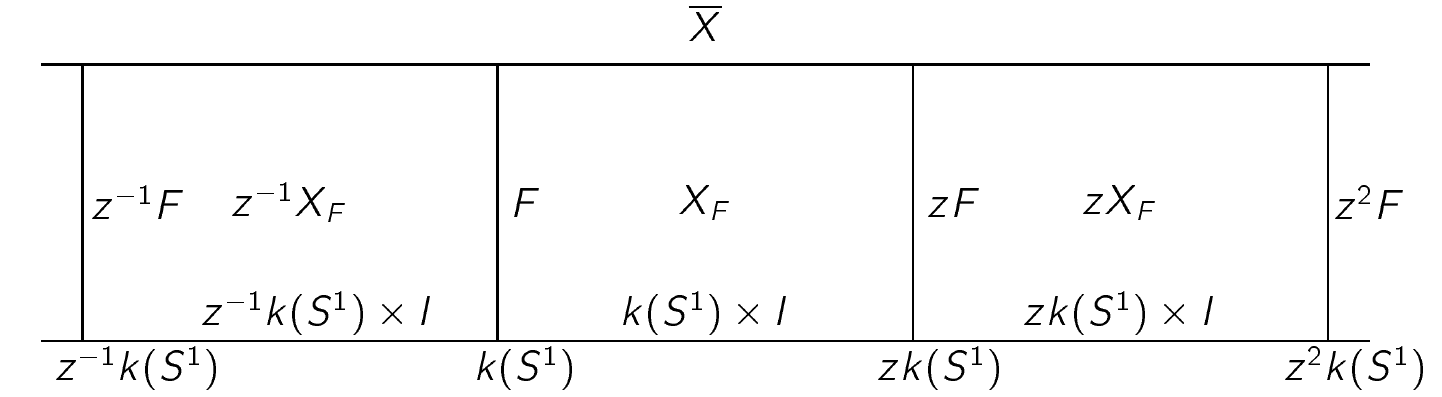}
$$

A curve $v:\SSS^1 \subset X_F\subset X \backslash F$  
and a curve $w:\SSS^1 \subset F$ are disjoint curves in $\SSS^3$, and have a well defined linking number in $\ZZ$.
This defines a $\ZZ$-module isomorphism
$$H_1(X_F;\ZZ) \xymatrix{\ar[r]^-{\cong}&} H_1(F;\ZZ)^*;~v \mapsto (w \mapsto {\rm {\rm lk}}(v,w)).$$
The inclusions $i_+:F \hookrightarrow X_F$ and $i_-:zF \hookrightarrow  X_F$ induce dual $\ZZ$-module morphisms
$$\begin{array}{l}
(i_+)_*=\Sigma:~H_1(F;\ZZ) \to H_1(X_F;\ZZ)~\cong~H_1(F;\ZZ)^*,\\
(i_-)_*=\Sigma^*:~H_1(F;\ZZ) \to H_1(X_F;\ZZ)~\cong~H_1(F;\ZZ)^*
\end{array}$$
and there is defined a short exact sequence of $\ZZ[z,z^{-1}]$-modules
$$\xymatrix@C+10pt{0 \ar[r] & H_1(F;\ZZ)[z,z^{-1}] \ar[r]^-{\Sigma-z\Sigma^*}& H_1(F;\ZZ)^*[z,z^{-1}] \ar[r] & H_1(\overline{X};\ZZ) \ar[r] & 0}.$$
This construction is due to Hirsch and Neuwirth~\cite{hirschneuwirth}.
It follows from $p_*:H_*(X)\cong H_*(\SSS^1_c)$ that in the knot case $c=1$ the $\ZZ[z,z^{-1}]$-module morphism $1-z:H_1(\overline{X}) \to H_1(\overline{X})$ is an isomorphism.
The {\it Alexander polynomial} of $k$
$$\Delta_k(z)={\rm det}(z\Sigma -  \Sigma^{\star}) \in \ZZ[z]$$
is such that $\Delta_k(z)H_1(\overline{X};\ZZ)=\{0\}$.
The Alexander polynomial is an isotopy invariant of the link $L$, which was introduced in~\cite{alexander} in a much more algebraic (and obscure) way.
The expression of $\Delta_k(z)$ in terms of $\Sigma$ is due to Seifert~\cite{seifert}.
The modern version of $\Delta_k(z)$ is the {\it Alexander-Conway polynomial}, the determinant of $u\Sigma - u^{-1}\Sigma^*$ in $\ZZ[u,u^{-1}]$ (with $u=z^2$), which is related to the Jones polynomial.

\subsection{Signatures of knots :  a tale from the sixties}\label{omegasignatures}

For a knot $k:\SSS^1 \subset \SSS^3$ Seifert~\cite{seifert} used his matrix $\Sigma$ to compute the homology groups $H_*(\overline{X}_a)$ (in fact, only the Betti numbers) of the finite cyclic covers $\overline{X}_a=\overline{X}/\{z^a\}$ ($a \geqslant   1$) of the knot exterior $X$, which are isotopy invariants.

In 1962, Trotter~\cite{trotter} published a remarkable paper providing a detailed analysis of the homology $H_*(\overline{X})$ of $\overline{X}$ itself.
In particular, he observed that for any Seifert matrix $\Sigma$ the signature of the symmetric matrix $\Sigma + \Sigma^{\star}$ over $\ZZ$ is also invariant (\cite[Proposition (v)1]{trotter}).
This is the {\it signature} of the knot.
In general, the matrix $\Sigma + \Sigma^*$ is only invertible over $\QQ$.

In 1965, Murasugi~\cite{murasugi} used this signature to get concrete topological consequences.
For instance, he showed that an {\it alternate} knot has a positive signature and therefore cannot be isotopic to its mirror image since one easily checks that mirror images have opposite signatures.
More importantly in our context, he showed that the signature can be used to study cobordism.
This concept had recently been introduced in paper by Fox and Milnor~\cite{foxmilnor} (which actually only appeared in 1966).

A knot is {\it slice} if it bounds an embedded disc in the $4$-ball.
Two knots $k_0,k_1$ are  {\it cobordant} if one can embed a cylinder $\SSS^1 \times [0,1]$ in $\SSS^3 \times [0,1]$ whose boundary is $k_0\times \{0 \}$ and $k_1 \times \{ 1\}$.

If a knot is slice, it is easy to show that the homology of a Seifert surface contains a subspace of half dimension on which the Seifert form identically vanishes.
It follows indeed that the signature vanishes.

Two independent papers by Tristram~\cite{tristram} and Levine~\cite{levine1} appeared in 1969 and extended this idea, both defining the {\it $\omega$-signature} of a Seifert matrix $\Sigma$ of a knot $k:\SSS^1 \hookrightarrow \SSS^3$ for $\omega \in \SSS^1\subset \CC$ to be the signature of the {\it hermitian} form $(1-\omega)\Sigma+ (1-\overline{\omega}) \Sigma^{\star}$
$$\tau_{\omega}(\Sigma)=\tau((1-\omega)\Sigma+ (1-\overline{\omega}) \Sigma^{\star}) \in \ZZ,$$
establishing cobordism invariance for appropriate $\omega$.  The determinant of the hermitian form is such that
$$\det((1-\omega)\Sigma+ (1-\overline{\omega}) \Sigma^{\star})=-(1-\overline{\omega})\Delta_k(\omega) \in \CC$$
so that for $\omega \neq 1 \in \SSS^1$ the hermitian form is nonsingular if and only if  $\Delta_k(\omega) \neq 0 \in \CC$.  In~\cite{tristram} only  the following $\omega$ were considered
$$\omega_p=\exp((p-1)\pi i/p)~\hbox{for an odd prime $p$, and}~\omega_2=\exp(\pi i)=-1$$
and it was proved that $\Delta_k(\omega_p) \neq 0$ and that the $\omega_p$-signatures $\tau_{\omega_p}(k) \in \ZZ$ are cobordism invariants. In~\cite{levine1} it was proved that the  $\omega$-signature function
$$\tau:~\omega \in \SSS^1 \mapsto \tau_{\omega}(\Sigma) \in \ZZ$$
takes constant values between the roots $\omega \in \SSS^1$ of the Alexander polynomial $\Delta_k(z)$, and that these values are cobordism invariants.
See section~\ref{modern} below for a discussion of the jumps in the $\omega$-signature function.

The connected sum of knots defines an operation which provides the space of cobordism of knots with the structure of an abelian group $C_1$.
Tristram~\cite{tristram} uses these $\omega$ signatures to find some epimorphism from $C_1$ to $\ZZ^{\infty}$.
The independent paper by Levine~\cite{levine1} contains similar results, expressed in a more algebraic terminology, but also showing that the $\omega$-signatures determine the cobordism class of a high-dimensional knot $k:\SSS^{2i-1} \hookrightarrow \SSS^{2i+1}$ modulo torsion - see Ranicki~\cite{ranickiknot} for an account of the high-dimensional knot cobordism groups using the $L$-theory of $\ZZ[z,z^{-1}]$ with the involution $\bar{z}=z^{-1}$.

Viro~\cite{viro} and Kauffman and Taylor~\cite{kauffmantaylor} identified $\tau(k)$ with the signature of the 4-dimensional manifold $N$ which is a double cover of $\DDD^4$ branched along $k$.
See more details  in Kauffman~\cite{kauffman} and Lickorish~\cite{lickorish}.

\subsection{Two great papers by Milnor and the modern approach}\label{modern}

The 1968 paper by J.~Milnor~\cite{milnorinfinite} is a pure gem and contains the ``right definition'' of the Murasugi-Tristram-Levine signatures.
There is no more any need for the non canonical choice of a Seifert surface and one gets quadratic forms {\it canonically} attached to a knot.

As a motivation, let us introduce first the concept of {\it fibred} knot.
A knot $k:\SSS^1 \hookrightarrow \SSS^3$  is {\it fibred} its complement has the structure of an {\it open book}.
In other words, there should exist a locally trivial fibration $p:\SSS^3\backslash k(\SSS^1) \to \SSS^1$.
In a tubular neighborhood of $k$, identified with $\SSS^1 \times \DDD^2$, the map $p$ is defined outside $k$ and should be of the form $p(x,z)=\text{Arg} (z) \in \SSS^1$.
For an arbitrary $* \in \SSS^1$ the fibre $p^{-1}(*)$ is such that $\Sigma=k(\SSS^1) \cup p^{-1}(*)$ is a  Seifert surface for $k$.
The fibres of $p$ look indeed like the pages of a book whose binding is $k$.
Many interesting knots are fibred.
For instance, if $P(u,v)$ is a polynomial in two complex variables $(u,v)$ such that $(0,0)$ is an isolated singularity of the curve $P(u,v)=0$, one can consider the knot $k$ which is the intersection of $P(u,v)=0$ with a small sphere $\SSS^3_{\epsilon} = \vert u \vert ^2 + \vert v \vert ^2 = \epsilon^2$ and the map $p=P(u,v)/\vert P(u,v) \vert \in \SSS^1$ provides such a fibration.
See the book by Milnor~\cite{milnorsingular} on singularities of hypersurfaces.
However, many knots are not fibred.

\begin{center}
\includegraphics[width=.7\linewidth]{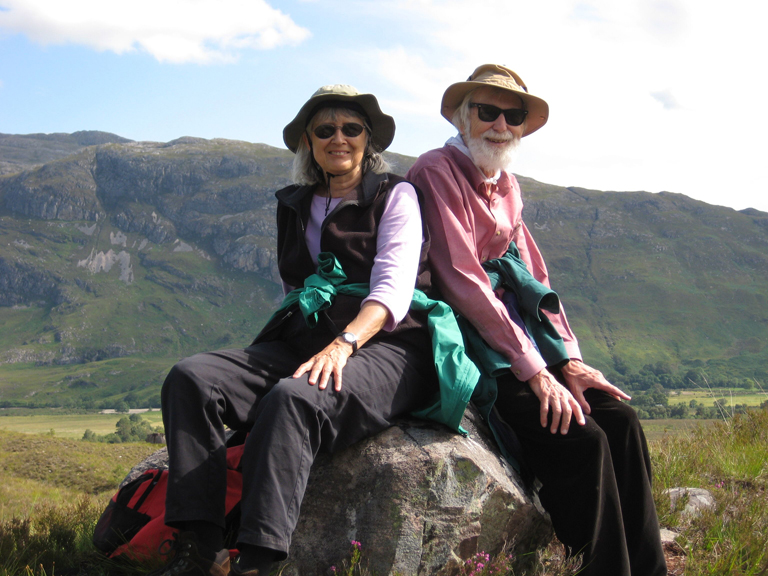}

Dusa McDuff and John Milnor, somewhere in Scotland.
\end{center}

As before, let
$$(X,\partial X)=({\rm closure}(\SSS^3 \backslash k(\SSS^1) \times \DDD^2),k(\SSS^1)\times \SSS^1)$$
and let $(\overline{X},\overline{\partial X})$ be the canonical infinite cyclic cover.
In the case where $k$ is fibred,
$$(\overline{X},\overline{\partial X})=(F,\partial F) \times \RR$$
with the generating covering translation $z:\overline{X} \to \overline{X}$ given by
$$z: (x,t) \in F \times \RR \mapsto (A(x),t+1) \in F \times \RR$$
for some {\it monodromy} diffeomorphism $A:F \to F$, which is the identity in a neighborhood of the boundary.

Such a picture does not hold in the general non fibred case, but it does hold at the homology level, at least if one works with real coefficients.
Milnor shows the following, for any knot $k$:
\begin{itemize}
\item $H^1(\overline{X}, \partial \overline{X},  \RR)$ has a finite even dimension,
\item $H^2(\overline{X}, \partial \overline{X}, \RR) \cong \RR$,
\item The evaluation of the cup product
$$
\theta :~H^1(\overline{X}, \partial \overline{X},  \RR) \times H^1(\overline{X}, \partial \overline{X},  \RR) \to H^2(\overline{X}, \partial \overline{X}, \RR) \cong \RR
$$
defines a symplectic vector space $(H^1(F;\RR),\theta)$ over $\RR$.
\end{itemize}

Let us call a linear automorphism $A$ {\it fibred} if $A-I$ is invertible, i.e. if 1 is not an eigenvalue.
The monodromy acts on the symplectic vector space $(H^1(\overline{X}, \partial \overline{X};  \RR),\theta)$ by a fibred  linear automorphism, still denoted by $A$, such that $A^*\theta A=\theta$.

In this way, {\it we can attach canonically a fibred symplectic automorphism $A$ of a symplectic vector space to any knot $k$}.
This is the fundamental invariant of a knot.

For any fibred  automorphism $A$ of a symplectic vector space $(H,\theta)$ over $\RR$ there is defined an abstract Seifert form (= asymmetric bilinear form) over $\RR$
$$\Sigma=\theta(I-A)^{-1}:~H \to H^*$$
such that $A=\Sigma^{-1}\Sigma^*$, with
$$\Sigma-\Sigma^*=\theta:~H \to H^*.$$
This allows the $\omega$-signature  $\tau_{\omega}(A) \in \ZZ$ to be defined for any $\omega \in \SSS^1$ to be the signature of the hermitian form
$$(1-\omega)\Sigma+(1-\overline{\omega})\Sigma^*:~ \CC\otimes_\RR H \to \CC \otimes_\RR H^*$$
as in section~\ref{omegasignatures}.
For any knot $k$ it is possible to choose a  Seifert surface $F$ such that
the  Seifert form $\Sigma:H_1(F;\RR) \to H_1(F;\RR)^*$ is an isomorphism,
in which case the exact sequence
$$\xymatrix@C+10pt{0 \ar[r] & H_1(F;\RR)[z,z^{-1}] \ar[r]^-{\Sigma-z\Sigma^*}& H_1(F;\RR)^*[z,z^{-1}] \ar[r] & H_1(\overline{X};\RR) \ar[r] & 0}$$
gives an isomorphism of symplectic forms over $\RR$
$$(H_1(F;\RR),\Sigma-\Sigma^*) \cong (H^1(\overline{X};\RR),\theta)$$ and   the  fibred automorphism of the symplectic form $(H^1(\overline{X};\RR),\theta)$  over $\RR$ is  given by
$$A=\Sigma^{-1}\Sigma^*:~(H^1(\overline{X};\RR),\theta) \to (H^1(\overline{X};\RR),\theta)$$
and $\tau_{\omega}(k)=\tau_{\omega}(A)$.

In the second step, one should extract invariants from symplectic automorphisms. As we know, any symplectic vector space is isomorphic with $\RR^{2n}$ equipped with the bilinear form:
$$
\Omega ((x_1,\dots,x_{2n}),(y_1,\dots, y_{2n})) = x_1y_{n+1}+\dots x_ny_{2n}- x_{n+1}y_1 - \dots - x_{2n}y_n.
$$
An $2n \times 2n$ matrix $A$ is said to be {\it symplectic} if it preserves the bilinear form $\Omega$. The {\it symplectic group}
$${\rm Sp}(2n,\RR)={\rm Aut}_\RR H_-(\RR^n)$$
consists of the symplectic $2n \times 2n$ matrices - we shall describe it in more detail later on. Therefore, any knot defines a symplectic matrix, unique up to conjugacy.

{\it Conjugacy invariants of matrices}, real or complex, are very well known and are described by the classical Jordan normal forms.
In the generic case, when eigenvalues are distinct, matrices are diagonalizable (over the complex) and the spectrum is the only invariant.
The situation is very different in the skew-symmetric case and even in the simplest situations, the spectrum does not contain enough information to characterize the conjugacy class.
This is the role of the signature.

To give a very simple example, consider rotations $R_1,R_2$ in the plane $\RR^2$ of angles $\alpha$ and $-\alpha$.
They are both symplectic automorphisms of $\RR^2$ equipped with the standard symplectic form.
They have the same spectrum $\exp{ \pm 2 \pi i \alpha}$.
Of course, they are conjugate in ${\rm GL}(2,\RR)$ but not in ${\rm SL}(2,\RR)$ : they don't rotate in the same direction.

The description of conjugacy classes of symplectic automorphisms is fundamental in hamiltonian dynamics since the flows under consideration are symplectic.
For some reasons, dynamicists and topologists did not collaborate too much on these questions.
The first (almost complete) solution to the problem has been given by Williamson~\cite{williamson} in 1936, with dynamical motivations.
Then, a better understanding, still with dynamical motivations, was given in the wonderful book by Yakubovich and Starzhinskii~\cite{yakubovich}.
See  Ekeland~\cite{ekeland},
 Long and Dong~\cite{longdong}, Gutt~\cite{gutt} for recent and detailed presentations, also with dynamical motivations. From the point of view of algebraic topology, the best reference is another beautiful paper of Milnor~\cite{milnorinner} on isometries of inner product spaces, which deals with the general case over any field.

We shall not give the full description and content ourselves with the part which is relevant for the definition of the signatures.
Our presentation is only slightly different and uses basic ideas from Krein's theory.

From the symplectic form $\Omega$ on $\RR^{2n}$, one produces canonically a hermitian form $G$ on $\CC^{2n}$.
The hermitian product of $u,v$ in $\CC^{2n}$ is given by:
$$
G(u,v) = i \Omega(u,\overline{v}).
$$
Note that all real vectors in $\RR^{2n}$ are isotropic for $G$.
Also, note that any symplectic automorphism of $(\RR^{2n},\Omega)$ is also a unitary automorphism of $(\CC^{2n}, G)$.

Given a complex subspace $E$ of $\CC^{2n}$ one can consider the restriction of $G$ to $E$.
As any hermitian form has a signature, it follows that any subspace $E$ has a signature $\tau(E)$.
Of course complex conjugate subspaces have opposite signatures.

Finally, let $A$ be a symplectic automorphism (or more generally a unitary automorphism of $\CC^{2n}$) and $\lambda$ a complex number.
Let $E_{\lambda}(A)$ be the corresponding characteristic subspace:
$$E_{\lambda}(A) = \bigcup\limits_{j\geqslant   1} {\rm ker} (A-\lambda Id)^j$$
(which is non trivial only if $\lambda$ is an eigenvalue of $A$).
The signature of the restriction of $G$ to this subspace is called the {\it $\lambda$-signature} of $A$ $$\tau^{\lambda}(A)=\tau(E_{\lambda},G\vert) \in \ZZ.$$
In other words, we have a map
$$A \in {\rm Sp}(2n,\RR) \mapsto \bigoplus\limits_{\lambda}\tau^{\lambda}(A) \in \bigoplus\limits_{\lambda \in \CC}\ZZ$$
with each $\tau^{\lambda}(A) \in \ZZ$ a conjugacy invariant.

The spectrum of symplectic $A$ can be decomposed in disjoint groups (possibly empty) of four different types.
\begin{itemize}
\item 4-tuples $\{ \lambda, \lambda^{-1}, \overline{\lambda}, \overline{\lambda}^{-1} \}$,
\item pairs of the form $\{ \lambda, \lambda^{-1}\}$ for some real $\lambda$ different from $\pm 1$,
\item $\pm 1$,
\item pairs of the form $\{ \lambda, \lambda^{-1}\}$ for some $\lambda$ of modulus 1.
\end{itemize}

It is very easy to see that the corresponding $\tau^{\lambda}(A)$ can be non trivial only in the case of a non real eigenvalue of modulus one.

Summing up this discussion, we conclude that one can attach to each symplectic matrix $A$ a function on the circle:
$$
 \lambda \in \SSS^1 \mapsto \tau^{\lambda}(A) \in \{-2n, \dots, 2n\}.
$$
Note that $\tau^{\lambda}(A) =- \tau_{\overline{\lambda}}(A)$ so that one could assume that the imaginary part of $\lambda$ is positive, without losing information.
Note also that $\tau^{\lambda}(A)=0$ if $\lambda$ is not in the spectrum of $A$.

When applied to the symplectic automorphism $A$ attached to a knot $k:\SSS^1 \hookrightarrow \SSS^3$, one gets (a slight modification of)  Milnor's signatures. For any $\omega_1=e^{2\pi i \theta_1}$, $\omega_2\in e^{2\pi i \theta_2} \in \SSS^1$
with $0 \leqslant \theta_1 < \theta_2 < 2\pi$ and a single eigenvalue
$\lambda=e^{2\pi i \mu}$ with $\theta_1 < \mu < \theta_2$ the jump in
the $\omega$-signatures is
$$\begin{array}{ll}
\sigma_{\omega_2}(k) -\sigma_{\omega_1}(k)&
=~\sigma_{\omega_2}(A) -\sigma_{\omega_1}(A)\\
&=~2\tau^{\lambda}(A)\in \ZZ
\end{array}$$
(Matumoto\cite{matumoto}, Ranicki~\cite[Prop. 40.10]{ranickiknot}).

All the previously discussed signatures can be obtained from these $\tau^{\lambda}$.
For instance, the Murasugi signature is the sum of all the $\tau^{\lambda}$ for all $\lambda$ with  $\Im (\lambda) >0$.

Let us work out a simple example.
Let
$$
A=\begin{pmatrix} \cos \alpha & - \sin \alpha \\ \sin \alpha & \cos \alpha \end{pmatrix}.
$$
The eigenvalues are $\exp(\pm i \alpha)$ and the corresponding eigenvectors are $(1, \mp i)$.
Computing the $G$ norm of these vectors, we get
$$G((1,\mp i),(1, \mp i))= i \Omega ((1,\mp i),(1, \pm i))= \pm 2$$
so that $\tau^{\exp (\pm i \alpha)}= \pm 1$.

\subsection{Braids}

See Epple~\cite{epple1} for the history of braids in the 19th century, starting with Gauss: here is a braid drawn by him in 1833:
$$
\includegraphics[width=.3\linewidth]{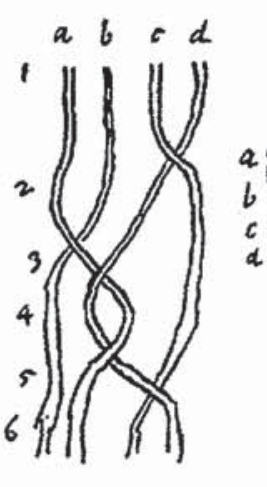}
$$

Fix $n \geqslant   2$ and $n$ distinct points $z_1,z_2,\dots,z_n \in \DDD^2$.
An {\it $n$-strand braid} $b$ is an embedding
$$\beta:~\coprod\limits_n I=\{1,2,\dots,n\} \times I \hookrightarrow \DDD^2 \times I$$
such that each of the composites
$$\xymatrix@C+15pt{I \ar[r]^-{\beta(k,-)} & \DDD^2 \times I \ar[r]^-{\text{projection}} &I}~(1 \leqslant k \leqslant n)$$
is a homeomorphism, and
$$\beta(k,0)=(z_k,0)~\in \DDD^2 \times \{0\},~\beta(k,1)=(z_{\sigma(k)},1) \in \DDD^2 \times \{1\}$$
for some permutation $\sigma $ of $\{1,2,\dots,n\}$.
Such a $\beta$ defines $n$ disjoint forward paths $t \mapsto \beta(k,t)$ in $\DDD^2 \times I \hookrightarrow \RR^3$ from $(z_k,0)$ to $(z_{\sigma(k)},1)$, such that each section
$$\beta(\{1,2,\dots,n\} \times I) \cap (\DDD^2 \times \{t\})~(t \in I)$$
consists of $n$ points.

\begin{center}
\includegraphics[width=.6\linewidth]{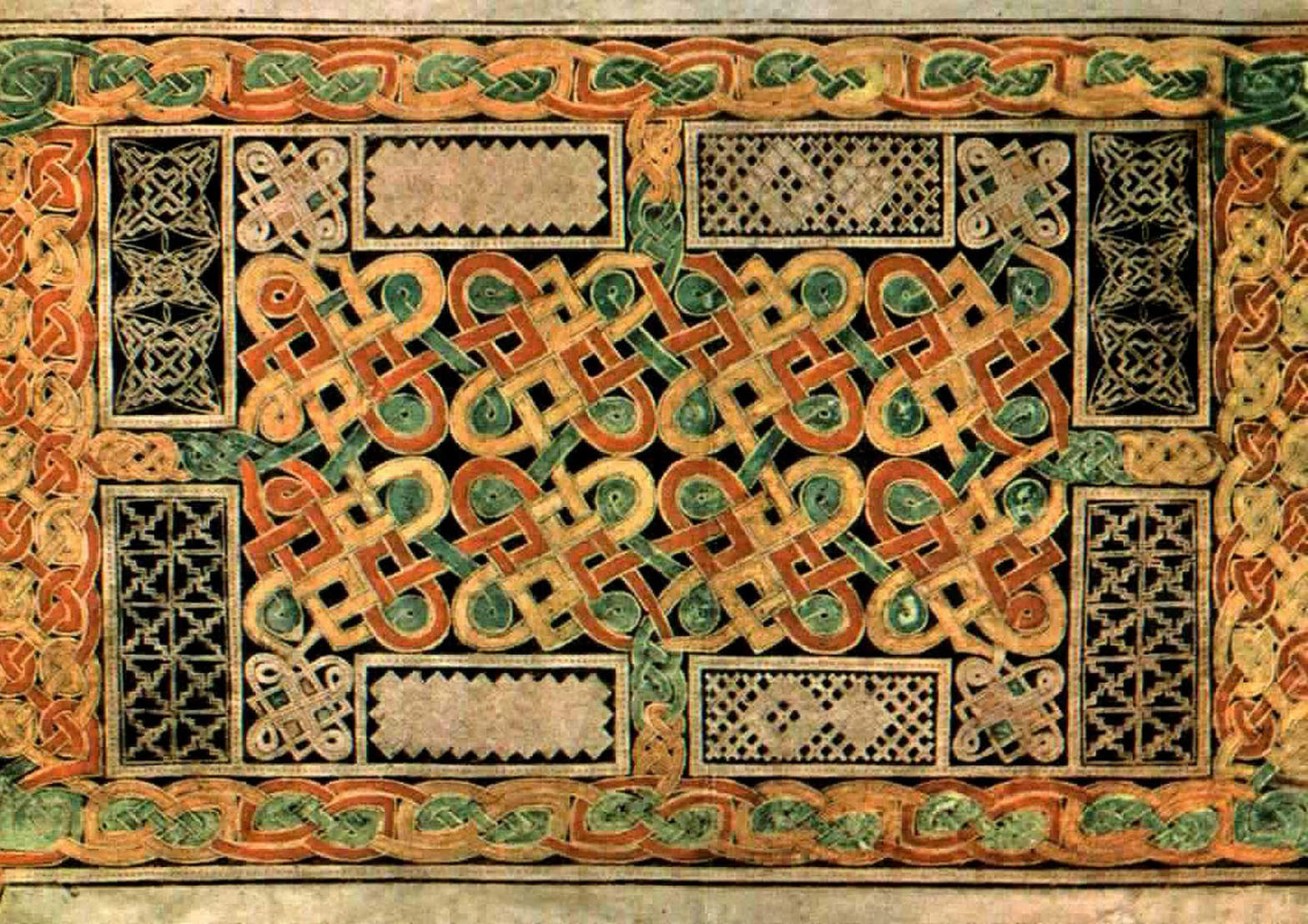}

\text{Braids in the Book of Durrow (Ireland)}
\end{center}

Artin~\cite{artin} founded the modern theory of braids.

The {\it trivial} $n$-strand braid is
$$\sigma_0: t_i \in \coprod\limits_n I   \mapsto  (i,t_i,0) \in \RR^3 $$
$$
\includegraphics[width=.5\linewidth]{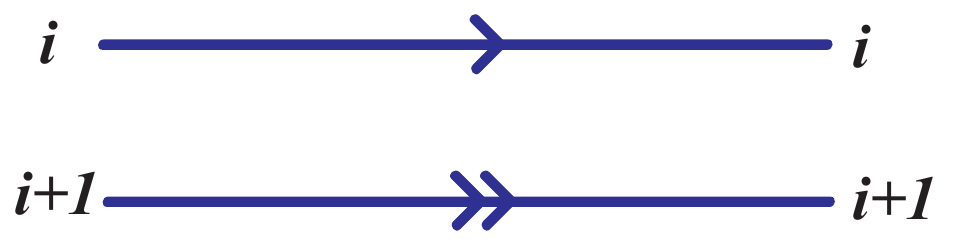}
$$
For $i=1,2,\dots,n-1$ the {\it elementary $n$-strand braid}
$\sigma_i$ is obtained from $\sigma_0$ by introducing an overcrossing of the
$i$th strand and the $(i+1)$th strand,
with the transposition $(i,i+1) $.
$$
\includegraphics[width=.5\linewidth]{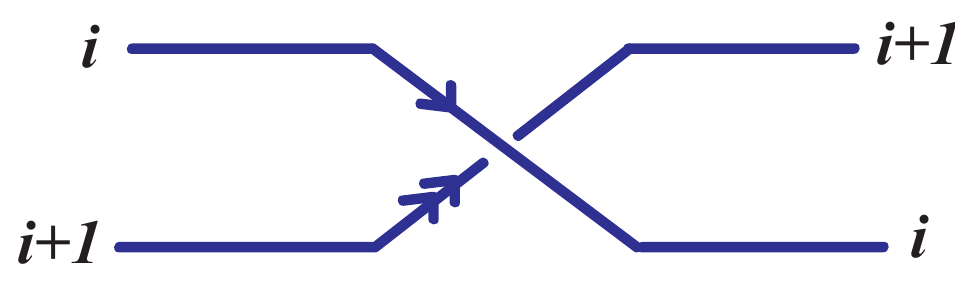}
$$
The inverse {\it elementary $n$-strand braid}
$\sigma_i^{-1}$ is defined in the same way but with an under crossing.
$$
\includegraphics[width=.5\linewidth]{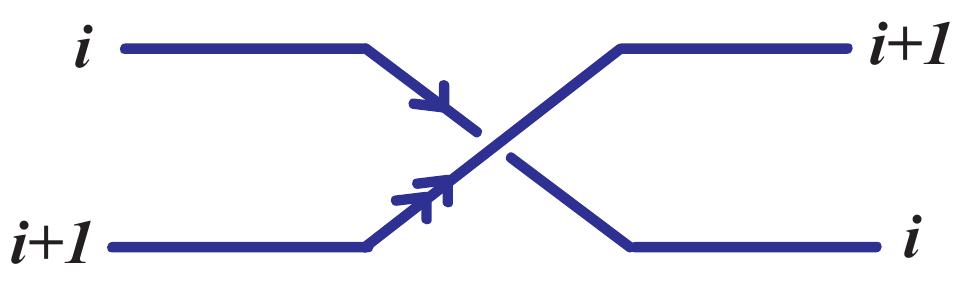}
$$
The Artin group  of isotopy classes of $n$-strand braids under concatenation is denoted by $B_n$ - it has generators $\sigma_1,\sigma_2,\dots,\sigma_{n-1}$ and relations
$$\begin{cases}
\sigma_i\sigma_j=\sigma_j\sigma_i&\text{if}~\vert i-j \vert \geqslant   2\\
\sigma_i\sigma_j\sigma_i=\sigma_j\sigma_i\sigma_j&\text{if}~
\vert i-j \vert=1.
\end{cases}$$
Every $n$-strand braid $\beta$ is represented by a word in $B_n$ in $\ell$ generators, corresponding to a sequence of $\ell$  crossings in a plane projection.

The {\it closure} of an $n$-strand braid $\beta$ is the $c$-component link
$$\widehat{\beta}=\beta \cup \sigma_0:~\coprod\limits_n I\cup_{\sigma}\coprod\limits_n I=\SSS^1_c \hookrightarrow \RR^3$$
with $c= \vert \{1,2,\dots,n\}/\sigma \vert$ the number of cycles in $\sigma$.
Alexander~\cite{alexlem} proved that every link is the closure $\widehat{\beta}$ of some braid $\beta$.
For example, the Hopf link is the closure of the concatenation $\sigma_1\sigma_1$
$$
\includegraphics[width=.8\linewidth]{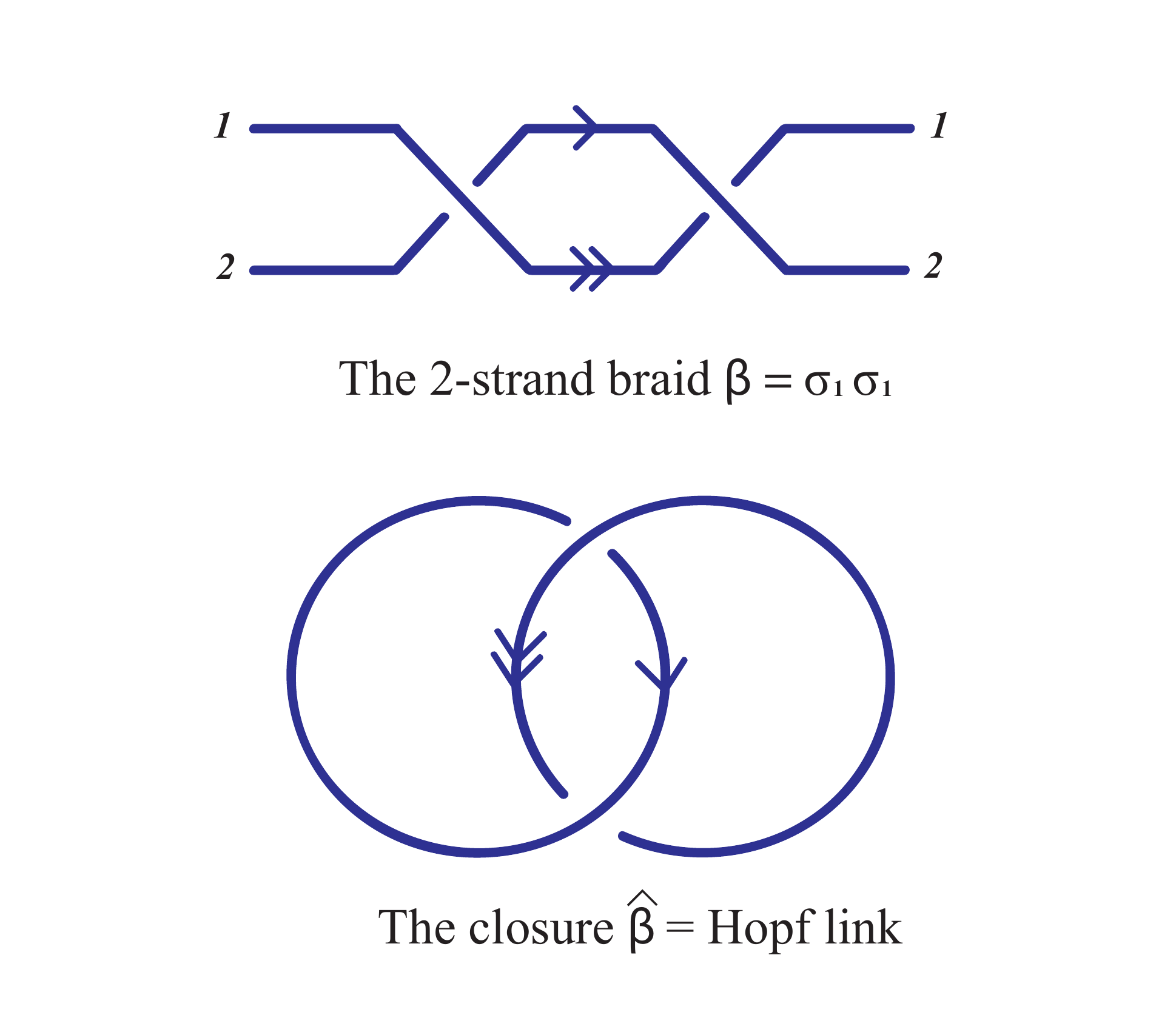}
$$

An $n$-strand braid $\beta$ with $\ell$ crossings is represented by a word
in $B_n$ of length $\ell$ in the generators $\sigma_1,\sigma_2,\dots,\sigma_{n-1}$, so that   $\beta=\beta_1\beta_2 \dots \beta_{\ell}$ is the concatenation of $\ell$ elementary braids.

Stallings~\cite{stallings} observed that the closure $\widehat{\beta}$ has a canonical projection with $n$ Seifert circles and $\ell$ intersections, and hence a canonical Seifert surface with $n$ 0-handles and $\ell$ 1-handles
$$F_{\beta}=\coprod\limits_n \DDD^2 \cup \coprod\limits_\ell \DDD^1 \times \DDD^1 \hookrightarrow \RR^3$$
and hence a canonical Seifert matrix $\Psi_\beta$.
This surface $F_\beta$ is homotopy equivalent to the $CW$ complex
$$X_\beta=\coprod\limits^n_{i=1}e^0_i \cup \coprod\limits^{\ell}_{j=1} e^1_j$$\\
 with $\partial e^1_j=e^0_i \cup e^0_{i+1}$ if $j$th crossing is between strands $i,i+1$
$$\begin{array}{ll}
H_1(F_\beta)=H_1(X_\beta)&=~\text{ker}(d:C_1(X_\beta)\to C_0(X_\beta))\\
&=~\text{ker}(d:\ZZ^\ell\to \ZZ^n)=\ZZ^m.
\end{array}$$
\indent We refer to  Stallings~\cite{stallings}, Gambaudo and Ghys~\cite{gambaudoghys3}, Cohen and van Wijk~\cite{cohenvanwijk}, Collins~\cite{collins},
Bourrigan~\cite{bourrigan} and Palmer~\cite{palmer} for more detailed accounts of the Seifert surfaces and Seifert matrices of braids.

\begin{center}
\includegraphics[width=.4\linewidth]{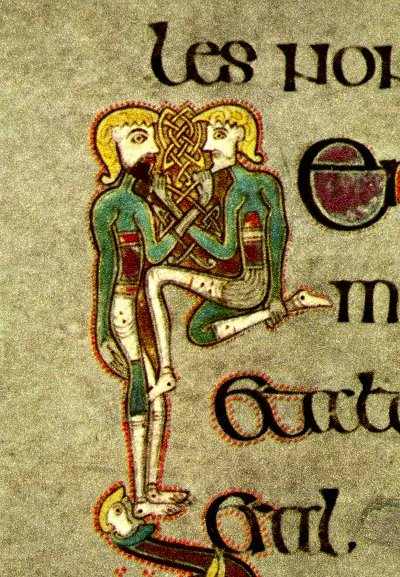}

Beard pullers and braids in the Book of Kells (Ireland)
\end{center}

\subsection{Knots and signatures in higher dimensions}

There is also a high-dimensional knot theory, for knots (or links)
$$k: \SSS^n \hookrightarrow \SSS^{n+2}$$
in all dimensions $n \geqslant   1$.
As noted by Adams~\cite{adams2} :

\begin{quote}
{\it Of course, from the point of view of the rest of mathematics, knots in higher-dimensional space deserve just as much attention as knots in 3-space}.
\end{quote}

The high-dimensional theory was initiated by Kervaire~\cite{kervairenoeuds} using surgery methods, including the use of the plumbing construction to realize any bilinear pairing $\Sigma:H \times H \to \ZZ$ on a f.g. free $\ZZ$-module $H$ with $(H,\Sigma+(-1)^{\ell}\Sigma^*)$ a nonsingular $(-1)^{\ell}$-symmetric form over $\ZZ$ as the Seifert form of a $(2\ell-1)$-knot $k:\SSS^{2\ell-1} \hookrightarrow \SSS^{2\ell+1}$ with an $(\ell-1)$-connected Seifert surface $F^{2\ell} \hookrightarrow \SSS^{2\ell+1}$ such that $H_{\ell}(F;\ZZ)=H$, for any $\ell \geqslant   1$.
The infinite cyclic cover $\overline{X}$ is $(\ell-1)$-connected;
Levine~\cite{levine2} proved that for $\ell \geqslant   2$ the isotopy classes of such {\it simple} $(2\ell-1)$-knots $k:\SSS^{2\ell-1} \hookrightarrow \SSS^{2\ell+1}$ are in one-one correspondence with the
$S$-equivalence classes of Seifert matrices $\Sigma$ with $\Sigma+(-1)^{\ell}\Sigma^*$ invertible.

A {\it cobordism} of $n$-knots $k_0,k_1:\SSS^n \hookrightarrow \SSS^{n+2}$ is an embedding
$$\ell:~\SSS^n \times I \hookrightarrow \SSS^{n+2} \times I$$
such that
$$\ell(u,i)=(k_i(u),i)~(u \in \SSS^n,i\in \{0,1\}).$$

A knot $k$ is {\it slice} if it extends to an embedding of the ball $\DDD^{n+1}$ in the sphere $\DDD^{n+3}$.

The set of cobordism classes of $n$-knots is an abelian group $C_n$, with addition by connected sum.

\begin{theorem}[Kervaire~\cite{kervairenoeuds}] \leavevmode

1. For $n \geqslant   2$ every $n$-knot is cobordant to one with $p:X\to \SSS^1$
a homotopy equivalence (resp. $\ell$-connected) if $n=2\ell$ (resp. $n=2\ell-1$).

2. $C_{2\ell}=0$ for $\ell \geqslant   1$.

3. There are canonical maps $C_{2\ell-1} \to C_{2\ell+3}$ which are isomorphisms for $\ell \geqslant   2$ and a surjection for $\ell=1$.

\end{theorem}

The knot cobordism groups $C_{2\ell-1}$ for $\ell \geqslant   2$ were computed  by Levine~\cite{levine1,levine3}
$$C_{2\ell-1}=\bigoplus\limits_{\infty}\ZZ \oplus \bigoplus\limits_{\infty}\ZZ/2\ZZ \oplus\bigoplus\limits_{\infty}\ZZ/4\ZZ~ (\text{countable}~\infty'\text{s}).$$
The high dimensional cobordism class $k \in C_{2\ell-1}$ of a $(2\ell-1)$-knot $k:\SSS^{2\ell-1}\hookrightarrow \SSS^{2\ell+1}$ is determined modulo torsion by primary invariants of the signature type, with one $\ZZ$-valued signature  for each root $e^{i\theta}\in \SSS^1$ (an algebraic integer) of the Alexander polynomial $\Delta_k(z)\in \ZZ[z,z^{-1}]$ with $0<\theta<\pi$, and for $\ell \equiv 0(\bmod\,2)$ also the signature. These computations were extended by
Sheiham~\cite{sheiham} to the cobordism groups of high-dimensional boundary links $L:\SSS^{2\ell-1}_c \hookrightarrow \SSS^{2\ell+1}$ (i.e. links with
$c$-component Seifert surfaces).

See Ranicki~\cite{ranickiknot} for an account of high-dimensional knot theory, including the many applications of the signature in the computation of $C_{2\ell-1}$ for $\ell \geqslant   2$.
For the relationship between knot invariants involving the infinite cyclic cover of the knot complement and Seifert surfaces see Blanchfield~\cite{blanchfield}, Kearton~\cite{kearton} and Ranicki~\cite{ranickibs}.

The computation of the knot cobordism groups is much more complicated for classical knots, i.e. for $C_1$. Casson-Gordon~\cite{cassongordon} showed in  1975 that  the surjection $C_1 \to C_{4{\star}+1}$ has non-trivial kernel, detected by the signatures of quadratic forms over cyclotomic fields.
The kernel has been studied by Cochran, Orr and Teichner~\cite{cochranorrteichner} using $L^2$-signatures and noncommutative Blanchfield forms.

\newpage

\section{The symplectic group and the Maslov class}\label{maslov}

\subsection{A very brief history of the symplectic group}

It is difficult, if not impossible, to describe the history of the symplectic group.
Brouzet~\cite{brouzet} mentions a double origin.
The first is {\it projective geometry} during the nineteenth century, in particular connected with the concept of duality.
Duality with respect to a quadric, i.e. a quadratic form, was generalized to duality with respect to a {\it symplectic} form, and this turned out to be relevant in the study of the ``space of lines" and in particular in the study of ``complexes'' (families of lines in $3$-space depending on two parameters).

The second is related to {\it algebraic geometry}, in particular with the foundational work of Riemann on abelian integrals.
Indeed, in modern terminology, the first cohomology of a Riemann surface is canonically equipped with a  symplectic structure,  defined by the cup-product,  in duality with the intersection of cycles.

Surprisingly, this paper does not mention a third origin, which may even be more important, at least from our point of view: the {\it geometrization of mechanics} all along the nineteenth century.
One should mention many names.
Among them Lagrange, Laplace, Legendre, Poisson, Hamilton, Liouville should be emphasized.
Notice that they all have their lunar craters.

\begin{figure}[ht]
\centering
\begin{minipage}[b]{0.3\linewidth}
\includegraphics[height=.9\textwidth]{lagrangecrater.jpg}
\center{Lagrange}
\end{minipage}
\hfill
\begin{minipage}[b]{0.3\linewidth}
\includegraphics[height=.9\textwidth]{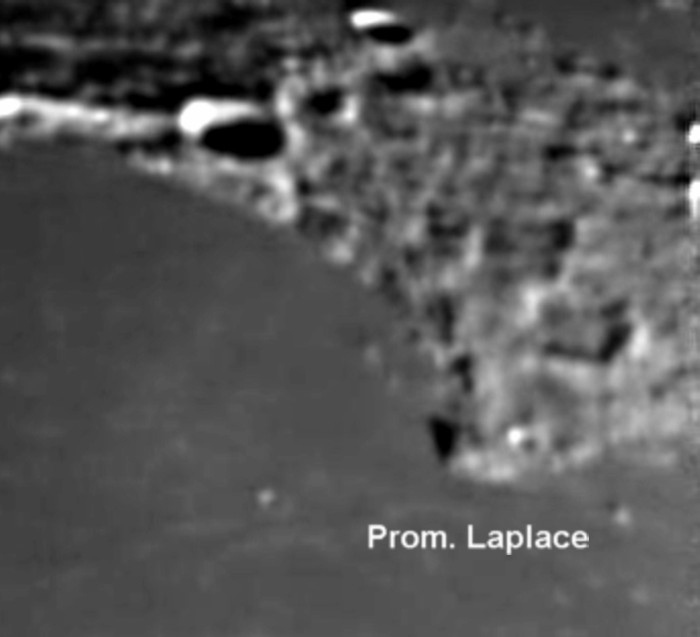}
\center{Laplace }
\end{minipage}
\hfill
\begin{minipage}[b]{0.3\linewidth}
\includegraphics[height=.9\textwidth]{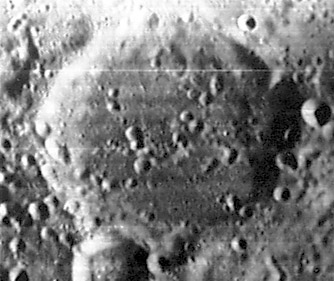}
\center{Legendre }
\end{minipage}
\end{figure}
\begin{figure}[ht]
\centering
\begin{minipage}[b]{0.3\linewidth}
\includegraphics[height=.9\textwidth]{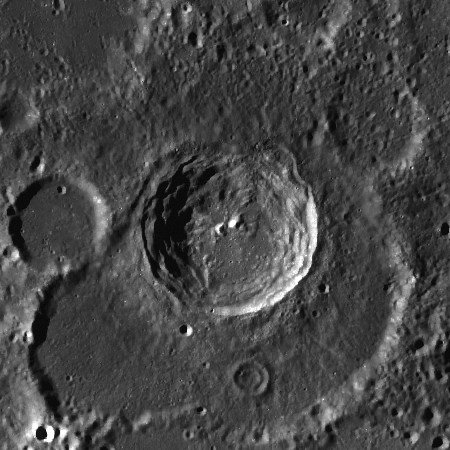}
\center{Hamilton}
\end{minipage}
\hfill
\begin{minipage}[b]{0.3\linewidth}
\includegraphics[height=.9\textwidth]{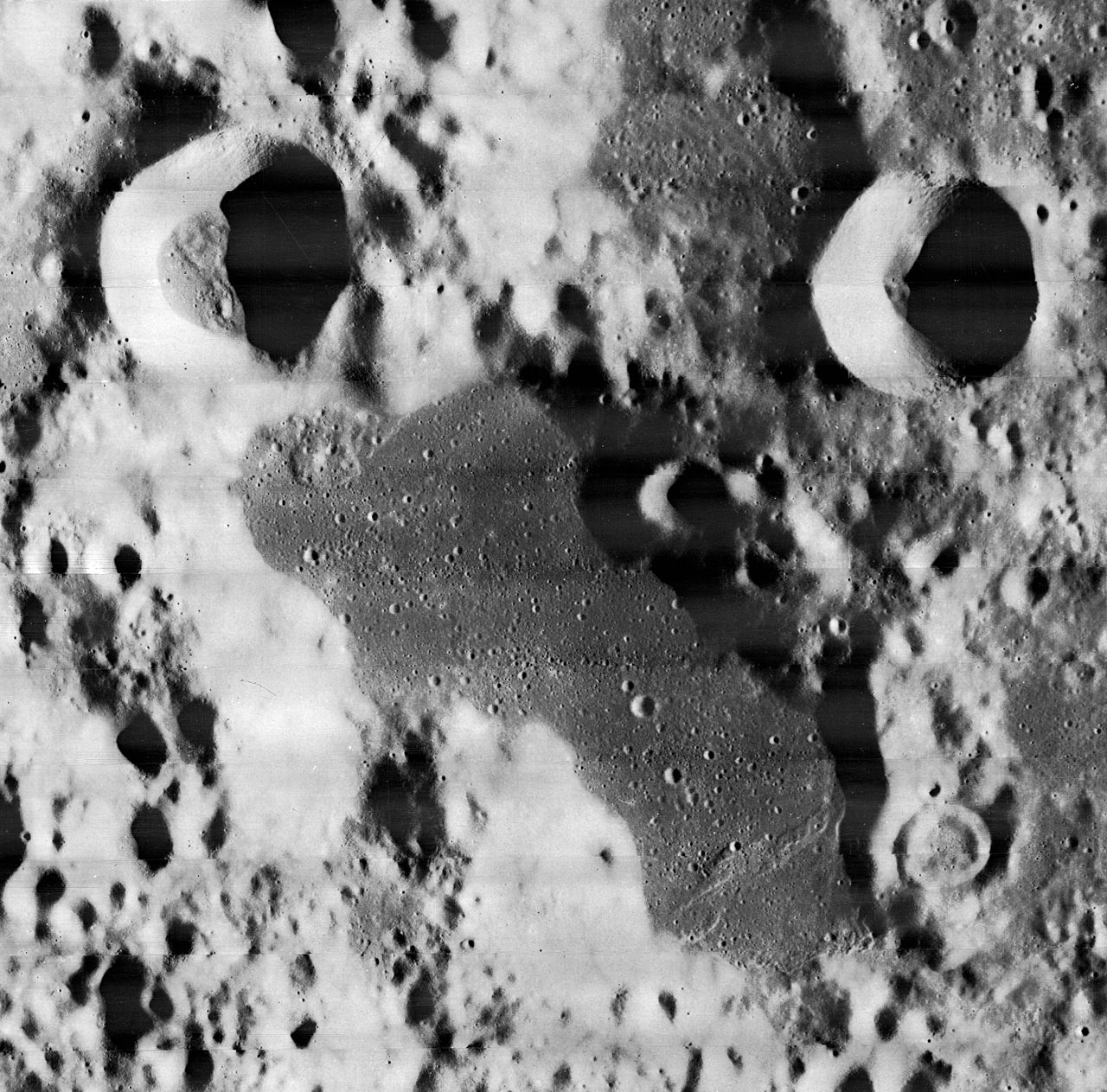}
\center{Liouville }
\end{minipage}
\hfill
\begin{minipage}[b]{0.3\linewidth}
\includegraphics[height=.9\textwidth]{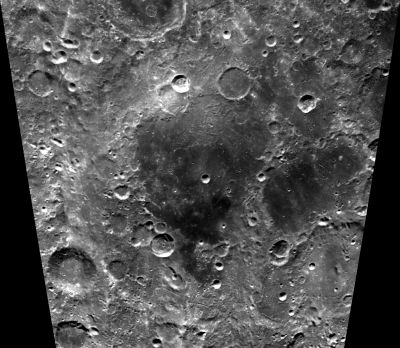}
\center{Poincar\'e }
\end{minipage}
\end{figure}

Towards the end of the nineteenth century, it was understood that the equations of motion of a mechanical system can be written in {\it canonical Hamilton form}:
$$
\frac{dp}{dt} = \frac{\partial H(p,q)}{\partial q} \quad ; \quad \frac{dq}{dt} = -\frac{\partial H(p,q)}{\partial p}
$$
where $q=(q_1,\dots, q_n)$ and $p=(p_1,\dots,p_n)$ denote the positions and momenta and $H$ is the hamiltonian function defined on $\RR^{2n}$.
This ordinary differential equation generates of flow $\phi^t$ of diffeomorphisms of ${\RR}^{2n}$ describing the motion.
Even though one should probably not attribute a full credit to Poincar\'e, it seems to us that his famous memoir on the three body problem contains the first explicit formulation of the basic properties of symplectic geometry :
\begin{itemize}
\item The flow $\phi^t$ preserves the closed form $\Omega=\sum\limits_{i=1}^{n} dp_i\wedge dq_i$.
\item It $f$ is any diffeomorphism of $\RR^{2n}$ preserving $\Omega$ (a canonical transformation) then the flow associated to $H\circ f$ is $f \circ \phi^t \circ f^{-1}$.
\end{itemize}
Using this, Poincar\'e founded the qualitative theory of symplectic (or hamiltonian) dynamics.
For instance, he used the preservation of the volume form $\Omega^n$ and his {\it recurrence theorem} to prove some weak stability (``\`a la Poisson'') of the solar system.
He also began a thorough analysis of periodic orbits.
Taking the differential of $\phi^t$ at a fixed point, one gets what is called today a symplectic matrix and any conjugacy invariant of this matrix (for instance its eigenvalues), gives some information on the dynamics.
In particular, Poincar\'e knew that if $\lambda$ is an eigenvalue of such a matrix, so are $\lambda^{-1}, \overline{\lambda}, \overline{\lambda}^{-1}$.
He also used in many circumstances what we call today ``symplectic manifolds''.

The systematic study of the symplectic group, as a real algebraic group, began later.
The terminology ``symplectic'' is due to Hermann Weyl in his 1939 book~\cite{weylclass} on the classical groups.
\begin{quote}
{\it The name ``complex group'' formerly advocated by me in allusion to line complexes, as they are defined by the vanishing of antisymmetric forms, has become more and more embarrassing through collision with the word ``complex'' in the connotation of complex number. I therefore propose to replace it by the corresponding Greek adjective ``symplectic''. Dickson calls the group ``Abelian linear group'' in homage to Abel who first studied it. }
\end{quote}
The terminology ``symplectic'' has survived and fortunately the word ``Abelian'' is not used any more in this context since this would have created a big confusion!

The {\it entr\'ee en sc\`ene} of the symplectic group in topology came even later.
Of course, the same Poincar\'e had introduced  homology and intersections of cycles in manifolds and ``proved'' his duality theorem but he does not seem to have used the symplectic property in any way.
Recall that for Poincar\'e homology was expressed as Betti numbers rather than as groups. The  homology groups were introduced by Emmy Noether.

During the twentieth century, there was an incredible blooming of symplectic dynamics and topology, in particular thanks to the fundamental works of Arnold and the introduction by Gromov of the techniques of pseudo-holomorphic curves - see McDuff and Salamon~\cite{mcduffsalamon}.

The name of Arnold was given to the asteroid 10031 Vladarnolda, $2.5891829$ astronomical units from the Sun,  and whose orbit has a pretty large eccentricity  $0.1991397$.

\begin{center}
\includegraphics[width=.65\linewidth]{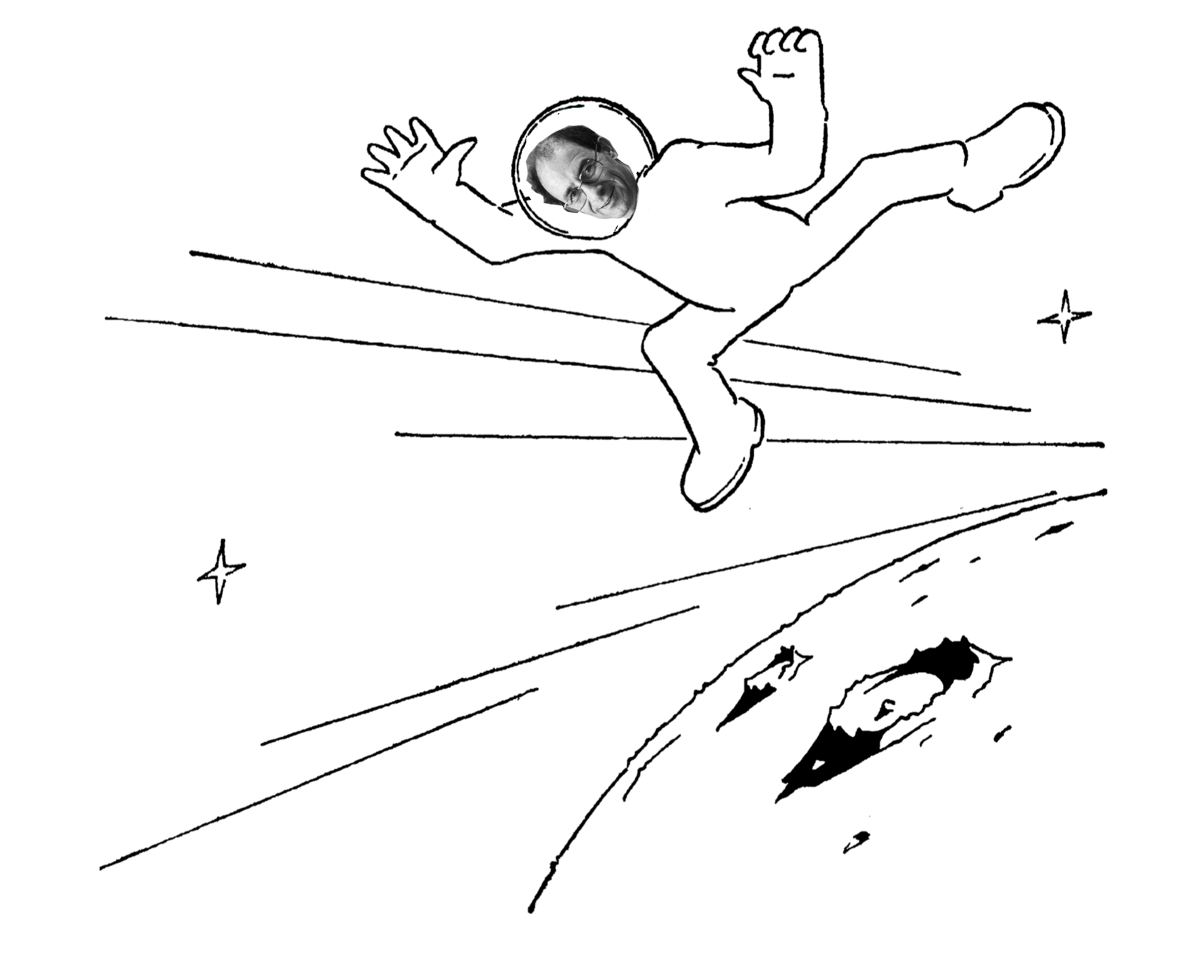}

\text{Vladimir Arnold, orbiting ``his'' asteroid:}

\text{a pastiche of Beletsky~\cite[p.254]{beletsky}}
\end{center}

\subsection{The Maslov class}

The Maslov class appears in different guises in many different parts of mathematics.

\begin{itemize}

\item In mathematical physics, as originally introduced by Maslov in~\cite{maslov}.

\item In hamiltonian dynamics, for example though the introduction of the Conley-Zehnder index for periodic orbits.

\item In topology where it is closely related to the signature of $4k$-dimensional manifolds, the basic invariant of surgery theory.

\item In algebra, specifically in the study of Witt groups, and their
non-simply-connected generalizations the Wall surgery obstruction groups.

\item In number theory, since it is related to some automorphic forms.

\end{itemize}

There are excellent survey papers on the Maslov class, for example~Cappell, Lee and Miller~\cite{cappellleemiller}, de Gosson~\cite{gosson} and Py~\cite{py}, providing a unified presentation of the many facets of this object.
The reader will find a rather complete bibliography at the website~\cite{ranickimaslov}.

Here, we shall limit  ourselves to a small selection, relevant to the other papers in this volume.
We can claim neither for originality nor for novelty and the best we can do is to quote the paper by Arnold on this topic~\cite{arnold1}:

\begin{quote}
\emph{In such a classical area as Sturm's theory it is hard to follow all the predecessors, and I can only say, like Bott and Edwards, that I do not make any claim as to the novelty of the results.
In connection with this I remark that numerous authors writing on the Maslov index, symplectic geometry, geometric quantization, Lagrangian analysis, etc., starting with [2], have not noticed the earlier works of Lidskii, as well as the earlier works of Bott [3] and Edwards [4], in which was constructed a Hermitian version of the theory of the Maslov index and Sturm intersections.}
\end{quote}

Equip $\RR^{2n}$ with the standard symplectic form
$$\Omega((x_i),(y_i))= \sum_{i=1}^n (x_i y_{n+i}-x_{i+n}y_i) \in \RR.$$

A {\it  lagrangian} is an $n$-dimensional subspace of $\RR^n$ on which $\Omega$ vanishes identically.
This terminology was introduced by Arnol’d~\cite{arnold1}  after Maslov had defined in 1965 a ``Lagrange manifold'' as an $n$-dimensional submanifold of $\RR^{2n}$ in which all tangent spaces are lagrangian subspaces.
Weinstein develops in~\cite{weinstein} the credo that
\begin{quote}
{\it Everything is a Lagrange manifold.}
\end{quote}

 The space of lagrangian subspaces of $\RR^{2n}$ is a compact manifold traditionally denoted by $\Lambda_n$.
 Of course, the symplectic group ${\rm Sp}(2n,\RR)$ acts on $\Lambda_n$ and it is not hard to check that this action is transitive, so that $\Lambda_n$ is a homogeneous space.

 The unitary group ${\rm U}(n)$ consists of complex unitary automorphisms of $\CC^n$ equipped with the standard hermitian form $\sum\limits_{j=1}^n z_j \overline{z'}_j$.
 In the identification between $\CC^n$ and $\RR^{2n}$, the imaginary part of the hermitian form is the symplectic form so that the unitary group ${\rm U}(n)$ is a (compact) subgroup of ${\rm Sp}(2n,\RR)$.
 It turns out that ${\rm U}(n)$ is indeed a maximal compact subgroup, and that $\Lambda_n$ is also homogenous under ${\rm U}(n)$ and is isomorphic to the quotient ${\rm U}(n)/{\rm O}(n)$.

 The homogenous space ${\mathcal H}_n={\rm Sp}(2n,\RR)/{\rm U}(n)$ is called the {\it Siegel domain} and carries very interesting geometry.
 This is a symmetric space with non positive curvature.
 This is also a {\it complex holomorphic} domain, that one can also see as the set of complex symmetric $n\times n$ matrices chose imaginary part is positive definite.
 Siegel devoted a long paper to this geometry in 1943~\cite{siegel}.
 His motivation came from number theory and the generalization of modular forms.

 It is always a good idea to look at the case $n=1$ where ${\rm Sp}(2,\RR)$ becomes ${\rm SL}(2,\RR)$, the lagrangian Grassmannian $\Lambda_1$ is the real projective line (diffeomorphic to a circle),
 and the Siegel space is the Poincar\'e upper half space.

 The Maslov class may be seen in many related ways.

 \begin{itemize}
 \item An element of $H^1(\Lambda_N,\ZZ)$,
 \item A cohomology class of degree 2 of ${\rm Sp}(2n,\RR)$ defining its universal cover.
 \item A two cocycle on $\Lambda_n$ invariant under the action of  ${\rm Sp}(2n,\RR)$.
 \item  A two cocycle on ${\mathcal H}_n$ invariant under the action of  ${\rm Sp}(2n,\RR)$.
 \end{itemize}

 We shall now describe these objects.

 \subsection{A short recollection of group cohomology}

Let $\Gamma$ be any group, equipped with the discrete topology.
On can consider the simplicial complex whose vertices are the elements of $\Gamma$ and having a $k$-simplex for every $(k+1)$-tuple of elements of $\Gamma$.
This space is obviously contractible and carries an obvious free action by left translation of $\Gamma$ on itself.
Clearly, the quotient space is a model for the Eilenberg MacLane space $K(\Gamma, 1)$.
The cohomology of this quotient space with coefficients in some abelian group $A$  is called the {\it group cohomology} of $\Gamma$ and denoted by $H^{\star}(\Gamma, A)$.
It is computed from the differential graded algebra consisting of cochains $c: \Gamma^{k+1} \to A$ which are invariant under translations :
$$
c(\gamma \gamma_0,\dots, \gamma \gamma_k)= c(\gamma_0, \dots, \gamma_k).
$$
 These cochains are called {\it homogeneous} for obvious reasons.
 Every homogeneous cochain $c$ defines a non-homogeneous cochain $\overline{c} :\Gamma^k \to A$ by
 $$
 \overline{c}(g_1,g_2,\dots,g_k)= c(1,g_1,g_1g_2, \dots, g_1g_2\dots , g_k)
 $$
 and conversely any non-homogeneous cochain defines a unique homogeneous cochain.
 Therefore, one can compute the cohomology of a group using non-homogeneous cochains.

 For instance, in the non-homogeneous presentation, a 1-cochain is just a map $f : \Gamma \to A$ and its coboundary is the map
 $$
(g_1,g_2) \in \Gamma^2 \mapsto df(g_1,g_2) = f(g_1g_2)-f(g_1)-f(g_2).
 $$
 Cohomology is typically used to describe central extensions
$$
\xymatrix{\{0\}  \ar[r] &A  \ar[r]^-{i} & \widetilde{\Gamma}  \ar[r]^-{\pi}& \Gamma \ar[r] &\{1\}}.$$
One chooses any set theoretic section
$s$ from $\Gamma$ to $ \widetilde{\Gamma}$ (i.e. such that $\pi \circ s = Id$).
Then, given two elements $g_1,g_2$ in $\Gamma$, the element $s(g_1g_2)s(g_2)^{-1}s(g_1)^{-1}$ projects by $\pi$ to the identity so that it is an element of the kernel of $i$
and can be identified with some element of $A$.
In this way, we get a function $c$ depending on $g_1,g_2$ with  values in the abelian group $A$.
This turns out to be a cocycle and changing the section $s$ changes $c$ by a coboundary so that the cohomology class of $c$ is well defined.
This is a class in $H^2(\Gamma,A)$ associated to the central extension.
Conversely, a cohomology class of degree 2 defines a central extension.

For an excellent introduction to group cohomology, we recommend~\cite{brown}.

 \subsection{The topology of $\Lambda_n$}

 Since the determinant of a matrix in ${\rm O}(n)$ is $\pm 1$ and the determinant of a unitary matrix has modulus 1, the homogeneous space $\Lambda_n= {\rm U}(n)/{\rm O}(n)$ maps to the circle:
 $$p:~{\rm U}(n)/{\rm O}(n) \to \SSS^1;~A \mapsto \det (A^2).$$
 This is a locally trivial fibration with simply connected fibres ${\rm SU}(n)/{\rm SO}(n)$.
 As a consequence, $H^1(\Lambda_n;\ZZ)$  is isomorphic to $\ZZ$.

 As a first definition, the {\it Maslov class} is the corresponding generator of $H^1(\Lambda_n,\ZZ)$.
 Any loop in $\Lambda_n$ has a {\it Maslov index}, the index of its projection by $p$.

 Note that the symplectic group ${\rm Sp}(2n,\RR)$ acts canonically on $\Lambda_n$ but this action does not preserve the fibres of $p$.
 However, we shall see later that the action ``almost preserves'' the fibres.

 Since ${\rm U}(n)$ is a maximal compact subgroup of ${\rm Sp}(2n,\RR)$, the quotient  ${\mathcal H}_n={\rm Sp}(2n,\RR)/{\rm U}(n)$ is contractible.
 It follows that the fundamental group of ${\rm Sp}(2n,\RR)$ is also isomorphic to $\ZZ$ so that its universal cover $\widetilde{{\rm Sp}}(2n,\RR)$ is a central extension
 $$\xymatrix{\{0\} \ar[r] & \ZZ \ar[r]^-{i} & \widetilde{{\rm Sp}}(2n,\RR)      \ar[r]^-{\pi}& {\rm Sp}(2n,\RR) \ar[r] &\{1\}}.$$
See Rawnsley~\cite{rawnsley} for an explicit description of $\widetilde{{\rm Sp}}(2n,\RR)$.

The {\it Maslov class} is the cohomology class in $H^2({\rm Sp}(2n,\RR),\ZZ)$ (now using the discrete topology on ${\rm Sp}(2n,\RR)$), corresponding to this central extension.

Another (equivalent) way to understand this class is to use the fact (proved by Hopf) that any homology class of degree 2 is represented by a surface.
Let $\phi: \pi_1(\Sigma_g) \to {\rm Sp}(2n,\RR)$ be some homomorphism from the fundamental group of the compact oriented surface $\Sigma_g$ of genus $g$ to the symplectic group.
The group $\pi_1(\Sigma_g)$ has a well known presentation
$$\langle a_1, \dots, a_g, \dots , b_1, \dots b_g \,  \vert \, [a_1,b_1]\dots [a_g,b_g]=1 \rangle.$$
Choose elements $\widetilde{\phi(a_i)},\widetilde{\phi(b_i)}$ in $\widetilde{{\rm Sp}}(2n,\RR)$ lifting $\phi(a_i),\phi(b_i)$.
The product of commutators $[\widetilde{\phi(a_1)}, \widetilde{\phi(b_1)}]\dots [\widetilde{\phi(a_g}),\widetilde{\phi(b_g)}]$ is an integer independent of all the choices.
This is the evaluation of the Maslov class on the homology class $\phi\in H_2({\rm Sp}(2n,\RR))$.

\subsection{Maslov indices and the universal cover of $\Lambda_n$}

Let us begin with the simple but crucial observation of Leray~\cite{leray}.
Suppose $L_1,L_2$ are two transverse lagrangian subspaces of $\RR^{2n}$.
The pairing $\Omega : L_1 \times L_2 \to \RR$ is nonsingular so that $L_2$ is identified with the dual $L_1^{\star}$ of $L_1$.
In this way, the space $\RR^{2n}$ is identified with $L_1 \oplus L_1^{\star}$ and under this identification the symplectic structure on
$L_1 \oplus L_1^{\star}$ is simply given by:
$$
\Omega((x,x^{\star}),(y,y^{\star}))=y^{\star}(x)-x^{\star}(y) \in \RR.
$$
Suppose now that we are given a third lagrangian $L_3$ which is transverse to $L_2$ so that it is the graph of some linear map $l$ from $L_1$ to $L_1^{\star}$.
The fact that $L_3$ is lagrangian means that $l$ is a symmetric form, i.e. $l(x)(y)-l(y)(x) \equiv 0$.
In other words, {\it if $L_1$ and $L_3$ are transverse to $L_2$, the lagrangian space $L_1$ is canonically equipped with a quadratic form.}
One defines the {\it Wall-Maslov ternary index} $W(L_1,L_2,L_3)$ as the signature of this quadratic form.

If turns out that this definition can be generalized without any transversality conditions.
One of the possibilities is the following.
As in  Definition~\ref{tri1} 2., given three lagrangians $L_1,L_2,L_3$
in $H_-(\RR^n)=(\RR^{2n},\Omega)$ one considers the vector space
$$V=\{(v_1,v_2,v_3) \in L_1 \oplus L_2 \oplus L_3\,\vert\,
v_1+v_2+v_3=0 \in \RR^{2n}\},$$
and the canonical quadratic form on $V$ defined by
$$
Q_{L_1,L_2,L_3}(v_1,v_2,v_3) = \Omega(v_1,v_2)=\Omega(v_2,v_3)=\Omega(v_3,v_1).
$$
One defines the {\it Wall-Maslov triple index} in general as the signature of this quadratic form.

So far, we worked with the symplectic group over the real numbers.
This construction can however be generalized over any field $K$.
The symplectic group ${\rm Sp}(2n,K)$ now acts on the Grassmannian  $\Lambda_n(K)$ of lagrangians in $H_-(K^n)$ and we can associate a quadratic form to three lagrangians, exactly as above.
The Wall-Maslov index of the triformation $(L_1,L_2,L_3)$ is now a class of in the symmetric Witt group $W(K)$.

Let us give the main properties of this index:

\begin{itemize}

\item The Wall-Maslov ternary index $W: \Lambda_n(K)^3 \to W(K)$ is invariant under the action of ${\rm Sp}(2n,K)$.

\item $W$ is an alternating function of its three arguments $(L_1,L_2,L_3)$.

\item $W$ is a {\it cocycle} : for any four lagrangians
$L_1,L_2,L_3,L_4$, one has $$W(L_2,L_3,L_4)-W(L_1,L_3,L_4)+W(L_1,L_2,L_4)-W(L_1,L_2,L_3)=0$$ in $W(K)$.

\end{itemize}

Only the third property is a tricky exercise in linear algebra.
See~Py~\cite{py} for a short proof.

Let us come back to the case of real numbers, which is our main example.
Since $H^1(\Lambda_n; \ZZ)$ is isomorphic to $\ZZ$, this defines an infinite cyclic cover $\pi: \widetilde{\Lambda}_n \to \Lambda_n$.
Let us denote by $T : \widetilde{\Lambda}_n \to \widetilde{\Lambda}_n$ the generator of the deck transformations.
The fibration $p : \Lambda_n \to \SSS^1$ lifts to a fibration  $\widetilde{p} : \widetilde{\Lambda}_n \to \RR$ such that $\widetilde{p} (T (\widetilde{L}))=\widetilde{p}(\widetilde{L}) +1$.
Of course, the group $\widetilde{{\rm Sp}}(2n,\RR)$ acts on $\widetilde{\Lambda}_n$.

Several important properties should be mentioned.

The first concerns the topology of $\Lambda_n$.
If one fixes a lagrangian $L$, one can consider the space $\Lambda_{n,\pitchfork L}$ consisting of lagrangians transverse to $L$.
As noticed above, this space is identified with the space of quadratic forms on any lagrangian supplementary space of $L$.
More precisely, $\Lambda_{n,\pitchfork L}$ is an affine space modeled on the space of symmetric matrices : the ``difference'' between two elements of $\Lambda_{n,\pitchfork L}$ being a  quadratic form.
This provides an atlas for $\Lambda_n$.
The complement of $\Lambda_{n,\pitchfork L}$ in $\Lambda_n$ is called by Arnold the {\it train} of $L$.

The 3-cocycle $W$ is a coboundary when lifted to $\widetilde{\Lambda}_n$.
This means that there is a function $m : \widetilde{\Lambda}_n^2 \to \ZZ$ such that for any three elements of $\widetilde{\Lambda}_n$, we have:
$$
m(\widetilde{L}_1,\widetilde{L}_2)+m(\widetilde{L}_2,\widetilde{L}_3)+m(\widetilde{L}_3,\widetilde{L}_1) = W( \pi(\widetilde{L}_1),\pi(\widetilde{L}_2), \pi(\widetilde{L}_3)).
$$
Moreover, this binary index is unique if one imposes the conditions
$$m(\widetilde{L}, T(\widetilde{L}))=n$$
and the invariance under the action of $\widetilde{{\rm Sp}}(2n,\RR)$.
The existence and uniqueness of this function $m$ is explained in a crystal clear way in~ Arnold~\cite{arnold1}.
Suppose $L_t$ is a continuous path of lagrangians {\it which are all transversal to the same} $L$.
Then the ``difference'' of $L_1$ and $L_2$ define a quadratic form having some signature.
One then lifts the path to a path $\widetilde{L}_t$ in  $\widetilde{\Lambda}_n$ and one defines $m(\widetilde{L}_1,\widetilde{L}_0)$ as being this signature.
One can check that this is independent of the path connecting these two lagrangians.
This only defines $m$ in the very special case of two elements of $\widetilde{\Lambda}_n$ which are connected by a path everywhere transverse to a given lagrangian.
One can then  extend the definition of $m$ to all pairs of elements of $\widetilde{\Lambda}_n$ by using the required property that $M$ is the coboundary of $m$.

The second property is that $\widetilde{\Lambda}_n$ is canonically equipped with a {\it partial ordering} which is important in hamiltonian dynamics.
Any $\Lambda_{n,\pitchfork L}$, being contractible, can be lifted to countably many contractible sets in $\widetilde{\Lambda}_n$.
In each if these contractible sets, on defines $\widetilde{L_1} \leqslant \widetilde{L}_2$ if the difference is a non negative quadratic form.
One should check that these definitions are indeed compatible in the intersections of the $\Lambda_{n,\pitchfork L}$.
Then one defines in general $\widetilde{L_1} \leqslant \widetilde{L}_2$  if there is a chain of elements of elements of $\widetilde{\Lambda}_n$ connecting them,
so that two consecutive elements of the chain belong to one of these contractible parts, and forming an ascending chain for the partially defined ordering.
These verifications are not difficult.

\subsection{The Maslov class as a bounded cohomology class}\label{central}

One of the key properties of the Maslov class is that it is {\it bounded}.
From the na\"{\i}ve point of view, this simply means that the Maslov triple index takes values in a finite set $\{-n,\dots, n \}$.

Given a space $X$, its singular homology is defined using singular simplices, i.e. continuous maps $\tau$ from the standard simplex $\Delta_k$ in $X$.
The space $C_k(X)$ of singular chains is by definition the vector space generated by $k$ simplices.
One can equip this space with the $\ell ^1$-norm, i.e. the sum of the absolute values of the coefficients.
Singular cochains are defined as the duals of singular chains.
A cochain is called bounded if it is a continuous linear form, with respect of the $\ell ^1$-norm.
More concretely, a cochain $c$ is bounded if there is a constant $C$ such that the evaluation $c(\tau)$ on any $k$-simplex has absolute value less that $C$.
The vector space of bounded cochains is obviously a graded differential complex.
Its cohomology is called the {\it bounded cohomology} of $X$ and denoted by $H^{\star}_b(X)$.
This  cohomology was introduced by Gromov~\cite{gromov}.

As a typical example, consider a compact oriented surface $\Sigma$ of genus $g\geqslant  2$ so that it can be equipped with a riemannian metric of curvature $-1$ and that one can identify its universal cover with Poincar\'e disc.
Consider a $2$-simplex $\tau : \Delta_2 \to \Sigma$.
Lift it to the Poincar\'e disc as a simplex $\widetilde{\tau}$ and define $c(\tau)$ as the (oriented) area of the geodesic triangle having the same vertices as $\widetilde{\tau}$.
One checks easily that this is defines a cocycle $c$ which is bounded since the area of geodesic triangles in the Poincar\'e disc is less than $\pi$ by Gauss-Bonnet theorem.
The cohomology class of this $2$-cocycle in the usual cohomology of $\Sigma$ is easy to compute : this is the area of $\Sigma$.
This example is somehow typical and bounded cohomology is vaguely related to negative curvature.

Given a space $X$, there is a classifying map to the Eilenberg MacLane space $K(\pi_1(X),1)$.
Gromov showed that this map induces an isomorphism in bounded cohomology.
In particular, {\it simply connected spaces have a trivial bounded cohomology}.

Therefore, all information concerning bounded cohomology comes from the fundamental group and should be computed from the group theoretical bounded cohomology.
Let $\Gamma$ be any group, equipped with the discrete topology.
One considers the differential graded algebra consisting of {\it bounded} homogeneous cochains $c: \Gamma^{k+1} \to \RR$.
Its homology is the {\it bounded cohomology} of the group $\Gamma$ and denoted by $H^{\star}_b(\Gamma, \RR)$.

Let us state some of the main properties of this cohomology, especially in degree 2.
See~Grigorchuk~\cite{grigorchuk} for a survey of the second bounded cohomology of a group.

If $\Gamma$ is amenable (for example if it is solvable) all bounded cohomology groups vanish.

If a group $\Gamma$ is {\it uniformly perfect}, i.e. if there is an integer $k$ such that any element is a product of at most $k$ commutators, then the canonical map from $H^2_b(\Gamma,\RR)$ to $H^2(\Gamma, \RR)$ is injective.
Indeed, suppose a bounded 2-cocycle is the coboundary of some a priori unbounded 1-cochain $f : \Gamma \to \RR$.
This means that there is a constant $C>0$ such that for every $g_1,g_2$ one has
$$
\vert f(g_1g_2)- f(g_1) - f(g_2) \vert \leqslant C.
$$
Such maps $f$ are called {\it quasi homomorphisms}.
It is easy to see that a quasi-homomorphisms is uniformly bounded on the set of products of a given number of commutators.
Hence, the assumption that $\Gamma$ is uniquely perfect implies that $f$ is bounded, which means precisely that the mapping $H^2_b(\Gamma,\RR)$ to $H^(\Gamma, \RR)$ is injective.

In general, the kernel of $H^2_b(\Gamma,\RR)$ to $H^2(\Gamma, \RR)$ is described by quasi-homomorphisms {\it up to} bounded functions.
There is an elementary way to get rid of these bounded functions using ``homogenization''.
Starting from a quasi homomorphism $f$, one defines its homogenization by
$$
\overline{f} (g) = \lim_{k\to \infty} \frac{f(g^k)}{k}
$$
which is another quasi-homomorphism such that $\vert \overline{f} - f \vert$ is bounded.
In this way, one sees that the kernel under consideration is identified with {\it homogeneous} quasi-homomorphisms, i.e. satisfying the additional condition that $\overline{f}(g^k)=k \overline{f}(g)$.

It turns out that the symplectic group ${\rm Sp}(2n,\RR)$ is uniformly perfect.

The Maslov class is a typical example of a bounded class.

Chose a lagrangian $L$ in $\Lambda_n$.
If $g_0,g_1,g_2$ are three elements of ${\rm Sp}(2n,\RR)$, the ternary index $W(g_0(L),g_1(L),g_2(L))$ defines a bounded 2-cocycle since its values are bounded by $2n$.
This is the {\it bounded Maslov class}.

Clearly, the pull back this cocycle to $\widetilde{{\rm Sp}}(2n,\ZZ)$ is exact and is the coboundary of the $1$-cochain associating to $(\widetilde{g}_0,\widetilde{g}_1)$ the {\it binary} index $m(\widetilde{g}_0(\widetilde{L}),\widetilde{g}_1(\widetilde{L}))$. The Maslov class in $H^2({\rm Sp}(2n,\RR),\ZZ)$ is associated to the central extension
$$\xymatrix{\{0\} \ar[r] & \ZZ \ar[r]^-{i} & \widetilde{{\rm Sp}}(2n,\RR)      \ar[r]^-{\pi}& {\rm Sp}(2n,\RR) \ar[r] &\{1\}}.$$

Of course, the above definition of the Maslov cocycle depends on the choice of a lagrangian $L$, as a base point in $\Lambda_n$,  but the class does not.
 In order to define a canonical cocycle, one can proceed in two ways, leading to two interesting concepts.

In a first approach, one can pull back some bounded Maslov cocycle, associated to some choice of $L$, to $\widetilde{{\rm Sp}}(2n,\RR)$ where, as we noticed already, it is exact.
Therefore, there is a map $r: \widetilde{{\rm Sp}}(2n,\RR) \to \RR $ such that
\begin{itemize}
\item $t(\widetilde{g}_1\widetilde{g_2})^{-1}) t(\widetilde{g}_1)t(\widetilde{g}_2)=  W(L,\pi(\widetilde{g}_1(L)),\pi(\widetilde{g}_2(L)))$
\item $t( T \widetilde{g}) = t(\widetilde{g}) +1$.
\end{itemize}

This function $t$ is of course a quasi-homomorphism which depends on $L$ but its homogenization $\rho$ does not.
The function $trans : \widetilde{{\rm Sp}}(2n,\RR) \to \RR $ is called the {\it symplectic translation number}.
This is the {\it unique homogeneous quasi-homomorphism $trans : \widetilde{{\rm Sp}}(2n,\RR) \to \RR $ taking the value 1 on the generator $T$ of the center}.
Modulo $\ZZ$, one gets the {\it rotation number} $\rho$ on the symplectic group, with values in $\RR / \ZZ$, useful in dynamics.

In a second approach, one can double the dimension of $\RR^{2n}$ and consider $\RR^{2n}\oplus \RR^{2n}$ with the symplectic form $\Omega \oplus - \Omega$.
The graph of a symplectic automorphism $g:H_-(\RR^{2n}) \to H_-(\RR^{2n})$ is a lagrangian $$\Gamma_g=\{(x,g(x))\,\vert\, x \in \RR^{2n}\}.$$
Three elements $g_0,g_1,g_2$  of ${\rm Sp}(2n,\RR)$ define three lagrangians in $\RR^{2n}\oplus \RR^{2n}$,
for which one can compute the Wall-Maslov ternary index.
This defines a  Maslov $2$-cocycle on the symplectic group which is canonical in the sense that it is invariant under conjugacy and does not depend on the choice of a lagrangian.

Finally, we mention another geometrical interpretation of the Maslov class.
The Siegel domain ${\mathcal H}_n={\rm Sp}(2n,\RR)/{\rm U}(n)$ is a K\"ahler manifold whose curvature is non positive.
It can be compactified in an equivariant way by adding a sphere at infinity.
Contrary to the case of the Poincar\'e disc (when $n=1$), the action of ${\rm Sp}(2n,\RR)$ on this sphere is not transitive.
The structure of this heterogeneous sphere is well understood and is described by a Tits building, that we cannot describe here.
It suffices to say that one piece of this building is identified with $\Lambda_n$ which is therefore equivariantly embedded in the sphere at infinity.

Now consider three points $x,y,z$ in ${\mathcal H}_n$. One can connect them by three geodesic arcs $[x,y],[y,z],[z,x]$ to build a triangle.
However, there is no totally geodesic plane containing the three points so that one cannot construct a usual triangle having $x,y,z$ as vertices.
Nevertheless, one can choose any smooth triangle having the union of $[x,y],[y,z],[z,x]$ as its boundary and integrate the K\"ahler form on this triangle.
Since the form is closed, this ``area'' does not depend on the choice of the triangle so that one can indeed define some number $area(x,y,z) \in \RR$.
This is a bounded cocycle.
Choosing some base point $\star$ in ${\mathcal H}_n$ one gets 2-cocycles in ${\rm Sp}(2n,\RR)$ as $area(\star, g_1(\star), g_2(\star))$.
Changing the base point does not change the cohomology class, as one checks easily.
When $\star$ goes to the sphere at infinity and converges to some lagrangian, the corresponding cocycle converges to the previously defined Maslov cocycle.
In other words, the Maslov class can also be seen as some incarnation of the symplectic form of Siegel space.

Suppose now that we restrict our study to {\it integral} symplectic matrices in ${\rm Sp}(2n,\ZZ)$.
Realizing $g_0^{-1}g_1$, $g_0^{-1}g_1 \in {\rm Sp}(2n,\ZZ)$
by automorphisms $\alpha,\beta:\Sigma_n \to \Sigma_n$ there is defined a 4-dimensional manifold with boundary $(T(\alpha,\beta),\partial T(\alpha,\beta))$ such that
\begin{enumerate}
\item  the mapping torus
$$T(\alpha)=\Sigma_n \times [0,1]/\{(x,0) \sim (\alpha(x),1)\,\vert\, x \in \Sigma_n\}$$
is a fibre bundle $\Sigma_n \to T(\alpha) \to \SSS^1$ with monodromy $\alpha$,
\item the double mapping torus $T(\alpha,\beta)$ is a fibre bundle $\Sigma_n \to T(\alpha,\beta) \to P$
over the 2-dimensional pair of pants $P$ with $\partial P=\SSS^1 \sqcup \SSS^1 \sqcup \SSS^1$, such that $\partial T(\alpha,\beta)=T(\alpha) \sqcup T(\beta) \sqcup T(\alpha\circ \beta)$.
\end{enumerate}
The signature of $T(\alpha,\beta)$  is the Wall non-additivity invariant (\ref{tri1})
$$\begin{array}{ll}
{\rm Meyer}(g_0,g_1,g_2)&=~\tau(T(\alpha,\beta))\\
&=~W(H_1(\Sigma_n;\RR),\phi_{\Sigma_n};\RR^n \oplus 0,\alpha(\RR^n\oplus 0),\beta(\RR^n \oplus 0))\\
&=~W(H_-(\RR^{2n});\Gamma_{g_0},\Gamma_{g_1},\Gamma_{g_2}) \in \ZZ
\end{array}$$
which is called the Meyer cocycle on ${\rm Sp}(2n,\RR)$ (\cite{meyer}) and which is clearly independent of all choices - this is a very useful cocycle in topology.

\subsection{Picture in the case $n=2$}

The case $n=1$ is well known.
$\Lambda_1$ is the projective line over the reals, so that it is a circle.
The symplectic group in this case is ${\rm SL}(2,\RR)$ and acts on $\Lambda_1$ through its quotient ${\rm PSL}(2,\RR)$ in the well known way.
The Siegel domain ${\mathcal H}_1$ is the Poincar\'e disc.
Any element of ${\rm PSL}(2,\RR)$ can be seen as an orientation preserving homeomorphism of the circle, having a therefore a well defined rotation number in $\RR / \ZZ$.

It might be useful to draw a picture in the case $n=2$ which  clarify the situation.
We follow the remarkable paper by Arnold~\cite{arnold2}.

Let us  consider  $\Lambda_2= {\rm U}(2)/{\rm O}(2)$.
It fibres over the circle with fibres diffeomorphic to ${\rm SU}(2)/{\rm SO}(2)$, which is diffeomorphic to the 2-sphere.
The universal cover $\widetilde{\Lambda}_2$ is therefore diffeomorphic to ${\mathbb S}^2 \times \RR$.
The second coordinate is the lift of the map $\det ^2$.
It is not difficult to check that the generator $T$ of the deck transformations acts as $T(x,t)= (-x,t+1)$ so that one should be careful that $\Lambda_2$ is not orientable.
If one chooses some lagrangian $L$, its train, that is the set of lagrangians which are not transverse to $L$ is homeomorphic to a 2-sphere in which two antipodal points have been identified.
Lifting this train to $\widetilde{\Lambda}_2 = {\mathbb S}^2 \times \RR$, one gets a string of two spheres, glued along points.

The complement of this ``rosary'' consists of disjoint copies of $\Lambda_{2,\pitchfork L}$, each being identified to the space of quadratic forms in $\RR^2$, i.e. to the 3-space $\RR^3$.

The following picture shows $\widetilde{\Lambda}_2 = {\mathbb S}^2 \times \RR$. The grey part represents $\Lambda_{2,\pitchfork L}$. The three local pictures represent the situation in the neighborhood of the three singular points $(x,t), (-x,t+1),(x,t+2)$.

$$\includegraphics[width=\textwidth]{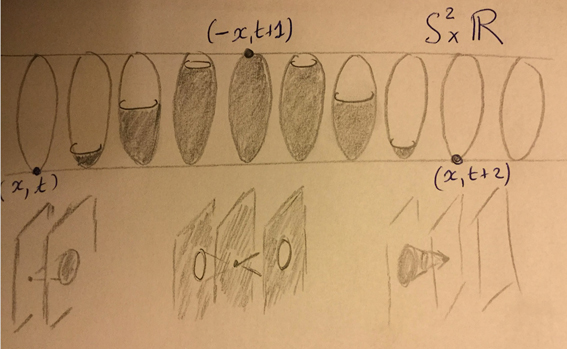}$$

The spheres ${\mathbb S}^2 \times \{t \}$ are definitely not invariant by the symplectic  group $\widetilde{{\rm Sp}}(4,\RR)$.
However, it follows from the boundedness of the Maslov class that if one looks at the image $\widetilde{g} ({\mathbb S}^2 \times \{t \})$ of such a sphere by an element of $\widetilde{{\rm Sp}}(4,\RR)$, its projection on the $\RR$ factor is an interval of bounded length.
In other words, even if the action of $\widetilde{{\rm Sp}}(4,\RR)$ does not descend to an action on $\RR$, it ``quasi-descends'' if one accepts to define a ``quasi-action'' in such a way that the image of a point might be an interval of bounded length.

Concerning the Siegel domain ${\mathcal H}_2$, its dimension is 6, so that its complex dimension is 3.
Its geometrical boundary is a 5-sphere ${\mathbb S}^5$ that we now describe in terms of the Tits building associated to this situation.

As mentioned earlier, ${\mathbb S}^5$ contains a copy of $\Lambda_2$.
It also contains a copy of $\mathbb{RP}^3$, the projective space of lines in $\RR^4$.
The sphere ${\mathbb S}^5$ is a ``restricted link'' of $\mathbb{RP}^3$ and $\Lambda_2$.
This means that every point on the sphere can be considered as a ``barycenter'' $\lambda D + (1-\lambda) L$
of some line $D$ and some lagrangian plane $L$ {\it containing} $D$ for some $\lambda \in [0,1]$.

\newpage

\section{Dynamics}\label{dynamics}

In this section, we briefly describe three topics from dynamical systems where some ``Maslov type'' ideas are relevant.

\subsection{The Calabi and Ruelle invariants}

Let us explain the construction of some invariants of symplectic diffeomorphisms in the simplest case.
Let $B$ be the unit open ball in $\RR^{2n}$ and denote by ${\rm Symp}_c(B)$ the group of symplectic diffeomorphisms of $B$ with compact support.

There are two basic dynamical invariants for elements of  ${\rm Symp}_c(B)$.

Choose some primitive for the symplectic form $\Omega$, for instance
$$\lambda = \sum p_i dq_i.$$
Let $\phi$ be an element of ${\rm Symp}_c(B)$ so that $d( \phi^{\star}(\lambda)-\lambda)=0$.
It follows that there is a function $H: B \to \RR$ such that
$$dH = \phi^{\star}(\lambda)-\lambda$$
and which is uniquely defined if one imposes the condition that $H$ is zero near the boundary of $B$.
One  defines the {\it Calabi invariant} of $\phi$ as the integral of $H$:
$${\mathfrak C}(\phi) = \int_B H\,  \Omega^n.$$
It is not hard to see that $\mathfrak C$ is a homomorphism from ${\rm Symp}_c(B)$ to $\RR$.
It is much harder to prove that the kernel of $\mathfrak C$ is a simple group, so that every non injective homomorphisms from ${\rm Symp}_c(B)$ to a group factors through $\mathfrak C$.
See~Banyaga~\cite{banyaga} for details.

Except in dimension 2, the dynamical meaning of this homomorphism is unclear.
In dimension 2, it is related to some kind of linking between orbits (see for instance~Gambaudo and Ghys~\cite{gambaudoghys1}).

The {\it Ruelle invariant} is related to the Maslov index (see Ruelle~\cite{ruelle}), and has been discussed in~Barge and Ghys~\cite{bargeghys}.
Let $\phi$ be an element of ${\rm Symp}_c(B)$ and consider its differential:
$$
d\phi : x \in B \mapsto d\phi (x) \in {\rm Sp}(2n,\RR).
$$
One can lift this to $\widetilde{{\rm Sp}}(2n,\RR)$ as a map
$$
\widetilde{d\phi} : x \in B \mapsto \widetilde{d\phi} (x) \in \widetilde{{\rm Sp}}(2n,\RR)
$$
which is the identity near the boundary of $B$. Therefore one can compute the symplectic translation number $t(\widetilde{d\phi}(x))$.
Its integral over $B$ defines a quasi-homomorphism on ${\rm Symp}_c(B)$ whose homogenization is called the \emph{Ruelle invariant} $\mathfrak R$:
$$
{\mathfrak R } :\phi \in {\rm Symp}_c(B) \mapsto  \lim_{k\to \infty}   \frac{1}{k} \int_B  t(\widetilde{d\phi^k}(x)) \, \Omega^n \in \RR.
$$
Somehow, one can describe this invariant as an ``elliptic Lyapunov exponent''.
Recall that Lyapunov exponents describe the exponential growth of a vector under iteration of a diffeomorphism.
In our case, the symplectic translation number only takes into account the eigenvalues of modulus one of the differential.

As a very simple example, one can consider the case $n=1$ and the hamiltonian function $H : B \to \RR$ which is of the form $h(\sqrt{p_1^2+q_1^2})$
for some real function $h: [0,1] \to \RR$ which is constant in a neighborhood of 0 and which is 0 in a neighborhood of 1.
Let $\phi$ be the time one of the hamiltonian flow associated to $H$.
One can easily compute the invariants of Calabi and Ruelle and see that they are different:
$$
{\mathfrak C} (\phi) = \int_0^1 2 \pi r h(r) \, dr \quad \text{and} \quad {\mathfrak R}(\phi)= -h(0).
$$
See~\cite{bargeghys,gambaudoghys1} for more information about these invariants.
Let us simply mention that similar invariants can also be defined on the groups of symplectomorphisms of more general symplectic manifolds.
For the state of the art on this topic, we recommend Borman and Zapolsky~\cite{bormanzapolsky}
or~Shelukhin~\cite{shelukhin}. Some of these invariants are related to Floer homology.

\subsection{The Conley-Zehnder index}

Recall that a {\it symplectic manifold} is an even dimensional manifold $M$ equipped with a closed non degenerate 2 form $\Omega$.
A function $H : M \to \RR $ on such a manifold defines a vector field $X_H$ called its symplectic gradient such that $\Omega(X_H, -)= dH (-1)$.
The flow generated by $X_H$ is the hamiltonian flow $\phi^t_H$ associated to $H$.
As a matter of fact, one usually considers hamiltonians $H_t$ depending periodically on a time $t$, i.e. from $M\times \RR / \ZZ$ to $\RR$.
One therefore has a vector field depending periodically on time generating some non autonomous flow $\phi^t_H$.
The symplectic diffeomorphism $\Phi=\phi^1_H$ is by definition a {\it hamiltonian diffeomorphism}.

As it is well known, at least since Poincar\'e, a great deal of information about the dynamics is contained in periodic orbits of $\Phi$ and it has been a constant problem to find tools proving the existence of these orbits. As Poincar\'e wrote in  the ``M\'ethodes Nouvelles de la M\'ecanique C\'eleste''~\cite[page 82]{poincare}:

\begin{quote}
{\it Il semble d'abord que [l'existence de solutions p\'eriodiques] ne puisse \^etre d'aucun int\'er\^et pour la pratique.
En effet, il y a une probabilit\'e nulle pour que les conditions initiales du mouvement soient pr\'ecis\'ement celles qui correspondent \`a une solution p\'eriodique.
Mais il peut arriver qu'elles en diff\`erent tr\`es peu, et cela a lieu justement dans les cas o\`u les m\'ethodes anciennes ne sont plus applicables.
On peut alors avec avantage prendre la solution p\'eriodique comme premi\`ere approximation [$\dots$]
D'ailleurs, ce qui rend ces solutions p\'eriodiques aussi pr\'ecieuses, c'est qu'elles sont, pour ainsi dire, la seule br\`eche par o\`u nous puissions essayer de p\'en\'etrer dans une place jusqu'ici r\'eput\'ee inabordable.}
\end{quote}

Assume for simplicity that the manifold $M$ is $\RR^{2n}$ and consider a periodic point $x$ of period 1, so that $\Phi(x)=x$.
Taking the differential along the orbit, one gets a path
$$
t\in [0,1] \mapsto \gamma(t)= d\phi^t_H (x) \in {\rm Sp}(2n,\RR).
$$
Even though the periodic orbit is a closed loop, the above path is of course not necessarily a loop.
We know how to associate a Maslov index to a loop in the symplectic group ${\rm Sp}(2n,\RR)$.
Conley and Zehnder defined an index, in the same spirit as Maslov for a path $\gamma(t)$ in the symplectic group such that $\gamma(0)=I_{2n}$ and $\gamma(1)$ does not have 1 as an eigenvalue (which is a generic condition). See for instance~\cite{gosson} for a definition.
In this way, every fixed point of a hamiltonian diffeomorphism has a {\it Conley Zehnder index}.
This can be generalized in many symplectic manifolds.

This index has been remarkably useful in the context of Floer homology, which is the most powerful tool to construct periodic orbits.
It is impossible to give any detail of this theory here and we refer to Audin and Damian~\cite{audindamian} for a good exposition.

Let $M$ be a closed symplectic manifold.
Assume for simplicity that $M$ is 2-connected.
Denote by ${\mathcal L} M$ the space of loops, i.e. the space of smooth maps from $\RR / \ZZ$ to $M$.
Let $H_t$ be a time periodic hamiltonian.

We define the {\it action functional} ${\mathcal A} : {\mathcal L} M \to \RR$ in the following way.
If $\gamma : \RR / \ZZ \to M$ is a loop, choose a disc $D$ with boundary $\gamma$ and let
$$
{\mathcal A} (\gamma) = -\int_D \Omega - \int_{\RR / \ZZ}{H_t(\gamma(t))}{dt}.
$$
It is easy to see that the critical points of this functional are precisely the 1-periodic points we are trying to detect.
It is therefore tempting to define some Morse type homology on this loop space, which is called the {\it Floer homology}.
One introduces a generic almost complex structure compatible with the symplectic form.
This enables us to define the gradient lines in the loop space.
Given two critical points, i.e. two periodic orbits $\gamma_+, \gamma_-$, one can look at the space of gradient lines $\gamma_{s\in \RR}$ with $\gamma_+$ such that $\gamma_s$ converges to $\gamma_+$ (resp. $\gamma_-$) when $s$ goes to $+\infty$ (resp. $- \infty$).
In classical Morse theory, the dimension of the intersections of stable and unstable manifolds of two critical points of generic Morse functions is completely determined by the difference of the indices of the critical points.
In this new context of Floer homology, indices are replaced by the Conley-Zehnder indices.
Again, we refer to~\cite{audindamian} for details.

\subsection{Knots and dynamics}

In this subsection, we describe some analogies between the topology of knots and links and the dynamics of volume preserving flows on 3-dimensional manifolds, especially on the $3$-sphere.
More details can be found in~Ghys~\cite{ghys}.

Recall that if $\mathfrak G$ is a Lie algebra, one can define a canonical {\it Killing} symmetric form by
$$\langle x,y \rangle = \text{trace} (ad(x) ad(y)).$$
If $\mathfrak G$ is a simple Lie algebra, this is nonsingular and this is the unique bilinear form which is invariant under the adjoint action (up to a constant factor).

In some cases, infinite dimensional Lie algebras possess such a Killing form.
The main example is the Lie algebra of divergence free vector fields on the 3-sphere, or more generally on a rational homology sphere $\Sigma$ of dimension 3.
Let $vol$ be a volume form on $\Sigma$ (of total mass 1) and consider the Lie algebra $\mathfrak{sdiff}(M)$ of divergence free vector fields $X$ on $M$, i.e. such that the Lie derivative $L_Xvol=0$.
This condition is equivalent to the fact that the 2-form $i_X vol$ is closed and hence exact, because of our assumptions. Let $\alpha_X$ be a choice of a primitive.
Then one can define
$$
\langle X_1, X_2 \rangle =~\int_\Sigma \alpha_{X_1} \wedge d\, \alpha_{X_2} .
$$
This is a symmetric nonsingular bilinear form on $\mathfrak{sdiff}(\Sigma)$ which is of course invariant under conjugacy.
In particular, $\langle X, X\rangle $ is an invariant of a divergence free vector field, that we shall call the {\it Arnold invariant}, following Arnold's Principle:
\begin{quote}
{\it If a notion bears a personal name, then this is not the name of the discoverer.}
\end{quote}
and its complement, the Berry Principle:
\begin{quote}
{\it The Arnold Principle is applicable to itself (\cite{arnold3}).}
\end{quote}
Indeed, the invariant was introduced by many authors, including Arnold (and Moffatt, Morau), under the name of Hopf invariant. It is also called {\it helicity}.

Arnold gave a nice dynamical or topological interpretation of $ \langle X_1, X_2 \rangle$.
Choose some auxiliary riemannian metric on $\Sigma$.
Let $\phi^t_1,\phi^t_2$ be the flows generated by $X_1$ and $X_2$ and pick two points $x_1,x_2$.
One can follow the piece of trajectory of $X_1$ from $x_1$ to $\phi^{t_1}_1(x_1)$ and then follow some minimal geodesic connecting the two endpoints of this arc.
In this way, we produce a closed loop $k_{X_1}(x_1,t_1)$.
In the same way, we can consider the loop $k_{X_2}(x_2,t_2)$.
It turns out that for a generic choice of $x_1,x_2,t_1,t_2$ these two loops are disjoint so that one can compute their linking number.
Then, a form of the ergodic theorem shows that the limit
$${\rm lk}(x_1,x_2)=\lim_{t_1,t_2 \to + \infty} \frac{1}{t_1t_2} {\rm lk} (k_{X_1}(x_1,t_1), k_{X_2}(x_2,t_2))$$
exists for almost every pair of points $(x_1,x_2)$.
The dynamical interpretation is then that the Arnold invariant is an asymptotic linking number:
$$
\langle X_1, X_2 \rangle = \int_\Sigma \int_\Sigma {\rm lk}(x_1,x_2) \, {\rm dvol}(x_1) \, {\rm dvol}(x_2).
$$

The main open conjecture in this topic concerns the topological invariance.
Suppose that the two flows $\phi^t_1,\phi^t_2$  are conjugate by some volume preserving {\it homeomorphism}.
Does that imply that the Arnold invariants of $X_1$ and $X_2$ are equal?

A partial result in this direction is obtained in~Gambaudo and Ghys~\cite{gambaudoghys1}.
If the manifold $\Sigma$ contains an invariant solid torus $D \times {\mathbb S}^1$ in such a way that the fibres $D \times \{ \star \}$ are transversal to the vector field $X$, one can relate the contribution of helicity in the solid torus to the Calabi invariant of the first return map on a disc. Since it is known that the Calabi invariant in dimension 2 is invariant by area preserving homeomorphisms, this provides a wide class of vector fields for which the conjecture is true.

The idea of considering a divergence free vector field as some kind of diffuse knot is certainly not new.
The closed form $i_X vol$ can be considered as a closed 1-current.

It is therefore tempting to try to generalize to volume preserving vector fields some other concepts coming from knot theory.
In particular, in~Gambaudo and Ghys~\cite{gambaudoghys2}, we discussed the signatures of a divergence free vector field.
As before, we consider the loop $k_X(x,t)$.
For almost every choice of $x,t$, this is a knot and one shows that the limit of signatures
$$\tau_X(x)=\lim_{t\to +\infty} \frac{1}{t^2} \tau (k_X(x,t))$$
exists for almost every point $x$.
Moreover, we show that if the vector field is ergodic, this limit $\tau_X(x)$ is almost everywhere constant and coincides with (one half of the) helicity of $X$.
A wide generalization is obtained in Baader and March\'e~\cite{baadermarche} to invariants of finite type in knot theory. They show, assuming ergodicity, that if $v$ is a finite type invariant of order $n$, the limit
$$\lim_{t\to +\infty} \frac{1}{t^n} v( k_X(x,t))$$
exists for almost every point $x$ and is equal to $C_v \langle X , X \rangle ^n$ for some constant $C_v$ depending only on the invariant $v$.

\newpage

\section{Number theory, topology and signatures}\label{sing}

\subsection{The modular group ${\rm SL}(2,\ZZ)$}\label{modular}

 The {\it modular group} ${\rm SL}(2,\ZZ)$ is the group of $2 \times 2$ integer matrices
$$A=\begin{pmatrix} a & b \\ c & d\end{pmatrix}$$
such that
$$\text{det}(A)=ad-bc=1 \in \ZZ.$$
Implicitly, the modular group was introduced by Gauss in his famous {\it Disquisitiones arithmeticae}~\cite{gauss} (1801), discussing in particular quadratic forms on two variables {\it with integral coefficients}.
It has been thoroughly investigated by many mathematicians, starting with Dedekind~\cite{dedekind} in his commentary on Riemann's work on elliptic functions.
This group is related to almost every area of mathematics and at least one of the authors of the present paper\footnote{In fact, both authors.} considers it as ``his favorite group''.
See Serre~\cite{serrearith} or Mumford, Series and Wright~\cite{mumfordserieswright} for a modern account.

In this final section, we will restrict ourselves to a very specific aspect of the modular group: its relation with low dimensional topology and signatures.

The center of  ${\rm SL}(2,\ZZ)$ consists of $\{ \pm I_2 \}$ and the quotient by the center is denoted by ${\rm PSL}(2,\ZZ)$.
The fundamental object is the action of ${\rm PSL}(2,\ZZ)$ on Poincar\'e upper half plane ${\cal H} = \{ z \in \CC \vert \Im (z) >0 \}$.
$$
( \left[ \begin{matrix} a & b \\ c & d\end{matrix}\right], z) \in {\rm PSL}(2,\ZZ) \times {\cal H} \mapsto \frac{az+b}{cz+d} \in {\cal H}
$$
(we use the bracket notation for the quotient by the center).
This action preserves the hyperbolic metric and there is a well-known fundamental domain:
$$
\{ z \in {\cal H} \vert - \frac{1}{2} \leqslant  \Re(z) \leqslant  \frac{1}{2} \text{ and } \vert z \vert \geqslant   1 \}.
$$
The action is not free but the quotient $X= {\cal H } / {\rm PSL}(2,\ZZ)$ defines a Riemann surface which is isomorphic to $\CC$.
This isomorphism is realized by the famous $j$-invariant:
$$
j(z )={1 \over q}+744+196884q+21493760q^{2}+864299970q^{3}+20245856256q^{4}+\cdots
$$
with $q = \exp(2 \pi i z )$.
However, it is frequently useful to think of $X = \ {\cal H } / {\rm PSL}(2,\ZZ)$ as an {\it orbifold}, taking into account the two non-free orbits, with stabilizers $\ZZ/2 \ZZ$ and $\ZZ / 3 \ZZ$.
The universal cover of $X$, {\it as an orbifold}, is ${\cal H}$ and its fundamental group, again as an orbifold, is ${\rm PSL}(2,\ZZ)$.

The following figure is extracted from a famous paper by Klein~\cite{klein}, dated 1878, in which he explains that this tessellation is due to Dedekind, who himself acknowledges the influence of Gauss $\dots$

\begin{center}
\includegraphics[width=.9\textwidth]{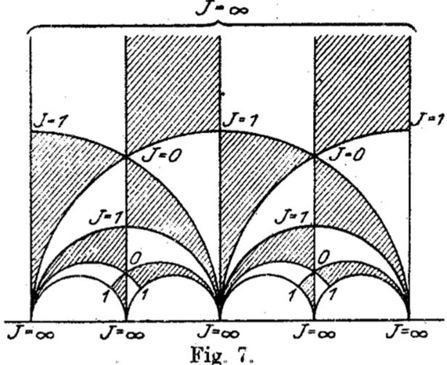}
\end{center}

The action on the boundary of ${\cal H}$ is the natural action of ${\rm PSL}(2,\ZZ)$ on $ \RR{\mathbb P}^1= \RR \cup \{\infty \}$.
The key point is that this action is transitive on rational numbers.
For $A \in {\rm SL}(2,\ZZ)$  with $c \neq 0$  the Euclidean algorithm gives a regular $\chi \in \ZZ^n$  with  $\vert\chi_k\vert \geqslant   2$, such that
$$
A=\begin{pmatrix} a & b \\ c & d \end{pmatrix}=
\begin{pmatrix} 0 & -1 \\ 1 & 0  \end{pmatrix}
\begin{pmatrix} \chi_1 & -1 \\ 1 & 0  \end{pmatrix}
\begin{pmatrix} \chi_2 & -1 \\ 1 & 0  \end{pmatrix}\dots\begin{pmatrix} \chi_n & -1 \\
1 & 0  \end{pmatrix}
$$
$$
a/c=[\chi_1,\chi_2,\dots,\chi_n]=\chi_1- \dfrac{1}{\chi_2-\dfrac{1}{\chi_3- \ddots\lower10pt\hbox{$-\dfrac{1}{\chi_n}$} }}.
$$
This is essentially the same matrix product as in Section~\ref{sturm}.

One can express the same fact in several equivalent ways.
It one sets
$$
S=\left[
\begin{matrix} 0 & -1 \\ 1 & 0\end{matrix}
\right];
T=\left[
\begin{matrix} 1 & 1 \\ 0 & 1\end{matrix}
\right];
\text{ and }
U=TS=
\left[
\begin{matrix} 1 &  1 \\ 1 & 0\end{matrix}
\right],
$$
one has $S^2= U^3=I_2$ and  ${\rm PSL}(2,\ZZ)$ is the {\it free product of the corresponding subgroups of order 2 and 3}.

One could also say that any element of  ${\rm PSL}(2,\ZZ)$ can be written in a unique way as a product of matrices which are alternatively of the form
$$
\left[\begin{matrix} 1 & n \\ 0 & 1\end{matrix}\right] \text{ and } \left[\begin{matrix} 1 & 0 \\ n & 1\end{matrix}\right]
$$
with $n\neq 0$.

{\it Modular forms} of weight $2k$ are expressions of the type $f(z)\, dz^k$, with $f$ holomorphic on ${\cal H}$, which are invariant under ${\rm PSL}(2,\ZZ)$.
For an introduction, we recommend the wonderful~\cite{serrearith}.
We will only recall that there is essentially a unique modular form of weight $12$,
$$
\Delta(z) = q \prod_{n=1}^{\infty}(1-q^n)^{24}
$$
again with  $q = \exp(2 \pi i z )$.
To explain the existence of this $\Delta$, one could use the kernel $\Gamma_2$ of the abelianization of ${\rm PSL}(2,\ZZ)$, of index $6$, acting freely on ${\cal H}$.
The quotient of ${\cal H}$ by $\Gamma_2$ is a {\it bona fide} Riemann surface, covering 6 times $X$,
which is actually a punctured elliptic curve, and which therefore carries a non vanishing holomorphic form.
The sixth power of this holomorphic form goes down to the modular form $\Delta$ on $X$.

\subsection{Torus bundles and their signatures}

We refer to Stillwell~\cite{stillwell} for a survey of the early history of 3-manifolds.

There are basically two ways of constructing a 3-manifold from a matrix in ${\rm SL}(2,\ZZ)$.
In this subsection, we describe torus bundles. In the next subsection we describe the Heegaard twisted double construction
of the lens spaces.

Every element $A \in {\rm SL}(2,\ZZ)$ is induced by an automorphism of the torus $\SSS^1 \times \SSS^1=\RR ^2 / \ZZ^2$ by a linear transformation of $\RR^2$ preserving the lattice of integral points $\ZZ^2$.
It turns out that any orientation preserving diffeomorphism of the torus is isotopic to a linear mapping so that
${\rm SL}(2,\ZZ)$ is isomorphic to the mapping class group of the torus.

Given a matrix $A$, one can construct the mapping torus
$$T^3_A=\RR ^2 / \ZZ ^2\times \RR /\{(x,t) \sim (A(x),t+1)\}.$$
By definition, $T^3_A$ has $\RR ^2 / \ZZ ^2 \times \RR$ as an infinite cyclic cover.
The natural projection of $T^3_A$ on $\SSS^1=\RR  / \ZZ$ is a locally trivial fibration, whose fibres are tori.

These examples were introduced by Poincar\'e~\cite{poincare1892} in his foundational paper on Topology.
His motivation was to study examples of manifolds where the homology (that he defined himself) was not sufficient to distinguish non homeomorphic manifolds.
Indeed, he defined the fundamental group based on these examples.
He showed that the fundamental groups of $T^3_A$ and $T^3_B$ are isomorphic if and only if $A$ and $B^{\pm 1}$ are conjugate in ${\rm GL}(2, \ZZ)$.

In order to connect these manifolds with the geometry of the modular orbifold $X = \ {\cal H } / {\rm PSL}(2,\ZZ)$, one should recall the following construction.

Consider the linear action of ${\rm SL} (2, \ZZ)$ on $\CC^2$.
The open set
$${\cal L}_{marked}=\{ (\omega_1,\omega_2) \in  \CC^{\bullet 2} \, \vert \, \omega_1/ \omega_2 \in {\cal H} \}$$
is the set of lattices in $\CC$ {\it marked with a positive basis}.
The quotient ${\cal L}= {\cal L}_{marked} / {\rm SL}(2, \ZZ)$ is the space of (unmarked) lattices in $\Lambda \subset \CC$.
For such a lattice $\Lambda$ one defines
$$
g_2(\Lambda) = 60 \sum_{\omega \in \Lambda\backslash \{0\}} \omega^{-4} \text{ and }  g_3(\Lambda) = 140 \sum_{\omega \in \Lambda\backslash \{0\}} \omega^{-6}.
$$
It turns out that $(g_2,g_3)$ maps ${\cal L}$  biholomorphically onto the complement of the discriminant locus $\{(g_2,g_3) \vert  g_2^3-27 g_3^2=0\} \subset \CC^2$.
By homogeneity, ${\cal L}$ is homeomorphic to the product by $\RR^{+}$ of the complement of a trefoil knot in the $3$-sphere $\{ (g_2,g_3) \in \CC \vert \, \vert g_2 \vert^2+ \vert g_3\vert^2=1 \}$.
Its fundamental group is the braid group $B_3$, a central extension of ${\rm SL}(2, \ZZ)$.

Note that the maps
$$
 (\omega_1,\omega_2) \in {\cal L}_{marked}  \to \omega_1/\omega_2 \in  {\cal H}
$$
is ${\rm SL}(2,\ZZ)$-equivariant and produces a holomorphic map ${\cal L}\to X$ (with fibres isomorphic to $\CC^{\bullet}$).

Each lattice $\Lambda \subset \CC$ defines an elliptic curve $\CC / \ZZ$, which is topologically a torus $\TT^2$.
Therefore there is a tautological torus bundle $E$ over $\cal L$.
The action of the fundamental group of the basis, i.e. $B_3$, on the homology $\ZZ^2$ of the torus fibres, factors through the projection $B_3 \to {\rm SL}(2, \ZZ)$
and reduces to canonical action of ${\rm SL}(2, \ZZ)$ on $\ZZ^2$.

Given a loop in $\cal L$, one can consider its pre-image in $E$.
One gets a $3$-manifold which is of course diffeomorphic to some $T^3_A$.
This was the main motivation of Poincar\'e in his attempt to generalize Riemann's ideas from complex curves to complex surfaces.

If $A$ is an element  of ${\rm SL}(2, \ZZ)$, one can choose a path in ${\rm SL}(2, \RR)$ connecting $A$ to $I_2$.
This provides a trivialization $t$ of the tangent bundle of $T^3_A$.
We can therefore consider the associated Hirzebruch signature defect : $3 \tau(W) - p_1(W, \partial W)$ for any $4$-manifold $W$ with boundary $T^3_A$ where the class $p_1$ is evaluated relative to its trivialization $t$ on the boundary.
We have seen earlier that this does not depend on the choice of $W$.
It is easy to show that the defect does not depend either on the choice of the path ${\rm SL}(2, \RR)$ since the effect of a change of path factors through $\pi_1(SO) = \ZZ / 2 \ZZ$.
Therefore, there is a well defined signature defect
$$
\delta : {\rm SL}(2, \ZZ) \to \ZZ.
$$
There are several ways to compute this number as a function of $A$.
The first is to find explicitly a manifold $W_A$ having $T^3_A$ as its boundary.
If $A,B$ are two elements of $ {\rm SL}(2, \ZZ)$, one can construct a $4$-manifold $W_{A,B}=T(A,B)$, which is a fibration over a pair of pants (i.e. a disc minus two discs), with fibres diffeomorphic to $\RR^2/ \ZZ^2$ and such that the monodromies of the boundary components are $A,B,AB$.
Therefore, the boundary of $W_{A,B}$ is $T^3_{AB}-T^3_{A}-T^3_{B}$.
If we can construct manifolds with boundaries $T^3_A$ and $T^3_B$, we can glue them to $W_{A,B}$ and this will provide a manifold with boundary $T^3_{AB}$.

When $A =\begin{pmatrix} 1 & n \\ 0 & 1\end{pmatrix}$, it is easy to describe $T^3_A$ : it is a principal circle fibration over the torus $\RR^2/ \ZZ^2$ with Euler number $n$.
Filling each fibre by a disc, one constructs immediately a $4$ manifold $W_A$ with boundary $T^3_A$.
Alternatively, $W_A$ is the unit disc bundle over $\RR^2/ \ZZ^2$ in the 2-dimensional vector bundle of Chern class $n$ : the zero section has self intersection $n$.

Since we know that any $A$ in $ {\rm SL}(2, \ZZ)$ can be written uniquely as a product of upper and lower triangular matrices, this provides an explicit construction of a manifold $W_A$ with boundary $T^3_A$.
It is now easy to compute the intersection form of $W_A$ and to get the numerical value of $\delta(A)$.

In order to state the result, we use the so called {\it Rademacher function}
$$
R: {\rm PSL}(2, \ZZ) \to \ZZ
$$
defined in the following way.
Every element of ${\rm PSL}(2, \ZZ)$ can be written as $U^{\epsilon_0}SU^{\epsilon_1}... U^{\epsilon_{k+1}}$
 with $\epsilon_0, \epsilon_{k+1}\in \{-1,0,1\}$ and $\epsilon_j\in \{-1,1\}$ for $j\neq 0, k+1$.
 We set
 $$
 R(A) = \sum_{i=0}^{k+1} \epsilon_i.
 $$
 For simplicity, suppose that $A$ is hyperbolic, i.e. that ${\rm trace}\,(A) >2$.
 Assume that in the above decomposition $\epsilon_0=1$ and $\epsilon_{k+1}=0$, then one gets
 $$
 \delta(A) = 4 R(A).
 $$
 See~\cite{bargeghys} for this computation.
 Another approach, also from~\cite{bargeghys} (and also Kirby and Melvin~\cite{kirbymelvin}), is to observe that the cocycle
 $$(A,B) \in {\rm SL}(2,\ZZ) \times {\rm SL}(2,\ZZ) \mapsto \tau(W_{A,B}) \in \ZZ$$
 represents $4 \in H^2({\rm SL}(2,\ZZ);\ZZ)=\ZZ/12\ZZ$, with 
 $$(A,B) \in {\rm SL}(2,\ZZ) \times {\rm SL}(2,\ZZ) \mapsto 3\tau(W_{A,B})
 =  \delta(AB)-\delta(A)-\delta(B)  \in \ZZ.$$
The canonical extension 
$$(A,B) \in {\rm SL}(2,\RR) \times {\rm SL}(2,\RR) \mapsto \tau(W_{A,B}) \in \ZZ$$
 is the cocycle of Meyer~\cite{meyer},  representing $4 \in H^2({\rm SL}(2,\RR);\ZZ)=\ZZ$. 
The boundedness of this cocycle and the fact that every element of ${\rm SL}(2, \RR) $ is a product of a bounded number of commutators enables us to identify quickly its primitive and to recover this Rademacher formula.

A nice example where such a formalism can be applied consists of Hilbert modular surfaces (see~\cite{hirzebruch1973}).
For instance, consider the {\it Hilbert modular group}  $\Gamma_D={\rm PSL}(2, \ZZ[\sqrt{D}])$ (for $D$ positive integer with no square factor, say such that $D$ is not congruent to 1 modulo 4).
It embeds in ${\rm PSL}(2, \RR)\times {\rm PSL}(2, \RR)$ as a discrete group through the two embeddings of $\ZZ[\sqrt{D}]$ in $\RR$.
One can therefore construct the quotient $W_D$ of ${ \cal H} \times {\cal H}$ by the action of $\Gamma$ (as an orbifold).
This is a 4-dimensional orbifold having a finite number of ends.
Each of its ends is associated to an element of the ideal class group of $\ZZ[\sqrt{D}]$  and is of the form $T^3_A\times \RR^+$.
Therefore $W_D$ provides a cobordism between all the $T^3_A$ associated to ideals in $\ZZ[\sqrt{D}]$.

\subsection{Lens spaces}\label{lens}

The {\it lens space} $L(c,a)$ is the closed, parallelizable 3-dimensional manifold defined for coprime $c,a \in \ZZ$
$$L(c,a)=\begin{cases}
\SSS^3 / (\ZZ/c\ZZ)&{\rm if}~c \neq 0\\
\SSS^1 \times \SSS^2&{\rm if}~c=0.
\end{cases}$$
where the unit sphere $\SSS^3$ is identified with $\{(u,v) \in \CC^2\,\vert \, \vert u \vert^2+\vert v\vert^2=1\}$ and
$\ZZ /c \ZZ $ acts by $(u,v) \mapsto (\zeta u,\zeta^av)$ with $\zeta=\exp{2\pi i /c}.$
Lens spaces were introduced by Tietze~\cite{tietze} in 1908.
They provided the first examples of homotopy equivalent manifolds which are not homeomorphic, $L(7,1)$ and $L(7,2)$ (detected by Reidemeister torsion).

The lens space $L(c,a)$ has
a genus 1 Heegaard decomposition
$$L(c,a)=\SSS^1 \times \DDD^2\cup_A \SSS^1 \times \DDD^2$$
for any $A=\begin{pmatrix} a & b \\ c & d \end{pmatrix} \in {\rm SL}(2,\ZZ)$, corresponding to the symplectic (= $(-1)$-symmetric)  formation
$(H_-(\ZZ);\ZZ\oplus \{0\},L)$ over $\ZZ$ with
$$L=A(\ZZ \oplus \{0\})=\{(ax,cx)\,\vert\, x \in \ZZ\} \subset \ZZ \oplus \ZZ.$$
If $c \neq 0$
$$\pi_1(L(c,a))=H_1(L(c,a))=\ZZ / c \ZZ~,~H_*(L(c,a);\QQ)=0$$  
and for any $q \in \ZZ$ there is a canonical orientation-preserving diffeomorphism
$L(c,a) \cong L(c,a+qc)$ inducing the isomorphism of symplectic
formations
$$\begin{array}{l}
\begin{pmatrix} 1 & q \\ 0 & 1 \end{pmatrix}~:~
(H_-(\ZZ);\ZZ \oplus \{0\},\{(ax,cx)\,\vert\, x \in \ZZ\})\\[1ex]
\hskip25pt
\xymatrix{\ar[r]^-{\cong}&} (H_-(\ZZ);\ZZ \oplus \{0\},\{((a+qc)x,cx)\,\vert\, x \in \ZZ\})~.
\end{array}$$
\indent The lens spaces are the only 3-dimensional manifolds having a Heegaard decomposition of genus 1.

As in Hirzebruch, Neumann and Koh~\cite[p.51]{hirzebruchneumannkoh} assume that $c$ is odd, $a$ is even, so that there is a unique expression of $c/a$ as an improper continued fraction
$$\begin{array}{ll}
c/a&=\chi_1-\dfrac{1}{\chi_2-\dfrac{1}{\chi_3-\ddots\lower10pt\hbox{$-\dfrac{1}{\chi_n}$} }}\\
&=[\chi_1,\chi_2,\dots,\chi_n] \in \QQ
\end{array}$$
for $\chi_k \neq 0 \in 2\ZZ$ ($1 \leqslant k \leqslant n$), corresponding to the iterations of the Euclidean algorithm for ${\rm gcd}(a,c)=1$, as in section~\ref{modular} but with the roles of $a,c$ reversed.
The expression is realized by the graph $4$-dimensional manifold
$$(M(\chi),\partial M(\chi))=(\DDD^4 \cup \bigcup\limits_n \DDD^2 \times \DDD^2, L(c,a))$$ obtained by $n$ successive geometric plumbings using the tree $A_n$ weighted by $\chi_k \in \pi_1({\rm SO}(2))=\pi_1(\SSS^1)=\ZZ$
$$\xymatrix{\hbox{$A_n:$}
\raise5pt\hbox{$\displaystyle{\chi_1}\atop \bullet$} \ar@{-}[r] &
\raise5pt\hbox{$\displaystyle{\chi_2} \atop \bullet$} \ar@{-}[r] &
\raise5pt\hbox{$\displaystyle{\chi_3} \atop \bullet$} \ar@{-}[r] & \dots \ar@{-}[r] &
\raise5pt\hbox{$\displaystyle{\chi_{n-1}} \atop \bullet$} \ar@{-}[r] &
\raise5pt\hbox{$\displaystyle{\chi_n} \atop \bullet$}
}$$
The manifold $M(\chi)$ resolves the singularity at $(0,0,0)$ of the 2-dimensional complex space
$$\{(w,z_1,z_2) \in \CC^3\,\vert\,w^c=z_1(z_2)^{c-a}\}.$$
See Jung~\cite{jung}, Hirzebruch~\cite{hirzebruchvier} and de la Harpe~\cite{delaharpe}.

The intersection matrix of $M(\chi)$ is the integral symmetric tridiagonal $n \times n$ matrix
$$\phi_{M(\chi)}=\text{Tri}(\chi)=\begin{pmatrix}
\chi_1 & 1 & 0 & \dots & 0 & 0 \\
1 & \chi_2 & 1 & \dots & 0 & 0 \\
0 & 1& \chi_3 & \dots & 0 & 0 \\
\vdots & \vdots & \vdots & \ddots & \vdots & \vdots \\
0 & 0 & 0 & \dots & \chi_{n-1} & 1\\
0 & 0 & 0 & \dots & 1 & \chi_n
\end{pmatrix},$$
which is obtained by the corresponding $n$ successive algebraic plumbings (\ref{plumbing2}). The symmetric form over $\ZZ$
$$(H_2(M(\chi);\ZZ),\phi_{M(\chi)})=(\ZZ^n,{\rm Tri}(\chi))$$
is nondegenerate, meaning that ${\rm det}({\rm Tri}(\chi)) \neq 0$.
As in Proposition~\ref{tri2} the principal minors
$\mu_k=\mu_k(M(\chi)) \neq 0 \in \ZZ$ ($1 \leqslant k \leqslant n$) are such that
$$\mu_k/\mu_{k-1}=[\chi_k,\chi_{k-1},\dots,\chi_1]~,~\mu_{k-2}+\mu_k=\chi_k\mu_{k-1}~(\mu_{-1}=0,\mu_0=1)$$
and by Theorem~\ref{sylvestertheorem} the signature of  $M(\chi)$ is 
$$\tau(M(\chi))=\sum\limits_{k=1}^n{\rm sign}(\mu_k/\mu_{k-1})~.$$
See section \ref{localize} below for a further discussion of the intersection matrix of $M(\chi)$ in the context of the Witt group
 localization exact sequence for $W(\ZZ) \to W(\QQ)$.

The hypothesis $\vert \chi_k \vert \geqslant   2$ gives a simplification:

\begin{proposition} \label{signtri}
{\rm [Hirzebruch, Neumann and Koh~\cite[Lemma 8.12]{hirzebruchneumannkoh} }
A vector $\chi=(\chi_1,\chi_2,\dots,\chi_n) \in \ZZ^n$ with $\vert \chi_k \vert \geqslant   2$ ($k=1,2,\dots,n$)
is regular (so that the continued fractions $[\chi_k,\chi_{k-1},\dots,\chi_1] \in \QQ$ are well-defined) with
$${\rm sign}([\chi_k,\chi_{k-1},\dots,\chi_1])={\rm sign}(\chi_k)$$
and
$$\tau({\rm Tri}(\chi))=\sum^n_{k=1}{\rm sign}([\chi_k,\chi_{k-1},\dots,\chi_1])=\sum\limits^n_{k=1}{\rm sign}(\chi_k) \in \ZZ.$$
\end{proposition}
\begin{proof} Assume inductively that for some $k \geqslant   1$
\begin{itemize}
\item[(a)] $\vert [\chi_k,\chi_{k-1},\dots,\chi_1]\vert >1$,
\item[(b)] ${\rm sign}([\chi_k,\chi_{k-1},\dots,\chi_1])= {\rm sign}(\chi_k)$.
\end{itemize}
By the triangle inequality
$$\begin{array}{ll}
\vert [\chi_{k+1},\chi_k,\dots,\chi_1]\vert&=~
\vert \chi_{k+1} - [\chi_k,\chi_{k-1},\dots,\chi_1]^{-1} \vert\\
&\geqslant   \vert \chi_{k+1}\vert - \vert [\chi_k,\chi_{k-1},\dots,\chi_1]^{-1} \vert > 1
\end{array}
$$
and
$${\rm sign}(\chi_{k+1})=
{\rm sign}([\chi_{k+1},\chi_k,\dots,\chi_1]^{-1}) = {\rm sign}([\chi_{k+1},\chi_k,\dots,\chi_1])$$
giving the inductive step. It now follows from Sylvester's Theorem~\ref{sylvestertheorem} that
$$\begin{array}{ll}
\tau({\rm Tri}(\chi))&=\sum\limits^n_{k=1}
{\rm sign}([\chi_k,\chi_{k-1},\dots,\chi_1])\\
&=\sum\limits^n_{k=1}{\rm sign}(\chi_k)\in \ZZ.
\end{array}$$
\end{proof}

See Hirzebruch and Mayer~\cite{hirzebruchmayer} and Hirzebruch, Neumann and Koh~\cite{hirzebruchneumannkoh} for
early accounts of plumbing.

\subsection{Dedekind sums}

The Dedekind $\eta$ function defined on ${\cal H}$ by
$$
\eta(z) = {\rm exp}( \pi i z /12) \prod_{n=1}^{\infty}(1-{\rm exp}(2 \pi i z )^n.
$$
It clearly does not vanish.
Its 24th power is the modular $\Delta$ that we mentioned earlier
$$
\eta^{24}(\frac{az+b}{cz+d})= \eta^{24}(z) (cz+d)^{12}
$$
which does not vanish.
Hence, there is a holomorphic determination of its logarithm $\ln \eta$ on ${\cal H}$ and taking the logarithm on both sides of the previous identity, we get:
$$
24 \ln \eta(\frac{az+b}{cz+d})= 24 \ln \eta (z)  + 6 \ln (-(cz+d)^{2})+ 2 \pi i R(A).
$$
for some function $R$ on ${\rm PSL}(2, \ZZ)$ (the second log in the right hand side is chosen with imaginary part in $]-\pi, \pi]$).
The numerical determination of $R(A)$ has been a challenge, and turned out to be related to many different topics, in particular number theory, topology, and combinatorics.
The inspiring paper by M. Atiyah~\cite{atiyahomni} contains an ``omnibus theorem'' proving that seven definitions of $R$ are equivalent!
As the notation suggests, one of the incarnation of $R$ is the previously defined Rademacher function.

We now give a more combinatorial description.

For a real number $x$, one defines $((x))$ by

$$((x)) =\begin{cases}
\{x\}-1/2&\text{if}~x \in \RR\backslash \ZZ\\
0&\text{if}~x \in \ZZ
\end{cases}$$
with $\{x\}\in [0,1)$ the fractional part of $x \in \RR$. Nonadditive:
$$((x))+((y))-((x+y))=\begin{cases}
-1/2&\text{if}~0<\{x\}+\{y\}<1\\
1/2&\text{if}~1<\{x\}+\{y\}<2\\
0&\text{if}~x~\in \ZZ~\text{or}~y \in \ZZ~\text{or}~x+y \in \ZZ.
\end{cases}$$

\begin{center}
\includegraphics[width=.9\textwidth]{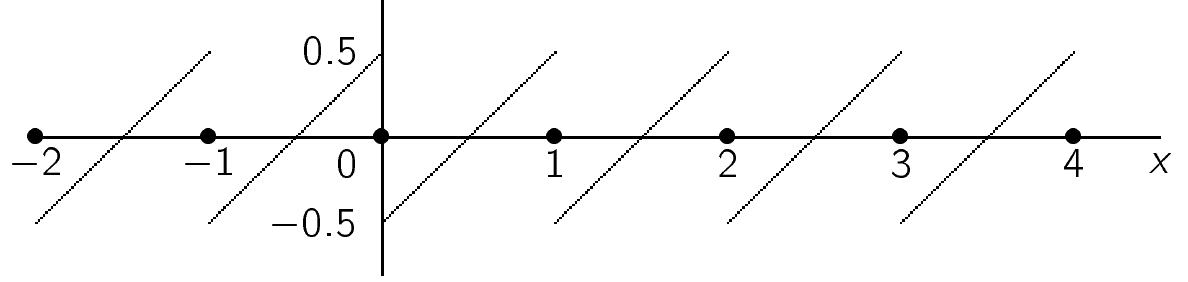}
\end{center}

The sawtooth function $x \mapsto((x))$ was used by Dedekind~\cite{dedekind} to count the $\pm 2\pi$ jumps in the complex logarithm
$$\text{log}(re^{i\theta})=\text{log}(r)+i(\theta+2n \pi) \in \CC(n \in \ZZ).$$
The {\it Dedekind sum} for $a,c \in \ZZ$ with $c \neq 0$ is
$$s(a,c)=\sum\limits^{\vert c \vert -1}_{k=1}
\big(\big( \frac{k}{c} \big)\big) \big(\big(\frac{ka}{c}\big)\big)
=\dfrac{1}{4\vert c\vert}\sum\limits^{\vert c \vert -1}_{k=1}
\text{cot}\big( \frac{k\pi}{c} \big) \text{cot}\big(\frac{ka\pi}{c}\big)\in \QQ$$
(Rademacher and Grosswald~\cite{rademachergrosswald},  Hirzebruch~\cite{hirzebruch1971},~\cite{hirzebruch1973},  Hirzebruch and Zagier~\cite{hirzebruchzagier}, Kirby and Melvin
\cite{kirbymelvin}).

For $\begin{pmatrix} a & b \\ c & d \end{pmatrix} \in {\rm SL}(2,\ZZ)$ and $\chi=(\chi_1,\chi_2,\dots,\chi_n)\in \ZZ^n$, $M(\chi)$ as in Proposition~\ref{signtri}, the signature defect of $M(\chi)$ is given by
$$
\tau(\text{Tri}(\chi))-(\sum\limits^n_{j=1}\chi_j)/3=
\begin{cases}
b/3d&\text{if}~c=0\\
(a+d)/3-4\text{sign}(c)s(a,c)&\text{if}~c \neq 0.
\end{cases}
$$
See Barge and Ghys~\cite{bargeghys} and Kirby and Melvin~\cite{kirbymelvin} for proofs and a hyperbolic geometry interpretation, related to the action of ${\rm SL}(2,\ZZ)$ on the upper half plane ${\cal H}$.

\newpage

\section{Appendix: Algebraic $L$-theory of rings with involution
and the localization exact sequence\\[0.5ex]
{\small  by Andrew Ranicki}}
\label{appendix}

In these notes we have largely worked with forms over commutative rings such as fields and their subrings.
However, the applications to the topology of manifolds  as well as purely algebraic considerations require the study of forms over an arbitrary ring with involution $R$ such as a group ring $\ZZ[\pi]$ with $\pi$ the fundamental group.
This has led to the development of the algebraic cobordism groups $L^n(R)$ (resp. $L_n(R)$) of $n$-dimensional f.g. free $R$-module chain complexes $C$ with a symmetric (resp. quadratic) Poincar\'e duality $\phi:H^{n-*}(C) \cong H_*(C)$ for
$n \geqslant   0$, with $L^0(R)=W(R)$ the symmetric Witt group of $R$. The quadratic $L$-groups $L_*(R)$ are the surgery obstruction groups of Wall~\cite{wall1971},
which are of central significance in the classification of the homotopy types of manifolds of dimension $>4$.
The symmetric $L$-groups $L^*(R)$ were introduced by Mishchenko~\cite{mishchenko} and developed further by Ranicki \cite{ranickiats1,ranickiats2,ranickiexact}.
We shall be mainly concerned with $L^*(R)$ here.

We note some basic properties of the algebraic $L$-groups:
\begin{itemize}
\item a compact oriented $n$-dimensional manifold $M$ with universal cover $\widetilde{M}$ has a {\it symmetric signature} $\sigma(M)=(C(\widetilde{M}),\phi_M)\in L^n(\ZZ[\pi_1(M)])$
which is both a cobordism and a homotopy invariant, and which for $n=4k$ has image $[\sigma(M)]=\tau(M) \in L^{4k}(\ZZ)=\ZZ$ the ordinary signature,
\item the intersection properties of even-dimensional manifolds,  the
linking properties of odd-dimensional manifolds, and the relationship  between the two for an even-dimensional manifold with boundary, are  seen in the symmetric Poincar\'e complex $(C(\widetilde{M}),\phi_M)$ of a closed manifold and its analogue for a manifold with boundary,
\item the quadratic $L$-groups are 4-periodic, $L_*(R)\cong L_{*+4}(R)$,  with $L_{2k}(R)$ (resp. $L_{2k+1}(R)$) the Witt group of nonsingular $(-1)^k$-quadratic forms (resp. formations) over $R$,
\item the symmetric  $L$-groups of a Dedekind ring $R$ are 4-periodic, $L^*(R)=L^{*+4}(R)$, with $L^{2k}(R)$ (resp. $L^{2k+1}(R)$) the Witt group of nonsingular $(-1)^k$-symmetric forms (resp. formations) over $R$,
\item the algebraic $L$-groups of $\ZZ$ are
$$\begin{array}{l}
L_n(\ZZ)~=~\begin{cases} \ZZ&{\rm if}~n \equiv 0(\bmod\,4)~({\rm signature/8)}\\
\ZZ/2\ZZ&{\rm if}~n \equiv 2(\bmod\,4)~(\hbox{\rm Arf-Kervaire invariant})\\
0&{\rm if}~n \equiv 1(\bmod\,2)~,
\end{cases}\\[6ex]
L^n(\ZZ)~=~\begin{cases} \ZZ&{\rm if}~n \equiv 0(\bmod\,4)~({\rm signature)}\\
\ZZ/2\ZZ&{\rm if}~n \equiv 1(\bmod\,4)~(\hbox{\rm de Rham invariant})\\
0&{\rm if}~n \equiv 2,3(\bmod\, 4)~,
\end{cases}
\end{array}$$
\item the function
$$A \in {\rm Sp}(R)=\bigcup\limits^{\infty}_{n=1}{\rm Sp}(2n,R) \mapsto (H_-(R^n);R^n,A(R^n)) \in L^3(R)$$
is a surjective group morphism, for any ring with involution $R$ with 4-periodic symmetric $L$-groups, in which case $L^3(R)$ is the Witt group of symplectic formations over $R$,
\item the symmetrization maps $L_*(R) \to L^*(R)$ are isomorphisms modulo 2-primary torsion, and actually isomorphisms if
 $1/2 \in R$,
\item the symmetric and quadratic $L$-groups are the same modulo 2-primary torsion, $L_*(R)[1/2] \cong L^*(R)[1/2]$, so they have the same signature invariants, i.e. ${\rm Hom}_{\QQ}(\QQ\otimes_{\ZZ}L_*(R),\QQ)={\rm Hom}_ {\QQ}(\QQ\otimes_{\ZZ}L^*(R),\QQ)$,
\item for any morphism $i:R \to R'$ of rings with involution there are relative symmetric $L$-groups $L^n(i)$ of the cobordism classes of $n$-dimensional
symmetric Poincar\'e pairs over $R'$ with the boundary $R'\otimes_R(C,\phi)$
for an $(n-1)$-dimensional symmetric Poincar\'e complex $(C,\phi)$ over $R$, with a long exact sequence
$$ \xymatrix{\dots \ar[r] & L^n(R) \ar[r]^-{i} & L^n(R') \ar[r] & L^n(i) \ar[r] & L^{n-1}(R) \ar[r] &\dots}$$
and similarly for quadratic $L$-groups,
\item for the localization $i:R \to R'=S^{-1}R$ of  a ring with involution $R$ inverting a multiplicative subset $S \subset R$ of central non-zero divisors with $1 \in S$ and $\bar{S}=S$ there is defined a long exact sequence of symmetric $L$-groups
$$\xymatrix{ \dots \ar[r] & L^n(R) \ar[r]^-{i} & L^n(S^{-1}R) \ar[r]^-{\partial}& L^n(R,S) \ar[r]& L^{n-1}(R) \ar[r]&\dots}$$ with $L^n(R,S)$ the cobordism group of $(n-1)$-dimensional
symmetric Poincar\'e complexes $(C,\phi)$ over $R$ such that $H_*(S^{-1}C)=0$, and similarly  for quadratic $L$-groups,
\item $L^0(S^{-1}R)=W(S^{-1}R)$ is the Witt group of nonsingular symmetric forms over $S^{-1}R$, and $L^0(R,S)=W(R,S)$ is the Witt group of nonsingular symmetric $S^{-1}R/R$-valued linking forms $(T,\lambda:T \times T \to S^{-1}R/R)$ on f.g. $R$-modules $T$ of homological dimension 1; for any $p_0 \in S$ there is defined a nonsingular symmetric form
$(S^{-1}R,p_0)$ over $S^{-1}R$ with $\partial(S^{-1}R,p_0)=(R/p_0R,1/p_0) \in W(R,S)$;
 for sequences of ``Sturm functions'' $p=(p_0,p_1,\dots,p_n) \in S^{n+1}$,
$q=(q_1,q_2,\dots,q_n) \in R^n$ such that 
$p_kq_k=p_{k-1}+p_{k+1}$, $p_n=1$, $p_{n+1}=0$ the Witt class 
$$(S^{-1}R^n,{\rm Tri}(q))=\sum\limits^n_{k=1}(S^{-1}R,p_{k-1}/p_k) \in W(S^{-1}R)$$
is such that
$$\begin{array}{l}
\partial(S^{-1}R,p_{k-1}/p_k)=
(R/p_{k-1}R,p_k/p_{k-1})\oplus (R/p_kR,-p_{k+1}/p_k)~,\\[1ex]
\partial  (S^{-1}R^n,{\rm Tri}(q))=(R/p_0,p_1/p_0) \in W(R,S)~,
\end{array}
$$
\item see Wall \cite[p.\ 297]{Wall1963} for an early instance in the
case $(R,S)=(\ZZ,\ZZ\backslash \{0\})$ of  this connection
between the Euclidean algorithm, tridiagonal symmetric forms and symmetric linking forms $(T,\lambda)$, with $T=\ZZ/p_0\ZZ$ a cyclic finite abelian group,  
\item the Witt group localization exact sequence for $(R,S)=(\RR[X],\RR[X]\backslash \{0\})$, $S^{-1}R=\RR(X)$ includes the split exact sequence 
$$\begin{array}{l}
\xymatrix{0 \ar[r] & W(\RR[X])=W(\RR)=\ZZ \ar[r] & W(\RR(X))}\\[1ex] \hskip50pt \xymatrix{\ar[r]^-{\partial} &W(\RR[X],S)=\ZZ[\RR]\oplus
\ZZ/2\ZZ[{\cal H}] \ar[r] &0;}
\end{array}$$
for a  regular degree $n+2m$ polynomial $P(X) \in \RR[X]$ with
Sturm functions $(P_0(X),P_1(X),\dots,P_{n+2m}(X))$, $Q(X)=(Q_1(X),Q_2(X),\dots,\allowbreak Q_{n+2m}(X))$ and roots $x_1,x_2,\dots,x_n \in \RR$,
$z_1,\overline{z}_1,\dots,z_m,\overline{z}_m \in {\cal H} \cup\overline{\cal H}$
 the Witt class 
$$\begin{array}{l}
(\RR(X),{\rm Tri}(Q(X)))~=~\sum\limits^{n+2m}_{i=1}(\RR(X),P_{i-1}(X)/P_i(X))\\[1ex]
=~\sum\limits^n_{k=1}(\RR(X),X-x_k) +\sum\limits^m_{j=1}(\RR(X),
(X-z_j)(X-\overline{z}_j)) -m(\RR(X),1) \\[1ex]
\hskip150pt\in W(\RR(X))
\end{array}$$
(Theorem \ref{nexttolast}) has image under $\partial$ 
$$\begin{array}{l}
\partial (\RR(X),{\rm Tri}(Q(X)))=(\RR[X]/(P_0(X)),P_1(X)/P_0(X))\\[1ex]
=\sum\limits^n_{k=1}1.x_k +\sum\limits^m_{j=1}1.z_j \in W(\RR[X],S)=\ZZ[\RR] \oplus \ZZ/2\ZZ[{\cal H}]~,
\end{array}$$
\item generically, for  a localization $S^{-1}\ZZ[\pi]$ of the group ring $\ZZ[\pi]$ of a group $\pi$, the quadratic $L$-groups $L_*(S^{-1}\ZZ[\pi])$ have a finite number of signature invariants (i.e. the $\QQ$-vector spaces $\QQ\otimes_{\ZZ}L_*(S^{-1}\ZZ[\pi])$ are finite-dimensional) for a finite group $\pi$ and an infinite number of signature invariants for an infinite
group $\pi$.
\end{itemize}
We refer to Ranicki~\cite{ranickiats1,ranickiats2,ranickiexact,ranickiknot,ranickitopman} for a development of algebraic $L$-theory.
We shall not attempt to summarize the full theory here!
Our aim in this appendix is to extend the theory of forms defined in Section~\ref{algebra} over fields to forms over a ring with involution $R$, and to indicate some of the connections between the material in sections 1-6 and the general theory.

\subsection{Forms over a ring with involution}

An {\it involution} on a ring $R$ is a function $r \in R \mapsto \bar{r} \in R$ such that $\overline{r+s}=\bar{r}+\bar{s}$, $\overline{rs}=\bar{s}.\bar{r}$, $\bar{1}=1$ and $\bar{\bar{r}}=r$. The three main examples are:
\begin{itemize}
\item $R$ is a commutative ring and $\bar{r}=r$ is the identity involution,
\item $R=\CC$ with $\overline{x+iy}=x-iy$ complex conjugation,
\item $R =\ZZ[\pi]$ is a group ring, and $\bar{g}=g^{-1}$ for $g \in \pi$.
\end{itemize}

Let then $R$ be a ring with involution.
The \emph{dual} of a (left) $R$-module $V$ is the (left) $R$-module $V^*={\rm Hom}_R(V,R)$ with $R$ acting by
$$(r,f) \in R \times V^* \mapsto (v \in V \mapsto f(v) \overline{r}) \in V^*.$$
The natural $R$-module morphism
$$v \in V  \mapsto (f \mapsto \overline{f(v)}) \in V^{**}$$
is an isomorphism for f.g. free $R$-modules $V$, in which case we use it to identify $V^{**}=V$.

Let $\epsilon=1$ or $-1$.
An {\it $\epsilon$-symmetric form $(V,\phi)$ over $R$}  is a f.g. free $R$-module $V$ with a bilinear pairing
$$\phi: (v,w) \in V \times V  \mapsto \phi(v,w) \in R$$
such that $\phi(w,v)=\epsilon \overline{\phi(v,w)} \in R$ for all $v,w \in V$.
This is equivalent to an $R$-module morphism
$$\phi: v \in V \mapsto (w \mapsto \phi(v,w)) \in V^*$$
such that $\phi^*=\epsilon \phi$. For $\epsilon=1$ the form is called {\it symmetric}, while for $\epsilon=-1$ it is called {\it skew-symmetric}.
A {\it morphism} $f:(V,\phi) \to (V',\phi')$ of $\epsilon$-symmetric forms over $R$ is an $R$-module morphism $f:V \to V'$ such that $f^*\phi'f=\phi$, or equivalently such that $\phi'(f(v) ,f(w))=\phi(v,w) \in R.$ An {\it isomorphism} is a morphism such that $f:V \to V'$ is an $R$-module isomorphism.

The {\it transpose} of an $n \times n$ matrix $A=(a_{ij})$ is the $n \times n$ matrix $A^*=(\bar{a}_{ji})$.
From the matrix point of view, an $\epsilon$-symmetric form $(V,\phi)$ with a choice of basis $(b_1,b_2,\dots,b_n)$ for $V$ corresponds to the  $n\times n$ matrix $S=(S_{ij})$ such that $S^*=\epsilon S$, with
$$\phi(b_i,b_j)=S_{ij}=\epsilon \bar{S}_{ji} \in R.$$
An isomorphism of forms corresponds to  linear congruence of matrices: if $f:(V,\phi) \to (V',\phi')$ is an isomorphism of $\epsilon$-symmetric forms and $f:V \to V'$ has invertible  matrix $A$ with respect to the bases chosen for $V$ and $V'$, then $A^*S'A=S$.

An $\epsilon$-symmetric form $(V,\phi)$ is {\it nonsingular} if the $R$-module morphism $\phi:V \to V^*$ is an isomorphism.

Given an $\epsilon$-symmetric form $(V,\phi)$ over $R$ and a submodule $W \subseteq V$ define the {\it orthogonal} submodule
$$W^{\perp}=\{v \in V\,\vert\, \phi(v,w)=0 \in R~\hbox{for all}~w \in W\} \subseteq V.$$
Submodules $V_1,V_2 \subseteq V$ are {\it orthogonal} if
$$\phi(v_1,v_2)=0 \in R~\hbox{for all}~v_1 \in V_1,v_2 \in V_2,$$
or equivalently $V_1 \subseteq V_2^{\perp}$.

A {\it subform} $(V',\phi') \subseteq (V,\phi)$  of an $\epsilon$-symmetric form $(V,\phi)$ over $R$ is the $\epsilon$-symmetric form defined by the restriction $\phi'=\phi\vert_{V'}$ on a direct summand $V'\subseteq V$ such that $V'^{\perp} \subseteq V$ is a direct summand.
Note that $V'^{\perp} \subseteq V$ is a direct summand if the $R$-module morphism
$$\phi\vert:v \in V  \mapsto (v' \mapsto \phi(v,v')) \in  {V'}^*$$
is onto, as is automatically the case if $(V,\phi)$ is nonsingular or if $R$ is a field.

The {\it radical} of $(V,\phi)$ is $V^{\perp} \subseteq V$.
If ${\rm im}(\phi) \subseteq V^*$ is a direct summand (e.g. if $R$ is a field) then $V^{\perp} \subseteq V$ is a direct summand, $(V/V^{\perp},[\phi])$ is a nonsingular $\epsilon$-symmetric form over $R$, and there are defined isomorphisms
$$(V,\phi) \,\cong\, (V^{\perp},0) \oplus ({\rm im}(\phi),\Phi),~ (V/V^{\perp},[\phi])\,\cong\, ({\rm im}(\phi),\Phi)$$
where
$$\Phi:~{\rm im}(\phi) \times {\rm im}(\phi) \to R;~(\phi(v),\phi(w)) \mapsto \phi(v)(w)=\epsilon \phi(w)(v). .$$
For a symmetric form $(V,\phi)$ over $\RR$
$$\tau(V,\phi)=\tau(V/V^{\perp},[\phi]) \in \ZZ$$

Let $(V,\phi)$ be an $\epsilon$-symmetric form over $R$.
A {\it sublagrangian} of an  $\epsilon$-symmetric form $(V,\phi)$ over $R$ is a subform of the type $(L,0) \subseteq (V,\phi)$. Then $L \subseteq V$ is isotropic (self-orthogonal), with $L,L^{\perp} \subseteq V$ direct summands such that $L \subseteq L^{\perp}$, with an induced $\epsilon$-symmetric form $(L^{\perp}/L,[\phi])$. A {\it lagrangian} is a sublagrangian $L$ such that $L=L^{\perp}$, in which case $(V,\phi)$ is nonsingular and  ${\rm dim}_R(L)={\rm dim}_R(V)/2$.

\begin{example} For any nonsingular $\epsilon$-symmetric form $(V,\phi)$ over $R$ the diagonal
$$\Delta=\{(v,v)\in V \oplus V\,\vert\, v \in V\} \subset V \oplus V$$
is a lagrangian of $(V,\phi)\oplus (V,-\phi)$.
\end{example}

\begin{definition} \leavevmode

1. The {\it metabolic $\epsilon$-symmetric form}  is the nonsingular $\epsilon$-symmetric form over $R$ defined for any $\epsilon$-symmetric form $(L^*,\theta)$ over $R$ by
$$H_{\epsilon}(L^*,\theta)=(L \oplus L^*,\begin{pmatrix} 0 & 1 \\ \epsilon & \theta \end{pmatrix})$$
with
$$\begin{array}{l}
\begin{pmatrix} 0 & 1 \\
\epsilon & \theta\end{pmatrix}:~L \oplus L^* \to (L \oplus L^*)^*;\\
(v,f)\mapsto ((w,g) \mapsto g(v) +\epsilon f(w)+\theta(f)(g))
\end{array}$$
$H_{\epsilon}(L^*,\theta)$ has lagrangian $L$.

2. The {\it hyperbolic $\epsilon$-symmetric form} is defined for any f.g. free $R$-module $L$ by
$$H_{\epsilon}(L)=H_{\epsilon}(L^*,0)=(L \oplus L^*,\begin{pmatrix} 0 & 1 \\ \epsilon & 0 \end{pmatrix}).$$

\end{definition}

\begin{proposition} \label{extension} \leavevmode

1. The inclusion of a sublagrangian $(L,0) \subseteq (V,\phi)$ in an $\epsilon$-symmetric form extends to an isomorphism of forms
$$H_{\epsilon}(L^*,\theta) \oplus (L^{\perp}/L,[\phi])\xymatrix{\ar[r]^-{\cong}&}(V,\phi).$$
In particular, a nonsingular $\epsilon$-symmetric form $(V,\phi)$ admits a lagrangian $L$ if and only if $(V,\phi)$ is isomorphic to $H_{\epsilon}(L^*,\theta)$ for some $(L^*,\theta)$.

2. If $1/2 \in R$ there is defined an isomorphism
$$\begin{pmatrix} 1 & \epsilon\theta/2 \\
0 & 1 \end{pmatrix}:~
H_{\epsilon}(L^*,\theta) \xymatrix{\ar[r]^-{\cong}&}H_{\epsilon}(L)$$
so a  nonsingular $\epsilon$-symmetric form admits a lagrangian $L$ if and only if it is isomorphic to $H_{\epsilon}(L)$ for some $L$.
\end{proposition}
\begin{proof} See Ranicki\cite[Prop. 2.2]{ranickiats2}.
\end{proof}

We finally get  to the definition of the Witt groups of a ring with involution.

\begin{definition} The {\it $\epsilon$-symmetric Witt group} $W_{\epsilon}(R)$ of a ring with involution $R$  is the group of equivalence classes of nonsingular $\epsilon$-symmetric forms $(V,\phi)$ over $R$ with
$$\begin{array}{ll}
(V,\phi) \sim (V',\phi')&\hbox{if and only if there exists an isomorphism}\\
&f:(V,\phi) \oplus H_{\epsilon}(L^*,\theta) \xymatrix{\ar[r]^-{\cong}&} (V',\phi') \oplus H_{\epsilon}({L'}^*,\theta')\\
&\hbox{for some $\epsilon$-symmetric forms $(L^*,\theta)$, $((L')^*,\theta')$ over $R$.}
\end{array}$$
This equivalence relation is called {\rm stable isomorphism}.
\end{definition}

The Witt group $W_{\epsilon}(R)$ is the quotient of the Grothendieck group of isomorphism classes $[V,\phi]$ of nonsingular $\epsilon$-symmetric forms over $R$ with the relations
$$\text{$[V,\phi]=0$ if $(V,\phi)$ admits a lagrangian $L$}$$
or equivalently
$$\text{$[V,\phi]=[L^{\perp}/L,[\phi]]$ if $(V,\phi)$ admits
a sublagrangian $L$}.$$
Addition and inverses are by
$$[V,\phi]+[V',\phi']=[V \oplus V',\phi\oplus \phi'],~-[V,\phi]=[V,-\phi] \in W_{\epsilon}(R).$$

We shall write the symmetric and skew-symmetric Witt groups as
$$W_{+1}(R)=W(R),~W_{-1}(R)=W_-(R)$$
and similarly for the hyperbolic and metabolic forms
$$\begin{array}{l}
H_{+1}(L^*,\theta)=H(L^*,\theta),~H_{+1}(L)=H(L),\\
H_{-1}(L^*,\theta)=H_-(L^*,\theta),~H_{-1}(L)=H_-(L).
\end{array}
$$

\begin{remark}
By Witt's Theorem~\ref{witt} for $R$ a field of characteristic $\neq 2$ (with the identity involution) nonsingular symmetric forms are isomorphic if and only if they are stably isomorphic.
This is false for a field of characteristic 2 and also for rings which are not fields: for example if $R=\FF_2$ or $\ZZ$ the nonsingular symmetric forms $H(R)$, $H(R,1)$ are stably isomorphic but not isomorphic.
\end{remark}

\begin{remark} (Barge and Lannes \cite{bargelannes})
Let $R$ be a commutative ring. For any $\ell \geqslant 1$ let $\text{Sym}_{\ell}(R)$ be the pointed set of nonsingular symmetric forms $(R^\ell,\phi)$ over $R$, based at $(R^\ell,I_\ell)$. 
For $k \geqslant 0$ let ${\rm Lag}_k(R)$ be the pointed set of lagrangians $L\subset R^k \oplus R^k$
in the hyperbolic nonsingular symplectic form $H_-(R^k)$ over $R$,
based at $R^k \oplus \{0\}$. The symplectic group ${\rm Sp}(2k,R)={\rm Aut}(H_-(R^k))$ acts transitively on ${\rm Lag}_k(R)$ by
$${\rm Sp}(2k,R) \times {\rm Lag}_k(R) \to {\rm Lag}_k(R)~;~(\alpha,L) \mapsto \alpha(L)~.$$
A symmetric form $(R^k,q)$ over $R$ determines an elementary symplectic matrix
$$E(q)~=~\begin{pmatrix} q & -1 \\ 1 & 0 \end{pmatrix} \in 
{\rm Sp}(2k,R)$$
with
$$E(q)(R^k \oplus \{0\})~=~\{(q(x),x)\,\vert\, x \in R^k\} \in {\rm Lag}_k(R)~. $$
Given an expression of $\alpha\in {\rm Sp}(2k,R)$ as a product of $n$ elementary symplectic matrices
$$\alpha~=~E(q_1)E(q_2) \dots E(q_n)  \in {\rm Sp}(2k,R)$$
there is defined  a symmetric form $(R^{2kn},{\rm Tri}(q))$ over $R$ with 
$${\rm Tri}(q)~=~\begin{pmatrix}
q_1 & 1 & 0 &\dots & 0\\
1 & q_2 & 1 &  \dots & 0 \\
0 &1 & q_3 & \dots & 0\\
\vdots & \vdots &\vdots & \ddots & \vdots \\
0 & 0 & 0 & \dots & q_n \end{pmatrix}$$
a generalized tridiagonal symmetric matrix. An element
$\alpha \in {\rm Sp}(2k,R[X])$ can be regarded as an algebraic path
$[0,1] \to {\rm Lag}_k(R)$ from
$\alpha(0)(R^k \oplus \{0\})$ to $\alpha(1)(R^k \oplus \{0\})$.
Let 
$$\begin{array}{l}
\Omega {\rm Lag}_k(R)\\[1ex]
=~\{\alpha \in {\rm Sp}(2k,R[X])\,\vert\, 
\alpha(1)(R^k \oplus \{0\})=\alpha(0)(R^k \oplus \{0\}) \in
{\rm Lag}_k(R)\}~,
\end{array}$$
the pointed set of algebraic loops in ${\rm Lag}_k(R)$.  For a regular ring $R$
(i.e. noetherian of finite global homological dimension)  with $1/2 \in R$ it is proved in \cite{bargelannes}  that for sufficiently large $k,n$ every $\alpha \in \Omega{\rm Lag}_k(R)$ has  an expression  
$$\alpha~=~E(q_1)E(q_2) \dots E(q_n)  \in {\rm Sp}(2k,R[X])$$
with $(R[X]^k,q_i)$ symmetric forms over $R[X]$ such that 
$(R^{2nk},{\rm Tri}(q)(0))$, \allowbreak
$(R^{2kn},{\rm Tri}(q)(1)) \in
{\rm Sym}_{2nk}(R)$.  It is then proved that the Maslov index map
$$\begin{array}{l}
\Omega\text{Lag}_k(R)\to \text{Sym}_{2nk}(R)~;\\[1ex]
\alpha \mapsto \text{Maslov}(\alpha)=(R^{nk},\text{Tri}(q)(1)) \oplus (R^{nk},-\text{Tri}(q)(0))
\end{array}$$
induces the algebraic Bott periodicity isomorphism
$$\varinjlim_k \pi_1(\text{Lag}_k(R)) ~\cong~ \varinjlim_\ell \pi_0(\text{Sym}_\ell(R))$$
with
$$\text{Sym}_{\ell}(R) \to \text{Sym}_{\ell+1}(R)~;~(R^{\ell},\phi)
\mapsto (R^{\ell} \oplus R,\phi \oplus 1)~.$$
(ii) The case $k=1$, $R=\RR$ of (i) is particularly relevant. Every 1-dimensional subspace $L \subset \RR \oplus \RR$ is a lagrangian in $H_-(\RR)$. The function
$$S^1 \to \text{Lag}_1(\RR)~=~\RR\,\PP^1~;~
e^{2\pi i x} \mapsto \{(\cos \pi x , \sin \pi x )\}$$
is a diffeomorphism, such that the image of $S^1 \backslash \{1\}\cong \RR$ is the contractible subspace
$$\text{Lag}_1(\RR)_0~=~\text{Lag}_1(\RR)\backslash \{ \{0\}\oplus \RR\} \subset \text{Lag}_1(\RR)~.$$
For generic degree $n$ $P(X) \in \RR[X]$ with $0,1 \in \RR$ regular the algebraic path
$\alpha=E(Q_1(X))E(Q_2(X))\dots E(Q_n(X)) \in \text{Sp}(2,\RR[X])$ given by the Sturm sequence corresponds to the actual path
$$\alpha~:~[0,1]\to \text{Lag}_1(\RR)~;~ x \mapsto \{(P(x),P'(x))\}$$
with $\alpha(0),\alpha(1) \in \text{Lag}_1(\RR)_0$ since $P(0),P(1) \neq 0$.
Use the congruence $\tau(\RR^\ell,\phi) \equiv \ell (\bmod \,2)$ and the Sylvester Law of Inertia to define the bijection
$$\varinjlim\limits_\ell\pi_0(\text{Sym}_\ell(\RR)) \to \ZZ^+~;~(\RR^{\ell},\phi)
 \mapsto \tau_-(\RR^{\ell},\phi)= (\ell-\tau(\RR^{\ell},\phi))/2$$
sending the base point  $(\RR^{\ell},1)$ to 0. The degree of the actual loop
$$\begin{array}{l}
[\alpha]~:~S^1~=~[0,1]/\{0,1\} \to \text{Lag}_1(\RR)/\text{Lag}_1(\RR)_0~\simeq~S^1~;\\[1ex]
\hskip100pt e^{2\pi i x} \mapsto \{(P(x),P'(x))\}~ (0 \leqslant x \leqslant 1)
\end{array}$$
is computed by
$$\begin{array}{ll}
\text{degree}([\alpha])&=~\vert [\alpha]^{-1}(\{0\} \oplus \RR)\vert\\[1ex]
&=~\hbox{number of real roots of $P(X)$ in $[0,1]$}\\[1ex]
&=~\text{var}(P_0(0),P_1(0),\dots,P_n(0)) -\text{var}(P_0(1),P_1(1),\dots,P_n(1))\\[1ex]
&=~\big(\tau(\text{Tri}(Q)(1))-\tau(\text{Tri}(Q)(0))\big)/2\\[1ex]
&\hskip50pt \hbox{(by  the algorithm of Sturm's Theorem \ref{sturmtheorem})}\\[1ex]
&=~{\rm Maslov}(\alpha)  \in W(\RR)~=~\ZZ~.
\end{array}$$
\end{remark}

\subsection{The localization exact sequence}\label{localize}

Let $R$ be a ring with involution, and let $S \subset R$ be a multiplicative subset of central non-zero divisors such that $1\in S$ and $\overline{S}=S$.
The {\it localization of $R$ inverting $S$} is the ring with involution
$$S^{-1}R~=~S \times R/\{(s,r) \sim (s',r') \,\vert\, rs'=r's \in R\} $$
with the equivalence class of $(r,s)$ denoted $r/s$.
The natural morphism of rings with involution $i:r \in R \mapsto r/1 \in S^{-1}R$ is injective. An $R$-module morphism $f:M \to N$ is an {\it $S$-isomorphism} if it induces an $S^{-1}R$-module isomorphism $f:S^{-1}M=S^{-1}R\otimes_RM \to S^{-1}N$.

The localization exact sequence of Ranicki~\cite[Chapter III]{ranickiexact} (and many other authors: Milnor and Husemoller, Karoubi, Pardon, Vogel, $\dots$)
$$\xymatrix@C-5pt{\dots \ar[r]& L^1(R,S)  \ar[r] & L^0(R) \ar[r]^-{i} & L^0(S^{-1}R)
\ar[r]^-{\partial} & L^0(R,S) \ar[r] & L^{-1}(R) \ar[r]& \dots}$$
is defined - there is also a version for quadratic $L$-theory.
The 0-dimensional $L$-group $L^0(S^{-1}R)=W(S^{-1}R)$ can be identified with the Witt group of symmetric forms $(V,\phi)$ over $R$ which are {\it $S$-nonsingular}, i.e. 
such that $\phi:V \to V^*$ is an $S$-isomorphism, or equivalently such that $S^{-1}(V,\phi)=S^{-1}R\otimes_R(V,\phi)$ is nonsingular over $S^{-1}R$. The $L$-group $L^0(R,S)=W(R,S)$ is the Witt group of nonsingular symmetric {\it linking forms over $(R,S)$}, pairs $(T,\lambda)$ with $T$ an $S$-torsion  f.g. $R$-module of homological dimension 1, and $\lambda:T \times T \to S^{-1}R/R$ a nonsingular symmetric pairing. Thus $T$ is required to have a f.g. projective $R$-module resolution
$$\xymatrix{0 \ar[r] & F_1 \ar[r]^-{d}& F_0 \ar[r] & T  \ar[r] & 0}$$
with $d$ an  $S$-isomorphism, so that $S^{-1}T=0$.

Every nonsingular symmetric form $(W,\Phi)$ over $S^{-1}R$ is isomorphic to $S^{-1}(V,\phi)$ for some $S$-nonsingular symmetric form $(V,\phi)$ over $R$. The {\it boundary} nonsingular symmetric linking form over $(R,S)$ is
$$(T,\lambda)=({\rm coker}(\phi:V \to V^*), (x,y) \mapsto x(\phi^{-1}(y)))~,$$
with 
$$\partial~:~W(S^{-1}R) \to W(R,S)~;~(W,\Phi)=S^{-1}(V,\phi) \mapsto 
\partial (V,\phi)~.$$
 Note that if $S^{-1}(V,\phi)$ is isomorphic to $S^{-1}(V',\phi')$ then the
boundary linking forms $\partial(V,\phi)$, $\partial (V',\phi')$ are only Witt-equivalent, and not in general isomorphic.

\begin{example} For any $(R,S)$ let $(V,\phi)=(R,s)$ with $\overline{s}=s \in S$. Then
$$\partial S^{-1}(V,\phi)~=~(T,\lambda)~=~(R/sR,1/s) \in W(R,S)$$
with 
$$\lambda~=~1/s~:~T \times T~=~R/sR \times R/sR \to S^{-1}R/R~;~(x,y) \mapsto
\overline{x}y/s.$$
\end{example}

\begin{example} In the special case $i:R=\ZZ \to S^{-1}R=\QQ$ the
boundary map $\partial:(W,\Phi) \mapsto (T,\lambda)$ associates to a nonsingular symmetric form $(W,\Phi)$ over $\QQ$ a nonsingular symmetric linking form $(T,\lambda)$, as follows. First choose a lattice $V \subset W$ (i.e. a f.g. free $\ZZ$-submodule such that $\QQ\otimes_{\ZZ}V=W$ and $\Phi(V \times V) \subseteq \ZZ$), define the finite abelian group $V^{\#}/V$ with the dual lattice
$$V^{\#}= \{w \in W\,\vert\, \Phi(V,w) \subseteq \ZZ \subset \QQ\} \subset W$$
such that there is defined a $\ZZ$-module isomorphism
$$V^{\#} \to V^*={\rm Hom}_\ZZ(V,\ZZ)~;~w_1 \mapsto (w_2 \mapsto \Phi(w_1,w_2))~.$$
The boundary linking form is given up to isomorphism by
$$(T,\lambda)~=~\partial (V,\phi)~=~(V^{\#}/V,\lambda)$$
with
$$\lambda~:~V^{\#}/V \times V^{\#}/V \to \QQ/\ZZ~;~
 (w_1,w_2) \mapsto \Phi(w_1,w_2)~.$$
\end{example}
\begin{example} Let $(R,S)=(\ZZ,\ZZ\backslash \{0\})$, so that $S^{-1}R=\QQ$ as in the previous example.\\
1. If $N$ is a $(2k-2)$-connected $(4k-1)$-dimensional
$\QQ$-homology sphere there is defined a nonsingular symmetric linking form $(H_{2k-1}(N),\lambda_N)$ over $(\ZZ,\ZZ\backslash \{0\})$
with $H_{2k-1}(N)$ a finite abelian group and
$$\lambda_N: (x,y) \in H_{2k-1}(N) \times H_{2k-1}(N) \mapsto z(x)/n \in \QQ/\ZZ$$
with $z \in C^{2k-1}(N)$ , $n \geqslant 1$ such that $ny=d^*z \cap [N]\in C_{2k-1}(N)$. \\
2. If $(M,\partial M)$ is a $(2k-1)$-connected $4k$-dimensional
manifold with boundary such that $\partial M$ is a  $(2k-2)$-connected $(4k-1)$-dimensional $\QQ$-homology sphere
(as in 1.), then
$$(H_{2k}(M;\QQ),\phi_M)=S^{-1}(H_{2k}(M),\phi_M)$$
is a nonsingular symmetric form over $\QQ$ with boundary the
nonsingular symmetric linking form  over $(\ZZ,\ZZ\backslash \{0\})$
$$\partial(H_{2k}(M;\QQ),\phi_M)=(H_{2k-1}(\partial M),\lambda_{\partial M}).$$
\end{example}

\begin{example} \label{lenslinking}
As before, let $(R,S)=(\ZZ,\ZZ\backslash \{0\})$, with $S^{-1}R=\QQ$.\\
1. The symmetric  linking form of the lens space 
$$L(c,a)=\SSS^1 \times \DDD^2 \cup_{\begin{pmatrix} a& b \\ c & d\end{pmatrix}} \SSS^1 \times \DDD^2~(c \neq 0)$$
is
$$(H_1(L(c,a)),\lambda)~=~(\ZZ/c\ZZ,a/c)$$
with
$$\lambda~:~(x,y) \in \ZZ/c\ZZ \times \ZZ/c\ZZ \mapsto axy/c \in \QQ/\ZZ.$$
The effect of the surgery on half a Heegaard decomposition
$\SSS^1 \times \DDD^2\subset L(c,a)$ is
 $$\SSS^1 \times \DDD^2\cup_{\begin{pmatrix} -c& -d \\ a & b\end{pmatrix}}\SSS^1 \times \DDD^2=L(a,-c)=-L(a,c)$$
where $-L(a,c)$ denotes $L(a,c)$ with the opposite orientation. The trace of the surgery is the  oriented cobordism
$(M(a,c);L(c,a),-L(a,c))$ with
$$(M(a,c),\partial M(a,c))~=~(L(a,c) \times I \cup \DDD^2 \times \DDD^2,L(c,a) \cup L(a,c))~.$$
(See Ranicki \cite[Chapter 2]{ranickisurgery} for the language of geometric surgery, such as the "effect" and "trace".)
The symmetric intersection form of $M(a,c)$ is the $S$-nonsingular symmetric form over $\ZZ$ 
$$(H_2(M(a,c)),\phi)~=~(\ZZ,ac)~.$$
In view of the isomorphism of symmetric linking forms
$$(\ZZ/c\ZZ,a/c)\oplus (\ZZ/a\ZZ,c/a) \xymatrix{\ar[r]^-{\cong}&} (\ZZ/ac\ZZ,1/ac)~;~
(x,y) \mapsto  ax+cy$$
(with inverse $z \mapsto (dz,-bz)$) the boundary symmetric linking form over $(\ZZ,S)$ is
$$\begin{array}{ll}
\partial (H_2(M(a,c)),\phi)&=~(H_1(L(a,c)),\lambda) \oplus (H_1(L(c,a)),\lambda)\\[1ex]
&=~(\ZZ/ac,1/ac)~=~(\ZZ/c\ZZ,a/c)\oplus (\ZZ/a\ZZ,c/a)~.
\end{array}$$
Thus for any coprime $p_0,p_1 \in S$ 
and $q_1 \in \ZZ$ there is defined an oriented cobordism $(M(p_0,p_1);L(p_0,p_1),L(p_1,p_2))$ with
$p_2=p_1q_1-p_0$ and 
$$\begin{array}{l}
(H_2(M(p_0,p_1)),\phi)=(\ZZ,p_0p_1)~,~L(p_1,p_2)=L(p_1,-p_0)=-L(p_1,p_0)~,\\[1ex]
\partial(H_2(M(p_0,p_1)),\phi)=\partial(\ZZ,p_0p_1)=(\ZZ/p_0\ZZ,p_1/p_0) \oplus (\ZZ/p_1\ZZ,-p_2/p_1)~.
\end{array}$$
2. For any $p_0,p_1,p_2\in S$ with $p_0,p_1$ and $p_1,p_2$ coprime and
$$p_0+p_2~=~p_1q_1 \in \ZZ$$
for some $q_1\in \ZZ$ there is defined an oriented cobordism 
$$\begin{array}{l}
(M(p_0,p_1,p_2);L(p_0,p_1),
-L(p_2,p_1))\\[1ex]
=~(M(p_0,p_1);L(p_0,p_1),L(p_1,p_2)) \cup
(M(p_1,p_2);L(p_1,p_2),-L(p_2,p_1))
\end{array}$$ 
with 
$$L(p_1,p_2)~=~L(p_1,-p_0)~=~-L(p_1,p_0)~.$$
The symmetric intersection form of $M(p_0,p_1,p_2)$ is the $S$-nonsingular symmetric form over $\ZZ$ 
$$(H_2(M(p_0,p_1,p_2)),\phi)~=~(\ZZ \oplus \ZZ,\begin{pmatrix}
p_0p_1 & p_0 \\ p_0 & q_1 \end{pmatrix})~.$$
In view of the isomorphism of nonsingular symmetric forms over $\QQ$
$$\begin{pmatrix} 1 & -1 \\ 0 & p_1\end{pmatrix}~:~
(\QQ \oplus \QQ,\begin{pmatrix} p_0p_1 & 0 \\ 0 & p_1p_2 \end{pmatrix})
\xymatrix{\ar[r]^-{\cong}&}  \QQ\otimes_{\ZZ}(H_2(M(p_0,p_1,p_2)),\phi)$$
the signature is
$$\tau(H_2(M(p_0,p_1,p_2)),\phi)~=~{\rm sign}(p_0p_1)+{\rm sign}(p_1p_2)~.$$
The boundary map in the localization exact sequence for $\ZZ \subset \QQ$ is such that
$$\begin{array}{l}
\partial~:~W(\QQ) \to W(\ZZ,S);\\[1ex]
(H_2(M(p_0,p_1,p_2);\QQ),\phi) \mapsto (H_1(L(p_0,p_1)),\lambda) \oplus -(H_1(L(p_2,p_1)),\lambda)~.
\end{array}
$$   
3. More generally, for any sequences $p=(p_0,p_1,\dots,p_n)\in S^{n+1}$, $q=(q_1,q_2,\dots,q_n)\in \ZZ^n$ such that
$$p_{k-1},p_k~\hbox{are coprime, and}~p_{k-1}+p_{k+1}=p_kq_k~(1 \leqslant k \leqslant n)$$
there is defined an oriented cobordism
$$(M(p);L(p_0,p_1),L(p_n,p_{n+1}))~=~\bigcup\limits^n_{k=1}
(M(p_{k-1},p_k);L(p_{k-1},p_k),L(p_k,p_{k+1}))$$
such that
$$\begin{array}{l}
\partial:W(\QQ) \to W(\ZZ,S);\\[1ex]
(H_2(M(p);\QQ),\phi)=\bigoplus\limits^n_{k=1}(\QQ,p_{k-1}p_k) 
\mapsto (H_1(L(p_0,p_1)),\lambda)-(H_1(L(p_n,p_{n+1})),\lambda)\\[1ex]
\hphantom{(H_2(M(p);\QQ),\phi)=\bigoplus\limits^n_{k=1}(\QQ,p_{k-1}p_k) 
\mapsto}=(\ZZ/p_0\ZZ,p_1/p_0)-(\ZZ/p_n\ZZ,p_{n+1}/p_n).
\end{array}$$
If  $p_n=1$, $p_{n+1}=0$ then $L(p_n,p_{n+1})=L(1,0)=\SSS^3$ and
$$M~=~M(p)\cup_{\SSS^3}\DDD^4~=~M(q_1,q_2,\dots,q_n)$$ 
is the 4-dimensional
graph manifold (\ref{graphmanifold} and  section \ref{lens}) obtained by the $A_n$-plumbing weighted by $q_n,q_{n-1},\dots,q_1$
starting from $\DDD^4$, with boundary $\partial M = L(p_0,p_1)$, $\ZZ$-coefficient intersection form
$$(H_2(M),\phi)~=~(\ZZ^n,{\rm Tri}(q_1,q_2,\dots,q_n)),$$
$\QQ$-coefficient intersection form
$$(H_2(M;\QQ),\phi)~=~\bigoplus\limits^n_{k=1}(\QQ,p_{k-1}p_k)~,$$
and boundary linking form
$$\partial (H_2(M;\QQ),\phi)~=~(H_1(L(p_0,p_1)),\lambda)~=~(\ZZ/p_0\ZZ,p_1/p_0)~.$$
The signature is
$$\begin{array}{ll}
\tau(M)&=~\sum\limits^n_{k=1}{\rm sign}(p_k/p_{k-1})~(\hbox{from~$\tau(M(p_{k-1},p_k))={\rm sign}(p_{k-1}p_k)$})\\[1ex]
&=~\tau({\rm Tri}(q_1,q_2,\dots,q_n))~=~\sum\limits^n_{k=1}{\rm sign}(p^*_k/p^*_{k-1})\\[1ex]
&\hskip25pt (\hbox{from Theorem \ref{sjgf}, with $p^*_k={\rm det}({\rm Tri}(q_1,q_2,\dots,q_k)))$}
\end{array}$$
giving a topological proof of Sylvester's Duality Theorem~\ref{sylvestertheorem}.
See Proposition \ref{trilinking} below for an abstract version of the construction.
\end{example}

We shall now develop an abstract version for any $(R,S)$ of the relationship between Sturm functions and the Witt group localization exact sequence
of $(\RR[X],\RR[X]\backslash \{0\})$ obtained in 
section \ref{order}.

\begin{proposition} \label{trilinking}
1. For coprime $s,t \in S$ such that $\bar{s}=s$, $\bar{t}=t$ there is defined an isomorphism of nonsingular symmetric linking forms over $(R,S)$
$$(R/sR,t/s) \oplus (R/tR,s/t)~\cong~(R/stR,1/st)~=~\partial(S^{-1}R,st)~,$$
and in the Witt group of $(R,S)$
$$\begin{array}{ll}
(R/sR,t/s) \oplus (R/tR,s/t)&=~(R/stR,1/st)\\[1ex]
&=~(R/stR,s/t)~=~(R/stR,t/s) \in W(R,S)~.
\end{array}$$
2. If $(V',\phi')$ is the $S$-nonsingular symmetric form over $R$ obtained from an $S$-nonsingular symmetric form $(V,\phi)$ 
over $R$ by plumbing
$$(V',\phi')~=~(V \oplus R,\begin{pmatrix} \phi & v^* \\ v & w
\end{pmatrix})$$
with $v \in V^*$, $w=\overline{w} \in R$ such that
$w-v \phi^{-1} v^* \in S$, then
$$\partial S^{-1}(V',\phi')~=~\partial S^{-1}(V,\phi)
\oplus \partial (S^{-1}R,w-v\phi^{-1}v^*) \in W(R,S)~.$$
3. Let $\chi=(\chi_1,\chi_2,\dots,\chi_n) \in R^n$, $\mu=(\mu_0,\mu_1,\dots,\mu_n) \in S^{n+1}$ be such that $\bar{\chi}_k=\chi_k$, $\bar{\mu}_k=\mu_k$ and
$$\mu_k+\mu_{k-2}~=~\chi_k \mu_{k-1}~( 1\leqslant k \leqslant n,~\mu_0=1,~\mu_{-1}=0)$$
so that each $(R^k,{\rm Tri}(\chi_1,\chi_2,\dots,\chi_k))$ is an 
$S$-nonsingular symmetric form over $R$,
with 
$$\mu_k~=~{\rm det}({\rm Tri}(\chi_1,\chi_2,\dots,\chi_k)) \in S~,$$
such that up to isomorphism of symmetric forms over $S^{-1}R$
$$(S^{-1}R^n,{\rm Tri}(\chi))~=~\bigoplus\limits^n_{k=1}(S^{-1}R,\mu_k/\mu_{k-1})~.$$
The pairs  $\mu_{k-1},\mu_k \in S$ 
$(0 \leqslant k \leqslant n)$ are coprime, and for $k=1,2,\dots,n$ up to isomorphism of 
symmetric linking forms over $(R,S)$
$$\begin{array}{ll}
\partial(S^{-1}R,\mu_k/\mu_{k-1})&=~
(R/\mu_{k-1}R,\mu_k/\mu_{k-1})\oplus (R/\mu_kR,\mu_{k-1}/\mu_k)~(by~1.)\\[1ex]
&=~
(R/\mu_{k-1}R,-\mu_{k-2}/\mu_{k-1})\oplus (R/\mu_kR,\mu_{k-1}/\mu_k)~.
\end{array}$$
Thus in the Witt group of symmetric linking forms over $(R,S)$
$$\begin{array}{ll}
\partial (S^{-1}R^n,{\rm Tri}(\chi))&=~
\sum\limits^n_{k=1}\big((R/\mu_{k-1}R,-\mu_{k-2}/\mu_{k-1})\oplus (R/\mu_kR,\mu_{k-1}/\mu_k)\big)\\[1ex]
&=~(R/\mu_nR,\mu_{n-1}/\mu_n)  \in W(R,S)~.
\end{array}$$
\end{proposition}
\begin{proof} 1.   Let $a,b \in R$ be such that $as+bt=1 \in R$. The $R$-module morphism
$$f~:~R/stR \to R/sR \oplus R/tR~;~w \mapsto (bw,aw)$$
has inverse
$$f^{-1}~:~R/sR \oplus R/tR \to R/stR~;~(x,y) \mapsto tx+sy$$
and defines an isomorphism of nonsingular symmetric linking forms over $(R,S)$
$$f~:~(R/stR,1/st) \xymatrix{\ar[r]^-{\cong}&} (R/sR,t/s) \oplus (R/tR,s/t)~.$$
It follows from the isomorphism of nonsingular symmetric forms over $S^{-1}R$ 
$$s~:~(S^{-1}R,s/t) \to (S^{-1}R,1/st)$$
that $\partial(S^{-1}R,s/t)=\partial (S^{-1}R,1/st) \in W(R,S)$, and similarly with the roles of $s,t$ reversed.\\
2. Immediate from the isomorphism of nonsingular symmetric forms over $S^{-1}R$
$$\begin{pmatrix} 1 & \phi^{-1}v^* \\ 0 & 1 \end{pmatrix}~:~
S^{-1}(V',\phi') \to S^{-1}(V,\phi) \oplus (S^{-1}R,w-v\phi^{-1}v^*)~.$$
3.  Consider the plumbing (Definition \ref{plumbing})
$$\begin{array}{ll}
(R^k,{\rm Tri}(\chi_1,\chi_2,\dots,\chi_k))&=~(V',\phi')~=~
(V \oplus R,\begin{pmatrix} \phi & v^* \\ v & w \end{pmatrix})\\[1ex]
&=~(R^{k-1}\oplus R,\begin{pmatrix} {\rm Tri}(\chi_1,\chi_2,\dots,\chi_{k-1}) & \begin{pmatrix} 0 \\ \vdots  \\ 0 \\ 1 \end{pmatrix}\\
\begin{pmatrix} 0 & \dots & 0 & 1 \end{pmatrix} & \chi_k \end{pmatrix})
\end{array}
$$
with the proof of Proposition \ref{signtri0} giving
$$w-v\phi^{-1}v^*~=~\chi_k-\mu_{k-2}/\mu_{k-1}~=~\mu_k/\mu_{k-1}~.$$
By 2., up to isomorphism of symmetric forms over $S^{-1}R$
$$\begin{array}{l}
S^{-1}(R^k,{\rm Tri}(\chi_1,\chi_2,\dots,\chi_k))\\[1ex]
\hskip25pt
=~S^{-1}(R^{k-1},{\rm Tri}(\chi_1,\chi_2,\dots,\chi_{k-1})) \oplus (S^{-1}R,\mu_k/\mu_{k-1})
\end{array}$$
and a fortiori in the Witt group $W(S^{-1}R)$. The ring elements $a_k,b_k \in R$ defined by
$$\begin{array}{l}
a_k~=~\begin{cases} 0&{\rm if}~k=0\\
1&{\rm if}~k=1\\
{\rm det}({\rm Tri}(\chi_2,\chi_3,\dots,\chi_k))&{\rm if}~2 \leqslant k \leqslant n~,
\end{cases}\\[5ex]
b_k~=~\begin{cases} 1&{\rm if}~k=0\\
0&{\rm if}~k=1\\
-1&{\rm if}~k=2\\
{\rm det}({\rm Tri}(\chi_2,\chi_3,\dots,\chi_{k-1}))&{\rm if}~3 \leqslant k \leqslant n
\end{cases}
\end{array}
$$
are such that
$$a_k\mu_{k-1}+b_k\mu_k~=~1~.$$
Thus $\mu_{k-1},\mu_k \in S$ are coprime, and by 1.
$$\begin{array}{ll}
\partial (S^{-1}R,\mu_{k-1}\mu_k)&=~
(R/(\mu_{k-1}\mu_k),1/\mu_{k-1}\mu_k)\\[1ex]
&=~
(R/(\mu_{k-1}),\mu_k/\mu_{k-1}) \oplus (R/(\mu_k),\mu_{k-1}/\mu_k)
\in W(R,S)~.
\end{array}$$
Finally, note that it follows from
$$\mu_k/\mu_{k-1}+\mu_{k-2}/\mu_{k-1}=\chi_k \in R \subset S^{-1}R$$
that there is an identity of linking forms
$$(R/(\mu_{k-1}),\mu_k/\mu_{k-1})~=~(R/(\mu_{k-1}),-\mu_{k-2}/\mu_{k-1})~(2 \leqslant k \leqslant n)~.$$
\end{proof} 

Let $R$ be a Dedekind ring with involution, and $S=R \backslash \{0\}\subset R$, so that $S^{-1}R=K$ is the field of fractions. See Ranicki~\cite[Chapter 4]{ranickiexact},  \cite[Prop. 38.5]{ranickiknot} for a detailed account of the
Witt group localization exact sequence 
$$\xymatrix@C-10pt{L^1(R,S)=0 \ar[r] & W(R) \ar[r]^-{i} & W(K) \ar[r]^-{\partial} & W(R,S)  \ar[r] & L^{-1}(R) \ar[r] & L^{-1}(K)=0}
$$
with $L^{-1}(R)=\widetilde{K}_0(R)/2\widetilde{K}_0(R)$, $\widetilde{K}_0(R)$ the ideal class group (= 0 for a principal ideal domain $R$). An $S$-torsion $R$-module $T$ of homological dimension 1 splits as a
direct sum
$$T=\bigoplus\limits_{\pi=\bar{\pi} \triangleleft R\,\text{maximal}} T_{\pi}$$
with $\pi$ the involution-invariant maximal (= non-zero prime) ideals of $R$, and
$T_{\pi}=\{x \in T\,\vert\, \pi^kx=0~\text{for~some}~k \geqslant 0\}$. Similarly for a linking form $(T,\lambda)$. The morphism
$$\bigoplus\limits_{\pi=\bar{\pi} \triangleleft R\,\text{maximal}}W(R/\pi) \to W(R,S)~;~(R/\pi,q) \mapsto (R/\pi,q/p)$$
is an isomorphism, with $R/\pi$ the residue field and  $p \in \pi$ a uniformizer (e.g. a generator for a principal ideal), and the inverse given by d\'evissage.  The composite
$W(R/\pi) \to W(R,S) \to L^{-1}(R)$ sends $(R/\pi,q)$ to 
$$(H_-(R);R \oplus 0,\text{im}(\begin{pmatrix}
q \\ p\end{pmatrix}:R \to R \oplus R))~=~[\pi] \in
L^{-1}(R)~=~\widetilde{K}_0(R)/2\widetilde{K}_0(R)~.$$
If there exists a Sturm sequence $p_*=(p_0,p_1,\dots,p_n)$, $q_*=(q_1,q_2,\dots,q_n)$ in $R$ with 
$$p_{k-1}+p_{k+1}=p_kq_k~,~
p_0=p,~p_1=q,~p_n \in R^{\bullet},~p_{n+1}=0$$ 
(e.g. if $R$ is a Euclidean domain) then
$$\begin{array}{l}
(H_-(R);R \oplus 0,\text{im}(\begin{pmatrix}
q \\ p\end{pmatrix}:R \to R \oplus R))~=~0 \in L^{-1}(R)~,\\[1ex]
\partial (K^n,\text{Tri}(q_*))~=~(R/\pi,q/p)\\[1ex]
\in \text{ker}(W(R,S)\to L^{-1}(R))~=~\text{im}(\partial:W(K) \to W(R,S))~,\\[1ex]
(K^n,\text{Tri}(q_*))~=~\bigoplus\limits^n_{k=1}(K,p_{k-1}/p_k) \in W(K)~.
\end{array}$$
The localizations of $R$ away from $\pi$ and at $\pi$ are the Dedekind rings with
involution
$$\pi^{-1}R~=~R[p^{-1}]~,~R_{(\pi)}~=~(R\backslash\pi)^{-1}R$$
which fit into a cartesian square
$$\xymatrix{R \ar[r] \ar[d] & R_{(\pi)}\ar[d] \\
\pi^{-1}R \ar[r] & K}$$
The $\pi$-adic completions of $R$ and $K$ are Dedekind rings with involution
$$\widehat{R}_\pi~=~\varprojlim_k R/\pi^k~,~
\widehat{K}_\pi~=~\pi^{-1}\widehat{R}_\pi~=~\widehat{R}_\pi[p^{-1}]$$
which fit into a cartesian square
$$\xymatrix{R_{(\pi)}\ar[d] \ar[r] & \widehat{R}_\pi \ar[d]\\
K \ar[r] &\widehat{K}_\pi}~.$$
Thus there are isomorphisms of abelian groups
$$\begin{array}{l}
R/\pi~\cong~R_{(\pi)}/pR_{(\pi)}~\cong~\widehat{R}_{\pi}/p\widehat{R}_{\pi}~,\\[1ex]
\varinjlim\limits_k R/p^kR~=~\pi^{-1}R/R~\cong~K/R_{(\pi)}~\cong~\widehat{K}_\pi/\widehat{R}_\pi
\end{array}$$
and excision isomorphisms of the torsion $L$-groups
$$L^*(R,\pi^{\infty})~\cong~L^*(R_{(\pi)},\pi^{\infty})~\cong~L^*(\widehat{R}_\pi,\pi^{\infty})~.$$
The residue map is the composite
$$\begin{array}{l}
\text{res}_\pi~:~W(K) \to W(\widehat{K}_\pi) \to W(\widehat{R}_\pi,\pi^{\infty})~=~W(R/\pi)~;\\[1ex]
\hskip20pt
(K,p^iq) \mapsto 
\begin{cases}
0&\text{if $i\equiv 0(\bmod\,2)$}\\
(R/\pi,q)&\text{if $i\equiv 1(\bmod\,2)$}
\end{cases}~([q] \in (R/\pi)^{\bullet})~.
\end{array}
$$

\begin{example} (i) Let $K$ be a field of characteristic $\neq 2$, and let $(R,S)=(K[X],K[X]\backslash \{0\})$, with localization the function field
$S^{-1}K[X]=K(X)$. The maximal ideals of $K[X]$ are the principal ideals $(\pi) \triangleleft K[X]$ with $\pi \in K[X]$ irreducible monic, with an injection
$$K[X]/\pi\to \pi^{-1}K[X]/K[X]=\varinjlim_k K[X]/\pi^k~;~P(X) \mapsto P(X)/\pi~.$$
The local ring $K_\pi[X]=K[X]_{(\pi)}$ has the unique maximal ideal
$(\pi) \triangleleft K_{\pi}[X]$, with $K_\pi[X]/(\pi)=K[X]/(\pi)$ the residue field.
The $\pi$-adic completion of $K[X]$
$$\widehat{K[X]}_\pi~=~\varprojlim_k K_\pi[X]/(\pi^k)~=~\varprojlim_k K[X]/(\pi^k)$$
has field of fractions the $\pi$-adic completion $\pi^{-1}\widehat{K[X]}_\pi=K_\pi(X)$ of $K(X)$ which 
has already appeared in section \ref{order}, along with the isomorphisms
$$W(K[X],S)~=~\bigoplus\limits_\pi W(K_\pi[X],\pi)~=
\bigoplus\limits_\pi W(\widehat{K[X]}_\pi,\pi)~=~\bigoplus \limits_\pi W(K[X]/\pi)$$
and the residue maps ${\rm res}_{\pi}:W(K(X)) \to W(K_\pi(X))\to W(K[X]/\pi)$.\\
(ii) Let $K=\RR$. The maximal ideals $(X-x)$, $(X-z)(X-\overline{z})\triangleleft \RR[X]$ are given by $x \in \RR$ and $z \in {\cal H}$. Now
$$\RR[X]/(X-x)~=~\RR~,~\RR[X]/(X-z)(X-\overline{z})~=~\CC$$
so that
$$W(\RR[X],S)~=~W(\RR)[\RR] \oplus W(\CC)[{\cal H}]~=~ 
\ZZ[\RR] \oplus \ZZ/2\ZZ[{\cal H}]~.$$
\end{example}

Here are some examples of how the Witt group localization exact sequence has appeared in this paper:

\begin{itemize}
\item The computation of $W(\QQ)$ in section~\ref{order} fits into the localization exact sequence
$$\xymatrix{0 \ar[r] & W(\ZZ) \ar[r]^-{i} & W(\QQ) \ar[r]^-{\partial} & 
W(\ZZ,S) \ar[r] & 0}$$
with $S=\ZZ\backslash\{0\} \subset \ZZ$ and $W(\ZZ,S)=\bigoplus\limits_{\infty} \ZZ/2\ZZ\oplus \bigoplus\limits_{\infty} \ZZ/4\ZZ$ the Witt group of nonsingular symmetric linking pairings $(T,\lambda:T \times T \to \QQ/\ZZ)$ on finite abelian groups.
In particular, this is the exact sequence used by Milnor and Husemoller~\cite{milnorhusemoller} to recover the computation that the signature map $\tau:W(\ZZ) \to \ZZ$ is an isomorphism (Serre~\cite{serrearith}).
\item For any field $K$ let $(R,S)=(K[X],K[X]\backslash \{0\})$ with
$S^{-1}R=K(X)$. The computation of $W(K(X))$ in section~\ref{order} for  $K$ of characteristic $\neq 2$ fits into the localization exact sequence
$$\xymatrix@C-5pt{0 \ar[r] & W(K[X]) \ar[r]^-{i} & W(K(X)) \ar[r]^-{\partial}& W(K[X],S)=\bigoplus\limits_\pi W(K[X]/\pi) \ar[r] & 0}$$
with $\pi\in K[X]$ running over all the irreducible monic polynomials - see Ranicki \cite[39A]{ranickiknot}. By a theorem of Karoubi~\cite{karoubi} (Ojanguren~\cite{ojanguren}) $W(K[X])=W(K)$.
By \cite[Prop. 39.2]{ranickiknot}
$W(K[X],S)$ is the Witt group of triples $(V,\phi,C)$ with $(V,\phi)$ a nonsingular symmetric form over $K$ and $C:V \to V$ a $K$-linear map 
with characteristic polynomial $\det(X I_n-C) \in S$ ($n=\dim_K(V)$), and $\phi(Cv,w)=\phi(v,Cw)$; such an object is the same as a nonsingular symmetric linking form $(K[X],S)$, and $\partial$ is split by
$$W(K[X],S) \to W(K(X))~;~(V,\phi,C) \mapsto (V(X),\phi(X-C))~.$$
. For any regular degree $n$ polynomial $P(X) \in K[X]$ the Witt class $(K(X),P(X)) \in W(K(X))$ has image
$$\partial (K(X),P(X))=(K^n,S(P),C(P)) \in W(K[X],S)$$
with $S(P)$ the Jacobi-Hermite matrix of Theorem~\ref{hermite} and $C(P)$ the companion matrix, such that $\det(X I_n-C(P))=P(X) \in S$. 
See \cite[Example 39.9]{ranickiknot} for the computation
$$W(\RR(X))=\ZZ \oplus \ZZ[\RR] \oplus \ZZ/2\ZZ[{\mathcal H}]$$ 
in section~\ref{order} from this point of view.
\item Let $R$ be an integral domain (with the identity involution) and let $K=S^{-1}R$ be the field of fractions, with $S=R\backslash \{0\}$. Given sequences $(p_0,p_1,\dots,p_n) \in S^{n+1}$, $(q_1,q_2,\dots,q_n) \in R^n$ satisfying the modified Euclidean algorithm (cf. Example \ref{euclid2})
with $p_n={\rm gcd}(p_0,p_1)\in S$ and recurrences
$$p_{k-1}+p_{k+1}=p_kq_k~(1 \leqslant k \leqslant n),~p_{n+1}=0$$
define
$$(\chi_1,\chi_2,\dots,\chi_n)=(q_n,q_{n-1},\dots,q_1)~,~
(\mu_0,\mu_1,\dots,\mu_n)=(p_n,p_{n-1},\dots,p_0)$$
so that as in Proposition \ref{trilinking}
$$\mu_k+\mu_{k-2}~=~\chi_k \mu_{k-1}~( 1\leqslant k \leqslant n,~\mu_0=1,~\mu_{-1}=0)$$
Define the tridiagonal $S$-nonsingular symmetric form $(R^n,{\rm Tri}(q))$ over $R$ with
$${\rm Tri}(q)=\begin{pmatrix} q_1 & 1 & 0 & \dots & 0 \\
1 &q_2 & 1 & \dots & 0 \\
0 & 1 & q_3 & \dots  & 0 \\
\vdots & \vdots & \vdots & \ddots & \vdots \\
0 & 0 & 0 & \dots & q_n \end{pmatrix}.$$
such that ${\rm det}({\rm Tri}(q))=p_0/p_n$ (Proposition \ref{tri2}). The invertible $n \times n$ matrix in $K$
 $$A~=~ \begin{pmatrix} 
1 & 0 & 0 & \dots & 0 \\
-p_2/p_1 & 1 & 0 & \dots &0\\ 
p_3/p_1 & -p_3/p_2 & 1 & \dots & 0\\
\vdots & \vdots & \vdots & \ddots & \vdots \\
(-1)^{n-1}p_n/p_1 & (-)^{n-2}p_n/p_2 & (-)^{n-3}p_n/p_3 &\dots & 1 
\end{pmatrix}$$ 
is such that
$$A^*{\rm Tri}(q)A~=~
\begin{pmatrix} p_0/p_1 & 0 & \dots & 0 \\
0 & p_1/p_2  &  \dots & 0\\
\vdots & \vdots &  \ddots & \vdots \\
0& 0 & \dots & p_{n-1}/p_n\end{pmatrix}~,$$
defining an isomorphism of symmetric forms over $K$
$$(K^n,{\rm Tri}(q))~\cong~\bigoplus\limits^n_{k=1}(K,p_{k-1}/p_k)~,$$ 
with ${\rm det}({\rm Tri}(q))=p_0/p_n \in K$. By Proposition \ref{trilinking} 3.
$$\begin{array}{ll}
\partial(K,p_{k-1}/p_k)&=
(R/p_{k-1}R,p_k/p_{k-1})\oplus (R/p_kR,p_{k-1}/p_k)\\[1ex]
&=(R/p_{k-1}R,p_k/p_{k-1})\oplus (R/p_kR,-p_{k+1}/p_k)
\end{array}$$
up to isomorphism of linking forms, and in the Witt group
$$\begin{array}{ll}
\partial (K^n,{\rm Tri}(q))&=\bigoplus\limits^n_{k=1}
((R/p_{k-1}R,p_k/p_{k-1}) \oplus (R/p_kR,-p_{k+1}/p_k))\\[1ex]
&=(R/p_0R,p_1/p_0)  \in W(R,S).
\end{array}$$
\item Let $(R,S)=(\ZZ,\ZZ\backslash \{0\})$. For any
$\begin{pmatrix} a & b \\ c
& d \end{pmatrix} \in {\rm SL}(2,\ZZ)$
we have that $a,c \in \ZZ$ are coprime (assume $c>a>0$ for simplicity), so the modified Euclidean algorithm gives
sequences $(p_0,p_1,\dots,p_n) \in (\ZZ\backslash \{0\})^{n+1}$, $(q_1,q_2,\dots,q_n) \in \ZZ^n$ as above, with 
$$p_0=c~,~p_1=a~,~p_{k-1}+p_{k+1}=p_kq_k~(1 \leqslant k \leqslant n,~p_n=1,~p_{n+1}=0)~.$$ 
${\rm Tri}(q)$ is the symmetric intersection matrix over $\ZZ$ of the 4-dimensional manifold $M$ constructed by $n$ plumbings along the $A_n$-tree weighted by $q$ (Example \ref{lenslinking}), with determinant 
${\rm det}({\rm Tri}(q))=c$ and signature
$$\tau({\rm Tri}(q))~=~n-2\,{\rm var}(p_0,p_1,\dots,p_n)~=~\sum\limits^n_{k=1}{\rm sign}(p_{k-1}/p_k) \in \ZZ~.$$
The boundary of $M$ is the 3-dimensional lens space $\partial M=L(c,a)$ 
(as in section~\ref{lens}), with linking form 
$$(H_1(L(c,a)),\lambda_{L(c,a)} )~=~(\ZZ/p_0\ZZ,p_1/p_0)~=~(\ZZ/c\ZZ,a/c)~.$$  
\item Let $(R,S)=(\RR[X],\RR[X]\backslash \{0\})$. For  a 
regular degree $n+2m$ polynomial $P(X) \in \RR[X]$  with real roots
$x_1,x_2,\dots,x_n \in \RR$ and $2m$ complex roots
$\{z_1,z_2,\dots,z_m\} \cup \{\bar{z}_1,\bar{z}_2,\dots,\bar{z}_m\}
\in {\cal H} \cup \overline{\cal H}$, let $(\RR[X]^{n+2m},{\rm Tri}(Q(X)))$
be the $S$-nonsingular tridiagonal symmetric form over $\RR[X]$ given by the Sturm quotients. By the above and Theorem \ref{nexttolast}
$$\begin{array}{l}
\partial (\RR(X)^{n+2m},{\rm Tri}(Q(X)))\\[1ex]
=~(\RR[X]/P(X)\RR[X],P'(X)/P(X))~
=~(\sum\limits^n_{k=1}1.x_k,\sum\limits^m_{j=1}1.z_j)\\[1ex]
\in {\rm coker}(W(\RR[X]) \to W(\RR(X)))~=~W(\RR[X],S)~=~\ZZ[\RR] \oplus\ZZ/2\ZZ[{\cal H}]
\end{array}$$
counting the roots of $P(X)$.
\item Let the Laurent polynomial extension $\ZZ[z,z^{-1}]$ of $\ZZ$ have the involution $\bar{z}=z^{-1}$, and  let $P \subset \ZZ[z,z^{-1}]$ be the multiplicative subset of the polynomials $p(z) \in \ZZ[z,z^{-1}]$ such that $\overline{p(z)}=\pm z^ap(z^{-1})\in \ZZ[z,z^{-1}]$ for some $a \in \ZZ$. A finite f.g. free $\ZZ[z,z^{-1}]$-module chain complex
$C$ is such that $H_*(P^{-1}C)=0$ if and only if $H_*(\ZZ\otimes_{\ZZ[z,z^{-1}]}C)=0$ (i.e. $C$ is ``$\ZZ$-acyclic'').
 The complement $(X,\partial X)$ of an $n$-knot $k:{\mathbb S}^n \subset {\mathbb S}^{n+2}$ is equipped with a degree 1 map
$$f:(X,\partial  X) \to {\mathbb S}^1 \times (\DDD^{n+3},k(\SSS^n))$$
which induces isomorphisms $f_*:H_*(X) \cong H_*(\SSS^1)$, and
such that 
$$f_*~:~\pi_1(X) \to \pi_1(\SSS^1)=\ZZ~;~(u:\SSS^1 \subset X) \mapsto {\rm lk}(u,k)$$
is the canonical surjection. The chain complex kernel of $f$ is a $\ZZ$-acyclic $(n+2)$-dimensional symmetric Poincar\'e complex $(C,\phi)$ over $\ZZ[z,z^{-1}]$ with
$H_*(C)=H_{*+1}(\overline{f}:\overline{X} \to \RR \times \DDD^{n+3})$ the reduced homology $\ZZ[z,z^{-1}]$-modules of the infinite cyclic cover
$\overline{X}=f^*\RR$ of $X$.
The morphism defined on the cobordism group $C_n$ of $n$-knots
$$k \in C_n \mapsto (C,\phi) \in L^{n+3}(\ZZ[z,z^{-1}],P)$$
is a surjection  for $n=1$ and an isomorphism for $n \geqslant   2$.
For odd $n=2i-1\geqslant 3$ the knot cobordism class
$$(k:\SSS^{2i-1} \subset \SSS^{2i+1})=(C,\phi)\in 
C_{2i-1}=L^{2i+2}(\ZZ[z,z^{-1}],P)$$
 is the Witt class of the $(-1)^{i+1}$-symmetric Blanchfield linking form $(H_i(\overline{X}),\lambda)$ over $(\ZZ[z,z^{-1}],P)$ 
(Kearton \cite{kearton}, Ranicki~\cite[Chapter 7.9]{ranickiexact}). The discussion of the knot signature jumps in section~\ref{modern} can be expressed  in terms of the computation of $W(\RR(z))$ of the rational function field $\RR(z)$ with involution $\overline{z}=z^{-1}$ (cf. Ranicki~\cite[Chapter 39]{ranickiknot}), paralleling the formulation in section~\ref{order} and in the previous point of the theorems of Sturm and Sylvester in terms of the computation of  $W(\RR(X))$ ($\overline{X}=X$).
\end{itemize}

\newpage

\bibliographystyle{acm}

\bibliography{list}

\end{document}